\pdfoutput=1
\documentclass[pdflatex,sn-mathphys-num]{sn-jnl}

\usepackage{xcolor}               
\usepackage{array}                
\usepackage{booktabs}             
\usepackage{multirow}             
\usepackage{makecell}          
\usepackage{adjustbox}       

\usepackage{amsmath,amssymb,amsfonts}
\usepackage{amsthm}
\usepackage[title]{appendix}
\usepackage{algorithm}
\usepackage{algpseudocode}
\usepackage{algorithmicx}
\usepackage{bbm}
\usepackage{caption}
\usepackage{cleveref}
\usepackage{enumitem}
\usepackage{graphicx}
\usepackage{listings}
\usepackage{mathrsfs}
\usepackage{manyfoot}
\usepackage{placeins}
\usepackage{subcaption}
\usepackage{textcomp}
\usepackage{xparse}
\usepackage{ytableau}
\usepackage{tikz}              
\usepackage{tikzpagenodes}      

\usetikzlibrary{positioning}     
\usetikzlibrary{arrows.meta}     
\usetikzlibrary{calc,math}      

\numberwithin{equation}{section}  

\theoremstyle{thmstyleone}       
\newtheorem{theorem}{Theorem}[section] 
\newtheorem{lemma}[theorem]{Lemma}     
\newtheorem{proposition}[theorem]{Proposition}

\theoremstyle{thmstyletwo}        
\newtheorem{example}{Example}[section]

\theoremstyle{thmstylethree}      
\newtheorem{definition}{Definition}[section]  

\theoremstyle{thmstyleone}

\theoremstyle{thmstyletwo}      
\newtheorem{Algorithm}{Algorithm}[section]    

\newcommand\scalemath[2]{\scalebox{#1}{\mbox{\ensuremath{\displaystyle #2}}}}  
\newcommand{\Gr}[2]{\mathrm{Gr}(#1,#2)}

\newcommand{\x}{\mathbf{x}}      
    
\newcommand{\y}{\mathbf{y}}

\NewDocumentCommand{\hg}{ O{} O{} m m }{
	\path let
	\p1 = ($(#4)-(#3)$),
	\n1 = {atan2(\y1,\x1)}
	in
	coordinate (#3#4left1)  at ($(#3)!1/3!(#4) + ({\n1 + 90}:0.15)$)
	coordinate (#3#4right1) at ($(#3)!1/3!(#4) + ({\n1 - 90}:0.15)$)
	coordinate (#3#4left2)  at ($(#3)!2/3!(#4) + ({\n1 + 90}:0.15)$)
	coordinate (#3#4right2) at ($(#3)!2/3!(#4) + ({\n1 - 90}:0.15)$);
	
	\draw
	[\IfNoValueF{#1}{#1}]
	(#3) .. controls (#3#4left1) and (#3#4right2) .. (#4);
	
	\draw
	[\IfNoValueF{#2}{#2}]
	(#3) .. controls (#3#4right1) and (#3#4left2) .. (#4);
}

\newcommand{\Disk}[1]{
	\begin{tikzpicture}
		\draw (0,0) circle (1);
		\foreach \k in {1,...,12} {
			\coordinate (\k) at (120-30*\k:1);
		}
		\foreach \k in {1,2,...,12} {
			\fill (\k) circle (2pt);
		}
		#1
	\end{tikzpicture}
}

\newcommand{\GrDisk}[1]{
	\begin{tikzpicture}
		\draw (0,0) circle (1);
		\foreach \k in {1,...,8} {
			\coordinate (\k) at (135-45*\k:1);
		}
		\foreach \k in {1,2,...,8} {
			\fill (\k) circle (2pt);
		}
		#1
	\end{tikzpicture}
}

\newcommand{\Diska}[1]{
	\GrDisk{
		#1
		\fill[white] (1) circle (2pt);
		\draw (1) circle (2pt);
		\fill[white] (2) circle (2pt);
		\draw (2) circle (2pt);
		\fill[white] (3) circle (2pt);
		\draw (3) circle (2pt);
		\fill[white] (4) circle (2pt);
		\draw (4) circle (2pt);
	}
}

\newcommand{\Diskb}[1]{
	\GrDisk{
		#1
		\fill[white] (1) circle (2pt);
		\draw (1) circle (2pt);
		\fill[white] (2) circle (2pt);
		\draw (2) circle (2pt);
		\fill[white] (3) circle (2pt);
		\draw (3) circle (2pt);
		\fill[white] (5) circle (2pt);
		\draw (5) circle (2pt);
	}
}

\newcommand{\Diskc}[1]{
	\GrDisk{
		#1
		\fill[white] (1) circle (2pt);
		\draw (1) circle (2pt);
		\fill[white] (2) circle (2pt);
		\draw (2) circle (2pt);
		\fill[white] (3) circle (2pt);
		\draw (3) circle (2pt);
		\fill[white] (6) circle (2pt);
		\draw (6) circle (2pt);
	}
}

\newcommand{\Diskd}[1]{
	\GrDisk{
		#1
		\fill[white] (1) circle (2pt);
		\draw (1) circle (2pt);
		\fill[white] (2) circle (2pt);
		\draw (2) circle (2pt);
		\fill[white] (4) circle (2pt);
		\draw (4) circle (2pt);
		\fill[white] (5) circle (2pt);
		\draw (5) circle (2pt);
	}
}

\newcommand{\Diske}[1]{
	\GrDisk{
		#1
		\fill[white] (1) circle (2pt);
		\draw (1) circle (2pt);
		\fill[white] (2) circle (2pt);
		\draw (2) circle (2pt);
		\fill[white] (4) circle (2pt);
		\draw (4) circle (2pt);
		\fill[white] (6) circle (2pt);
		\draw (6) circle (2pt);
	}
}

\newcommand{\Diskf}[1]{
	\GrDisk{
		#1
		\fill[white] (1) circle (2pt);
		\draw (1) circle (2pt);
		\fill[white] (2) circle (2pt);
		\draw (2) circle (2pt);
		\fill[white] (4) circle (2pt);
		\draw (4) circle (2pt);
		\fill[white] (7) circle (2pt);
		\draw (7) circle (2pt);
	}
}

\newcommand{\Diskg}[1]{
	\GrDisk{
		#1
		\fill[white] (1) circle (2pt);
		\draw (1) circle (2pt);
		\fill[white] (2) circle (2pt);
		\draw (2) circle (2pt);
		\fill[white] (5) circle (2pt);
		\draw (5) circle (2pt);
		\fill[white] (6) circle (2pt);
		\draw (6) circle (2pt);
	}
}

\newcommand{\Diskh}[1]{
	\GrDisk{
		#1
		\fill[white] (1) circle (2pt);
		\draw (1) circle (2pt);
		\fill[white] (3) circle (2pt);
		\draw (3) circle (2pt);
		\fill[white] (5) circle (2pt);
		\draw (5) circle (2pt);
		\fill[white] (7) circle (2pt);
		\draw (7) circle (2pt);
	}
}

\begin{document}

\title[Web diagrams of cluster variables for Gr(4,8)]{Web Diagrams of Cluster Variables for Grassmannian Gr(4,8)}


\author[1]{\fnm{Wen Ting} \sur{Zhang}}\email{zhangwt@lzu.edu.cn}

\author*[2]{\fnm{Rui Zhi} \sur{Tang}}\email{tangrzh01@gmail.com}

\author[2]{\fnm{Jin Xing} \sur{Zhao}}\email{zhjxing@imu.edu.cn}

\affil[1]{\orgdiv{School of Mathematics and Statistics}, \orgname{Lanzhou University}, \orgaddress{ \city{Lanzhou}, \country{P. R. China}}}

\affil*[2]{\orgdiv{School of Mathematical Sciences}, \orgname{Inner Mongolia University}, \orgaddress{\city{Hohhot}, \country{P. R. China}}}

\abstract{Gaetz, Pechenik, Pfannerer, Striker, and Swanson introduced the concept of hourglass plabic graphs and provided a method for computing web diagrams and invariants corresponding to $4\times n$ Young tableaux, while Elkin, Musiker, and Wright applied Lam's method to explicitly compute the webs compatible with cluster variables in $\Gr{3}{n}$ and their twists, namely, the preimages of the immanant map introduced by Fraser, Lam, and Le. In this paper, we use these two methods to compute both the web diagrams and the dual webs corresponding to quadratic and cubic cluster variables in the Grassmannian cluster algebra $\mathbb{C}[\Gr{4}{8}]$.}




\maketitle

\section{Introduction}\label{sec:Int}

In \cite{fomin_cluster_2002} of Fomin and Zelevinsky, the definition of cluster algebra is introduced to consider a dual canonical basis and to understand and compute the canonical basis of quantum groups through this, over time, it emerged as a central theme in algebraic combinatorics, establishing deep connections with numerous other areas of mathematics. Scott's seminal work \cite{scott_grassmannians_2006} laid the groundwork for the cluster structure on the Grassmannian, introduced these structures and established that cluster algebras can exhibit finite type under certain conditions. Postnikov advanced the combinatorial study of these algebras in \cite{postnikov_total_2006} by introducing plabic graphs as a key tool for analyzing their structure. He also developed the boundary measurement map, which links Pl\"ucker coordinates to dimers on plabic graphs, thereby parametrizing the Grassmannian as a projective variety. Talaska later deepened the understanding of this map in \cite{talaska_formula_2008}. Berenstein, Fomin, and Zelevinsky in \cite{berenstein_parametrizations_1996} originally introduced the concept of the twist map and its connection to the Laurent expansions of Pl\"ucker coordinates. Building on this result, in \cite{marsh_twists_2016} of Marsh and Scott, and in \cite{muller_twist_2016} of Muller and Speyer, the authors respectively studied these Laurent expansions of Pl\"ucker coordinates through detailed analysis and new developments.

In \cite{kuperberg_spiders_1996}, Kuperberg introduced the notion of $sl_r$-webs, provided a web basis for cases $r=2$ and 3, and further proposed the concept of web invariants in $U_q(\hat{g})$-modules. Fomin and Pylyavskyy in \cite{FOMIN2016717} subsequently developed computational methods for tensor diagrams. Chang, Dual, Fraser, and Li \cite{chang_quantum_2020} combined this approach with Young tableaux representations of cluster monomials. Later, Gaetz, Pechenik, Pfannerer, Striker and Swanson generalized the methodology and established a rotation-invariant web basis for $sl_4$-webs in \cite{gaetz2023rotation}. In a separate direction, Lam introduced compatibility relations to construct Temperley-Lieb invariants for Pl\"ucker polynomials in \cite{lam_dimers_2015}. Fraser, Lam, and Le \cite{fraser_dimers_2019} further investigated connections between dimer models and web invariants, defining the immanant map that links Pl\"ucker polynomials with dual tensor diagrams. More recently, in \cite{elkin_twists_2023}, Elkin, Musiker, and Wright applied this mapping to compute images of Pl\"ucker polynomials under the twist map.

In this paper, we primarily study the quadratic and cubic cluster variables in $\mathbb{C}[\Gr{4}{8}]$. Specifically, to study the 120 quadratic and 174 cubic cluster variables in $\mathbb{C}[\Gr{4}{8}]$, we choose 3 quadratic and 14 cubic cluster variables as representatives, and the others can be obtained by dihedral translates or changing the labels of isolated vertices. The rotation-invariant $U_q(sl_4)$-web basis and the growth algorithm introduced in \cite{gaetz2023rotation} provide a computational method for web invariants of quadratic and cubic cluster variables in $\mathbb{C}[\mathrm{Gr}(k,n)]$, and represent the web diagrams as hourglass graphs. Additionally, we apply the method of compatibility mentioned in \cite{lam_dimers_2015} to determine the non-crossing matchings and non-elliptic webs compatible with quadratic and cubic cluster variables. According to the definition in \cite[Section 4.2]{lam_dimers_2015}, non-elliptic webs compatible with a certain Pl\"ucker polynomial are precisely those obtained by reducing triple dimer configurations that share the same boundary conditions as the polynomial. Consequently, the method of compatibility requires enumeration of all non-elliptic webs, so we refine the method for enumerating non-elliptic webs in \cite{elkin_twists_2023} and fully classify all non-elliptic webs with 4 black and 4 white boundary vertices. In \cite{fraser_dimers_2019}, it is proven that the webs obtained by compatibility conditions are exactly the dual webs given by the preimage of the immanant map.

In Section \ref{sec:Pre}, some necessary definitions, including cluster algebras, plabic graphs, webs and compatibility, are given. In Section \ref{sec:Inv}, we detail the method to compute the web diagrams using the growth algorithm introduced in \cite{gaetz2023rotation}. For a semi-standard Young tableau, its lattice word is constructed by taking the row indices of entries labeled $1,2,\ldots, n$ in the tableau as the sequence of letters forming the word. The algorithm constructs a boundary line and suspends downward-labeled segments corresponding to the letters in the lattice word. Following the growth rules illustrated in Figure \ref{fig:growthrule}, these segments are iteratively crossing and merged, ultimately yielding an hourglass graph, the web diagram associated with the Young tableau. According to \cite{chang_quantum_2020}, cluster variables can be associated with Young tableaux; therefore, we apply the above method to compute the web diagrams corresponding to quadratic and cubic cluster variables in the Grassmannian $\mathbb{C}[\Gr{4}{8}]$.

In Section \ref{sec:Inm}, we consider the compatibility condition introduced in \cite{lam_dimers_2015}, which enables the computation of webs compatible with cluster variables. Based on \cite{fraser_dimers_2019}, the resulting webs obtained through this approach are precisely the dual webs of the web diagrams. These dual webs play an important role in computing twists of cluster variables as demonstrated in \cite{elkin_twists_2023}. To construct such dual webs, we utilize the enumeration procedure in \cite{elkin_twists_2023} for generating all non-elliptic webs with 12 black boundary vertices corresponding to cluster variables in $\mathbb{C}[\Gr{4}{12}]$. Subsequently, we apply the contraction operation to sink vertices following \cite{russell_explicit_2013}, reducing them to non-elliptic webs with 4 black and 4 white boundary vertices that correspond to cluster variables in $\mathbb{C}[\Gr{4}{8}]$. Finally, by implementing Lam's compatibility framework \cite{lam_dimers_2015}, we explicitly determine all compatible webs for quadratic and cubic cluster variables in $\mathbb{C}[\Gr{4}{8}]$.	

\section{Preliminaries}\label{sec:Pre}
For $m, n \in \mathbb{Z}$, if $m < n$, denote the set $\{m, m+1,\ldots, n\}$ by $[m,n]$, by $[n]$ the set $[1,n]$, and by $\pm[n]$ the set $\{\pm 1, \pm 2, \ldots, \pm n\}$. For $k \in \mathbb{Z}$, if $0 \le k \le n$, we denote the set of all subsets with $k$ elements of $[n]$ by $\binom{[n]}{k}$.

\subsection{Grassmannian and the cluster structure}\label{ssc:PreGra}
In this subsection, we briefly review the structure of the Grassmannian and the definition of cluster algebra. 

Denote by $\Gr{k}{n}\subset \mathbb{P}^{\binom{n}{k}-1}$ the Grassmannian consists of all $k$-planes in $\mathbb{C}^n$, with its Pl\"ucker embedding in projective space. The point in $\Gr{k}{n}$ can be represented by a matrix $M \in \operatorname{Mat}_{k \times n}(\mathbb{C})$, and the embedding is defined as follows: For any $J \in \binom{[n]}{k}$, there is a corresponding projective called \textit{Pl\"ucker coordinate} $P_J$, and in the point $M$, the value of $P_J(M)$ is defined as the maximal minor of $M$ using the column set $J$.

Denote by $\mathbb{C}[\Gr{k}{n}]$ the homogeneous coordinate ring on $\Gr{k}{n}$, which is the affine cone over $\Gr{k}{n}$. In \cite{scott_grassmannians_2006}, Scott shows that $\mathbb{C}[\Gr{k}{n}]$ is a cluster algebra in the sense of \cite{fomin_cluster_2002}.

Now, we give a definition of cluster algebra.

\begin{definition}[\cite{fomin_cluster_2002}]\label{def:clualg}
	A \textit{quiver} $Q = (Q_0, Q_1, s, t)$ is a finite directed graph with no loops or $2$-cycles, where $Q_0$ is the set of vertices, $Q_1$ is the set of arrows, and $s, t: Q_1 \to Q_0$ are the maps assigning each arrow its source and target. The vertex set $Q_0$ is indexed by $[n]$, where the first $m$ vertices are mutable and the remaining $n-m$ vertices are frozen.
	
	For any $k \in [m]$, the \textit{mutated quiver} $\mu_k(Q)$ is formed by:
	\begin{itemize}
		\item adding an arrow $i \rightarrow j$ for each 2-length path $i \rightarrow k \rightarrow j$,
		\item reversing all arrows with source or target $k$,
		\item removing arrows that belong to a maximal collection of disjoint $2$-cycles.
	\end{itemize}
	
	Let $\mathcal{F}$ be a field isomorphic to a field of rational functions in $m$ independent variables. A \textit{seed} in $\mathcal{F}$ is a pair $(\mathbf{x}, Q)$, where $\mathbf{x} = (x_1, \ldots, x_m)$ is a free generating set of $\mathcal{F}$ and $Q$ is a quiver. The elements of $\mathbf{x}$ are \textit{cluster variables}, and $x_{m+1}, \ldots, x_n$ are \textit{frozen variables}.
	
	Given a seed $(\mathbf{x}, Q)$ and $k \in [m]$, the \textit{mutated seed} $\mu_k(\mathbf{x}, Q)$ is $(\mathbf{x}', \mu_k(Q))$, where $\mathbf{x}' = (x'_1, \ldots, x'_m)$ and $x'_k$ is given by
	\[
	x_k' x_k = \prod_{s(\alpha) = k} x_{t(\alpha)} + \prod_{t(\alpha) = k} x_{s(\alpha)},
	\]
	with $x'_j = x_j$ for all $j \neq k$.
	
	Starting from an initial seed, a seed is \textit{reachable} if it can be obtained by a sequence of mutations. The \textit{cluster algebra} is the $\mathbb{C}$-algebra generated by all cluster and frozen variables in all reachable seeds. Cluster variables are \textit{compatible} if two cluster variables appear together in a common cluster, and \textit{cluster monomials} are products of compatible cluster variables.
\end{definition}

For cluster algebra $\mathbb{C}[\Gr{k}{n}]$, it was proposed in \cite{scott_grassmannians_2006} that only $\Gr{2}{n}$, $\Gr{n-2}{n}$, $\Gr{3}{6}$, $\Gr{3}{7}$, $\Gr{4}{7}$, $\Gr{3}{8}$ and $\Gr{5}{8}$ are of finite type.

Taking $\Gr{3}{6}$ as an example, besides the rank $1$ cluster variables, which are precisely the monomial forms of Pl\"ucker coordinates, there also exist quadratic cluster variables.

Denote the two different quadratic cluster variables in $\mathbb{C}[\Gr{3}{6}]$ by $X=\operatorname{det}(v_1\times v_2 \hspace{0.3cm} v_3\times v_4 \hspace{0.3cm} v_5\times v_6)$ and $Y=\operatorname{det}(v_6\times v_1 \hspace{0.3cm} v_2\times v_3 \hspace{0.3cm} v_4\times v_5)$, where $v_i$ denotes by the vector of the $i$th column of the point $M \in \Gr{3}{6}$. From the relation

\begin{equation*}
	\operatorname{det}(u\hspace{0.2cm} v\hspace{0.2cm} w) = u \cdot (v \times w) = (u \times v) \cdot w
\end{equation*}
and
\begin{equation*}
	(u \times v) \cdot(w \times z)=
	\left| \begin{array}{cc}
		u\cdot w & u\cdot z \\
		v\cdot w & v\cdot z
	\end{array} \right| 
\end{equation*}
the cluster variables can be expressed as polynomials, for example, $X$ can also be expressed as quadratic difference $P_{134}P_{256}-P_{156}P_{234}$, $P_{124}P_{356}-P_{123}P_{456}$ or $P_{125}P_{346}-P_{126}P_{345}$. Similarly, in $\mathbb{C}[\Gr{3}{8}]$, we can find $24$ cubic cluster variables. In other Grassmannian cluster algebras, there are also higher rank cluster variables.

We count the occurrences of numbers in the subscript indices of each Pl\"ucker monomial appearing in a Pl\"ucker polynomial and record the multiplicity of each number as a sequence $\lambda = (\lambda_1, \lambda_2, \ldots, \lambda_n)$, where $\lambda_i$ denotes the number of occurrence of $i$. This sequence $\lambda$ is defined to be the \textit{boundary condition} of the Pl\"ucker polynomial. In particular, if $\lambda_i=0$, we call $i$ an \textit{isolated vertex} of the Pl\"ucker polynomial. In general, there is a $\mathbb{Z}^n$-grading of $\mathbb{C}[\Gr{k}{n}]$ given by the boundary condition $\lambda$, where
\begin{equation*}
	\mathbb{C}[\Gr{k}{n}]=\bigoplus_{r=1}^{\infty}\bigoplus_{\sum \lambda_i=rk}\mathbb{C}[\Gr{k}{n}]_{\lambda}
\end{equation*}

In this paper, we focus on the cluster algebra $\mathbb{C}[\operatorname{Gr}(4,8)]$, and study the quadratic and cubic cluster variables. From the computation in \cite{cheung_clustering_2022}, there are 120 quadratic and 174 cubic cluster variables in $\mathbb{C}[\operatorname{Gr}(4,8)]$. Explicitly, since $\mathbb{C}[\operatorname{Gr}(4,7)]$ is isomorphic to $\mathbb{C}[\operatorname{Gr}(3,7)]$, which contains 14 quadratic cluster variables and no higher-rank cluster variables, among the 120 quadratic cluster variables in $\mathbb{C}[\operatorname{Gr}(4,8)]$, $\binom{8}{7} \times 14=112$ of them originate from the cluster variables in $\mathbb{C}[\operatorname{Gr}(4,7)]$. Therefore, we select two cluster variables with $\lambda_1=2$ and $\lambda_8=0$ as representatives for computation, while the others are either dihedral translates of these two or involve relabeling of isolated points. For the remaining 8 quadratic cluster variables without isolated points, we choose one representative for computation, with the other seven being its dihedral translates. For the 174 cubic cluster variables, we classify them into 8 types based on their boundary conditions, and select a total of 14 distinct cluster variables as representatives. All other cubic cluster variables are dihedral translates of these 14 representatives.

\subsection{Plabic graph and dimer configuration}\label{ssc:PreDimers}

\begin{definition}[\cite{postnikov_total_2006}]\label{def:plabic}
	A \textit{plabic graph} $G$ is a planar bicolored graph embedded in a disk with $n$ boundary vertices, the degree of each vertex is $1$, labeled $1, \ldots, n$ in a clockwise order. The embedding of $G$ is \textit{proper}, meaning that edges do not cross, and every internal vertex is connected to at least one boundary vertex by a path.
\end{definition}

Clearly, all of plabic graphs are bipartite, and all boundary vertices are black. Refer to the set of vertices of $G$ as $V (G)$, the set of edges of $G$ as $E(G)$, and the set of faces of $G$ as $F(G)$.

\begin{definition}[\cite{postnikov_total_2006}]\label{def:zigzag}
	A path $i \rightarrow \cdots \rightarrow j$ ($i$, $j$ are boundary vertices) is called a \textit{zigzag path} (or \textit{Postnikov strand}), if 
	\begin{itemize}
		\item we reach a black vertex, turn right to the maximum extent towards the next vertex.
		\item we reach a white vertex, turn left to the maximum extent towards the next vertex.
	\end{itemize}
\end{definition}

A zigzag path is labeled with its terminal $j$, therefore denoted by $p_j$. For every face $f$ in $F(G)$, if $f$ is located at the left of $p_j$, then add the label $j$ to the face $f$. By this method, after labeling the faces on the left side of all the zigzag paths in this manner, each face will carry exactly $k$ labels in the plabic graph of $\Gr{k}{n}$.

\begin{definition}[\cite{postnikov_total_2006}]\label{def:dimer}
	A \textit{dimer configuration} (also called an \textit{almost perfect matching}) $D$ on a plabic graph $G$ is a subset of edges of $G$ such that
	\begin{itemize}
		\item each interior vertex of G is adjacent to exactly one edge in D,
		\item each boundary vertex of G is adjacent to either no edges or one edge in D.
	\end{itemize}
\end{definition}
The set of boundary vertices adjacent to one edge in $D$ is denoted by $\partial(D)$, which is also called the \textit{boundary condition} of $D$, for convenience. Sometimes we denote it as a sequence $\lambda=(\lambda_1,\ldots,\lambda_n)$, where $\lambda_i$ corresponds to the number of times the $i$-th boundary vertex appears in the boundary conditions. This notation corresponds to the boundary condition of Pl\"ucker polynomials, although it is a slight abuse of notation, it is unambiguous.

We denote by $\mathcal{D}(G)$ the set of all dimer configurations in plabic graph $G$, and $\mathcal{D}_J(G)$ the set of all dimer configurations whose boundary condition is exactly $J \subset [n]$.

The multiset of the edges in $G$ such that every interior vertex has degree exactly $m$ is called an $m$-fold dimer configuration, that is, it is also a superimposition of $m$ single dimer configurations of $G$.  We denote by $\mathcal{D}^m(G)$ the set of $m$-fold dimer configurations of $G$.

\subsection{Web invariant, $sl_r$-webs and web basis}\label{ssc:PreInv}

The quantum group $U_q(\mathfrak{sl}_r)$ is a $\mathbb{C}(q)$-algebra that arises from the theory of quantum integrable systems; see \cite{lusztig2010introduction} for a definition by generators and relations. The quantum group $U_q(\mathfrak{sl}_r)$ deforms the universal enveloping algebra $U(\mathfrak{sl}_r)$, so the classical theory of $\mathfrak{sl}_r$ is recovered when $q = 1$. 

For a standard $U_q(\mathfrak{sl}_r)$-module $V_q$ with standard $\mathbb{C}_q$-basis $v_1, \ldots, v_r$, its \textit{quantum exterior algebra} $\bigwedge_q^{c} V_q$ is a $U_q(\mathfrak{sl}_r)$-module (see \cite{berenstein2008braided} and \cite{Cautis_2014}). For the generators $v_i$ and $v_j$, the quantum exterior product satisfies the $q$-commutation relations
\begin{equation*}
	v_i \wedge_q v_j \;=\;
	\begin{cases}
		(-q)v_j \wedge_q v_i, & i<j,\\
		0, & i=j.\\
	\end{cases}
\end{equation*}
For $I=\{i_1> \cdots > i_c\}$, denote by $v_I$ the ordered quantum exterior product $v_{i_1}\wedge_q \cdots \wedge_q v_{i_c}$. The quantum exterior algebra $\bigwedge_q^{c} V_q$ has $\mathbb{C}_q$-basis $\{v_I:I\in\binom{[r]}{c}\}$.

The linear dual $(\bigwedge_q^{c} V_q)^\ast$ is also a $U_q(\mathfrak{sl}_r)$-module with dual basis $\{v_I^{\ast}\}$. For convenience, denoted by $\bigwedge_q^{-c} V_q$ the dual $(\bigwedge_q^{c} V_q)^\ast$, and by $v_{\overline{I}}$ its basis $\{v_I^{\ast}\}$.

\begin{definition}\label{def:wedge}
	Given a type $\underline{c} = (c_1, \ldots, c_n)$ where each $c_j\in \pm[r]$, let
	\begin{equation*}
		\bigwedge\nolimits_q^{\underline{c}} V_q = \bigwedge\nolimits_q^{c_1} V_q \otimes \cdots \otimes \bigwedge\nolimits_q^{c_n} V_q.
	\end{equation*}
\end{definition}
Elements in $H=\operatorname{Hom}_{U_q(\mathfrak{sl}_r)}(\bigwedge_q^{\underline{c}} V_q,\bigwedge_q^{\underline{d}} V_q)$ can be obtained by composition and tensoring from	the natural product map as well as natural $U_q(\mathfrak{sl}_r)$-module maps arising from coproducts, duals, evaluation, coevaluation, and the identity. A \textit{tensor invariant} is an element in 
\begin{equation*}
	\mathcal{W}_{\underline{c}}(U_q(\mathfrak{sl}_r))=\operatorname{Inv}_{U_q(\mathfrak{sl}_r)}(\bigwedge\nolimits_q^{\underline{c}} V_q)=\operatorname{Hom}_{U_q(\mathfrak{sl}_r)}(\bigwedge\nolimits_q^{\underline{c}} V_q,\mathbb{C}(q)).
\end{equation*}
In fact, the tensor invariant is denoted by $[W]_q$, where $W$ is a particular way of encoding tensor invariants by using planar diagrams called \textit{web diagrams}. The definition of webs were first introduced by Kuperberg in \cite{kuperberg_spiders_1996}. In \cite{morrison2007diagrammatic}, the notion of a tag was introduced and used in \cite{Cautis_2014}. In this paper, we use the definition in \cite{fraser_dimers_2019} with tags which are summarized in \cite{gaetz2023rotation}.

\begin{definition}\label{def:tagwebs}
	\cite[Definition 2.3]{gaetz2023rotation} A $U_q(\mathfrak{sl}_r)$\textit{-web} (or just \textit{web}) is a planar graph $W$ embedded in the disk such that
	\begin{itemize}
		\item $W$ is properly bicolored, with black and white vertices;
		\item all vertices on the boundary circle have degree 1;
		\item boundary vertices are labeled $1, 2, \ldots, n$ in clockwise order;
		\item all interior vertices have a special ``tag'' edge, which immediately terminates in the interior
		and not at a vertex;
		\item non-tag edges have multiplicities in $[r]$;
		\item all vertices in the interior of the disk have incident edge multiplicities summing to $r$.
	\end{itemize}
	These webs are defined up to planar isotopy fixing the boundary circle.
\end{definition}

In this paper, we focus on how to compute these diagrams refer to the methods in \cite{gaetz2023rotation}. For the full diagrammatic construction and its relationship with tensor invariants, see \cite{Cautis_2014}, \cite{kim2003graphical}, \cite{morrison2007diagrammatic}, \cite{fraser_dimers_2019}, and \cite{gaetz2023rotation} for details. In these papers, each vertex in a $U_q(\mathfrak{sl}_r)$-web corresponds to the relevant wedge, cowedge, or dual map, and the resulting equivalence classes of these webs under planar isotopy and certain specialized moves correspond precisely to tensor invariants as elements of $\operatorname{Inv}_{U_q(\mathfrak{sl}_r)}(\bigwedge\nolimits_q^{\underline{c}} V_q)$.

We now describe some special cases and their web bases applied in this paper. In Figure \ref{fig:web_basis}, examples for the cases of $r=2,3,4$ are given respectively.

\begin{figure}[h]
	\centering
	\begin{minipage}[t]{0.3\textwidth}
		\centering
		\begin{tikzpicture}
			\draw (0,0) circle (1);
			\foreach \k in {1,...,6} {
				\coordinate (\k) at (150-60*\k:1);
				\fill (\k) circle (2pt);
			}
			\draw (6) arc[start angle=-120, end angle=0, radius=0.58];
			\draw (2) arc[start angle=120, end angle=240, radius=0.58];
			\draw (4) arc[start angle=0, end angle=120, radius=0.58];
			\node[below] at (0,-1.2) {non-crossing matching};
		\end{tikzpicture}
	\end{minipage}
	\begin{minipage}[t]{0.3\textwidth}
		\centering
		\begin{tikzpicture}
			\draw (0,0) circle (1);
			\foreach \k in {1,...,9} {
				\coordinate (\k) at (130-40*\k:1);
				\fill (\k) circle (2pt);
			}
			\coordinate (A) at (70:0.7);
			\coordinate (B) at (-30:0.6);
			\coordinate (C) at (-130:0.7);
			\coordinate (D) at (150:0.7);
			
			\coordinate (a) at (150:0.1);
			
			\draw (1)--(A);		\draw (2)--(A);		\draw (3)--(B);		
			\draw (4)--(B);		\draw (5)--(B);		\draw (6)--(C);		
			\draw (7)--(C);		\draw (8)--(D);		\draw (9)--(D);
			\draw (a)--(A);		\draw (a)--(C);		\draw (a)--(D);
			
			\foreach \p in {A,B,C,D} {
				\fill[white] (\p) circle (2pt);
				\draw (\p) circle (2pt);
			}
			\fill (a) circle (2pt);
			\node[below] at (0,-1.2) {non-elliptic web};
		\end{tikzpicture}
	\end{minipage}
	\begin{minipage}[t]{0.3\textwidth}
		\centering
		\Disk{
			\coordinate (A) at (60:0.6);
			\coordinate (B) at (15:0.8);
			\coordinate (C) at (-30:0.6);
			\coordinate (D) at (-75:0.8);
			\coordinate (E) at (-120:0.6);
			\coordinate (F) at (-165:0.8);
			\coordinate (G) at (150:0.6);
			\coordinate (H) at (105:0.8);
			\coordinate (I) at (0:0);
			
			\coordinate (a) at (90:0.5);
			\coordinate (b) at (15:0.5);
			\coordinate (c) at (-60:0.5);
			\coordinate (d) at (-90:0.5);
			\coordinate (e) at (-165:0.5);
			\coordinate (f) at (120:0.5);
			
			\draw (1)--(H);		\draw (2)--(A);		\draw (3)--(B);		\draw (4)--(B);
			\draw (5)--(C);		\draw (6)--(D);		\draw (7)--(D);		\draw (8)--(E);
			\draw (9)--(F);		\draw (10)--(F);	\draw (11)--(G);	\draw (12)--(H);
			
			\draw (a)--(H);		\draw (a)--(I);		\hg{a}{A};
			\draw (b)--(A);		\draw (b)--(C);		\hg{b}{B};
			\draw (c)--(D);		\draw (c)--(I);		\hg{c}{C};
			\draw (d)--(D);		\draw (d)--(I);		\hg{d}{E};
			\draw (e)--(E);		\draw (e)--(G);		\hg{e}{F};
			\draw (f)--(H);		\draw (f)--(I);		\hg{f}{G};
			
			\foreach \p in {A,B,C,D,E,F,G,H,I} {
				\fill[white] (\p) circle (2pt);
				\draw (\p) circle (2pt);
			}
			\foreach \p in {a,b,c,d,e,f}{
				\fill (\p) circle (2pt);
			}
			\node[below] at (0,-1.2) {fully reduced};
			\node[below] at (0,-1.5) {hourglass plabic graph};
		}
	\end{minipage}
	\caption{Examples of web basis of $r=2,3,4$.}	
	\label{fig:web_basis}
\end{figure}
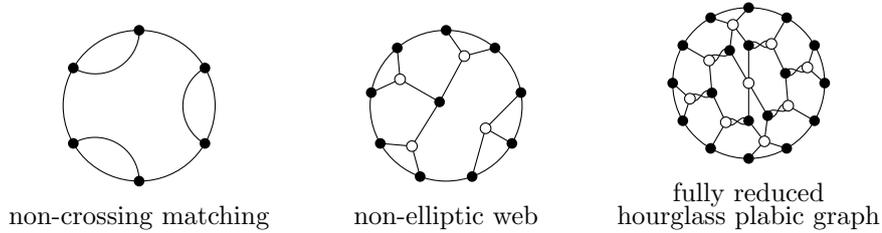

When $r=2$, $sl_2$-webs are spanned by \textit{non-crossing matchings}, which are partitions of the boundary vertices into pairs, such that the chords connecting each pair are non-crossing. 

When $r=3$, $sl_3$-webs are spanned by \textit{non-elliptic webs}, which are webs containing no interior faces bounded by four or fewer edges. That is, a non-elliptic web contains no contractible cycles, bigons, or squares. For an $sl_3$-web, we can express it as a linear combination of non-elliptic webs via skein relations in Figure \ref{fig:skein3}, which were introduced in \cite{kuperberg_spiders_1996}. 

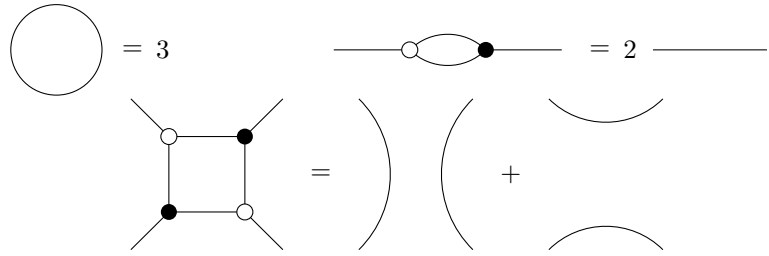
\begin{figure}[h]
	\centering
	\begin{minipage}{0.3\textwidth}
		\centering
		\begin{tikzpicture}
			\draw (-1,0) circle (0.6);
			\node at (0,0) {=};
			\node at (0.4,0) {3};
		\end{tikzpicture}
	\end{minipage}
	\begin{minipage}{0.6\textwidth}
		\centering
		\begin{tikzpicture}
			\draw (-3.5,0)--(-2.5,0);
			\draw (-2.5,0) to[out=-45, in=-135] (-1.5,0);
			\draw (-2.5,0) to[out=45, in=135] (-1.5,0);
			\draw (-1.5,0)--(-0.5,0);
			\fill[white] (-2.5,0) circle (3pt);
			\draw (-2.5,0) circle (3pt);
			\fill (-1.5,0) circle (3pt);
			\node at (0,0) {=};
			\node at (0.4,0) {2};
			\draw (0.7,0)--(2.2,0);
		\end{tikzpicture}
	\end{minipage}
	\begin{minipage}{\textwidth}
		\centering
		\begin{tikzpicture}
			\draw (-3,1)--(-2.5,0.5);
			\draw (-3,-1)--(-2.5,-0.5);
			\draw (-1,1)--(-1.5,0.5);
			\draw (-1,-1)--(-1.5,-0.5);
			\draw (-2.5,0.5)--(-2.5,-0.5);
			\draw (-2.5,0.5)--(-1.5,0.5);
			\draw (-1.5,-0.5)--(-1.5,0.5);
			\draw (-1.5,-0.5)--(-2.5,-0.5);
			\fill[white] (-2.5,0.5) circle (3pt);
			\draw (-2.5,0.5) circle (3pt);
			\fill[white] (-1.5,-0.5) circle (3pt);
			\draw (-1.5,-0.5) circle (3pt);
			\fill (-2.5,-0.5) circle (3pt);
			\fill (-1.5,0.5) circle (3pt);
			\node at (-0.5,0) {=};
			\draw (0,1) to[out=-45, in=45] (0,-1);
			\draw (1.5,1) to[out=-135, in=135] (1.5,-1);
			\node at (2,0) {+};
			\draw (2.5,1) to[out=-45, in=-135] (4,1);
			\draw (2.5,-1) to[out=45, in=135] (4,-1);
		\end{tikzpicture}
	\end{minipage}
	\caption{The skein relations of $sl_3$-webs.}
	\label{fig:skein3}
\end{figure}

\begin{lemma}\cite[Lemma 5.4]{elkin_twists_2023}\label{lem:interior}
	Given a non-elliptic web $W$, consider its web \textit{interior} $\mathring{W}$, which only consists of the internal vertices and edges of $W$. Suppose that $\mathring{W}$ is connected and consists only of cycles. Then there exist two adjacent vertices in $\mathring{W}$ that are bivalent in $\mathring{W}$.
\end{lemma}

When $r=4$, the situation becomes more intricate: multiple edges are permitted in $sl_4$-webs, so a rotation-invariant web basis is introduced in \cite{gaetz2023rotation} which is called an hourglass plabic graph. 

\begin{definition}\label{def:hggraph}
	\cite[Section 3.1]{gaetz2023rotation} An \textit{hourglass graph} $G$ is an underlying planar embedded graph $\widehat{G}$, together with a positive integer multiplicity $m(e)$ for each edge $e$. The hourglass graph $G$ is drawn in the plane by replacing each edge $e$ of $\widehat{G}$ with $m(e) > 1$ by $m(e)$ strands, twisted so that the clockwise order of these strands around the two incident vertices are the same. For $m(e) \geq 2$, we call this twisted edge an \textit{$m(e)$-hourglass}, and call an edge with $m(e) = 1$ a \textit{simple edge}.
	
	An \textit{hourglass plabic graph} is a bipartite hourglass graph G, with a fixed proper
	black-white vertex coloring, embedded in a disk, with all internal vertices of degree 4, and all
	boundary vertices of simple degree 1, labeled clockwise as $1, 2, \ldots, n$.
\end{definition}

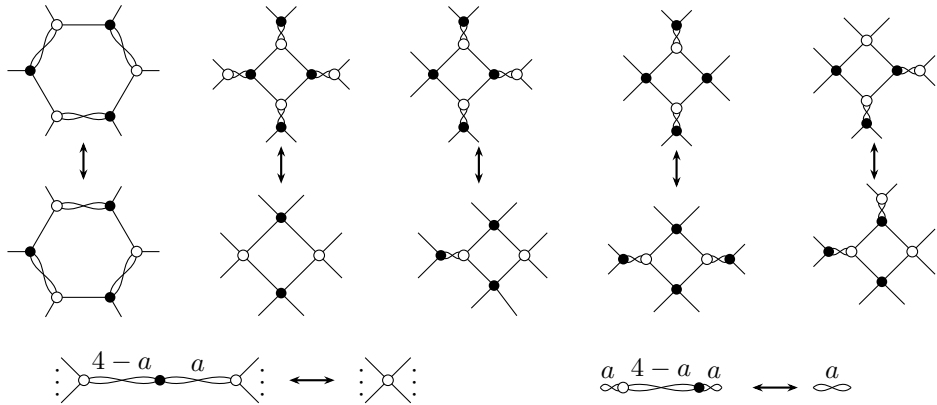
\begin{figure}[h]
	\centering
	\begin{minipage}{0.19\textwidth}
		\centering
		\begin{tikzpicture}
			\coordinate (A) at (-0.35,{1.2+0.35*sqrt(3)});
			\coordinate (B) at (0.7,1.2);
			\coordinate (C) at (-0.35,{1.2-0.35*sqrt(3)});
			\coordinate (D) at (-0.35,{-1.2+0.35*sqrt(3)});
			\coordinate (E) at (0.7,-1.2);
			\coordinate (F) at (-0.35,{-1.2-0.35*sqrt(3)});
			
			\coordinate (a) at (0.35,{1.2+0.35*sqrt(3)});
			\coordinate (b) at (-0.7,1.2);
			\coordinate (c) at (0.35,{1.2-0.35*sqrt(3)});
			\coordinate (d) at (0.35,{-1.2+0.35*sqrt(3)});
			\coordinate (e) at (-0.7,-1.2);
			\coordinate (f) at (0.35,{-1.2-0.35*sqrt(3)});
			
			\draw (A)--(-0.5,{1.2+0.5*sqrt(3)});		\draw (a)--(0.5,{1.2+0.5*sqrt(3)});
			\draw (B)--(1,1.2);							\draw (b)--(-1,1.2);
			\draw (C)--(-0.5,{1.2-0.5*sqrt(3)});		\draw (c)--(0.5,{1.2-0.5*sqrt(3)});
			\draw (D)--(-0.5,{-1.2+0.5*sqrt(3)});		\draw (d)--(0.5,{-1.2+0.5*sqrt(3)});
			\draw (E)--(1,-1.2);						\draw (e)--(-1,-1.2);
			\draw (F)--(-0.5,{-1.2-0.5*sqrt(3)});		\draw (f)--(0.5,{-1.2-0.5*sqrt(3)});
			
			\hg{A}{b};	\hg{B}{a};	\hg{C}{c};	\hg{D}{d};	\hg{E}{f};	\hg{F}{e};
			\draw (A)--(a);		\draw (B)--(c);		\draw (C)--(b);
			\draw (D)--(e);		\draw (E)--(d);		\draw (F)--(f);
			
			\draw[thick,<->,>={Stealth[scale=0.6]}] (0,0.25)--(0,-0.25);
			
			\foreach \p in {A,B,C,D,E,F} {
				\fill[white] (\p) circle (2pt);
				\draw (\p) circle (2pt);
			}
			\foreach \p in {a,b,c,d,e,f}{
				\fill (\p) circle (2pt);
			}
		\end{tikzpicture}
	\end{minipage}
	\begin{minipage}{0.19\textwidth}
		\centering
		\begin{tikzpicture}
			\coordinate (A) at (0,1.6);
			\coordinate (B) at (-0.7,1.2);
			\coordinate (C) at (0.7,1.2);
			\coordinate (D) at (0,0.8);
			\coordinate (E) at (-0.5,-1.2);
			\coordinate (F) at (0.5,-1.2);
			
			\coordinate (a) at (0,1.9);
			\coordinate (b) at (-0.4,1.2);
			\coordinate (c) at (0.4,1.2);
			\coordinate (d) at (0,0.5);
			\coordinate (e) at (0,-0.7);
			\coordinate (f) at (0,-1.7);
			
			\draw (a)--(0.2,2.1);		\draw (a)--(-0.2,2.1);	
			\draw (B)--(-0.9,1);		\draw (B)--(-0.9,1.4);							
			\draw (C)--(0.9,1);			\draw (C)--(0.9,1.4);		
			\draw (d)--(0.2,0.3);		\draw (d)--(-0.2,0.3);
			\draw (E)--(-0.8,-0.9);		\draw (E)--(-0.8,-1.5);	
			\draw (e)--(-0.3,-0.4);		\draw (e)--(0.3,-0.4);
			\draw (F)--(0.8,-0.9);		\draw (F)--(0.8,-1.5);
			\draw (f)--(-0.3,-2);		\draw (f)--(0.3,-2);
			
			\hg{A}{a};		\hg{B}{b};		\hg{C}{c};		\hg{D}{d};	
			\draw (A)--(b);		\draw (A)--(c);		\draw (D)--(b);		\draw (D)--(c);	
			\draw (E)--(e);		\draw (E)--(f);		\draw (F)--(e);		\draw (F)--(f);
			
			\draw[thick,<->,>={Stealth[scale=0.6]}] (0,0.25)--(0,-0.25);
			
			\foreach \p in {A,B,C,D,E,F} {
				\fill[white] (\p) circle (2pt);
				\draw (\p) circle (2pt);
			}
			\foreach \p in {a,b,c,d,e,f}{
				\fill (\p) circle (2pt);
			}
		\end{tikzpicture}
	\end{minipage}
	\begin{minipage}{0.19\textwidth}
		\centering
		\begin{tikzpicture}
			\coordinate (A) at (-0.2,1.6);
			\coordinate (B) at (0.5,1.2);
			\coordinate (C) at (-0.2,0.8);
			\coordinate (D) at (-0.2,-1.2);
			\coordinate (E) at (0.6,-1.2);
			
			\coordinate (a) at (-0.2,1.9);
			\coordinate (b) at (-0.6,1.2);
			\coordinate (c) at (0.2,1.2);
			\coordinate (d) at (-0.2,0.5);
			\coordinate (e) at (0.2,-0.8);
			\coordinate (f) at (-0.5,-1.2);
			\coordinate (g) at (0.2,-1.6);
			
			\draw (a)--(-0.4,2.1);		\draw (a)--(0,2.1);	
			\draw (B)--(0.7,1);			\draw (B)--(0.7,1.4);							
			\draw (b)--(-0.9,0.9);		\draw (b)--(-0.9,1.5);		
			\draw (d)--(-0.4,0.3);		\draw (d)--(0,0.3);
			\draw (E)--(0.9,-0.9);		\draw (E)--(0.9,-1.5);	
			\draw (e)--(-0.1,-0.5);		\draw (e)--(0.5,-0.5);
			\draw (f)--(-0.8,-0.9);		\draw (f)--(-0.8,-1.5);
			\draw (g)--(-0.1,-1.9);		\draw (g)--(0.5,-2);
			
			\hg{A}{a};		\hg{B}{c};		\hg{C}{d};		\hg{D}{f};	
			\draw (A)--(b);		\draw (A)--(c);		\draw (C)--(b);		\draw (C)--(c);	
			\draw (D)--(e);		\draw (D)--(g);		\draw (E)--(e);		\draw (E)--(g);
			
			\draw[thick,<->,>={Stealth[scale=0.6]}] (0,0.25)--(0,-0.25);
			
			\foreach \p in {A,B,C,D,E} {
				\fill[white] (\p) circle (2pt);
				\draw (\p) circle (2pt);
			}
			\foreach \p in {a,b,c,d,e,f,g}{
				\fill (\p) circle (2pt);
			}
		\end{tikzpicture}
	\end{minipage}
	\begin{minipage}{0.19\textwidth}
		\centering
		\begin{tikzpicture}
			\coordinate (A) at (0,1.6);
			\coordinate (B) at (0,0.8);
			\coordinate (C) at (-0.4,-1.2);
			\coordinate (D) at (0.4,-1.2);
			
			\coordinate (a) at (0,1.9);
			\coordinate (b) at (-0.4,1.2);
			\coordinate (c) at (0.4,1.2);
			\coordinate (d) at (0,0.5);
			\coordinate (e) at (0,-0.8);
			\coordinate (f) at (-0.7,-1.2);
			\coordinate (g) at (0.7,-1.2);
			\coordinate (h) at (0,-1.6);
			
			\draw (a)--(-0.2,2.1);		\draw (a)--(0.2,2.1);							
			\draw (b)--(-0.7,0.9);		\draw (b)--(-0.7,1.5);		
			\draw (c)--(0.7,0.9);		\draw (c)--(0.7,1.5);
			\draw (d)--(-0.2,0.3);		\draw (d)--(0.2,0.3);	
			\draw (e)--(-0.3,-0.5);		\draw (e)--(0.3,-0.5);
			\draw (f)--(-0.9,-1);		\draw (f)--(-0.9,-1.4);
			\draw (g)--(0.9,-1);		\draw (g)--(0.9,-1.4);
			\draw (h)--(-0.3,-1.9);		\draw (h)--(0.3,-1.9);
			
			\hg{A}{a};		\hg{B}{d};		\hg{C}{f};		\hg{D}{g};	
			\draw (A)--(b);		\draw (A)--(c);		\draw (B)--(b);		\draw (B)--(c);	
			\draw (C)--(e);		\draw (C)--(h);		\draw (D)--(e);		\draw (D)--(h);
			
			\draw[thick,<->,>={Stealth[scale=0.6]}] (0,0.25)--(0,-0.25);
			
			\foreach \p in {A,B,C,D} {
				\fill[white] (\p) circle (2pt);
				\draw (\p) circle (2pt);
			}
			\foreach \p in {a,b,c,d,e,f,g,h}{
				\fill (\p) circle (2pt);
			}
		\end{tikzpicture}
	\end{minipage}
	\begin{minipage}{0.19\textwidth}
		\centering
		\begin{tikzpicture}
			\coordinate (A) at (-0.1,1.6);
			\coordinate (B) at (0.6,1.2);
			\coordinate (C) at (-0.1,0.8);
			\coordinate (D) at (0.1,-0.5);
			\coordinate (E) at (-0.3,-1.2);
			\coordinate (F) at (0.5,-1.2);	
			
			\coordinate (a) at (-0.5,1.2);
			\coordinate (b) at (0.3,1.2);
			\coordinate (c) at (-0.1,0.5);
			\coordinate (d) at (0.1,-0.8);
			\coordinate (e) at (-0.6,-1.2);
			\coordinate (f) at (0.1,-1.6);
			
			\draw (A)--(-0.4,1.9);		\draw (A)--(0.2,1.9);							
			\draw (a)--(-0.8,0.9);		\draw (a)--(-0.8,1.5);		
			\draw (B)--(0.8,1);			\draw (B)--(0.8,1.4);
			\draw (c)--(-0.3,0.3);		\draw (c)--(0.1,0.3);	
			\draw (D)--(-0.1,-0.3);		\draw (D)--(0.3,-0.3);
			\draw (e)--(-0.8,-1);		\draw (e)--(-0.8,-1.4);
			\draw (F)--(0.8,-0.9);		\draw (F)--(0.8,-1.5);
			\draw (f)--(-0.2,-1.9);		\draw (f)--(0.4,-1.9);
			
			\hg{B}{b};		\hg{C}{c};		\hg{D}{d};		\hg{E}{e};	
			\draw (A)--(a);		\draw (A)--(b);		\draw (C)--(a);		\draw (C)--(b);	
			\draw (E)--(d);		\draw (E)--(f);		\draw (F)--(d);		\draw (F)--(f);
			
			\draw[thick,<->,>={Stealth[scale=0.6]}] (0,0.25)--(0,-0.25);
			
			\foreach \p in {A,B,C,D,E,F} {
				\fill[white] (\p) circle (2pt);
				\draw (\p) circle (2pt);
			}
			\foreach \p in {a,b,c,d,e,f}{
				\fill (\p) circle (2pt);
			}
		\end{tikzpicture}
	\end{minipage}
	\\\vspace{0.3cm}
	\begin{minipage}{0.48\textwidth}
		\centering
		\begin{tikzpicture}
			\coordinate (A) at (-3,0);
			\coordinate (B) at (-1,0);
			\coordinate (C) at (1,0);
			
			\coordinate (a) at (-2,0);
			
			\hg{A}{a};		\hg{B}{a};
			\draw (A)--(-3.3,0.3);		\draw (A)--(-3.3,-0.3);	
			\draw (B)--(-0.7,0.3);		\draw (B)--(-0.7,-0.3);
			\draw (C)--(0.7,0.3);		\draw (C)--(0.7,-0.3);	
			\draw (C)--(1.3,0.3);		\draw (C)--(1.3,-0.3);
			
			\draw[thick,<->,>={Stealth[scale=0.6]}] (-0.3,0)--(0.3,0);
			
			\foreach \p in {A,B,C} {
				\fill[white] (\p) circle (2pt);
				\draw (\p) circle (2pt);
			}
			\fill (a) circle (2pt);
			
			\node at (-3.35,0) [rotate=90] {\ldots};
			\node at (-0.65,0) [rotate=90] {\ldots};
			\node at (0.65,0) [rotate=90] {\ldots};
			\node at (1.35,0) [rotate=90] {\ldots};
			\node[above] at (-2.5,0) {$4-a$};	
			\node[above] at (-1.5,0) {$a$};	
		\end{tikzpicture}
	\end{minipage}
	\begin{minipage}{0.48\textwidth}
		\centering
		\begin{tikzpicture}
			\coordinate (A) at (-2,0);
			\coordinate (B) at (-2.3,0);
			\coordinate (C) at (-0.7,0);
			\coordinate (a) at (-1,0);
			\coordinate (b) at (0.5,0);
			\coordinate (c) at (1,0);
			
			\hg{A}{a};		\hg{A}{B};		\hg{a}{C};		\hg{b}{c}

			\draw[thick,<->,>={Stealth[scale=0.6]}] (-0.3,0)--(0.3,0);
			
			\fill[white] (A) circle (2pt);
			\draw (A) circle (2pt);
			
			\fill (a) circle (2pt);
			\node[above] at (-2.2,0) {$a$};	
			\node[above] at (-1.5,0) {$4-a$};	
			\node[above] at (-0.8,0) {$a$};	
			\node[above] at (0.75,0) {$a$};	
		\end{tikzpicture}
	\end{minipage}
	\caption{Benzene move, square moves and construction moves of hourglass plabic graphs.}
	\label{fig:hgmoves}
\end{figure}
To construct the web basis for hourglass plabic graphs, the benzene move, square moves, and contraction moves in Figure \ref{fig:hgmoves} are introduced as equivalence relations in \cite{gaetz2023rotation}. An hourglass plabic graph $G$ is defined to be fully reduced if every hourglass plabic graph in its equivalence class $[G]$ under these moves contains none of the following configurations:
\begin{itemize}
	\item an interior vertex incident to fewer than three other vertices;
	\item 2-cycle (treating an hourglass edge as a single edge);
	\item 4-cycle containing an hourglass edge.
\end{itemize}
All the skein relations of hourglass plabic graphs are listed in \cite[Figures 34-37]{gaetz2023rotation}.

\subsection{Compatibility and immanant map}\label{ssc:PreComp}

In \cite{lam_dimers_2015}, the definition of compatibility is introduced to establish a correspondence among webs, dimers and Pl\"ucker polynomials.

In \cite{lam_dimers_2015}, there is a method to reduce a dimer to a corresponding web. For a double dimer configuration, each internal vertex is either a bivalent vertex (which can be contracted) or lies in a 2-cycle. Therefore, we only need to consider whether each connected component connects two distinct boundary vertices, thereby allowing it to be reduced to a non-crossing matching.

For a triple dimer configuration $D$ on a plabic graph $G$, let $G_D$ be the subgraph of $G$ whose edge set consists of all distinct edges appearing in $D$. In \cite{lam_dimers_2015}, there is a method to create a web $W(D)$ corresponding to $D$ as follows:

\begin{enumerate}[label={(\arabic*)}]
	\item for each boundary vertex in the graph $G$, create a corresponding boundary vertex in $W(D)$. If a boundary vertex is adjacent to two edges in $D$, color it white; otherwise, color it black;
	
	\item for any cycle in $G_D$ composed only of bivalent vertices, add a cycle without vertices in $W(D)$ and orient it arbitrarily;
	
	\item for a chain of bivalent vertices in $G_D$ that connects two vertices $v$ and $v'$, add an arrow in $W(D)$ pointing from the degree-1 boundary vertex in $G_D$ to the degree-2 vertex, that is, from the black vertex in $W(D)$ to the white one;
	
	\item for a connected component in $G_D$ that contains at least one trivalent vertex, create a corresponding connected component in $W(D)$, and contract the two neighbors of the bivalent vertices within it.
\end{enumerate}
\begin{figure}[h]
	\centering
	\begin{minipage}[h]{0.3\textwidth}
		\centering
		\begin{tikzpicture}
			\draw (0,0) circle (1.5);
			\foreach \k in {1,...,8} {
				\coordinate (\k) at (135-45*\k:1.5);
			}
			
			\coordinate (A) at (90:1.2);
			\coordinate (B) at (45:1.05);
			\coordinate (C) at (0:1.05);
			\coordinate (D) at (-45:1.05);
			\coordinate (E) at (-90:1.275);
			\coordinate (F) at (-135:1.275);
			\coordinate (G) at (180:1.275);
			\coordinate (H) at (135:0.9);
			\coordinate (I) at (82.5:0.375);
			\coordinate (J) at (45:0.675);
			\coordinate (K) at (-37.5:0.375);
			\coordinate (L) at (-112.5:0.9);
			\coordinate (M) at (-157.5:0.9);
			\coordinate (N) at (-157.5:0.375);
			
			\coordinate (a) at (90:0.9);
			\coordinate (b) at (22.5:0.9);
			\coordinate (c) at (-22.5:0.9);
			\coordinate (d) at (-90:1.05);
			\coordinate (e) at (-135:1.05);
			\coordinate (f) at (180:1.05);
			\coordinate (g) at (180:0.675);
			\coordinate (h) at (142.5:0.375);
			\coordinate (i) at (22.5:0.375);
			\coordinate (j) at (-97.5:0.375);
			
			\draw[red] (1)--(A);		\draw[red] (2)--(B);		
			\draw[red] (3)--(C);		\draw[red] (8)--(H);
			\draw[red] (a)--(I);		\draw[red] (b)--(J);		
			\draw[red] (c)--(D);		\draw[red] (d)--(E);		
			\draw[red] (e)--(F);		\draw[red] (f)--(G);		
			\draw[red] (g)--(M);		\draw[red] (h)--(N);		
			\draw[red] (i)--(K);		\draw[red] (j)--(L);
			\draw (4)--(D);		\draw (5)--(E);		\draw (6)--(F);		\draw (7)--(G);		
			\draw (a)--(A);		\draw (a)--(B);		\draw (a)--(H);		\draw (a)--(J);		\draw (b)--(B);		\draw (b)--(C);		\draw (c)--(C);		\draw (c)--(K);		\draw (d)--(D);		\draw (d)--(L);		\draw (e)--(L);		\draw (e)--(M);		\draw (f)--(H);		\draw (f)--(M);		\draw (g)--(H);		\draw (g)--(N);		\draw (h)--(H);		\draw (h)--(I);		\draw (i)--(I);		\draw (i)--(J);			\draw (j)--(K);		\draw (j)--(N);
			
			\foreach \k in {1,2,...,8} {
				\fill (\k) circle (2pt);
			}
			\foreach \p in {A,B,C,D,E,F,G,H,I,J,K,L,M,N} {
				\fill[white] (\p) circle (2pt);
				\draw (\p) circle (2pt);
			}
			\foreach \p in {a,b,c,d,e,f,g,h,i,j}{
				\fill (\p) circle (2pt);
			}
		\end{tikzpicture}
	\end{minipage}
	\begin{minipage}[h]{0.3\textwidth}
		\centering
		\begin{tikzpicture}
			\draw (0,0) circle (1.5);
			\foreach \k in {1,...,8} {
				\coordinate (\k) at (135-45*\k:1.5);
			}
			
			\coordinate (A) at (90:1.2);
			\coordinate (B) at (45:1.05);
			\coordinate (C) at (0:1.05);
			\coordinate (D) at (-45:1.05);
			\coordinate (E) at (-90:1.275);
			\coordinate (F) at (-135:1.275);
			\coordinate (G) at (180:1.275);
			\coordinate (H) at (135:0.9);
			\coordinate (I) at (82.5:0.375);
			\coordinate (J) at (45:0.675);
			\coordinate (K) at (-37.5:0.375);
			\coordinate (L) at (-112.5:0.9);
			\coordinate (M) at (-157.5:0.9);
			\coordinate (N) at (-157.5:0.375);
			
			\coordinate (a) at (90:0.9);
			\coordinate (b) at (22.5:0.9);
			\coordinate (c) at (-22.5:0.9);
			\coordinate (d) at (-90:1.05);
			\coordinate (e) at (-135:1.05);
			\coordinate (f) at (180:1.05);
			\coordinate (g) at (180:0.675);
			\coordinate (h) at (142.5:0.375);
			\coordinate (i) at (22.5:0.375);
			\coordinate (j) at (-97.5:0.375);
			
			\draw[blue] (1)--(A);		\draw[blue] (4)--(D);
			\draw[blue] (5)--(E);		\draw[blue] (6)--(F);
			\draw[blue] (a)--(J);		\draw[blue] (b)--(B);
			\draw[blue] (c)--(C);		\draw[blue] (d)--(L);
			\draw[blue] (e)--(M);		\draw[blue] (f)--(G);
			\draw[blue] (g)--(H);		\draw[blue] (h)--(N);
			\draw[blue] (i)--(I);		\draw[blue] (j)--(K);
			\draw (2)--(B);		\draw (3)--(C);		\draw (7)--(G);		\draw (8)--(H);
			\draw (a)--(A);		\draw (a)--(B);		\draw (a)--(H);		\draw (a)--(I);			\draw (b)--(C);		\draw (b)--(J);		\draw (c)--(D);		\draw (c)--(K);		\draw (d)--(D);		\draw (d)--(E);		\draw (e)--(F);		\draw (e)--(L);			\draw (f)--(H);		\draw (f)--(M);		\draw (g)--(M);		\draw (g)--(N);		\draw (h)--(H);		\draw (h)--(I);		\draw (i)--(J);		\draw (i)--(K);			\draw (j)--(L);		\draw (j)--(N);
			
			\foreach \k in {1,2,...,8} {
				\fill (\k) circle (2pt);
			}
			\foreach \p in {A,B,C,D,E,F,G,H,I,J,K,L,M,N} {
				\fill[white] (\p) circle (2pt);
				\draw (\p) circle (2pt);
			}
			\foreach \p in {a,b,c,d,e,f,g,h,i,j}{
				\fill (\p) circle (2pt);
			}
		\end{tikzpicture}
	\end{minipage}
	\begin{minipage}[h]{0.3\textwidth}
		\centering
		\begin{tikzpicture}
			\draw (0,0) circle (1.5);
			\foreach \k in {1,...,8} {
				\coordinate (\k) at (135-45*\k:1.5);
			}
			
			\coordinate (A) at (90:1.2);
			\coordinate (B) at (45:1.05);
			\coordinate (C) at (0:1.05);
			\coordinate (D) at (-45:1.05);
			\coordinate (E) at (-90:1.275);
			\coordinate (F) at (-135:1.275);
			\coordinate (G) at (180:1.275);
			\coordinate (H) at (135:0.9);
			\coordinate (I) at (82.5:0.375);
			\coordinate (J) at (45:0.675);
			\coordinate (K) at (-37.5:0.375);
			\coordinate (L) at (-112.5:0.9);
			\coordinate (M) at (-157.5:0.9);
			\coordinate (N) at (-157.5:0.375);
			
			\coordinate (a) at (90:0.9);
			\coordinate (b) at (22.5:0.9);
			\coordinate (c) at (-22.5:0.9);
			\coordinate (d) at (-90:1.05);
			\coordinate (e) at (-135:1.05);
			\coordinate (f) at (180:1.05);
			\coordinate (g) at (180:0.675);
			\coordinate (h) at (142.5:0.375);
			\coordinate (i) at (22.5:0.375);
			\coordinate (j) at (-97.5:0.375);
			
			\draw[green] (2)--(B);		\draw[green] (3)--(C);		
			\draw[green] (4)--(D);		\draw[green] (7)--(G);	
			\draw[green] (a)--(A);		\draw[green] (b)--(J);
			\draw[green] (c)--(K);		\draw[green] (d)--(E);		
			\draw[green] (e)--(F);		\draw[green] (f)--(M);		
			\draw[green] (g)--(N);		\draw[green] (h)--(H);
			\draw[green] (i)--(I);		\draw[green] (j)--(L);		
			\draw (1)--(A);		\draw (5)--(E);		\draw (6)--(F);		\draw (8)--(H);
			\draw (a)--(B);		\draw (a)--(H);		\draw (a)--(I);		\draw (a)--(J);		\draw (b)--(B);		\draw (b)--(C);		\draw (c)--(C);		\draw (c)--(D);			\draw (d)--(D);		\draw (d)--(L);		\draw (e)--(L);		\draw (e)--(M);		\draw (f)--(G);		\draw (f)--(H);		\draw (g)--(H);		\draw (g)--(M);			\draw (h)--(I);		\draw (h)--(N);		\draw (i)--(J);		\draw (i)--(K);		\draw (j)--(K);		\draw (j)--(N);
			
			\foreach \k in {1,2,...,8} {
				\fill (\k) circle (2pt);
			}
			\foreach \p in {A,B,C,D,E,F,G,H,I,J,K,L,M,N} {
				\fill[white] (\p) circle (2pt);
				\draw (\p) circle (2pt);
			}
			\foreach \p in {a,b,c,d,e,f,g,h,i,j}{
				\fill (\p) circle (2pt);
			}
		\end{tikzpicture}
	\end{minipage}
	\caption{The 3 different dimers $D_1$, $D_2$, $D_3$ (from left to right) as the bases of the $3$-fold dimer configuration.}
	\label{fig:3dimers}
\end{figure}
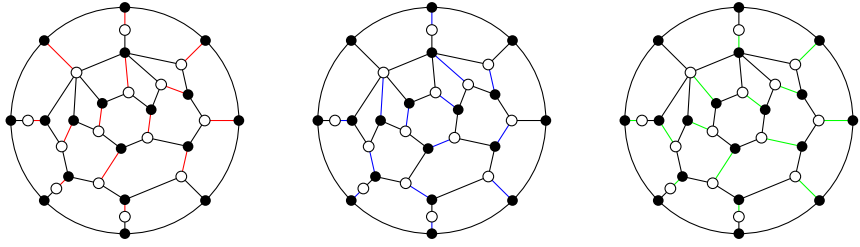
\begin{example}
	Consider a $3$-fold dimer configuration $D \in \mathcal{D}^3(G)$ in a plabic graph $G$ drawn by the method from \cite{postnikov_total_2006} with the boundary condition $\lambda=(2,2,2,2,1,1,1,1)$. Then $D$ can be constructed by folding three dimer configurations, for example, the configurations $D_1$, $D_2$, $D_3$ displayed in \Cref{fig:3dimers}, whose boundary conditions are $\partial(D_1)=\{1,2,3,8\}$, $\partial(D_2)=\{1,4,5,6\}$ and $\partial(D_3)=\{2,3,4,7\}$.
	
	Figure \ref{fig:dimer_to_web} displays the 3-fold dimer configuration $D$ (left) and its corresponding web (right). Although the web is not non-elliptic, it can be expressed as a linear combination of non-elliptic webs through the skein relations shown in Figure \ref{fig:skein3}.
	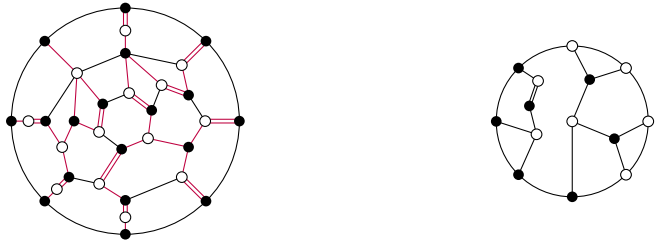
\begin{figure}[h]
		\centering
		\begin{minipage}[h]{0.6\textwidth}
			\centering
			\begin{tikzpicture}
				\draw (0,0) circle (1.5);
				\foreach \k in {1,...,8} {
					\coordinate (\k) at (135-45*\k:1.5);
				}
				
				\coordinate (A) at (90:1.2);
				\coordinate (B) at (45:1.05);
				\coordinate (C) at (0:1.05);
				\coordinate (D) at (-45:1.05);
				\coordinate (E) at (-90:1.275);
				\coordinate (F) at (-135:1.275);
				\coordinate (G) at (180:1.275);
				\coordinate (H) at (135:0.9);
				\coordinate (I) at (82.5:0.375);
				\coordinate (J) at (45:0.675);
				\coordinate (K) at (-37.5:0.375);
				\coordinate (L) at (-112.5:0.9);
				\coordinate (M) at (-157.5:0.9);
				\coordinate (N) at (-157.5:0.375);
				
				\coordinate (a) at (90:0.9);
				\coordinate (b) at (22.5:0.9);
				\coordinate (c) at (-22.5:0.9);
				\coordinate (d) at (-90:1.05);
				\coordinate (e) at (-135:1.05);
				\coordinate (f) at (180:1.05);
				\coordinate (g) at (180:0.675);
				\coordinate (h) at (142.5:0.375);
				\coordinate (i) at (22.5:0.375);
				\coordinate (j) at (-97.5:0.375);
				
				\draw[purple] [double distance=1pt] (1)--(A);		
				\draw[purple] [double distance=1pt] (2)--(B);		
				\draw[purple] [double distance=1pt] (3)--(C);			
				\draw[purple] [double distance=1pt] (4)--(D);		
				\draw[purple] (5)--(E);		\draw[purple] (6)--(F);
				\draw[purple] (7)--(G);		\draw[purple] (8)--(H);
				
				\draw[purple][double distance=1pt] (b)--(J);		
				\draw[purple] (a)--(A);		\draw[purple] (a)--(I);		
				\draw[purple] [double distance=1pt] (d)--(E);		
				\draw[purple] (a)--(J);		\draw[purple] (b)--(B);		
				\draw[purple] [double distance=1pt] (e)--(F);		
				\draw[purple] (c)--(C);		\draw[purple] (c)--(D);		
				\draw[purple] [double distance=1pt] (f)--(G);		
				\draw[purple] (c)--(K);		\draw[purple] (d)--(L);		
				\draw[purple] [double distance=1pt] (h)--(N);		
				\draw[purple] (e)--(M);		\draw[purple] (f)--(M);	
				\draw[purple] [double distance=1pt] (i)--(I);		
				\draw[purple] (g)--(H);		\draw[purple] (g)--(M);
				\draw[purple] [double distance=1pt] (j)--(L);		
				\draw[purple] (g)--(N);		\draw[purple] (h)--(H);		
				\draw[purple] (i)--(K);		\draw[purple] (j)--(K);		
				
				\draw (a)--(B);		\draw (a)--(H);		\draw (b)--(C);		\draw (d)--(D);			\draw (e)--(L);		\draw (f)--(H);		\draw (h)--(I);		\draw (i)--(J);		\draw (j)--(N);
				
				\foreach \k in {1,2,...,8} {
					\fill (\k) circle (2pt);
				}
				\foreach \p in {A,B,C,D,E,F,G,H,I,J,K,L,M,N} {
					\fill[white] (\p) circle (2pt);
					\draw (\p) circle (2pt);
				}
				\foreach \p in {a,b,c,d,e,f,g,h,i,j}{
					\fill (\p) circle (2pt);
				}
			\end{tikzpicture}
		\end{minipage}
		\begin{minipage}[h]{0.3\textwidth}
			\centering
			\Diska{
				\coordinate (A) at (-160:0.5);
				\coordinate (B) at (130:0.7);
				\coordinate (C) at (0:0);
				
				\coordinate (a) at (160:0.6);
				\coordinate (b) at (67.5:0.6);
				\coordinate (c) at (-22.5:0.6);
				
				\draw (1)--(b);		\draw (2)--(b);		\draw (3)--(c);		\draw (4)--(c);
				\draw (5)--(C);		\draw (6)--(A);		\draw (7)--(A);		\draw (8)--(B);
				\draw (a)--(A);		\draw (b)--(C);		\draw (c)--(C);		
				\draw[double] (a)--(B);
				
				\foreach \p in {A,B,C} {
					\fill[white] (\p) circle (2pt);
					\draw (\p) circle (2pt);
				}
				\foreach \p in {a,b,c}{
					\fill (\p) circle (2pt);
				}
			}
		\end{minipage}
		\caption{The $3$-fold dimer configuration (left) and the web correspond to it (right).}
		\label{fig:dimer_to_web}
	\end{figure}
\end{example}

The method above induced the following definitions. 

\begin{definition}\label{def:compatible2}
	\cite[Section 3.1]{lam_dimers_2015} Let $M$ be a matching with $n$ boundary vertices $v_1, \ldots, v_n$. Suppose $I, J\in \binom{[n]}{k}$, a monomial $P=P_I P_J$ and the matching $M$ are said to be \textit{compatible} if for each edge $\mu$ in $M$, $\mu$ matches a vertex in $I\backslash J$ with a vertex in $J\backslash I$, and boundary vertices in $I\cap J$ are isolated white vertices.  
\end{definition}

\begin{definition}\label{def:compatible3}
	\cite[Section 4.5]{lam_dimers_2015} Let $W$ be a web with $n$ boundary vertices $v_1, \ldots, v_n$. Suppose $I, J, K \in \binom{[n]}{k}$, a monomial $P = P_I P_J P_K$ and the web $W$ are said to be \textit{compatible} if there exists an edge coloring that satisfies the following conditions for all  $i\in [n]$:
	\begin{itemize}[leftmargin=0.5cm]
		\item the three edges incident to an internal vertex $v$ are in $3$ different colors $\alpha$,  $\beta$ and $\gamma$;
		
		\item the boundary vertex $v_i$ is black and adjacent to an edge in color $\alpha$ (resp. $\beta$, $\gamma$), if $i \in I\backslash (J\cup K)$ (resp. $J\backslash (I\cup K)$, $K\backslash (I\cup J)$);
		
		\item the boundary vertex $v_i$ is white and adjacent to an edge in color $\alpha$ (resp. $\beta$, $\gamma$), if $i \in J\cap K\backslash I$ (resp. $I\cap K\backslash J$, $I\cap J\backslash K$);
		
		\item otherwise, vertex $v_i$ is not adjacent to any edge.
	\end{itemize}
	We denote the number of edge colorings by the notation $a(I,J,K;W)$, and say the monomial $P_I P_J P_K$ and the $sl_3$-web $W$ are \textit{uniquely compatible} if $a(I,J,K;W)=1$.
\end{definition}

The method of compatibility provides an explicit realization of the immanant map from \cite{fraser_dimers_2019}, which maps the dual webs of web diagrams to Pl\"ucker polynomials. 

\begin{theorem}\label{thm:fll19}
	\cite[Section 4.1]{fraser_dimers_2019} The immanant map 
	\begin{equation*}
		\operatorname{Imm}:\mathcal{W}_{\lambda}(U)^* \rightarrow \mathbb{C}[\mathrm{Gr}(k,n)]_{\lambda}
	\end{equation*}
	defined by $\operatorname{Imm}(\phi)(\tilde{X}(N))=\phi(Web_r(N;\lambda))$ is an isomorphism.
\end{theorem}

This isomorphism implies that for a Pl\"ucker polynomial $P = \sum \operatorname{sgn}_i P_i$, each component $P_i$ is compatible with a web $W_i$, while $P$ itself exhibits compatibility with the signed web combination $\sum \operatorname{sgn}_i W_i$.

\section{Web diagrams of cluster variables in $\mathbb{C}[\operatorname{Gr}(4,8)]$]}\label{sec:Inv}
In this section, we compute the web diagrams which correspond to the web invariants of the 3 quadratic and 14 cubic cluster variables which mentioned in Section \ref{ssc:PreGra}. 

\subsection{The promotion and trip permutations}\label{ssc:WinPromTrip}
Before the computation of the web diagrams, we recall that the bijection between rectangular standard Young tableaux and hourglass plabic graphs introduced in \cite{gaetz2023rotation}. This bijection consists of two key components: trip permutations and the growth algorithm. For convenience, we denote the set of all $a\times b$ rectangular standard Young tableaux by $\operatorname{SYT}(a\times b)$. For convenience, in a symmetric group $S_n$, let $c=(12\cdots n)$ and $w_0=n(n-1)\cdots 1$.

\begin{definition}\label{def:PE}
	\cite{SCHUTZENBERGER197273} \textit{Promotion} is a bijection $\mathcal{P}: \operatorname{SYT}(a\times b) \rightarrow \operatorname{SYT}(a\times b)$ defined as follows. Given $T \in \operatorname{SYT}(a\times b)$, first delete the entry 1 from the box $b$. Fill the (now) empty box $b$ by sliding into $b$ the smaller of the two values appearing immediately to the right of $b$ and below $b$. This slide yields a new empty box $b'$. Repeat this sliding process until the empty box appears in the bottom right corner of the tableau. Finally, fill the bottom right corner with $ab + 1$ and then subtract 1 from all entries. The
	result is the promotion $\mathcal{P}(T) \in \operatorname{SYT}(a\times b)$.
	
	For $T \in \operatorname{SYT}(a\times b)$,  denote \textit{evacuation} of $T$ by $\mathcal{E}(T)\in \operatorname{SYT}(a\times b)$, which is obtained by rotating $T$ by 180$^\circ$ and then replacing each entry $i$ by $ab+1-i$.
\end{definition}

\begin{figure}[h]
	\begin{tikzpicture}[node distance=0.25cm, >={Stealth[scale=0.6]}]
		\node(E1){
			\scalemath{0.8}{\begin{ytableau}
					1 & 4 \\
					2 & 5 \\
					3 & 7 \\
					6 & 8 
			\end{ytableau}}
		};
		\node(E2)[right=of E1]{
			\scalemath{0.8}{\begin{ytableau}
					\phantom{1} & 4 \\
					2 & 5 \\
					3 & 7 \\
					6 & 8 
			\end{ytableau}}
		};
		\node(E3)[right=of E2]{
			\scalemath{0.8}{\begin{ytableau}
					2 & 4 \\
					\phantom{1} & 5 \\
					3 & 7 \\
					6 & 8 
			\end{ytableau}}
		};
		\node(E4)[right=of E3]{
			\scalemath{0.8}{\begin{ytableau}
					2 & 4 \\
					3 & 5 \\
					\phantom{1} & 7 \\
					6 & 8 
			\end{ytableau}}
		};
		\node(E5)[right=of E4]{
			\scalemath{0.8}{\begin{ytableau}
					2 & 4 \\
					3 & 5 \\
					6 & 7 \\
					\phantom{1} & 8 
			\end{ytableau}}
		};
		\node(E6)[right=of E5]{
			\scalemath{0.8}{\begin{ytableau}
					2 & 4 \\
					3 & 5 \\
					6 & 7 \\
					8 & \phantom{1} 
			\end{ytableau}}
		};
		\node(E7)[right=of E6]{
			\scalemath{0.8}{\begin{ytableau}
					2 & 4 \\
					3 & 5 \\
					6 & 7 \\
					8 & 9
			\end{ytableau}}
		};
		\node(E8)[right=of E7]{
			\scalemath{0.8}{\begin{ytableau}
					1 & 3 \\
					2 & 4 \\
					5 & 6 \\
					7 & 8
			\end{ytableau}}
		};
		\node(E0)[left=-0.1cm of E1]{$T=$};
		\node(E0)[right=-0.1cm of E8]{$=\mathcal{P}(T)$};
		\draw[->] (E1) -- (E2);
		\draw[->] (E2) -- (E3);
		\draw[->] (E3) -- (E4);
		\draw[->] (E4) -- (E5);
		\draw[->] (E5) -- (E6);
		\draw[->] (E6) -- (E7);
		\draw[->] (E7) -- (E8);
	\end{tikzpicture}
	\begin{tikzpicture}[node distance=0.25cm, >={Stealth[scale=0.6]}]
		\node(E1){
			\scalemath{0.8}{\begin{ytableau}
					1 & 4 \\
					2 & 5 \\
					3 & 7 \\
					6 & 8 
			\end{ytableau}}
		};
		\node(E2)[right=of E1]{
			\scalemath{0.8}{\begin{ytableau}
					8 & 6 \\
					7 & 3 \\
					5 & 2 \\
					4 & 1 
			\end{ytableau}}
		};
		\node(E3)[right=of E2]{
			\scalemath{0.8}{\begin{ytableau}
					1 & 3 \\
					2 & 6 \\
					4 & 7 \\
					5 & 8 
			\end{ytableau}}
		};
		\node(E0)[left=-0.1cm of E1]{$T=$};
		\node(E0)[right=-0.1cm of E3]{$=\mathcal{E}(T)$};
		\draw[->] (E1) -- (E2);
		\draw[->] (E2) -- (E3);
	\end{tikzpicture}
	\caption{An example for promotion and evacuation of a rectangular standard Young tableau $T$.}
	\label{fig:PTET}
\end{figure}
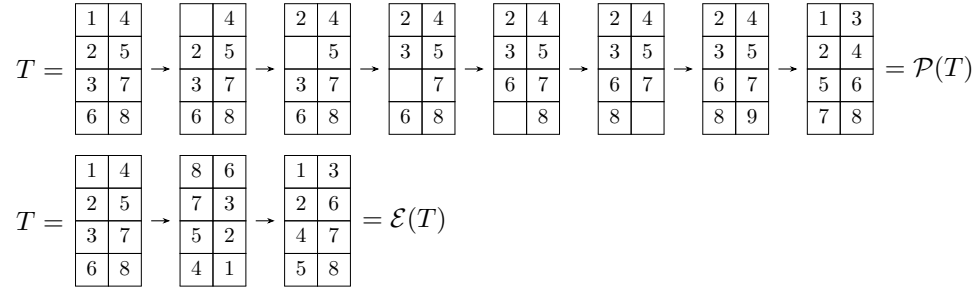

Note that $\mathcal{P}^{ab}=id$ and $\mathcal{E}$ is an involution, $\mathcal{P}$ and $\mathcal{E}$ generate a dihedral group action on $\operatorname{SYT}(a\times b)$.

Then recall the permutations of promotion and trip introduced in \cite{gaetz2024promotion} and \cite{gaetz2024web}. 

\begin{definition}\label{def:promiT}
	\cite{gaetz2024promotion} For $T \in \operatorname{SYT}(a\times b)$, $1 \leq i \leq a-1$ and $1\leq j \leq ab-1$, denote the unique entry of $\mathcal{P}^{j-1}(T)$ that moves from row $i+1$ to row $i$ in the process when applying promotion to $\mathcal{P}^{j-1}(T)$ by $p^{i,j}$. For $i\leq i\leq a-1$, the fixed-point-free permutation $\operatorname{prom}_i(T): j \mapsto p^{i,j}+j-1$ (mod $ab$) is called the $i$-th \textit{promotion permutation} of $T$.
\end{definition}

For an hourglass plabic graph $G$, similar to the definition of trip permutations in plabic graphs, for each integer $1 \leq j \leq ab-1$, we can find a path starting from the boundary vertex labeled $j$ in $G$. When passing through a white vertex, it turns to the $i$-th left edge; when passing through a black vertex, it turns to the $i$-th right edge. The path eventually terminates at boundary vertex $k$, thereby defining the $i$-th \textit{trip permutation} of $G$ as $\operatorname{trip}_i(G): j \mapsto k$. Additionally, let $\operatorname{rot}(G)$ be the hourglass plabic graph obtained by applying the map $i\mapsto i-1$ to the label of boundary vertices of $G$, and  $\operatorname{refl}(G)$ the hourglass plabic graph obtained by applying the map $i\mapsto ab+1-i$ to the label of boundary vertices of $G$.

Denote by $\operatorname{Aexc}(\pi) = \{i | \pi^{-1}(i)>i\}$ the \textit{anti-exceedance} set of a permutation $\pi$, with $\operatorname{rot}(\pi)=c^{-1} \circ \pi \circ c$ the \textit{rotation} of $\pi$ and with $\operatorname{refl}(\pi)=w_0\pi w_0$ the \textit{reverse-component} of $\pi$. The promotion permutations of $T$ and trip permutations of $G$ satisfy the following properties.

\begin{theorem}\label{thm:TG}
	\cite[Theorems.~4.7 and 6.1, Proposition.~4.11]{gaetz2023rotation} Let $T \in \operatorname{SYT}(4\times b)$, and $1 \leq i \leq a-1$. Then
	\begin{itemize}
		\item $\operatorname{rot}(\operatorname{prom}_i(T))=\operatorname{prom}_i(\mathcal{P}(T))$;
		\item $\operatorname{refl}(\operatorname{prom}_i(T))=\operatorname{prom}_i(\mathcal{E}(T))$;
		\item $\operatorname{prom}_i(T)=\operatorname{prom}_{a-i}(T)^{-1}$;
		\item $\operatorname{Aexc}(\operatorname{prom}_i(T))=\{e|e$ is in the first $i$ rows of $T\}$.
	\end{itemize} 
	For an equivalence class $[G]$ of fully reduced hourglass plabic graphs, there is a unique standard Young tableau $\mathcal{T}([G])$ corresponding to $[G]$ whose first $i$ rows contain the entries $\operatorname{Aexc} (\operatorname{trip}_i(G))$ for $i=1,2,3$. Furthermore, the map $\mathcal{T}$ is a bijection that satisfies
	\begin{itemize}
		\item $\operatorname{trip}_i(G)=\operatorname{prom}_i(\mathcal{T}([G]))$;
		\item $\mathcal{T}([\operatorname{rot}(G)])=\mathcal{P}(\mathcal{T}([G]))$;
		\item $\mathcal{T}([\operatorname{refl}(G)])=\mathcal{E}(\mathcal{T}([G]))$.
	\end{itemize}
\end{theorem}

For rectangular semi-standard Young tableaux, it is necessary to convert them into standard Young tableaux before applying the maps above. For example, given a semi-standard Young tableau $S=\scalemath{0.6}{\;{\ytableausetup{centertableaux}\ytableaushort{11,24,36,57}}}$ with the boundary condition $\lambda_S=(2,1,1,1,1,1,1)$, we relabel the entries labeled ``1'' from left to right by $1,2$, and systematically increment the labels for subsequent entries. This process yields a rectangular standard Young tableau $T=\scalemath{0.6}{\;{\ytableausetup{centertableaux}\ytableaushort{12,35,47,68}}}$. By relabeling the entries labeled ``7'' and ``8'' in $\mathcal{P}^2(T)=\scalemath{0.6}{\;{\ytableausetup{centertableaux} \ytableaushort{13,25,46,78}}}$ as $7$, $\mathcal{P}(S)=\scalemath{0.6}{\;{\ytableausetup{centertableaux} \ytableaushort{13,25,46,77}}}$ can be obtained. Hence the boundary condition of $\mathcal{P}(S)$ is $\lambda_{\mathcal{P}(S)}=(1,1,1,1,1,1,2)$, $\mathcal{P}^2(S)$ can also be obtained from $\mathcal{P}^3(T)$. 

\begin{example}
	For $T$ given in Figure \ref{fig:PTET}, we have
	\begin{equation*}
		\begin{aligned}
			\operatorname{prom}_1(T) &= 23756184,\quad \operatorname{Aexc}(\operatorname{prom}_1(T)) = \{1,4\},\\
			\operatorname{prom}_2(T) &= 37168425,\quad \operatorname{Aexc}(\operatorname{prom}_2(T)) = \{1,2,4,5\},\\
			\operatorname{prom}_3(T) &= 37168425,\quad \operatorname{Aexc}(\operatorname{prom}_3(T)) = \{1,2,3,4,5,7\}.
		\end{aligned}
	\end{equation*}
	For the hourglass plabic graph $G$ in Figure \ref{fig:extrip}, we have 
	\begin{equation*}
		\begin{aligned}
			\operatorname{trip}_1(G) &= 23756184= \operatorname{prom}_1(T),\\
			\operatorname{trip}_2(G) &= 37168425= \operatorname{prom}_2(T),\\
			\operatorname{trip}_3(G) &= 37168425= \operatorname{prom}_3(T).
		\end{aligned}
	\end{equation*}
	Then the Young tableau $T$ is the Young tableau corresponding to the hourglass plabic graph $G$.
	\begin{figure}[h]
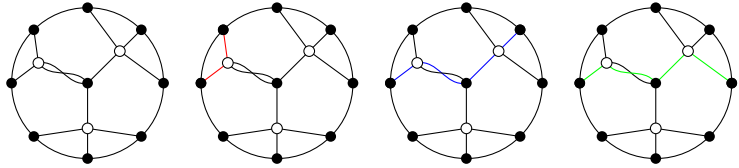

		\centering
		\GrDisk{
			\coordinate (A) at (45:0.6);
			\coordinate (B) at (-90:0.6);
			\coordinate (C) at (157.5:0.7);
			
			\coordinate (a) at (0:0);
			
			\draw (1)--(A);		\draw (2)--(A);		\draw (3)--(A);		\draw (4)--(B);
			\draw (5)--(B);		\draw (6)--(B);		\draw (7)--(C);		\draw (8)--(C);
			
			\draw (a)--(A);		\draw (a)--(B);		\hg{a}{C};
			
			\foreach \p in {A,B,C} {
				\fill[white] (\p) circle (2pt);
				\draw (\p) circle (2pt);
			}
			\fill (a) circle (2pt);
		}
		\GrDisk{
			\coordinate (A) at (45:0.6);
			\coordinate (B) at (-90:0.6);
			\coordinate (C) at (157.5:0.7);
			
			\coordinate (a) at (0:0);
			
			\draw (1)--(A);		\draw (2)--(A);		\draw (3)--(A);		\draw (4)--(B);
			\draw (5)--(B);		\draw (6)--(B);		\draw[red] (7)--(C);		\draw[red] (8)--(C);
			
			\draw (a)--(A);		\draw (a)--(B);		\hg{a}{C};
			
			\foreach \p in {A,B,C} {
				\fill[white] (\p) circle (2pt);
				\draw (\p) circle (2pt);
			}
			\fill (a) circle (2pt);
			\fill (7) circle (2pt);
			\fill (8) circle (2pt);
		}
		\GrDisk{
			\coordinate (A) at (45:0.6);
			\coordinate (B) at (-90:0.6);
			\coordinate (C) at (157.5:0.7);
			
			\coordinate (a) at (0:0);
			
			\draw (1)--(A);		\draw[blue] (2)--(A);		\draw (3)--(A);		\draw (4)--(B);
			\draw (5)--(B);		\draw (6)--(B);		\draw[blue] (7)--(C);		\draw (8)--(C);
			
			\draw[blue] (a)--(A);		\draw (a)--(B);		\hg[blue]{a}{C};
			
			\foreach \p in {A,B,C} {
				\fill[white] (\p) circle (2pt);
				\draw (\p) circle (2pt);
			}
			\fill (a) circle (2pt);
			\fill (2) circle (2pt);
			\fill (7) circle (2pt);
		}
		\GrDisk{
			\coordinate (A) at (45:0.6);
			\coordinate (B) at (-90:0.6);
			\coordinate (C) at (157.5:0.7);
			
			\coordinate (a) at (0:0);
			
			\draw (1)--(A);		\draw (2)--(A);		\draw[green] (3)--(A);		\draw (4)--(B);
			\draw (5)--(B);		\draw (6)--(B);		\draw[green] (7)--(C);		\draw (8)--(C);
			
			\draw[green] (a)--(A);		\draw (a)--(B);		\hg[][green]{a}{C};
			
			\foreach \p in {A,B,C} {
				\fill[white] (\p) circle (2pt);
				\draw (\p) circle (2pt);
			}
			\fill (a) circle (2pt);
			\fill (3) circle (2pt);
			\fill (7) circle (2pt);
		}
		\caption{An hourglass plabic graph $G$ as an example of $\operatorname{trip}_i(G)$, we can see that $\operatorname{trip}_1(G)(7)=8$(the red path in the second graph), $\operatorname{trip}_2(G)(7)=2$ (the blue path in the third graph) and $\operatorname{trip}_3(G)(7)=3$ (the green path in the fourth graph).}
		\label{fig:extrip}
	\end{figure}
\end{example}

\subsection{Growth algorithm}\label{ssc:WinGrowth}
In this subsection, we present the growth algorithm and growth rules introduced in \cite{gaetz2023rotation}. Using this method, we can construct an hourglass plabic graph associated with a given Young tableau. For a standard Young tableau $T$, where $L(T)=l_1l_2\cdots l_n$ is the \textit{lattice word} of $T$, the letter $l_i$ is the row index of the box labeled $i$ in $T$. If $T$ is a semi-standard Young tableau, we replace the letter $l_i$ in $L(T)$ with the multiset of row indices of all boxes labeled $i$.

\begin{figure}
	\centering
	\scalemath{0.8}{
		\begin{tikzpicture}[>=stealth]
			\centering
			\begin{scope}
				\draw[green!50!black, very thick] (0,0) rectangle (2,3); 
				\draw[thick] (0,1.5)--(2,1.5);
				
				\draw (0.5,2.75)--(0.5,2.25);
				\draw (0.5,2.25) .. controls (0.5,1.55) and (1.5,1.55) .. (1.5,2.25);
				\draw (1.5,2.75)--(1.5,2.25);
				\draw[->] (1,1.725)--(1.05,1.725);
				\node[left,red!70!black,font=\tiny] at (0.55,2.55) {$1$};
				\node[right,red!70!black,font=\tiny] at (1.45,2.55) {$\overline{1}$};
				
				\draw (0.5,1.25)--(0.5,0.75);
				\draw (0.5,0.75) .. controls (0.5,0.05) and (1.5,0.05) .. (1.5,0.75);
				\draw (1.5,1.25)--(1.5,0.75);
				\draw[->] (1,0.225)--(0.95,0.225);
				\node[left,red!70!black,font=\tiny] at (0.55,1.05) {$\overline{4}$};
				\node[right,red!70!black,font=\tiny] at (1.45,1.05) {$4$};
			\end{scope}
			\begin{scope}
				\draw[green!50!black, very thick] (2.2,0) rectangle (6.2,3); 
				\draw[thick] (2.2,1.5)--(6.2,1.5);
				\draw[thick] (4.2,3)--(4.2,0);
				
				\draw (2.7,2.75)--(3.7,1.75);
				\draw (3.7,2.75)--(2.7,1.75);
				\draw[->] (2.9,2.55)--(2.95,2.5);
				\draw[->] (3.5,1.95)--(3.55,1.9);
				\draw[->] (3.5,2.55)--(3.55,2.6);
				\draw[->] (2.9,1.95)--(2.95,2);
				\node[left,red!70!black,font=\tiny] at (2.9,2.5) {$1$};
				\node[right,red!70!black,font=\tiny] at (3.45,2.5) {$\overline{2}$};
				\node[left,red!70!black,font=\tiny] at (2.9,2) {$\overline{2}$};
				\node[right,red!70!black,font=\tiny] at (3.45,2) {$1$};
				
				\draw (4.7,2.75)--(5.7,1.75);
				\draw (5.7,2.75)--(4.7,1.75);
				\draw[->] (4.9,2.55)--(4.95,2.5);
				\draw[->] (5.5,1.95)--(5.55,1.9);
				\draw[->] (5.5,2.55)--(5.55,2.6);
				\draw[->] (4.9,1.95)--(4.95,2);
				\node[left,red!70!black,font=\tiny] at (4.9,2.5) {$2$};
				\node[right,red!70!black,font=\tiny] at (5.45,2.5) {$\overline{1}$};
				\node[left,red!70!black,font=\tiny] at (4.9,2) {$\overline{1}$};
				\node[right,red!70!black,font=\tiny] at (5.45,2) {$2$};
				
				\draw (2.7,1.25)--(3.7,0.25);
				\draw (3.7,1.25)--(2.7,0.25);
				\draw[->] (2.9,1.05)--(2.85,1.1);
				\draw[->] (3.5,0.45)--(3.45,0.5);
				\draw[->] (3.5,1.05)--(3.45,1);
				\draw[->] (2.9,0.45)--(2.85,0.4);
				\node[left,red!70!black,font=\tiny] at (2.9,1) {$\overline{3}$};
				\node[right,red!70!black,font=\tiny] at (3.45,1) {$4$};
				\node[left,red!70!black,font=\tiny] at (2.9,0.5) {$4$};
				\node[right,red!70!black,font=\tiny] at (3.45,0.5) {$\overline{3}$};
				
				\draw (4.7,1.25)--(5.7,0.25);
				\draw (5.7,1.25)--(4.7,0.25);
				\draw[->] (4.9,1.05)--(4.85,1.1);
				\draw[->] (5.5,0.45)--(5.45,0.5);
				\draw[->] (5.5,1.05)--(5.45,1);
				\draw[->] (4.9,0.45)--(4.85,0.4);
				\node[left,red!70!black,font=\tiny] at (4.9,1) {$\overline{4}$};
				\node[right,red!70!black,font=\tiny] at (5.45,1) {$3$};
				\node[left,red!70!black,font=\tiny] at (4.9,0.5) {$3$};
				\node[right,red!70!black,font=\tiny] at (5.45,0.5) {$\overline{4}$};
			\end{scope}
			\begin{scope}
				\draw[green!50!black, very thick] (6.4,0) rectangle (8.4,3); 
				\draw[thick] (6.4,1.5)--(8.4,1.5);
				
				\draw (6.9,2.75)--(7.9,1.75);
				\draw (7.9,2.75)--(6.9,1.75);
				\draw[->] (7.1,2.55)--(7.15,2.5);
				\draw[->] (7.7,1.95)--(7.75,1.9);
				\draw[->] (7.7,2.55)--(7.75,2.6);
				\draw[->] (7.1,1.95)--(7.15,2);
				\node[left,red!70!black,font=\tiny] at (7.1,2.5) {$2$};
				\node[right,red!70!black,font=\tiny] at (7.65,2.5) {$\overline{2}$};
				\node[left,red!70!black,font=\tiny] at (7.1,2) {$\overline{1}$};
				\node[right,red!70!black,font=\tiny] at (7.65,2) {$1$};
				
				\draw (6.9,1.25)--(7.9,0.25);
				\draw (7.9,1.25)--(6.9,0.25);
				\draw[->] (7.1,1.05)--(7.05,1.1);
				\draw[->] (7.7,0.45)--(7.65,0.5);
				\draw[->] (7.7,1.05)--(7.65,1);
				\draw[->] (7.1,0.45)--(7.05,0.4);
				\node[left,red!70!black,font=\tiny] at (7.1,1) {$\overline{3}$};
				\node[right,red!70!black,font=\tiny] at (7.65,1) {$3$};
				\node[left,red!70!black,font=\tiny] at (7.1,0.5) {$4$};
				\node[right,red!70!black,font=\tiny] at (7.65,0.5) {$\overline{4}$};
			\end{scope}
			\begin{scope}
				\draw[green!50!black, very thick] (8.6,0) rectangle (13.6,3); 
				\draw[thick] (8.6,1.5)--(13.6,1.5);
				\draw[thick] (11.1,0)--(11.1,3);
				
				\draw (9.1,2.75)--(10.1,1.75);
				\draw (10.1,2.75)--(9.1,1.75);
				\draw (10.6,2.75)--(10.6,1.75);
				\draw[->] (9.3,2.55)--(9.35,2.5);
				\draw[->] (9.9,1.95)--(9.95,1.9);
				\draw[->] (9.9,2.55)--(9.95,2.6);
				\draw[->] (9.3,1.95)--(9.35,2);
				\node[left,red!70!black,font=\tiny] at (9.3,2.5) {$1$};
				\node[right,red!70!black,font=\tiny] at (9.85,2.5) {$\overline{3}$};
				\node[left,red!70!black,font=\tiny] at (9.3,2) {$\overline{3}$};
				\node[right,red!70!black,font=\tiny] at (9.85,2) {$1$};
				\node[left,red!70!black,font=\tiny] at (10.65,2.4) {$\overline{1}$};
				\node[left,red!70!black,font=\tiny] at (10.65,2.1) {$2$};
				
				\draw (12.1,2.75)--(13.1,1.75);
				\draw (13.1,2.75)--(12.1,1.75);
				\draw (11.7,2.75)--(11.7,1.75);
				\draw[->] (12.3,2.55)--(12.35,2.5);
				\draw[->] (12.9,1.95)--(12.95,1.9);
				\draw[->] (12.9,2.55)--(12.95,2.6);
				\draw[->] (12.3,1.95)--(12.35,2);
				\node[left,red!70!black,font=\tiny] at (12.3,2.5) {$3$};
				\node[right,red!70!black,font=\tiny] at (12.85,2.5) {$\overline{1}$};
				\node[left,red!70!black,font=\tiny] at (12.3,2) {$\overline{1}$};
				\node[right,red!70!black,font=\tiny] at (12.85,2) {$3$};
				\node[left,red!70!black,font=\tiny] at (11.75,2.4) {$1$};
				\node[left,red!70!black,font=\tiny] at (11.75,2.1) {$\overline{2}$};
				
				\draw (9.6,1.25)--(10.6,0.25);
				\draw (10.6,1.25)--(9.6,0.25);
				\draw (9.2,1.25)--(9.2,0.25);
				\draw[->] (9.8,1.05)--(9.75,1.1);
				\draw[->] (10.4,0.45)--(10.35,0.5);
				\draw[->] (10.4,1.05)--(10.35,1);
				\draw[->] (9.8,0.45)--(9.75,0.4);
				\node[left,red!70!black,font=\tiny] at (9.8,1) {$\overline{2}$};
				\node[right,red!70!black,font=\tiny] at (10.35,1) {$4$};
				\node[left,red!70!black,font=\tiny] at (9.8,0.5) {$4$};
				\node[right,red!70!black,font=\tiny] at (10.35,0.5) {$\overline{2}$};
				\node[left,red!70!black,font=\tiny] at (9.25,0.9) {$\overline{4}$};
				\node[left,red!70!black,font=\tiny] at (9.25,0.6) {$3$};
				
				\draw (11.6,1.25)--(12.6,0.25);
				\draw (12.6,1.25)--(11.6,0.25);
				\draw (13.1,1.25)--(13.1,0.25);
				\draw[->] (11.8,1.05)--(11.75,1.1);
				\draw[->] (12.4,0.45)--(12.35,0.5);
				\draw[->] (12.4,1.05)--(12.35,1);
				\draw[->] (11.8,0.45)--(11.75,0.4);
				\node[left,red!70!black,font=\tiny] at (11.8,1) {$\overline{4}$};
				\node[right,red!70!black,font=\tiny] at (12.35,1) {$2$};
				\node[left,red!70!black,font=\tiny] at (11.8,0.5) {$2$};
				\node[right,red!70!black,font=\tiny] at (12.35,0.5) {$\overline{4}$};
				\node[left,red!70!black,font=\tiny] at (13.15,0.9) {$4$};
				\node[left,red!70!black,font=\tiny] at (13.15,0.6) {$\overline{3}$};
			\end{scope}
	\end{tikzpicture}}
	\\\vspace{0.1cm}
	\scalemath{0.8}{
		\begin{tikzpicture}[>=stealth]
			\centering
			\begin{scope}
				\draw[green!50!black, very thick] (0,0) rectangle (5,3); 
				\draw[thick] (0,1.5)--(5,1.5);
				\draw[thick] (2.5,0)--(2.5,3);
				
				\draw (0.5,2.75)--(1.5,1.75);
				\draw (1.5,2.75)--(0.5,1.75);
				\draw (2,2.75)--(2,1.75);
				\draw[->] (0.7,2.55)--(0.75,2.5);
				\draw[->] (1.3,2.55)--(1.25,2.5);
				\draw[->] (0.7,1.95)--(0.65,1.9);
				\draw[->] (1.3,1.95)--(1.35,1.9);
				\node[left,red!70!black,font=\tiny] at (0.7,2.5) {$1$};
				\node[right,red!70!black,font=\tiny] at (1.25,2.5) {$2$};
				\node[left,red!70!black,font=\tiny] at (0.7,2) {$2$};
				\node[right,red!70!black,font=\tiny] at (1.25,2) {$1$};
				\node[left,red!70!black,font=\tiny] at (2.05,2.4) {$\overline{1}$};
				\node[left,red!70!black,font=\tiny] at (2.05,2.1) {$2$};
				
				\draw (3.5,2.75)--(4.5,1.75);
				\draw (4.5,2.75)--(3.5,1.75);
				\draw (3.1,2.75)--(3.1,1.75);
				\draw[->] (3.7,2.55)--(3.65,2.6);
				\draw[->] (4.3,2.55)--(4.35,2.6);
				\draw[->] (3.7,1.95)--(3.75,2);
				\draw[->] (4.3,1.95)--(4.25,2);
				\node[left,red!70!black,font=\tiny] at (3.7,2.5) {$\overline{2}$};
				\node[right,red!70!black,font=\tiny] at (4.25,2.5) {$\overline{1}$};
				\node[left,red!70!black,font=\tiny] at (3.7,2) {$\overline{1}$};
				\node[right,red!70!black,font=\tiny] at (4.25,2) {$\overline{2}$};
				\node[left,red!70!black,font=\tiny] at (3.15,2.4) {$1$};
				\node[left,red!70!black,font=\tiny] at (3.15,2.1) {$\overline{2}$};
				
				\draw (1,1.25)--(2,0.25);
				\draw (2,1.25)--(1,0.25);
				\draw (0.6,1.25)--(0.6,0.25);
				\draw[->] (1.2,1.05)--(1.25,1);
				\draw[->] (1.8,1.05)--(1.75,1);
				\draw[->] (1.2,0.45)--(1.15,0.4);
				\draw[->] (1.8,0.45)--(1.85,0.4);
				\node[left,red!70!black,font=\tiny] at (1.2,1) {$3$};
				\node[right,red!70!black,font=\tiny] at (1.75,1) {$4$};
				\node[left,red!70!black,font=\tiny] at (1.2,0.5) {$4$};
				\node[right,red!70!black,font=\tiny] at (1.75,0.5) {$3$};
				\node[left,red!70!black,font=\tiny] at (0.65,0.9) {$\overline{4}$};
				\node[left,red!70!black,font=\tiny] at (0.65,0.6) {$3$};
				
				\draw (3,1.25)--(4,0.25);
				\draw (4,1.25)--(3,0.25);
				\draw (4.5,1.25)--(4.5,0.25);
				\draw[->] (3.2,1.05)--(3.15,1.1);
				\draw[->] (3.8,1.05)--(3.85,1.1);
				\draw[->] (3.2,0.45)--(3.25,0.5);
				\draw[->] (3.8,0.45)--(3.75,0.5);
				\node[left,red!70!black,font=\tiny] at (3.2,1) {$\overline{4}$};
				\node[right,red!70!black,font=\tiny] at (3.75,1) {$\overline{3}$};
				\node[left,red!70!black,font=\tiny] at (3.2,0.5) {$\overline{3}$};
				\node[right,red!70!black,font=\tiny] at (3.75,0.5) {$\overline{4}$};
				\node[left,red!70!black,font=\tiny] at (4.55,0.9) {$4$};
				\node[left,red!70!black,font=\tiny] at (4.55,0.6) {$\overline{3}$};
			\end{scope}
			\begin{scope}
				\draw[green!50!black, very thick] (5.2,0) rectangle (10.2,3); 
				\draw[thick] (5.2,1.5)--(10.2,1.5);
				\draw[thick] (7.7,0)--(7.7,3);
				
				\draw (5.7,2.75)--(6.7,1.75);
				\draw (6.7,2.75)--(5.7,1.75);
				\draw (7.2,2.75)--(7.2,1.75);
				\draw[->] (5.9,2.55)--(5.95,2.5);
				\draw[->] (6.5,2.55)--(6.45,2.5);
				\draw[->] (5.9,1.95)--(5.85,1.9);
				\draw[->] (6.5,1.95)--(6.55,1.9);
				\node[left,red!70!black,font=\tiny] at (5.9,2.5) {$1$};
				\node[right,red!70!black,font=\tiny] at (6.45,2.5) {$3$};
				\node[left,red!70!black,font=\tiny] at (5.9,2) {$3$};
				\node[right,red!70!black,font=\tiny] at (6.45,2) {$1$};
				\node[left,red!70!black,font=\tiny] at (7.25,2.55) {$\overline{1}$};
				\node[left,red!70!black,font=\tiny] at (7.25,2.25) {$2$};
				\node[left,red!70!black,font=\tiny] at (7.25,1.95) {$3$};
				
				\draw (8.7,2.75)--(9.7,1.75);
				\draw (9.7,2.75)--(8.7,1.75);
				\draw (8.3,2.75)--(8.3,1.75);
				\draw[->] (8.9,2.55)--(8.85,2.6);
				\draw[->] (9.5,2.55)--(9.55,2.6);
				\draw[->] (8.9,1.95)--(8.95,2);
				\draw[->] (9.5,1.95)--(9.45,2);
				\node[left,red!70!black,font=\tiny] at (8.9,2.5) {$\overline{3}$};
				\node[right,red!70!black,font=\tiny] at (9.45,2.5) {$\overline{1}$};
				\node[left,red!70!black,font=\tiny] at (8.9,2) {$\overline{1}$};
				\node[right,red!70!black,font=\tiny] at (9.45,2) {$\overline{3}$};
				\node[left,red!70!black,font=\tiny] at (8.35,2.55) {$1$};
				\node[left,red!70!black,font=\tiny] at (8.35,2.25) {$\overline{2}$};
				\node[left,red!70!black,font=\tiny] at (8.35,1.95) {$\overline{3}$};
				
				\draw (6.2,1.25)--(7.2,0.25);
				\draw (7.2,1.25)--(6.2,0.25);
				\draw (5.8,1.25)--(5.8,0.25);
				\draw[->] (6.4,1.05)--(6.45,1);
				\draw[->] (7,1.05)--(6.95,1);
				\draw[->] (6.4,0.45)--(6.35,0.4);
				\draw[->] (7,0.45)--(7.05,0.4);
				\node[left,red!70!black,font=\tiny] at (6.4,1) {$2$};
				\node[right,red!70!black,font=\tiny] at (6.95,1) {$4$};
				\node[left,red!70!black,font=\tiny] at (6.4,0.5) {$4$};
				\node[right,red!70!black,font=\tiny] at (6.95,0.5) {$2$};
				\node[left,red!70!black,font=\tiny] at (5.85,1.05) {$\overline{4}$};
				\node[left,red!70!black,font=\tiny] at (5.85,0.75) {$3$};
				\node[left,red!70!black,font=\tiny] at (5.85,0.45) {$2$};
				
				\draw (8.2,1.25)--(9.2,0.25);
				\draw (9.2,1.25)--(8.2,0.25);
				\draw (9.7,1.25)--(9.7,0.25);
				\draw[->] (8.4,1.05)--(8.35,1.1);
				\draw[->] (9,1.05)--(9.05,1.1);
				\draw[->] (8.4,0.45)--(8.45,0.5);
				\draw[->] (9,0.45)--(8.95,0.5);
				\node[left,red!70!black,font=\tiny] at (8.4,1) {$\overline{4}$};
				\node[right,red!70!black,font=\tiny] at (8.95,1) {$\overline{2}$};
				\node[left,red!70!black,font=\tiny] at (8.4,0.5) {$\overline{2}$};
				\node[right,red!70!black,font=\tiny] at (8.95,0.5) {$\overline{4}$};
				\node[left,red!70!black,font=\tiny] at (9.75,1.05) {$4$};
				\node[left,red!70!black,font=\tiny] at (9.75,0.75) {$\overline{3}$};
				\node[left,red!70!black,font=\tiny] at (9.75,0.45) {$\overline{2}$};
			\end{scope}
			\begin{scope}
				\draw[green!50!black, very thick] (10.4,0) rectangle (15.4,3); 
				\draw[thick] (10.4,1.5)--(15.4,1.5);
				\draw[thick] (12.9,0)--(12.9,3);
				
				\draw (10.9,2.75)--(11.9,1.75);
				\draw (11.9,2.75)--(10.9,1.75);
				\draw (12.4,2.75)--(12.4,1.75);
				\draw[->] (11.1,2.55)--(11.15,2.5);
				\draw[->] (11.7,2.55)--(11.65,2.5);
				\draw[->] (11.1,1.95)--(11.05,1.9);
				\draw[->] (11.7,1.95)--(11.75,1.9);
				\node[left,red!70!black,font=\tiny] at (11.1,2.5) {$1$};
				\node[right,red!70!black,font=\tiny] at (11.65,2.5) {$4$};
				\node[left,red!70!black,font=\tiny] at (11.1,2) {$4$};
				\node[right,red!70!black,font=\tiny] at (11.65,2) {$1$};
				\node[left,red!70!black,font=\tiny] at (12.45,2.7) {$\overline{1}$};
				\node[left,red!70!black,font=\tiny] at (12.45,2.4) {$2$};
				\node[left,red!70!black,font=\tiny] at (12.45,2.1) {$3$};
				\node[left,red!70!black,font=\tiny] at (12.45,1.8) {$4$};
				
				\draw (13.9,2.75)--(14.9,1.75);
				\draw (14.9,2.75)--(13.9,1.75);
				\draw (13.5,2.75)--(13.5,1.75);
				\draw[->] (14.1,2.55)--(14.05,2.6);
				\draw[->] (14.7,2.55)--(14.75,2.6);
				\draw[->] (14.1,1.95)--(14.15,2);
				\draw[->] (14.7,1.95)--(14.65,2);
				\node[left,red!70!black,font=\tiny] at (14.1,2.5) {$\overline{4}$};
				\node[right,red!70!black,font=\tiny] at (14.65,2.5) {$\overline{1}$};
				\node[left,red!70!black,font=\tiny] at (14.1,2) {$\overline{1}$};
				\node[right,red!70!black,font=\tiny] at (14.65,2) {$\overline{4}$};
				\node[left,red!70!black,font=\tiny] at (13.55,2.7) {$1$};
				\node[left,red!70!black,font=\tiny] at (13.55,2.4) {$\overline{2}$};
				\node[left,red!70!black,font=\tiny] at (13.55,2.1) {$\overline{3}$};
				\node[left,red!70!black,font=\tiny] at (13.55,1.8) {$\overline{4}$};
				
				\draw (11.4,1.25)--(12.4,0.25);
				\draw (12.4,1.25)--(11.4,0.25);
				\draw (11,1.25)--(11,0.25);
				\draw[->] (11.6,1.05)--(11.65,1);
				\draw[->] (12.2,1.05)--(12.15,1);
				\draw[->] (11.6,0.45)--(11.55,0.4);
				\draw[->] (12.2,0.45)--(12.25,0.4);
				\node[left,red!70!black,font=\tiny] at (11.6,1) {$1$};
				\node[right,red!70!black,font=\tiny] at (12.15,1) {$4$};
				\node[left,red!70!black,font=\tiny] at (11.6,0.5) {$4$};
				\node[right,red!70!black,font=\tiny] at (12.15,0.5) {$1$};
				\node[left,red!70!black,font=\tiny] at (11.05,1.2) {$\overline{4}$};
				\node[left,red!70!black,font=\tiny] at (11.05,0.9) {$3$};
				\node[left,red!70!black,font=\tiny] at (11.05,0.6) {$2$};
				\node[left,red!70!black,font=\tiny] at (11.05,0.3) {$1$};
				
				\draw (13.4,1.25)--(14.4,0.25);
				\draw (14.4,1.25)--(13.4,0.25);
				\draw (14.9,1.25)--(14.9,0.25);
				\draw[->] (13.6,1.05)--(13.55,1.1);
				\draw[->] (14.2,1.05)--(14.25,1.1);
				\draw[->] (13.6,0.45)--(13.65,0.5);
				\draw[->] (14.2,0.45)--(14.15,0.5);
				\node[left,red!70!black,font=\tiny] at (13.6,1) {$\overline{4}$};
				\node[right,red!70!black,font=\tiny] at (14.15,1) {$\overline{1}$};
				\node[left,red!70!black,font=\tiny] at (13.6,0.5) {$\overline{1}$};
				\node[right,red!70!black,font=\tiny] at (14.15,0.5) {$\overline{4}$};
				\node[left,red!70!black,font=\tiny] at (14.95,1.2) {$4$};
				\node[left,red!70!black,font=\tiny] at (14.95,0.9) {$\overline{3}$};
				\node[left,red!70!black,font=\tiny] at (14.95,0.6) {$\overline{2}$};
				\node[left,red!70!black,font=\tiny] at (14.95,0.3) {$\overline{1}$};
			\end{scope}
	\end{tikzpicture}}
	\\\vspace{0.1cm}
	\scalemath{0.8}{
		\begin{tikzpicture}[>=stealth]
			\centering
			\begin{scope}
				\draw[green!50!black, very thick] (0,0) rectangle (5,3); 
				\draw[thick] (0,1.5)--(5,1.5);
				\draw[thick] (2.5,0)--(2.5,3);
				
				\draw (0.5,2.75)--(1.5,1.75);
				\draw (1.5,2.75)--(0.5,1.75);
				\draw (2,2.75)--(2,1.75);
				\draw[->] (0.7,2.55)--(0.75,2.5);
				\draw[->] (1.3,2.55)--(1.25,2.5);
				\draw[->] (0.7,1.95)--(0.75,2);
				\draw[->] (1.3,1.95)--(1.25,2);
				\node[left,red!70!black,font=\tiny] at (0.7,2.5) {$1$};
				\node[right,red!70!black,font=\tiny] at (1.25,2.5) {$2$};
				\node[left,red!70!black,font=\tiny] at (0.7,2) {$\overline{3}$};
				\node[right,red!70!black,font=\tiny] at (1.25,2) {$\overline{4}$};
				\node[left,red!70!black,font=\tiny] at (2.05,2.4) {$\overline{3}$};
				\node[left,red!70!black,font=\tiny] at (2.05,2.1) {$4$};
				
				\draw (3.5,2.75)--(4.5,1.75);
				\draw (4.5,2.75)--(3.5,1.75);
				\draw (3.1,2.75)--(3.1,1.75);
				\draw[->] (3.7,2.55)--(3.65,2.6);
				\draw[->] (4.3,2.55)--(4.35,2.6);
				\draw[->] (3.7,1.95)--(3.65,1.9);
				\draw[->] (4.3,1.95)--(4.35,1.9);
				\node[left,red!70!black,font=\tiny] at (3.7,2.5) {$\overline{2}$};
				\node[right,red!70!black,font=\tiny] at (4.25,2.5) {$\overline{1}$};
				\node[left,red!70!black,font=\tiny] at (3.7,2) {$4$};
				\node[right,red!70!black,font=\tiny] at (4.25,2) {$3$};
				\node[left,red!70!black,font=\tiny] at (3.15,2.4) {$3$};
				\node[left,red!70!black,font=\tiny] at (3.15,2.1) {$\overline{4}$};
				
				\draw (1,1.25)--(2,0.25);
				\draw (2,1.25)--(1,0.25);
				\draw (0.6,1.25)--(0.6,0.25);
				\draw[->] (1.2,1.05)--(1.25,1);
				\draw[->] (1.8,1.05)--(1.75,1);
				\draw[->] (1.2,0.45)--(1.25,0.5);
				\draw[->] (1.8,0.45)--(1.75,0.5);
				\node[left,red!70!black,font=\tiny] at (1.2,1) {$3$};
				\node[right,red!70!black,font=\tiny] at (1.75,1) {$4$};
				\node[left,red!70!black,font=\tiny] at (1.2,0.5) {$\overline{1}$};
				\node[right,red!70!black,font=\tiny] at (1.75,0.5) {$\overline{2}$};
				\node[left,red!70!black,font=\tiny] at (0.65,0.9) {$\overline{2}$};
				\node[left,red!70!black,font=\tiny] at (0.65,0.6) {$1$};
				
				\draw (3,1.25)--(4,0.25);
				\draw (4,1.25)--(3,0.25);
				\draw (4.5,1.25)--(4.5,0.25);
				\draw[->] (3.2,1.05)--(3.15,1.1);
				\draw[->] (3.8,1.05)--(3.85,1.1);
				\draw[->] (3.2,0.45)--(3.15,0.4);
				\draw[->] (3.8,0.45)--(3.85,0.4);
				\node[left,red!70!black,font=\tiny] at (3.2,1) {$\overline{4}$};
				\node[right,red!70!black,font=\tiny] at (3.75,1) {$\overline{3}$};
				\node[left,red!70!black,font=\tiny] at (3.2,0.5) {$2$};
				\node[right,red!70!black,font=\tiny] at (3.75,0.5) {$1$};
				\node[left,red!70!black,font=\tiny] at (4.55,0.9) {$2$};
				\node[left,red!70!black,font=\tiny] at (4.55,0.6) {$\overline{1}$};
			\end{scope}
			\begin{scope}
				\draw[green!50!black, very thick] (5.2,0) rectangle (10.2,3); 
				\draw[thick] (5.2,1.5)--(10.2,1.5);
				\draw[thick] (7.7,0)--(7.7,3);
				
				\draw (5.7,2.75)--(6.7,1.75);
				\draw (6.7,2.75)--(5.7,1.75);
				\draw (7.2,2.75)--(7.2,1.75);
				\draw[->] (5.9,2.55)--(5.95,2.5);
				\draw[->] (6.5,2.55)--(6.45,2.5);
				\draw[->] (5.9,1.95)--(5.95,2);
				\draw[->] (6.5,1.95)--(6.45,2);
				\node[left,red!70!black,font=\tiny] at (5.9,2.5) {$1$};
				\node[right,red!70!black,font=\tiny] at (6.45,2.5) {$3$};
				\node[left,red!70!black,font=\tiny] at (5.9,2) {$\overline{2}$};
				\node[right,red!70!black,font=\tiny] at (6.45,2) {$\overline{4}$};
				\node[left,red!70!black,font=\tiny] at (7.25,2.55) {$\overline{2}$};
				\node[left,red!70!black,font=\tiny] at (7.25,2.25) {$\overline{3}$};
				\node[left,red!70!black,font=\tiny] at (7.25,1.95) {$4$};
				
				\draw (8.7,2.75)--(9.7,1.75);
				\draw (9.7,2.75)--(8.7,1.75);
				\draw (8.3,2.75)--(8.3,1.75);
				\draw[->] (8.9,2.55)--(8.85,2.6);
				\draw[->] (9.5,2.55)--(9.55,2.6);
				\draw[->] (8.9,1.95)--(8.85,1.9);
				\draw[->] (9.5,1.95)--(9.55,1.9);
				\node[left,red!70!black,font=\tiny] at (8.9,2.5) {$\overline{3}$};
				\node[right,red!70!black,font=\tiny] at (9.45,2.5) {$\overline{1}$};
				\node[left,red!70!black,font=\tiny] at (8.9,2) {$4$};
				\node[right,red!70!black,font=\tiny] at (9.45,2) {$2$};
				\node[left,red!70!black,font=\tiny] at (8.35,2.55) {$2$};
				\node[left,red!70!black,font=\tiny] at (8.35,2.25) {$3$};
				\node[left,red!70!black,font=\tiny] at (8.35,1.95) {$\overline{4}$};
				
				\draw (6.2,1.25)--(7.2,0.25);
				\draw (7.2,1.25)--(6.2,0.25);
				\draw (5.8,1.25)--(5.8,0.25);
				\draw[->] (6.4,1.05)--(6.45,1);
				\draw[->] (7,1.05)--(6.95,1);
				\draw[->] (6.4,0.45)--(6.45,0.5);
				\draw[->] (7,0.45)--(6.95,0.5);
				\node[left,red!70!black,font=\tiny] at (6.4,1) {$2$};
				\node[right,red!70!black,font=\tiny] at (6.95,1) {$4$};
				\node[left,red!70!black,font=\tiny] at (6.4,0.5) {$\overline{1}$};
				\node[right,red!70!black,font=\tiny] at (6.95,0.5) {$\overline{3}$};
				\node[left,red!70!black,font=\tiny] at (5.85,1.05) {$\overline{3}$};
				\node[left,red!70!black,font=\tiny] at (5.85,0.75) {$\overline{2}$};
				\node[left,red!70!black,font=\tiny] at (5.85,0.45) {$1$};
				
				\draw (8.2,1.25)--(9.2,0.25);
				\draw (9.2,1.25)--(8.2,0.25);
				\draw (9.7,1.25)--(9.7,0.25);
				\draw[->] (8.4,1.05)--(8.35,1.1);
				\draw[->] (9,1.05)--(9.05,1.1);
				\draw[->] (8.4,0.45)--(8.35,0.4);
				\draw[->] (9,0.45)--(9.05,0.4);
				\node[left,red!70!black,font=\tiny] at (8.4,1) {$\overline{4}$};
				\node[right,red!70!black,font=\tiny] at (8.95,1) {$\overline{2}$};
				\node[left,red!70!black,font=\tiny] at (8.4,0.5) {$3$};
				\node[right,red!70!black,font=\tiny] at (8.95,0.5) {$1$};
				\node[left,red!70!black,font=\tiny] at (9.75,1.05) {$3$};
				\node[left,red!70!black,font=\tiny] at (9.75,0.75) {$2$};
				\node[left,red!70!black,font=\tiny] at (9.75,0.45) {$\overline{1}$};
			\end{scope}
			\begin{scope}
				\draw[green!50!black, very thick] (10.4,0) rectangle (15.4,3); 
				\draw[thick] (10.4,1.5)--(15.4,1.5);
				\draw[thick] (12.9,0)--(12.9,3);
				
				\draw (10.9,2.75)--(11.9,1.75);
				\draw (11.9,2.75)--(10.9,1.75);
				\draw (12.4,2.75)--(12.4,1.75);
				\draw[->] (11.1,2.55)--(11.15,2.5);
				\draw[->] (11.7,2.55)--(11.65,2.5);
				\draw[->] (11.1,1.95)--(11.15,2);
				\draw[->] (11.7,1.95)--(11.65,2);
				\node[left,red!70!black,font=\tiny] at (11.1,2.5) {$2$};
				\node[right,red!70!black,font=\tiny] at (11.65,2.5) {$3$};
				\node[left,red!70!black,font=\tiny] at (11.1,2) {$\overline{1}$};
				\node[right,red!70!black,font=\tiny] at (11.65,2) {$\overline{4}$};
				\node[left,red!70!black,font=\tiny] at (12.45,2.7) {$\overline{1}$};
				\node[left,red!70!black,font=\tiny] at (12.45,2.4) {$\overline{2}$};
				\node[left,red!70!black,font=\tiny] at (12.45,2.1) {$\overline{3}$};
				\node[left,red!70!black,font=\tiny] at (12.45,1.8) {$4$};
				
				\draw (13.9,2.75)--(14.9,1.75);
				\draw (14.9,2.75)--(13.9,1.75);
				\draw (13.5,2.75)--(13.5,1.75);
				\draw[->] (14.1,2.55)--(14.05,2.6);
				\draw[->] (14.7,2.55)--(14.75,2.6);
				\draw[->] (14.1,1.95)--(14.05,1.9);
				\draw[->] (14.7,1.95)--(14.75,1.9);
				\node[left,red!70!black,font=\tiny] at (14.1,2.5) {$\overline{3}$};
				\node[right,red!70!black,font=\tiny] at (14.65,2.5) {$\overline{2}$};
				\node[left,red!70!black,font=\tiny] at (14.1,2) {$4$};
				\node[right,red!70!black,font=\tiny] at (14.65,2) {$1$};
				\node[left,red!70!black,font=\tiny] at (13.55,2.7) {$1$};
				\node[left,red!70!black,font=\tiny] at (13.55,2.4) {$2$};
				\node[left,red!70!black,font=\tiny] at (13.55,2.1) {$3$};
				\node[left,red!70!black,font=\tiny] at (13.55,1.8) {$\overline{4}$};
				
				\draw (11.4,1.25)--(12.4,0.25);
				\draw (12.4,1.25)--(11.4,0.25);
				\draw (11,1.25)--(11,0.25);
				\draw[->] (11.6,1.05)--(11.65,1);
				\draw[->] (12.2,1.05)--(12.15,1);
				\draw[->] (11.6,0.45)--(11.65,0.5);
				\draw[->] (12.2,0.45)--(12.15,0.5);
				\node[left,red!70!black,font=\tiny] at (11.6,1) {$2$};
				\node[right,red!70!black,font=\tiny] at (12.15,1) {$3$};
				\node[left,red!70!black,font=\tiny] at (11.6,0.5) {$\overline{1}$};
				\node[right,red!70!black,font=\tiny] at (12.15,0.5) {$\overline{4}$};
				\node[left,red!70!black,font=\tiny] at (11.05,1.2) {$\overline{4}$};
				\node[left,red!70!black,font=\tiny] at (11.05,0.9) {$\overline{3}$};
				\node[left,red!70!black,font=\tiny] at (11.05,0.6) {$\overline{2}$};
				\node[left,red!70!black,font=\tiny] at (11.05,0.3) {$1$};
				
				\draw (13.4,1.25)--(14.4,0.25);
				\draw (14.4,1.25)--(13.4,0.25);
				\draw (14.9,1.25)--(14.9,0.25);
				\draw[->] (13.6,1.05)--(13.55,1.1);
				\draw[->] (14.2,1.05)--(14.25,1.1);
				\draw[->] (13.6,0.45)--(13.55,0.4);
				\draw[->] (14.2,0.45)--(14.25,0.4);
				\node[left,red!70!black,font=\tiny] at (13.6,1) {$\overline{3}$};
				\node[right,red!70!black,font=\tiny] at (14.15,1) {$\overline{2}$};
				\node[left,red!70!black,font=\tiny] at (13.6,0.5) {$4$};
				\node[right,red!70!black,font=\tiny] at (14.15,0.5) {$1$};
				\node[left,red!70!black,font=\tiny] at (14.95,1.2) {$4$};
				\node[left,red!70!black,font=\tiny] at (14.95,0.9) {$3$};
				\node[left,red!70!black,font=\tiny] at (14.95,0.6) {$2$};
				\node[left,red!70!black,font=\tiny] at (14.95,0.3) {$\overline{1}$};			\end{scope}
	\end{tikzpicture}}
	\\\vspace{0.1cm}
	\scalemath{0.8}{
		\begin{tikzpicture}[>=stealth]
			\centering
			\begin{scope}
				\draw[red!80!black, very thick] (0,0) rectangle (7.6,3); 
				\draw[thick] (0,1.5)--(7.6,1.5);
				\draw[thick] (3.8,0)--(3.8,3);
				
				\draw (0.5,2.75)--(1.5,1.75);
				\draw (1.5,2.75)--(0.5,1.75);
				\draw (2.25,2.75)--(2.25,1.75);
				\draw (3.3,2.75)--(3.3,1.75);
				\draw[->] (0.7,2.55)--(0.75,2.5);
				\draw[->] (1.3,2.55)--(1.25,2.5);
				\draw[->] (0.7,1.95)--(0.65,1.9);
				\draw[->] (1.3,1.95)--(1.35,1.9);
				\node[left,red!70!black,font=\tiny] at (0.7,2.5) {$1$};
				\node[right,red!70!black,font=\tiny] at (1.25,2.5) {$4$};
				\node[left,red!70!black,font=\tiny] at (0.7,2) {$4$};
				\node[right,red!70!black,font=\tiny] at (1.25,2) {$1$};
				\node at (1.8,2.25) {$\Biggl($};
				\node[left,red!70!black,font=\tiny] at (2.3,2.4) {$\overline{2}$};
				\node[left,red!70!black,font=\tiny] at (2.3,2.1) {$\overline{3}$};
				\node at (2.55,2.25) {$\Biggr)$};
				\node at (2.7,2.7) {$\star$};
				\node[left,red!70!black,font=\tiny] at (3.35,2.7) {$\overline{1}$};
				\node[left,red!70!black,font=\tiny] at (3.35,2.4) {$2$};
				\node[left,red!70!black,font=\tiny] at (3.35,2.1) {$3$};
				\node[left,red!70!black,font=\tiny] at (3.35,1.8) {$4$};
				
				\draw (6.1,2.75)--(7.1,1.75);
				\draw (7.1,2.75)--(6.1,1.75);
				\draw (5.25,2.75)--(5.25,1.75);
				\draw (4.4,2.75)--(4.4,1.75);
				\draw[->] (6.3,2.55)--(6.25,2.6);
				\draw[->] (6.9,2.55)--(6.95,2.6);
				\draw[->] (6.3,1.95)--(6.35,2);
				\draw[->] (6.9,1.95)--(6.85,2);
				\node[left,red!70!black,font=\tiny] at (6.3,2.5) {$\overline{4}$};
				\node[right,red!70!black,font=\tiny] at (6.85,2.5) {$\overline{1}$};
				\node[left,red!70!black,font=\tiny] at (6.3,2) {$\overline{1}$};
				\node[right,red!70!black,font=\tiny] at (6.85,2) {$\overline{4}$};
				\node at (4.8,2.25) {$\Biggl($};
				\node[left,red!70!black,font=\tiny] at (5.3,2.4) {$2$};
				\node[left,red!70!black,font=\tiny] at (5.3,2.1) {$3$};
				\node at (5.55,2.25) {$\Biggr)$};
				\node at (5.7,2.7) {$\star$};
				\node[left,red!70!black,font=\tiny] at (4.45,2.7) {$1$};
				\node[left,red!70!black,font=\tiny] at (4.45,2.4) {$\overline{2}$};
				\node[left,red!70!black,font=\tiny] at (4.45,2.1) {$\overline{3}$};
				\node[left,red!70!black,font=\tiny] at (4.45,1.8) {$\overline{4}$};
				
				\draw (2.3,1.25)--(3.3,0.25);
				\draw (3.3,1.25)--(2.3,0.25);
				\draw (1.45,1.25)--(1.45,0.25);
				\draw (0.6,1.25)--(0.6,0.25);
				\draw[->] (2.5,1.05)--(2.55,1);
				\draw[->] (3.1,1.05)--(3.05,1);
				\draw[->] (2.5,0.45)--(2.45,0.4);
				\draw[->] (3.1,0.45)--(3.15,0.4);
				\node[left,red!70!black,font=\tiny] at (2.5,1.0) {$1$};
				\node[right,red!70!black,font=\tiny] at (3.05,1.0) {$4$};
				\node[left,red!70!black,font=\tiny] at (2.5,0.5) {$4$};
				\node[right,red!70!black,font=\tiny] at (3.05,0.5) {$1$};
				\node at (1,0.75) {$\Biggl($};
				\node[left,red!70!black,font=\tiny] at (1.5,0.9) {$\overline{3}$};
				\node[left,red!70!black,font=\tiny] at (1.5,0.6) {$\overline{2}$};
				\node at (1.75,0.75) {$\Biggr)$};
				\node at (1.9,1.2) {$\star$};
				\node[left,red!70!black,font=\tiny] at (0.65,1.2) {$\overline{4}$};
				\node[left,red!70!black,font=\tiny] at (0.65,0.9) {$3$};
				\node[left,red!70!black,font=\tiny] at (0.65,0.6) {$2$};
				\node[left,red!70!black,font=\tiny] at (0.65,0.3) {$1$};
				
				\draw (4.3,1.25)--(5.3,0.25);
				\draw (5.3,1.25)--(4.3,0.25);
				\draw (6.05,1.25)--(6.05,0.25);
				\draw (7.1,1.25)--(7.1,0.25);
				\draw[->] (4.5,1.05)--(4.45,1.1);
				\draw[->] (5.1,1.05)--(5.15,1.1);
				\draw[->] (4.5,0.45)--(4.55,0.5);
				\draw[->] (5.1,0.45)--(5.05,0.5);
				\node[left,red!70!black,font=\tiny] at (4.5,1) {$\overline{4}$};
				\node[right,red!70!black,font=\tiny] at (5.05,1) {$\overline{1}$};
				\node[left,red!70!black,font=\tiny] at (4.5,0.5) {$\overline{1}$};
				\node[right,red!70!black,font=\tiny] at (5.05,0.5) {$\overline{4}$};
				\node at (5.6,0.75) {$\Biggl($};
				\node[left,red!70!black,font=\tiny] at (6.1,0.9) {$3$};
				\node[left,red!70!black,font=\tiny] at (6.1,0.6) {$2$};
				\node at (6.35,0.75) {$\Biggr)$};
				\node at (6.5,1.2) {$\star$};
				\node[left,red!70!black,font=\tiny] at (7.15,1.2) {$4$};
				\node[left,red!70!black,font=\tiny] at (7.15,0.9) {$\overline{3}$};
				\node[left,red!70!black,font=\tiny] at (7.15,0.6) {$\overline{2}$};
				\node[left,red!70!black,font=\tiny] at (7.15,0.3) {$\overline{1}$};		
			\end{scope}
			\begin{scope}
				\draw[red!80!black, very thick] (7.8,0) rectangle (15.4,3); 
				\draw[thick] (7.8,1.5)--(15.4,1.5);
				\draw[thick] (11.6,0)--(11.6,3);
				
				\draw (8.3,2.75)--(9.3,1.75);
				\draw (9.3,2.75)--(8.3,1.75);
				\draw (10.05,2.75)--(10.05,1.75);
				\draw (11.1,2.75)--(11.1,1.75);
				\draw[->] (8.5,2.55)--(8.55,2.5);
				\draw[->] (9.1,2.55)--(9.05,2.5);
				\draw[->] (8.5,1.95)--(8.55,2);
				\draw[->] (9.1,1.95)--(9.05,2);
				\node[left,red!70!black,font=\tiny] at (8.5,2.5) {$2$};
				\node[right,red!70!black,font=\tiny] at (9.05,2.5) {$3$};
				\node[left,red!70!black,font=\tiny] at (8.5,2) {$\overline{1}$};
				\node[right,red!70!black,font=\tiny] at (9.05,2) {$\overline{4}$};
				\node at (9.6,2.25) {$\Biggl($};
				\node[left,red!70!black,font=\tiny] at (10.1,2.4) {$2$};
				\node[left,red!70!black,font=\tiny] at (10.1,2.1) {$3$};
				\node at (10.35,2.25) {$\Biggr)$};
				\node at (10.5,2.7) {$\star$};
				\node[left,red!70!black,font=\tiny] at (11.15,2.7) {$\overline{1}$};
				\node[left,red!70!black,font=\tiny] at (11.15,2.4) {$\overline{2}$};
				\node[left,red!70!black,font=\tiny] at (11.15,2.1) {$\overline{3}$};
				\node[left,red!70!black,font=\tiny] at (11.15,1.8) {$4$};
				
				\draw (13.9,2.75)--(14.9,1.75);
				\draw (14.9,2.75)--(13.9,1.75);
				\draw (13.05,2.75)--(13.05,1.75);
				\draw (12.2,2.75)--(12.2,1.75);
				\draw[->] (14.1,2.55)--(14.05,2.6);
				\draw[->] (14.7,2.55)--(14.75,2.6);
				\draw[->] (14.1,1.95)--(14.05,1.9);
				\draw[->] (14.7,1.95)--(14.75,1.9);
				\node[left,red!70!black,font=\tiny] at (14.1,2.5) {$\overline{3}$};
				\node[right,red!70!black,font=\tiny] at (14.65,2.5) {$\overline{2}$};
				\node[left,red!70!black,font=\tiny] at (14.1,2) {$4$};
				\node[right,red!70!black,font=\tiny] at (14.65,2) {$1$};
				\node at (12.6,2.25) {$\Biggl($};
				\node[left,red!70!black,font=\tiny] at (13.1,2.4) {$\overline{2}$};
				\node[left,red!70!black,font=\tiny] at (13.1,2.1) {$\overline{3}$};
				\node at (13.35,2.25) {$\Biggr)$};
				\node at (13.5,2.7) {$\star$};
				\node[left,red!70!black,font=\tiny] at (12.25,2.7) {$1$};
				\node[left,red!70!black,font=\tiny] at (12.25,2.4) {$2$};
				\node[left,red!70!black,font=\tiny] at (12.25,2.1) {$3$};
				\node[left,red!70!black,font=\tiny] at (12.25,1.8) {$\overline{4}$};
				
				\draw (10.1,1.25)--(11.1,0.25);
				\draw (11.1,1.25)--(10.1,0.25);
				\draw (9.25,1.25)--(9.25,0.25);
				\draw (8.4,1.25)--(8.4,0.25);
				\draw[->] (10.3,1.05)--(10.35,1);
				\draw[->] (10.9,1.05)--(10.85,1);
				\draw[->] (10.3,0.45)--(10.35,0.5);
				\draw[->] (10.9,0.45)--(10.85,0.5);
				\node[left,red!70!black,font=\tiny] at (10.3,1.0) {$2$};
				\node[right,red!70!black,font=\tiny] at (10.85,1.0) {$3$};
				\node[left,red!70!black,font=\tiny] at (10.3,0.5) {$\overline{1}$};
				\node[right,red!70!black,font=\tiny] at (10.85,0.5) {$\overline{4}$};
				\node at (8.8,0.75) {$\Biggl($};
				\node[left,red!70!black,font=\tiny] at (9.3,0.9) {$3$};
				\node[left,red!70!black,font=\tiny] at (9.3,0.6) {$2$};
				\node at (9.55,0.75) {$\Biggr)$};
				\node at (9.7,1.2) {$\star$};
				\node[left,red!70!black,font=\tiny] at (8.45,1.2) {$\overline{4}$};
				\node[left,red!70!black,font=\tiny] at (8.45,0.9) {$\overline{3}$};
				\node[left,red!70!black,font=\tiny] at (8.45,0.6) {$\overline{2}$};
				\node[left,red!70!black,font=\tiny] at (8.45,0.3) {$1$};
				
				\draw (12.1,1.25)--(13.1,0.25);
				\draw (13.1,1.25)--(12.1,0.25);
				\draw (13.85,1.25)--(13.85,0.25);
				\draw (14.9,1.25)--(14.9,0.25);
				\draw[->] (12.3,1.05)--(12.25,1.1);
				\draw[->] (12.9,1.05)--(12.95,1.1);
				\draw[->] (12.3,0.45)--(12.25,0.4);
				\draw[->] (12.9,0.45)--(12.95,0.4);
				\node[left,red!70!black,font=\tiny] at (12.3,1) {$\overline{3}$};
				\node[right,red!70!black,font=\tiny] at (12.85,1) {$\overline{2}$};
				\node[left,red!70!black,font=\tiny] at (12.3,0.5) {$4$};
				\node[right,red!70!black,font=\tiny] at (12.85,0.5) {$1$};
				\node at (13.4,0.75) {$\Biggl($};
				\node[left,red!70!black,font=\tiny] at (13.9,0.9) {$\overline{3}$};
				\node[left,red!70!black,font=\tiny] at (13.9,0.6) {$\overline{2}$};
				\node at (14.15,0.75) {$\Biggr)$};
				\node at (14.3,1.2) {$\star$};
				\node[left,red!70!black,font=\tiny] at (14.95,1.2) {$4$};
				\node[left,red!70!black,font=\tiny] at (14.95,0.9) {$3$};
				\node[left,red!70!black,font=\tiny] at (14.95,0.6) {$2$};
				\node[left,red!70!black,font=\tiny] at (14.95,0.3) {$\overline{1}$};
			\end{scope}
	\end{tikzpicture}}
	\caption{The growth rules in Algorithm \ref{alg:growth}. All of the 88 short rules are boxed in green and classified to 10 families. The vertical lines beside the crossing arrows are called \textit{witness} which should labeled by one of the presented letters. Two families of long rules are boxed in red, the $\star$ in the long rules means that there can be any number of witnesses enclosed in parentheses.}
	\label{fig:growthrule}
\end{figure}
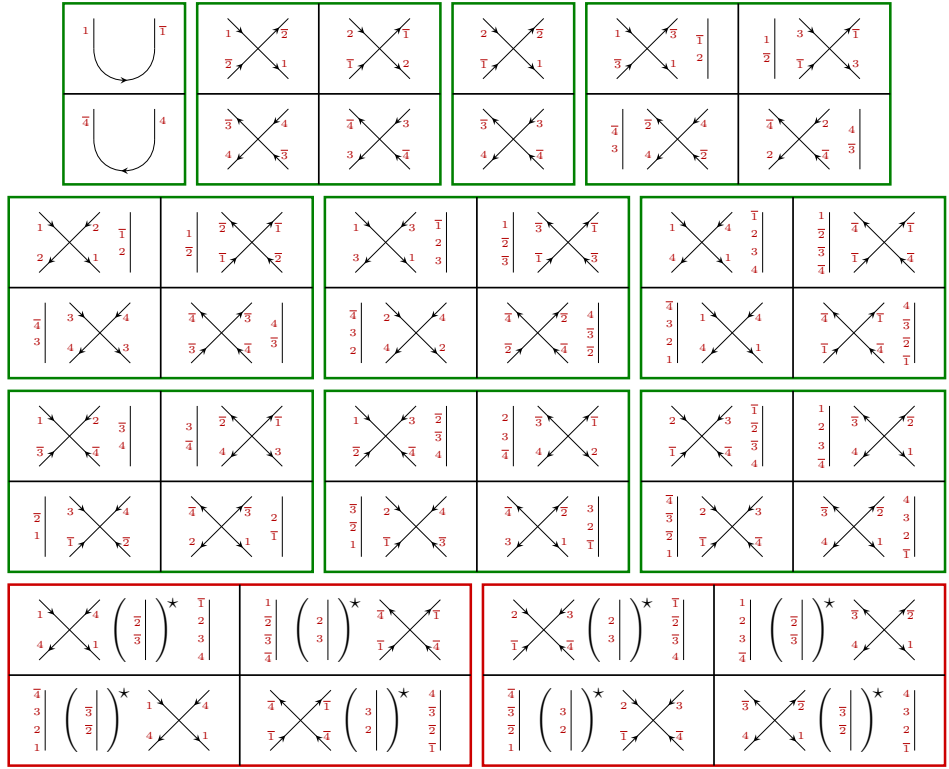

\begin{Algorithm}\label{alg:growth}
	\cite[Algorithm 5.1]{gaetz2023rotation} For $T \in \operatorname{SYT}(4\times b)$, $w=L(T)$.
	\begin{enumerate}[label=(\arabic*), topsep=0pt]
		\item Draw a horizontal line as the boundary edge. For each letter in $w = L(T)$, place a boundary vertex on the horizontal line and hang a dangling strand labeled with the letter itself. If a letter is replaced by a multiset, split it by treating each element in the multiset as an individual letter.
		
		\item Apply the growth rules given in \cite[Figure 22]{gaetz2023rotation} in Figure \ref{fig:growthrule}, until	there are no dangling strands. When applying the growth rules, letters 1, 2, 3, 4 represent a downward arrow, and the letters $\overline{1}$, $\overline{2}$, $\overline{3}$, $\overline{4}$ represent an upward arrow, more information can be found in \cite{gaetz2023rotation} and \cite{patrias2019promotion}.
		
		\item Pinch the endpoints of a boundary edge and convert the resulting intersection point into an hourglass edge or an internal vertex according to the orientation of edges in Figure \ref{fig:6verToHg}, thereby obtaining an hourglass plabic graph. If the resulting hourglass plabic graph contains boundary vertices obtained via multiset splitting in step 1, contract them into a single vertex; if multiple edges arise during contraction, represent them as hourglass edges.
	\end{enumerate}
\end{Algorithm}

\begin{figure}[H]
	\centering
	\begin{minipage}[h]{0.3\textwidth}
		\centering
		\begin{tikzpicture}[rotate=90, >=stealth]
			\draw (0.5,1)--(-0.5,0);
			\draw (-0.5,1)--(0.5,0);
			\draw[->] (0.3,0.8)--(0.25,0.75);
			\draw[->] (0.3,0.2)--(0.25,0.25);
			\draw[->] (-0.3,0.8)--(-0.25,0.75);
			\draw[->] (-0.3,0.2)--(-0.25,0.25);
			\draw[<->,thick] (0,-0.1)--(0,-0.7);
			\draw (0.5,-0.8)--(-0.5,-1.8);
			\draw (0.5,-1.8)--(-0.5,-0.8);
			\fill[white] (0,-1.3) circle (2pt);
			\draw (0,-1.3) circle (2pt);
		\end{tikzpicture}
	\end{minipage}
	\begin{minipage}[h]{0.3\textwidth}
		\centering
		\begin{tikzpicture}[rotate=90, >=stealth]
			\draw (0.5,1)--(-0.5,0);
			\draw (-0.5,1)--(0.5,0);
			\draw[->] (0.3,0.8)--(0.35,0.85);
			\draw[->] (0.3,0.2)--(0.35,0.15);
			\draw[->] (-0.3,0.8)--(-0.35,0.85);
			\draw[->] (-0.3,0.2)--(-0.35,0.15);
			\draw[<->,thick] (0,-0.1)--(0,-0.7);
			\draw (0.5,-0.8)--(-0.5,-1.8);
			\draw (0.5,-1.8)--(-0.5,-0.8);
			\fill (0,-1.3) circle (2pt);
		\end{tikzpicture}
	\end{minipage}
	\begin{minipage}[h]{0.3\textwidth}
		\centering
		\begin{tikzpicture}[rotate=90, >=stealth]
			\draw (0.5,1)--(-0.5,0);
			\draw (-0.5,1)--(0.5,0);
			\draw[->] (0.3,0.8)--(0.35,0.85);
			\draw[->] (0.3,0.2)--(0.35,0.15);
			\draw[->] (-0.3,0.8)--(-0.25,0.75);
			\draw[->] (-0.3,0.2)--(-0.25,0.25);
			\draw[<->,thick] (0,-0.1)--(0,-0.7);
			\draw (0.7,-0.8)--(0.3,-1.3);
			\draw (0.7,-1.8)--(0.3,-1.3);
			\draw (-0.7,-0.8)--(-0.3,-1.3);
			\draw (-0.7,-1.8)--(-0.3,-1.3);
			\coordinate (A) at (-0.3,-1.3);
			\coordinate (a) at (0.3,-1.3);
			\hg{A}{a};
			\fill (a) circle (2pt);
			\fill[white] (A) circle (2pt);
			\draw (A) circle (2pt);
		\end{tikzpicture}
	\end{minipage}
	\caption{The three distinct vertex configurations and their corresponding local realizations in hourglass plabic graphs.}
	\label{fig:6verToHg}
\end{figure}
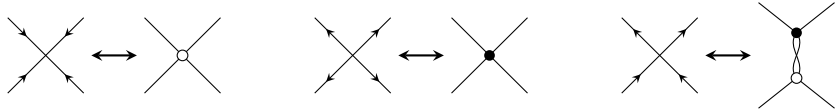

We keep using $T\in \operatorname{SYT}(a\times b)$ from Figure \ref{fig:PTET} as the example to generate its corresponding hourglass plabic graph. For convenience, we represent the growth rules in Figure \ref{fig:growthrule} by a map that maps the letters from top to bottom.

For $T=\scalemath{0.6}{\;{\ytableausetup{centertableaux}\ytableaushort{14,25,37,68}}}$ , $w=L(T)=12312434$, the graph on the left of Figure \ref{fig:exagrowth} is obtained using the growth rules, while the right shows the hourglass plabic graph produced by the entire growth algorithm.

\begin{figure}[h]
	\centering
	\begin{minipage}{0.69\textwidth}
		\centering
		\begin{tikzpicture}[>=stealth]
			\draw[thick] (-0.4,0)--(6,0);
			\foreach \k in {1,...,8} {
				\coordinate (\k) at (-0.8+0.8*\k,0);
				\fill (\k) circle (2pt);
				\draw (\k)--(-0.8+0.8*\k,-0.5);
				\node[font=\tiny, above] at (-0.8+0.8*\k,0) {\k};
			}
			\node[font=\tiny, left] at (0.05,-0.35) {1};
			\node[font=\tiny, left] at (0.85,-0.35) {2};
			\node[font=\tiny, left] at (1.65,-0.35) {3};
			\node[font=\tiny, left] at (2.45,-0.35) {1};
			\node[font=\tiny, left] at (3.25,-0.35) {2};
			\node[font=\tiny, left] at (4.05,-0.35) {4};
			\node[font=\tiny, left] at (4.85,-0.35) {3};
			\node[font=\tiny, left] at (5.65,-0.35) {4};
			
			\draw (0,-0.5)--(0,-1);
			\draw (0.8,-0.5)--(1.6,-1);
			\draw (1.6,-0.5)--(0.8,-1);
			\draw (2.4,-0.5)--(2.4,-1);
			\draw (3.2,-0.5)--(4,-1);
			\draw (4,-0.5)--(3.2,-1);
			\draw (4.8,-0.5)--(4.8,-1);
			\draw (5.6,-0.5)--(5.6,-1);
			
			\node[font=\tiny, left] at (0.05,-0.85) {1};
			\node[font=\tiny, left] at (0.9,-0.85) {$\overline{1}$};
			\node[font=\tiny, right] at (1.5,-0.85) {$\overline{4}$};
			\node[font=\tiny, left] at (2.45,-0.85) {1};
			\node[font=\tiny, left] at (3.3,-0.85) {$\overline{1}$};
			\node[font=\tiny, right] at (3.9,-0.85) {$\overline{3}$};
			\node[font=\tiny, left] at (4.85,-0.85) {3};
			\node[font=\tiny, left] at (5.65,-0.85) {4};
			
			\draw (0,-1)--(0,-1.2);
			\draw (0.8,-1)--(0.8,-1.2);
			\draw (1.6,-1)--(1.6,-1.5);
			\draw (2.4,-1)--(2.4,-1.2);
			\draw (3.2,-1)--(3.2,-1.2);
			\draw (4,-1)--(4.8,-1.5);
			\draw (4.8,-1)--(4,-1.5);
			\draw (5.6,-1)--(5.6,-1.5);
			
			\draw (0,-1.2)--(0.8,-1.2);
			\draw (2.4,-1.2)--(3.2,-1.2);
			\node[font=\tiny, left] at (1.65,-1.35) {$\overline{4}$};
			\node[font=\tiny, left] at (4.1,-1.35) {4};
			\node[font=\tiny, right] at (4.7,-1.35) {$\overline{4}$};
			\node[font=\tiny, left] at (5.65,-1.35) {4};
			
			\draw (1.6,-1.5)--(1.6,-1.7);
			\draw (4,-1.5)--(4,-1.7);
			\draw (4.8,-1.5)--(4.8,-1.7);
			\draw (5.6,-1.5)--(5.6,-1.7);
			
			\draw (1.6,-1.7)--(4,-1.7);
			\draw (4.8,-1.7)--(5.6,-1.7);

		\end{tikzpicture}
	\end{minipage}
	\begin{minipage}{0.29\textwidth}
		\GrDisk{
			\coordinate (A) at (45:0.5);
			\coordinate (B) at (-90:0.5);
			\coordinate (C) at (157.5:0.7);
			\coordinate (a) at (157.5:0.2);
			
			\draw (1)--(A); 	\draw (2)--(A);		\draw (3)--(A);		\draw (4)--(B);
			\draw (5)--(B);		\draw (6)--(B);		\draw (7)--(C);		\draw (8)--(C);
			
			\draw (a)--(A);		\draw (a)--(B);		\hg{a}{C};
			\foreach \p in {A,B,C} {
				\fill[white] (\p) circle (2pt);
				\draw (\p) circle (2pt);
			}
			\fill (a) circle (2pt);
		}
	\end{minipage}
	\caption{The graph obtained by growth rules (left) and the hourglass plabic graph (right) of $T$.}
	\label{fig:exagrowth}
\end{figure}
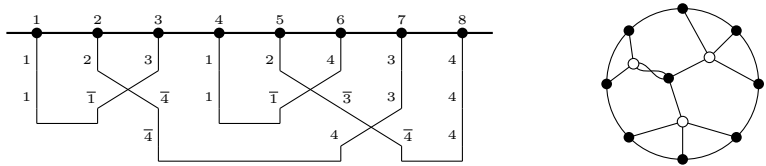

\subsection{Web diagrams of quadratic and cubic cluster variables in $\mathbb{C}[\Gr{4}{8}]$}\label{ssc:WinInv}
We compute the web diagrams of the 3 quadratic and 14 cubic cluster variables mentioned in Section \ref{ssc:PreGra}. Explicitly, we apply the growth algorithm to the Young tableaux corresponding to the cluster variables mentioned in \cite{chang_quantum_2020}.

In the Grassmannian cluster algebra $\mathbb{C}[\Gr{4}{8}]$, there are 120 quadratic cluster variables according to \cite{cheung_clustering_2022}, $\binom{8}{7}\times 14 =112$ of them can be obtained from $\mathbb{C}[\Gr{4}{7}]$, which are dihedral translates or isolated-vertex replacements of just two distinct cluster variables, whose boundary condition is $\lambda=(2,1,1,1,1,1,1,0)$, and the other 8 cluster variables with condition $\lambda=(1,1,1,1,1,1,1,1)$ are rotations obtained by someone of them. Therefore, it is necessary to discuss the 3 typical cluster variables, whose corresponding Young tableaux and web diagrams are listed below:
\begin{table}[h]
	\centering
	\setcellgapes{1mm}
	\makegapedcells
	\begin{tabular}{cccccc}
		\hline
		\makecell{tableau} & \makecell{web} & \makecell{tableau} & \makecell{web} & \makecell{tableau} & \makecell{web}\\
		\hline
		\makecell{\scalemath{0.9}{\ytableaushort{11,24,36,57}}} &  
		\makecell{
			\scalemath{0.8}{
				\GrDisk{
					\coordinate (A) at (45:0.5);
					\coordinate (B) at (-67.5:0.7);
					\coordinate (C) at (180:0.5);
					\coordinate (a) at (0:0);
					\draw (1)--(A);		\draw (1)--(C);		\draw (2)--(A);		\draw (3)--(A);		
					\draw (4)--(B);		\draw (5)--(B);		\draw (6)--(C);		\draw (7)--(C);
					\draw (a)--(A);		\draw (a)--(C);		\hg{B}{a};				
					\foreach \p in {A,B,C} {
						\fill[white] (\p) circle (2pt);
						\draw (\p) circle (2pt);}
					\fill (a) circle (2pt);}}} &
		\makecell{\scalemath{0.9}{\ytableaushort{11,23,45,67}}} &  
		\makecell{
			\scalemath{0.8}{
				\GrDisk{
					\coordinate (A) at (67.5:0.7);
					\coordinate (B) at (-22.5:0.7);
					\coordinate (C) at (-112.5:0.7);
					\coordinate (D) at (135:0.5);
					\coordinate (a) at (67.5:0.4);
					\coordinate (b) at (-112.5:0.4);
					\draw (1)--(A);		\draw (1)--(D);		\draw (2)--(A);		\draw (3)--(B);		
					\draw (4)--(B);		\draw (5)--(C);		\draw (6)--(C);		\draw (7)--(D);
					\draw (a)--(B);		\draw (a)--(D);		\hg{A}{a};	
					\draw (b)--(B);		\draw (b)--(D);		\hg{C}{b};			
					\foreach \p in {A,B,C,D} {
						\fill[white] (\p) circle (2pt);
						\draw (\p) circle (2pt);}
					\fill (a) circle (2pt);
					\fill (b) circle (2pt);}}} &
		\makecell{\scalemath{0.9}{\ytableaushort{14,25,37,68}}} &  
		\makecell{
			\scalemath{0.8}{
				\GrDisk{
					\coordinate (A) at (45:0.5);
					\coordinate (B) at (-90:0.5);
					\coordinate (C) at (157.5:0.7);
					\coordinate (a) at (157.5:0.2);
					\draw (1)--(A); 	\draw (2)--(A);		\draw (3)--(A);		\draw (4)--(B);
					\draw (5)--(B);		\draw (6)--(B);		\draw (7)--(C);		\draw (8)--(C);
					\draw (a)--(A);		\draw (a)--(B);		\hg{a}{C};
					\foreach \p in {A,B,C} {
						\fill[white] (\p) circle (2pt);
						\draw (\p) circle (2pt);}
					\fill (a) circle (2pt);}}} \\			
		\hline
	\end{tabular}
	\caption{Web diagrams of three typical cluster variables: two with boundary condition $\lambda=(2,1,1,1,1,1,1,0)$, and the rightmost variable selected from eight cluster variables with $\lambda=(1,1,1,1,1,1,1,1)$.}
	\label{tab:quainv}
\end{table}

The 174 cubic cluster variables are classified by their boundary conditions. Because $\sum_{i=1}^{8} \lambda_i=12$, up to dihedral translation of the component labels $i$ of $\lambda_i$ in $\lambda$, there are 8 different types of boundary conditions. In selecting typical cluster variables, we fix $\lambda_1=2$ and minimize the sum of indices i where $\lambda_i=2$. This yields 14 distinct cluster variables, with all other cubic cluster variables being dihedral translates of these.

\begin{table}[h]
	\centering
	\setcellgapes{1mm}
	\makegapedcells
	\begin{tabular}{ccccccccc}
		\hline
		\makecell{type} & \makecell{tableau} & \makecell{web} & \makecell{type} & \makecell{tableau} & \makecell{web} & \makecell{type} & \makecell{tableau} & \makecell{web}\\
		\hline
		\makecell{1} &
		\makecell{\scalemath{0.7}{\ytableaushort{112,233,445,678}}} &  
		\makecell{
			\scalemath{0.63}{
				\GrDisk{
					\coordinate (A) at (67.5:0.8);
					\coordinate (B) at (22.5:0.8);
					\coordinate (C) at (-22.5:0.8);
					\coordinate (D) at (-80:0.45);
					\coordinate (E) at (-112.5:0.8);
					\coordinate (F) at (180:0.7);
					\coordinate (G) at (112.5:0.8);
					\coordinate (H) at (90:0.1);
					\coordinate (a) at (90:0.5);
					\coordinate (b) at (22.5:0.5);
					\coordinate (c) at (-30:0.3);
					\coordinate (d) at (-135:0.5);
					\coordinate (e) at (157.5:0.4);
					\draw (1)--(A);		\draw (1)--(G);		\draw (2)--(A);		\draw (2)--(B);
					\draw (3)--(B);		\draw (3)--(C);		\draw (4)--(C);		\draw (4)--(D);		
					\draw (5)--(E);		\draw (6)--(E);		\draw (7)--(F);		\draw (8)--(G);
					\draw (a)--(G);		\draw (a)--(H);		\hg{A}{a};
					\draw (b)--(C);		\draw (b)--(H);		\hg{B}{b};
					\draw (c)--(C);		\draw (c)--(H);		\hg{D}{c};
					\draw (d)--(D);		\draw (d)--(F);		\hg{E}{d};
					\draw (e)--(G);		\draw (e)--(H);		\hg{F}{e};				
					\foreach \p in {A,B,C,D,E,F,G,H} {
						\fill[white] (\p) circle (2pt);
						\draw (\p) circle (2pt);}
					\foreach \p in {a,b,c,d,e}{
						\fill (\p) circle (2pt);}}}} &
		\multirow{5}{*}{\makecell{4}} &
		\makecell{\scalemath{0.7}{\ytableaushort{112,234,455,678}}} &  
		\makecell{
			\scalemath{0.63}{
				\GrDisk{
					\coordinate (A) at (67.5:0.8);
					\coordinate (B) at (22.5:0.8);
					\coordinate (C) at (-45:0.6);
					\coordinate (D) at (-100:0.5);
					\coordinate (E) at (-120:0.8);
					\coordinate (F) at (180:0.6);
					\coordinate (G) at (112.5:0.8);
					\coordinate (H) at (45:0.1);
					\coordinate (a) at (90:0.4);
					\coordinate (b) at (0:0.4);
					\coordinate (c) at (-67.5:0.3);
					\coordinate (d) at (-140:0.6);
					\coordinate (e) at (140:0.4);
					\draw (1)--(A);		\draw (1)--(G);		\draw (2)--(A);		\draw (2)--(B);
					\draw (3)--(B);		\hg{4}{C};			\draw (5)--(D);		\draw (5)--(E);		
					\draw (6)--(E);		\draw (7)--(F);		\draw (8)--(G);
					\draw (a)--(G);		\draw (a)--(H);		\hg{A}{a};
					\draw (b)--(C);		\draw (b)--(H);		\hg{B}{b};
					\draw (c)--(C);		\draw (c)--(H);		\hg{D}{c};
					\draw (d)--(D);		\draw (d)--(F);		\hg{E}{d};
					\draw (e)--(G);		\draw (e)--(H);		\hg{F}{e};
					\foreach \p in {A,B,C,D,E,F,G,H} {
						\fill[white] (\p) circle (2pt);
						\draw (\p) circle (2pt);}
					\foreach \p in {a,b,c,d,e}{
						\fill (\p) circle (2pt);}}}} &
		\makecell{6} &
		\makecell{\scalemath{0.7}{\ytableaushort{112,244,367,578}}} &  
		\makecell{
			\scalemath{0.63}{
				\GrDisk{
					\coordinate (A) at (70:0.5);
					\coordinate (B) at (22.5:0.8);
					\coordinate (C) at (-30:0.4);
					\coordinate (D) at (-67.5:0.8);
					\coordinate (E) at (-135:0.4);
					\coordinate (F) at (135:0.6);
					\coordinate (a) at (22.5:0.5);
					\coordinate (b) at (-75:0.5);
					\coordinate (c) at (120:0.3);
					\draw (1)--(A);		\draw (1)--(F);		\hg{2}{B}			\draw (3)--(B);
					\draw (4)--(C);		\draw (4)--(D);		\draw (5)--(D);		\draw (6)--(E);		
					\draw (7)--(E);		\draw (7)--(F);		\draw (8)--(F);
					\draw (a)--(A);		\draw (a)--(B);		\hg{C}{a};
					\draw (b)--(C);		\draw (b)--(E);		\hg{D}{b};
					\draw (c)--(E);		\draw (c)--(F);		\hg{A}{c};
					\foreach \p in {A,B,C,D,E,F} {
						\fill[white] (\p) circle (2pt);
						\draw (\p) circle (2pt);}
					\foreach \p in {a,b,c}{
						\fill (\p) circle (2pt);}}}} \\			
		\cmidrule(lr){1-3}
		\cmidrule(lr){5-9}
		\multirow{5}{*}{\makecell{2}} &
		\makecell{\scalemath{0.7}{\ytableaushort{112,233,455,678}}} &  
		\makecell{
			\scalemath{0.63}{
				\GrDisk{
					\coordinate (A) at (67.5:0.8);
					\coordinate (B) at (22.5:0.8);
					\coordinate (C) at (-22.5:0.8);
					\coordinate (D) at (-80:0.45);
					\coordinate (E) at (-112.5:0.8);
					\coordinate (F) at (180:0.7);
					\coordinate (G) at (112.5:0.8);
					\coordinate (H) at (90:0.1);
					\coordinate (a) at (90:0.5);
					\coordinate (b) at (22.5:0.5);
					\coordinate (c) at (-30:0.3);
					\coordinate (d) at (-135:0.5);
					\coordinate (e) at (157.5:0.4);
					\draw (1)--(A);		\draw (1)--(G);		\draw (2)--(A);		\draw (2)--(B);
					\draw (3)--(B);		\draw (3)--(C);		\draw (4)--(C);		\draw (5)--(D);		
					\draw (5)--(E);		\draw (6)--(E);		\draw (7)--(F);		\draw (8)--(G);
					\draw (a)--(G);		\draw (a)--(H);		\hg{A}{a};
					\draw (b)--(C);		\draw (b)--(H);		\hg{B}{b};
					\draw (c)--(C);		\draw (c)--(H);		\hg{D}{c};
					\draw (d)--(D);		\draw (d)--(F);		\hg{E}{d};
					\draw (e)--(G);		\draw (e)--(H);		\hg{F}{e};				
					\foreach \p in {A,B,C,D,E,F,G,H} {
						\fill[white] (\p) circle (2pt);
						\draw (\p) circle (2pt);}
					\foreach \p in {a,b,c,d,e}{
						\fill (\p) circle (2pt);}}}} & &
		\makecell{\scalemath{0.7}{\ytableaushort{113,225,447,568}}} & 
		\makecell{
			\scalemath{0.63}{
				\GrDisk{
					\coordinate (A) at (67.5:0.8);
					\coordinate (B) at (22.5:0.5);
					\coordinate (C) at (-22.5:0.8);
					\coordinate (D) at (-67.5:0.4);
					\coordinate (E) at (-135:0.6);
					\coordinate (F) at (157.5:0.8);
					\coordinate (G) at (112.5:0.5);
					\coordinate (a) at (67.5:0.5);
					\coordinate (b) at (-22.5:0.5);
					\coordinate (c) at (0:0);
					\coordinate (d) at (157.5:0.5);
					\draw (1)--(A);		\draw (1)--(G);		\draw (2)--(A);		\draw (2)--(B);
					\draw (3)--(C);		\draw (4)--(C);		\draw (4)--(D);		\draw (5)--(D);		
					\draw (5)--(E);		\draw (6)--(E);		\draw (7)--(F);		\draw (8)--(F);
					\draw (a)--(B);		\draw (a)--(G);		\hg{A}{a};
					\draw (b)--(B);		\draw (b)--(D);		\hg{C}{b};
					\draw (c)--(B);		\draw (c)--(D);		\draw (c)--(E);		\draw (c)--(G);
					\draw (d)--(E);		\draw (d)--(G);		\hg{F}{d};
					\foreach \p in {A,B,C,D,E,F,G} {
						\fill[white] (\p) circle (2pt);
						\draw (\p) circle (2pt);}
					\foreach \p in {a,b,c,d}{
						\fill (\p) circle (2pt);}}}} &
		\makecell{7} &
		\makecell{\scalemath{0.7}{\ytableaushort{112,245,366,578}}} &  
		\makecell{
			\scalemath{0.63}{
				\GrDisk{
					\coordinate (A) at (70:0.5);
					\coordinate (B) at (22.5:0.8);
					\coordinate (C) at (-50:0.6);
					\coordinate (D) at (-135:0.4);
					\coordinate (E) at (-170:0.7);
					\coordinate (F) at (125:0.7);
					\coordinate (a) at (22.5:0.5);
					\coordinate (b) at (120:0.3);
					\coordinate (c) at (160:0.5);
					\draw (1)--(A);		\draw (1)--(F);		\hg{2}{B}			\draw (3)--(B);
					\draw (4)--(C);		\draw (5)--(C);		\draw (5)--(D);		\draw (6)--(D);		
					\draw (6)--(E);		\draw (7)--(E);		\draw (8)--(F);
					\draw (a)--(A);		\draw (a)--(B);		\hg{C}{a};
					\draw (b)--(D);		\draw (b)--(F);		\hg{A}{b};
					\draw (c)--(D);		\draw (c)--(F);		\hg{E}{c};
					\foreach \p in {A,B,C,D,E,F} {
						\fill[white] (\p) circle (2pt);
						\draw (\p) circle (2pt);}
					\foreach \p in {a,b,c}{
						\fill (\p) circle (2pt);}}}} \\
		\cmidrule(lr){2-9}
		& \makecell{\scalemath{0.7}{\ytableaushort{113,225,347,568}}} & 
		\makecell{
			\scalemath{0.63}{
				\GrDisk{
					\coordinate (A) at (67.5:0.8);
					\coordinate (B) at (22.5:0.7);
					\coordinate (C) at (-45:0.6);
					\coordinate (D) at (-135:0.6);
					\coordinate (E) at (160:0.8);
					\coordinate (F) at (112.5:0.4);
					\coordinate (a) at (60:0.5);
					\coordinate (b) at (0:0);
					\coordinate (c) at (170:0.5);
					\draw (1)--(A);		\draw (1)--(F);		\draw (2)--(A);		\draw (2)--(B);
					\draw (3)--(B);		\draw (3)--(C);		\draw (4)--(C);		\draw (5)--(C);		
					\draw (5)--(D);		\draw (6)--(D);		\draw (7)--(E);		\draw (8)--(E);
					\draw (a)--(B);		\draw (a)--(F);		\hg{A}{a};
					\draw (b)--(B);		\draw (b)--(C);		\draw (b)--(D);		\draw (b)--(F);
					\draw (c)--(D);		\draw (c)--(F);		\hg{E}{c};
					\foreach \p in {A,B,C,D,E,F} {
						\fill[white] (\p) circle (2pt);
						\draw (\p) circle (2pt);}
					\foreach \p in {a,b,c}{
						\fill (\p) circle (2pt);}}}} &
		\multirow{9}{*}{\makecell{5}} &
		\makecell{\scalemath{0.7}{\ytableaushort{112,234,456,678}}} & 
		\makecell{
			\scalemath{0.63}{
				\GrDisk{
					\coordinate (A) at (67.5:0.8);
					\coordinate (B) at (22.5:0.8);
					\coordinate (C) at (-45:0.6);
					\coordinate (D) at (-112.5:0.8);
					\coordinate (E) at (-157.5:0.8);
					\coordinate (F) at (112.5:0.8);
					\coordinate (G) at (0:0);
					\coordinate (a) at (90:0.5);
					\coordinate (b) at (0:0.5);
					\coordinate (c) at (-90:0.4);
					\coordinate (d) at (170:0.5);
					\draw (1)--(A);		\draw (1)--(F);		\draw (2)--(A);		\draw (2)--(B);
					\draw (3)--(B);		\hg{4}{C};			\draw (5)--(D);		\draw (6)--(D);		
					\draw (6)--(E);		\draw (7)--(E);		\draw (8)--(F);
					\draw (a)--(F);		\draw (a)--(G);		\hg{A}{a};
					\draw (b)--(C);		\draw (b)--(G);		\hg{B}{b};
					\draw (c)--(C);		\draw (c)--(G);		\hg{D}{c};
					\draw (d)--(F);		\draw (d)--(G);		\hg{E}{d};
					\foreach \p in {A,B,C,D,E,F,G} {
						\fill[white] (\p) circle (2pt);
						\draw (\p) circle (2pt);}
					\foreach \p in {a,b,c,d}{
						\fill (\p) circle (2pt);}}}} & 
		\multirow{5}{*}{\makecell{8}} &
		\makecell{\scalemath{0.7}{\ytableaushort{113,245,367,578}}} & 
		\makecell{
			\scalemath{0.63}{
				\GrDisk{
					\coordinate (A) at (45:0.6);
					\coordinate (B) at (-45:0.6);
					\coordinate (C) at (-135:0.6);
					\coordinate (D) at (135:0.6);
					\coordinate (a) at (0:0);
					\draw (1)--(A);		\draw (1)--(D);		\draw (2)--(A);		\draw (3)--(A);
					\draw (3)--(B);		\draw (4)--(B);		\draw (5)--(B);		\draw (5)--(C);		
					\draw (6)--(C);		\draw (7)--(C);		\draw (7)--(D);		\draw (8)--(D);
					\draw (a)--(A);		\draw (a)--(B);		\draw (a)--(C);		\draw (a)--(D);
					\foreach \p in {A,B,C,D} {
						\fill[white] (\p) circle (2pt);
						\draw (\p) circle (2pt);}
					\foreach \p in {a}{
						\fill (\p) circle (2pt);}}}} \\
		\cmidrule(lr){1-3}
		\cmidrule(lr){5-6}
		\cmidrule(lr){8-9}
		\multirow{5}{*}{\makecell{3}} &
		\makecell{\scalemath{0.7}{\ytableaushort{112,233,456,678}}} &  
		\makecell{
			\scalemath{0.63}{
				\GrDisk{
					\coordinate (A) at (67.5:0.8);
					\coordinate (B) at (22.5:0.8);
					\coordinate (C) at (-22.5:0.8);
					\coordinate (D) at (-112.5:0.8);
					\coordinate (E) at (-157.5:0.8);
					\coordinate (F) at (112.5:0.8);
					\coordinate (G) at (0:0);
					\coordinate (a) at (90:0.4);
					\coordinate (b) at (0:0.4);
					\coordinate (c) at (-80:0.5);
					\coordinate (d) at (170:0.5);
					\draw (1)--(A);		\draw (1)--(F);		\draw (2)--(A);		\draw (2)--(B);
					\draw (3)--(B);		\draw (3)--(C);		\draw (4)--(C);		\draw (5)--(D);		
					\draw (6)--(D);		\draw (6)--(E);		\draw (7)--(E);		\draw (8)--(F);
					\draw (a)--(F);		\draw (a)--(G);		\hg{A}{a};
					\draw (b)--(C);		\draw (b)--(G);		\hg{B}{b};
					\draw (c)--(C);		\draw (c)--(G);		\hg{D}{c};
					\draw (d)--(F);		\draw (d)--(G);		\hg{E}{d};
					\foreach \p in {A,B,C,D,E,F,G} {
						\fill[white] (\p) circle (2pt);
						\draw (\p) circle (2pt);}
					\foreach \p in {a,b,c,d}{
						\fill (\p) circle (2pt);}}}} & &
		\makecell{\scalemath{0.7}{\ytableaushort{112,244,366,578}}} &  
		\makecell{
			\scalemath{0.63}{
				\GrDisk{
					\coordinate (A) at (67.5:0.7);
					\coordinate (B) at (22.5:0.8);
					\coordinate (C) at (-30:0.5);
					\coordinate (D) at (-67.5:0.8);
					\coordinate (E) at (-120:0.3);
					\coordinate (F) at (-170:0.8);
					\coordinate (G) at (112.5:0.8);
					\coordinate (a) at (22.5:0.5);
					\coordinate (b) at (-75:0.5);
					\coordinate (c) at (170:0.5);
					\coordinate (d) at (100:0.4);
					\draw (1)--(A);		\draw (1)--(G);		\hg{2}{B};			\draw (3)--(B);		
					\draw (4)--(C);		\draw (4)--(D);		\draw (5)--(D);		\draw (6)--(E);		
					\draw (6)--(F);		\draw (7)--(F);		\draw (8)--(G);	
					\draw (a)--(A);		\draw (a)--(B);		\hg{C}{a};
					\draw (b)--(C);		\draw (b)--(E);		\hg{D}{b};
					\draw (c)--(E);		\draw (c)--(G);		\hg{F}{c};
					\draw (d)--(E);		\draw (d)--(G);		\hg{A}{d};
					\foreach \p in {A,B,C,D,E,F,G} {
						\fill[white] (\p) circle (2pt);
						\draw (\p) circle (2pt);}
					\foreach \p in {a,b,c,d}{
						\fill (\p) circle (2pt);}}}} & &
		\makecell{\scalemath{0.7}{\ytableaushort{112,334,556,778}}} &  
		\makecell{
			\scalemath{0.63}{
				\GrDisk{
					\coordinate (A) at (50:0.5);
					\coordinate (B) at (22.5:0.85);
					\coordinate (C) at (-5:0.5);
					\coordinate (D) at (-67.5:0.7);
					\coordinate (E) at (-130:0.5);
					\coordinate (F) at (-157.5:0.85);
					\coordinate (G) at (175:0.5);
					\coordinate (H) at (112.5:0.7);
					\coordinate (I) at (0:0);
					\coordinate (a) at (80:0.3);
					\coordinate (b) at (22.5:0.55);
					\coordinate (c) at (-35:0.3);
					\coordinate (d) at (-100:0.3);
					\coordinate (e) at (-157.5:0.55);
					\coordinate (f) at (145:0.3);
					\draw (1)--(A);		\draw (1)--(H);		\draw (2)--(B);		\draw (3)--(B);
					\draw (3)--(C);		\draw (4)--(D);		\draw (5)--(D);		\draw (5)--(E);		
					\draw (6)--(F);		\draw (7)--(F);		\draw (7)--(G);		\draw (8)--(H);
					\draw (a)--(H);		\draw (a)--(I);		\hg{A}{a};
					\draw (b)--(A);		\draw (b)--(C);		\hg{B}{b};
					\draw (c)--(D);		\draw (c)--(I);		\hg{C}{c};
					\draw (d)--(D);		\draw (d)--(I);		\hg{E}{d};
					\draw (e)--(E);		\draw (e)--(G);		\hg{F}{e};
					\draw (f)--(H);		\draw (f)--(I);		\hg{G}{f};
					\foreach \p in {A,B,C,D,E,F,G,H,I} {
						\fill[white] (\p) circle (2pt);
						\draw (\p) circle (2pt);}
					\foreach \p in {a,b,c,d,e,f}{
						\fill (\p) circle (2pt);}}}} \\
		\cmidrule(lr){2-3}
		\cmidrule(lr){5-9}
		& \makecell{\scalemath{0.7}{\ytableaushort{112,234,366,578}}} &  
		\makecell{
			\scalemath{0.63}{
				\GrDisk{
					\coordinate (A) at (67.5:0.8);
					\coordinate (B) at (0:0.7);
					\coordinate (C) at (-67.5:0.8);
					\coordinate (D) at (-135:0.3);
					\coordinate (E) at (-157.5:0.8);
					\coordinate (F) at (112.5:0.8);
					\coordinate (a) at (90:0.4);
					\coordinate (b) at (-67.5:0.4);
					\coordinate (c) at (170:0.5);
					\draw (1)--(A);		\draw (1)--(F);		\draw (2)--(A);		\draw (2)--(B);
					\hg{3}{B};			\draw (4)--(C);		\draw (5)--(C);		\draw (6)--(D);		
					\draw (6)--(E);		\draw (7)--(E);		\draw (8)--(F);
					\draw (a)--(D);		\draw (a)--(F);		\hg{A}{a};
					\draw (b)--(B);		\draw (b)--(D);		\hg{C}{b};
					\draw (c)--(D);		\draw (c)--(F);		\hg{E}{c};
					\foreach \p in {A,B,C,D,E,F} {
						\fill[white] (\p) circle (2pt);
						\draw (\p) circle (2pt);}
					\foreach \p in {a,b,c}{
						\fill (\p) circle (2pt);}}}} & &
		\makecell{\scalemath{0.7}{\ytableaushort{112,234,366,578}}} &  
		\makecell{
			\scalemath{0.63}{
				\GrDisk{
					\coordinate (A) at (67.5:0.8);
					\coordinate (B) at (22.5:0.5);
					\coordinate (C) at (-22.5:0.8);
					\coordinate (D) at (-67.5:0.5);
					\coordinate (E) at (-112.5:0.8);
					\coordinate (F) at (-157.5:0.5);
					\coordinate (G) at (157.5:0.8);
					\coordinate (H) at (112.5:0.5);
					\coordinate (a) at (67.5:0.5);
					\coordinate (b) at (-22.5:0.5);
					\coordinate (c) at (-112.5:0.5);
					\coordinate (d) at (0:0);
					\coordinate (e) at (157.5:0.5);
					\draw (1)--(A);		\draw (1)--(H);		\draw (2)--(A);		\draw (2)--(B);
					\draw (3)--(C);		\draw (4)--(C);		\draw (4)--(D);		\draw (5)--(E);		
					\draw (6)--(E);		\draw (6)--(F);		\draw (7)--(G);		\draw (8)--(G);
					\draw (a)--(B);		\draw (a)--(H);		\hg{A}{a};
					\draw (b)--(B);		\draw (b)--(D);		\hg{C}{b};
					\draw (c)--(D);		\draw (c)--(F);		\hg{E}{c};
					\draw (d)--(B);		\draw (d)--(D);		\draw (d)--(F);		\draw (d)--(H);
					\draw (e)--(F);		\draw (e)--(H);		\hg{G}{e};
					\foreach \p in {A,B,C,D,E,F,G,H} {
						\fill[white] (\p) circle (2pt);
						\draw (\p) circle (2pt);}
					\foreach \p in {a,b,c,d,e}{
						\fill (\p) circle (2pt);}}}} & & & \\
		\hline	
	\end{tabular}
	\caption{Web diagrams of 14 distinct cubic cluster variables.}
	\label{tab:cubinv}
\end{table}

All of the 174 cubic cluster variables are dihedral translates of the 14 semi-standard Young tableaux in Table \ref{tab:cubinv}, which can be obtained by applying the promotion and evacuation maps in Theorem \ref{thm:TG}. Web diagrams of them can also be obtained by applying the rotation and reflection maps to the hourglass graphs in Table \ref{tab:cubinv} or in their equivalence classes.

\section{Dual webs compatible with cluster variables in $\mathbb{C}[\operatorname{Gr}(4,8)]$}\label{sec:Inm}
In this section, we discuss the $sl_r$-webs compatible with the quadratic and cubic cluster variables in $\mathbb{C}[\operatorname{Gr}(4,8)]$, where the $sl_r$-webs can be used to compute the twist of cluster variables. We use the definition of compatibility introduced in \cite{lam_dimers_2015}, and the method from \cite{elkin_twists_2023} that enumerates all non-elliptic webs to compute the compatible webs for the cubic cluster variables. For convenience, in this section, we refer to matchings as $sl_2$-webs and webs as $sl_3$-webs.

\subsection{Enumeration of non-elliptic webs}\label{ssc:ComSink}
To enumerate the non-elliptic webs of $\Gr{4}{8}$ which have 4 black and 4 white boundary vertices, we adopt a contracting approach which was first introduced in \cite{russell_explicit_2013} and is referred to as a sink vertex.

\begin{definition}\label{def:sink}
	If two adjacent black boundary vertices are connected to the same white internal vertex, then contract their corresponding two boundary edges and contract that white internal vertex with these two black boundary vertices into a single white boundary vertex, we call this single white boundary vertex a \textit{sink vertices}. We refer to the structure shown on the left side of Figure \ref{fig:sink}, where several adjacent boundary vertices are connected to the same white internal vertex, as a \textit{claw}. For convenience, to represent sink vertices, we will denote them by the set of two boundary vertices that correspond to them.
\end{definition}

\begin{figure}[h]
	\centering
	\begin{minipage}[h]{\textwidth}
		\centering
		\begin{tikzpicture}
			\draw (-4.5,0)--(-2.5,0);	\draw (-4,0)--(-3.5,0.25)--(-3,0);
			\draw (-3.5,0.25)--(-3.5,0.5);	\draw (-4,0.75)--(-3.5,0.5)--(-3,0.75);
			
			\fill [white] (-3.5,0.25) circle (2pt);
			\draw (-3.5,0.25) circle (2pt);
			\fill (-4,0) circle (2pt);
			\fill (-3,0) circle (2pt);
			\fill (-3.5,0.5) circle (2pt);
			
			\draw[-latex, >=stealth, line width=1pt] (-2,0.5) -- (-1.5,0.5);
			
			\draw (-1,0)--(1,0);	\draw (0,0)--(0,0.5);
			\draw (-0.5,0.7)--(0,0.5)--(0.5,0.75);
			
			\fill [white] (0,0) circle (2pt);
			\draw (0,0) circle (2pt);
			\fill (0,0.5) circle (2pt);
		\end{tikzpicture}
	\end{minipage}
	\caption{Contracting boundary edges to produce a sink vertex.}
	\label{fig:sink}
\end{figure}
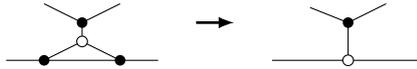

\begin{proposition}\label{prp:sink}
	Any non-elliptic web with $n$ black boundary vertices and $k$ white boundary vertices can be obtained by contracting a non-elliptic web with $n+2k$ boundary vertices.
\end{proposition}

\begin{proof}
	For any non-elliptic web with $n$ black boundary vertices and $k$ white boundary vertices, by the inverse transformation of sink vertices to their white boundary vertices, we can obtain a web with $n+2k$ black boundary vertices. Since this process neither changes the interior of the web (it may merely add a univalent vertex within the web's interior), nor alter the non-elliptic nature of the web, that is the resulting web remains non-elliptic.
\end{proof}

The concept of sink vertices, as formalized in Proposition \ref{prp:sink}, significantly reduces the complexity of enumeration. Thus, it is only necessary to enumerate the cases of non-elliptic webs with only black boundary vertices, since the corresponding contractions into sink vertices can solve the problem of enumerating the non-elliptic webs with 4 white and 4 black boundary vertices. Consequently, the non-elliptic webs with 12 black boundary vertices should be enumerated.

Let $W$ be a non-elliptic web, denote $V_{int}$ by the set of internal vertices, that is, $|V_{int}|$ is the number of internal vertices in $W$.
\begin{proposition}\cite[Proposition 5.3]{elkin_twists_2023} \label{prp:Vint}
	Let $W$ be a non-elliptic web with $n$ boundary vertices, $c$ cycles and $m$ connected components. Then $|V_{int}|=n+2c-2m$.
\end{proposition}

\begin{proposition}\cite[Proposition 5.5]{elkin_twists_2023} \label{prp:c_Vint}
	Let $W$ be a non-elliptic web with $c$ cycles,. Then
	
	\begin{itemize}[leftmargin=0.5cm]
		\item $|V_{int}|\geq 2c+4$, when $c\geq 1$,
		
		\item $|V_{int}|\geq 2c+6$, when $c\geq 2$,
		
		\item $|V_{int}|\geq 2c+7$, when $c\geq 3$,
		
		\item and $|V_{int}|\geq 2c+8$, when $c\geq 4$.
	\end{itemize}
	
\end{proposition}

By a similar proof of \Cref{prp:c_Vint}, we can obtain

\begin{itemize}[leftmargin=0.5cm]
	\item $|V_{int}|\geq 2c+9$, when $c\geq 5$,
	
	\item $|V_{int}|\geq 2c+10$, when $c\geq 6$,
	
	\item $|V_{int}|\geq 2c+10$, $c\geq 7$.
\end{itemize}

\begin{proposition}\cite[Section 2.1]{fraser_dimers_2019}\label{prp:n=kr}
	For any web with only $n$ black vertices and no white vertices on its boundary, the difference in the number of white internal vertices and black vertices is a constant $k$, referred to as the \textit{exceedance}. In particular, $k=n/r$.
\end{proposition}

Note that exceedance $k$ mentioned here is consistent with the $k$ referred to in $\Gr{k}{n}$.

\begin{lemma}\label{lem:1con_12}
	All non-elliptic webs with $12$ black boundary vertices and just one connected component are dihedral translates of the $9$ webs labeled from $W_1$ to $W_9$ in \Cref{fig:1con_12}.
\end{lemma}

\begin{figure}[h]
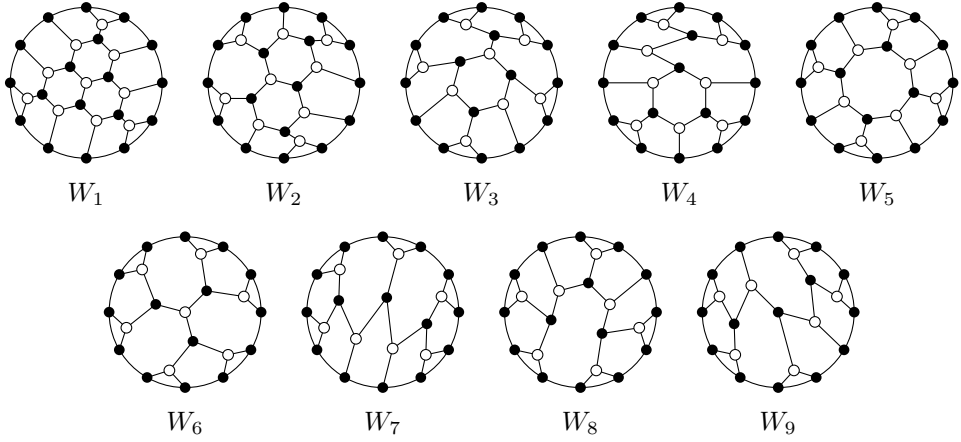

	\centering
	\begin{minipage}[t]{0.19\textwidth}
		\centering
		\Disk{
			\coordinate (A) at (75:0.8);
			\coordinate (B) at (-45:0.8);
			\coordinate (C) at (-165:0.8);
			\coordinate (D) at (45:{0.3*sqrt(3)});
			\coordinate (E) at (-15:{0.3*sqrt(3)});
			\coordinate (F) at (-75:{0.3*sqrt(3)});
			\coordinate (G) at (-135:{0.3*sqrt(3)});
			\coordinate (H) at (165:{0.3*sqrt(3)});
			\coordinate (I) at (105:{0.3*sqrt(3)});
			\coordinate (J) at (0:0);
			
			\coordinate (a) at (75:0.6);
			\coordinate (b) at (-45:0.6);
			\coordinate (c) at (-165:0.6);
			\coordinate (d) at (15:0.3);
			\coordinate (e) at (-105:0.3);
			\coordinate (f) at (135:0.3);
			
			\draw (1)--(A);		\draw (2)--(A);		\draw (3)--(D);		\draw (4)--(E);
			\draw (5)--(B);		\draw (6)--(B);		\draw (7)--(F);		\draw (8)--(G);
			\draw (9)--(C);		\draw (10)--(C);	\draw (11)--(H);	\draw (12)--(I);
			
			\draw (a)--(A);		\draw (a)--(D);		\draw (a)--(I);
			\draw (b)--(B);		\draw (b)--(E);		\draw (b)--(F);
			\draw (c)--(C);		\draw (c)--(G);		\draw (c)--(H);
			\draw (d)--(D);		\draw (d)--(E);		\draw (d)--(J);
			\draw (e)--(F);		\draw (e)--(G);		\draw (e)--(J);
			\draw (f)--(H);		\draw (f)--(I);		\draw (f)--(J);
			
			\foreach \p in {A,B,C,D,E,F,G,H,I,J} {
				\fill[white] (\p) circle (2pt);
				\draw (\p) circle (2pt);
			}
			\foreach \p in {a,b,c,d,e,f}{
				\fill (\p) circle (2pt);
			}
			\node[below] at (0,-1.2) {$W_1$};
		}
	\end{minipage}
	\begin{minipage}[t]{0.19\textwidth}
		\centering
		\Disk{
			\coordinate (A) at (45:0.8);
			\coordinate (B) at (-75:0.8);
			\coordinate (C) at (-165:0.8);
			\coordinate (D) at (135:0.8);
			\coordinate (E) at ({asin(sqrt(39)*2/13)-195}:{sqrt(13)*0.18});
			\coordinate (F) at ({165-asin(sqrt(39)*2/13)}:{sqrt(13)*0.18});
			\coordinate (G) at ({(asin(sqrt(21)/7))-15}:{sqrt(7)*0.18});
			\coordinate (H) at ({-(asin(sqrt(21)/7))-15}:{sqrt(7)*0.18});
			\coordinate (I) at (165:0.18);
			
			\coordinate (a) at ({asin(sqrt(39)*2/13)-15}:{sqrt(13)*0.18});
			\coordinate (b) at ({-asin(sqrt(39)*2/13)-15}:{sqrt(13)*0.18});
			\coordinate (c) at ({(asin(sqrt(21)/7))-195}:{sqrt(7)*0.18});
			\coordinate (d) at ({(165-asin(sqrt(21)/7))}:{sqrt(7)*0.18});
			\coordinate (e) at (-15:0.18);
			
			\draw (1)--(F);		\draw (2)--(A);		\draw (3)--(A);		\draw (4)--(G);
			\draw (5)--(H);		\draw (6)--(B);		\draw (7)--(B);		\draw (8)--(E);
			\draw (9)--(C);		\draw (10)--(C);	\draw (11)--(D);	\draw (12)--(D);
			
			\draw (a)--(A);		\draw (a)--(F);		\draw (a)--(G);
			\draw (b)--(B);		\draw (b)--(E);		\draw (b)--(H);
			\draw (c)--(C);		\draw (c)--(E);		\draw (c)--(I);
			\draw (d)--(D);		\draw (d)--(F);		\draw (d)--(I);
			\draw (e)--(G);		\draw (e)--(H);		\draw (e)--(I);
			
			\foreach \p in {A,B,C,D,E,F,G,H,I} {
				\fill[white] (\p) circle (2pt);
				\draw (\p) circle (2pt);
			}
			\foreach \p in {a,b,c,d,e}{
				\fill (\p) circle (2pt);
			}
			\node[below] at (0,-1.2) {$W_2$};
		}
	\end{minipage}
	\begin{minipage}[t]{0.19\textwidth}
		\centering
		\Disk{
			\coordinate (A) at (45:0.8);
			\coordinate (B) at (-15:0.8);
			\coordinate (C) at (-105:0.8);
			\coordinate (D) at (165:0.8);
			\coordinate (E) at (105:0.8);
			\coordinate (F) at (75:0.4);
			\coordinate (G) at (-45:0.4);
			\coordinate (H) at (-165:0.4);
			
			\coordinate (a) at (75:0.65);
			\coordinate (b) at (15:0.4);
			\coordinate (c) at (-105:0.4);
			\coordinate (d) at (135:0.4);
			
			\draw (1)--(E);		\draw (2)--(A);		\draw (3)--(A);		\draw (4)--(B);
			\draw (5)--(B);		\draw (6)--(G);		\draw (7)--(C);		\draw (8)--(C);
			\draw (9)--(H);		\draw (10)--(D);	\draw (11)--(D);	\draw (12)--(E);
			
			\draw (a)--(A);		\draw (a)--(E);		\draw (a)--(F);
			\draw (b)--(B);		\draw (b)--(F);		\draw (b)--(G);
			\draw (c)--(C);		\draw (c)--(G);		\draw (c)--(H);
			\draw (d)--(D);		\draw (d)--(F);		\draw (d)--(H);
			
			\foreach \p in {A,B,C,D,E,F,G,H} {
				\fill[white] (\p) circle (2pt);
				\draw (\p) circle (2pt);
			}
			\foreach \p in {a,b,c,d}{
				\fill (\p) circle (2pt);
			}
			\node[below] at (0,-1.2) {$W_3$};
		}
	\end{minipage}
	\begin{minipage}[t]{0.19\textwidth}
		\centering
		\Disk{
			\coordinate (A) at (45:0.8);
			\coordinate (B) at (-45:0.8);
			\coordinate (C) at (-135:0.8);
			\coordinate (D) at (105:0.8);
			\coordinate (E) at (-90:0.6);
			\coordinate (F) at (135:0.6);
			\coordinate (G) at (0:{sqrt(3)*0.2});
			\coordinate (H) at (180:{sqrt(3)*0.2});
			
			\coordinate (a) at (75:0.65);
			\coordinate (b) at ({-acos(sqrt(3/7))}:{sqrt(7)*0.2});
			\coordinate (c) at ({acos(sqrt(3/7))-180}:{sqrt(7)*0.2});
			\coordinate (d) at (90:0.2);
			
			\draw (1)--(D);		\draw (2)--(A);		\draw (3)--(A);		\draw (4)--(G);
			\draw (5)--(B);		\draw (6)--(B);		\draw (7)--(E);		\draw (8)--(C);
			\draw (9)--(C);		\draw (10)--(H);	\draw (11)--(F);	\draw (12)--(D);
			
			\draw (a)--(A);		\draw (a)--(D);		\draw (a)--(F);
			\draw (b)--(B);		\draw (b)--(E);		\draw (b)--(G);
			\draw (c)--(C);		\draw (c)--(E);		\draw (c)--(H);
			\draw (d)--(F);		\draw (d)--(G);		\draw (d)--(H);
			
			\foreach \p in {A,B,C,D,E,F,G,H} {
				\fill[white] (\p) circle (2pt);
				\draw (\p) circle (2pt);
			}
			\foreach \p in {a,b,c,d}{
				\fill (\p) circle (2pt);
			}
			\node[below] at (0,-1.2) {$W_4$};
		}
	\end{minipage}
	\begin{minipage}[t]{0.19\textwidth}
		\centering
		\Disk{
			\coordinate (A) at (75:0.8);
			\coordinate (B) at (-15:0.8);
			\coordinate (C) at (-105:0.8);
			\coordinate (D) at (165:0.8);
			\coordinate (E) at (30:0.5);
			\coordinate (F) at (-60:0.5);
			\coordinate (G) at (-150:0.5);
			\coordinate (H) at (120:0.5);
			
			\coordinate (a) at (75:0.5);
			\coordinate (b) at (-15:0.5);
			\coordinate (c) at (-105:0.5);
			\coordinate (d) at (165:0.5);
			
			\draw (1)--(A);		\draw (2)--(A);		\draw (3)--(E);		\draw (4)--(B);
			\draw (5)--(B);		\draw (6)--(F);		\draw (7)--(C);		\draw (8)--(C);
			\draw (9)--(G);		\draw (10)--(D);	\draw (11)--(D);	\draw (12)--(H);
			
			\draw (a)--(A);		\draw (a)--(E);		\draw (a)--(H);
			\draw (b)--(B);		\draw (b)--(E);		\draw (b)--(F);
			\draw (c)--(C);		\draw (c)--(F);		\draw (c)--(G);
			\draw (d)--(D);		\draw (d)--(G);		\draw (d)--(H);
			
			\foreach \p in {A,B,C,D,E,F,G,H} {
				\fill[white] (\p) circle (2pt);
				\draw (\p) circle (2pt);
			}
			\foreach \p in {a,b,c,d}{
				\fill (\p) circle (2pt);
			}
			\node[below] at (0,-1.2) {$W_5$};
		}
	\end{minipage}
	\\\vspace{0.2cm}
	\begin{minipage}[t]{0.19\textwidth}
		\centering
		\Disk{
			\coordinate (A) at (75:0.8);
			\coordinate (B) at (15:0.8);
			\coordinate (C) at (-45:0.8);
			\coordinate (D) at (-105:0.8);
			\coordinate (E) at (-165:0.8);
			\coordinate (F) at (135:0.8);
			\coordinate (G) at (0:0);
			
			\coordinate (a) at (45:0.4);
			\coordinate (b) at (-75:0.4);
			\coordinate (c) at (165:0.4);
			
			\draw (1)--(A);		\draw (2)--(A);		\draw (3)--(B);		\draw (4)--(B);
			\draw (5)--(C);		\draw (6)--(C);		\draw (7)--(D);		\draw (8)--(D);
			\draw (9)--(E);		\draw (10)--(E);	\draw (11)--(F);	\draw (12)--(F);
			
			\draw (a)--(A);		\draw (a)--(B);		\draw (a)--(G);
			\draw (b)--(C);		\draw (b)--(D);		\draw (b)--(G);
			\draw (c)--(E);		\draw (c)--(F);		\draw (c)--(G);
			
			\foreach \p in {A,B,C,D,E,F,G} {
				\fill[white] (\p) circle (2pt);
				\draw (\p) circle (2pt);
			}
			\foreach \p in {a,b,c}{
				\fill (\p) circle (2pt);
			}
			\node[below] at (0,-1.2) {$W_6$};
		}
	\end{minipage}
	\begin{minipage}[t]{0.19\textwidth}
		\centering
		\Disk{
			\coordinate (A) at (75:0.8);
			\coordinate (B) at (15:0.8);
			\coordinate (C) at (-45:0.8);
			\coordinate (D) at (-165:0.8);
			\coordinate (E) at (135:0.8);
			\coordinate (F) at (-75:0.5);
			\coordinate (G) at (-135:0.5);
			
			\coordinate (a) at (-15:0.6);
			\coordinate (b) at (165:0.6);
			\coordinate (c) at (75:0.2);
			
			\draw (1)--(A);		\draw (2)--(A);		\draw (3)--(B);		\draw (4)--(B);
			\draw (5)--(C);		\draw (6)--(C);		\draw (7)--(F);		\draw (8)--(G);
			\draw (9)--(D);		\draw (10)--(D);	\draw (11)--(E);	\draw (12)--(E);
			
			\draw (a)--(B);		\draw (a)--(C);		\draw (a)--(F);
			\draw (b)--(D);		\draw (b)--(E);		\draw (b)--(G);
			\draw (c)--(A);		\draw (c)--(F);		\draw (c)--(G);
			
			\foreach \p in {A,B,C,D,E,F,G} {
				\fill[white] (\p) circle (2pt);
				\draw (\p) circle (2pt);
			}
			\foreach \p in {a,b,c}{
				\fill (\p) circle (2pt);
			}
			\node[below] at (0,-1.2) {$W_7$};
		}
	\end{minipage}
	\begin{minipage}[t]{0.19\textwidth}
		\centering
		\Disk{
			\coordinate (A) at (75:0.8);
			\coordinate (B) at (-15:0.8);
			\coordinate (C) at (-75:0.8);
			\coordinate (D) at (-135:0.8);
			\coordinate (E) at (165:0.8);
			\coordinate (F) at (15:0.4);
			\coordinate (G) at (135:0.4);
			
			\coordinate (a) at (75:0.4);
			\coordinate (b) at (-45:0.4);
			\coordinate (c) at (-165:0.4);
			
			\draw (1)--(A);		\draw (2)--(A);		\draw (3)--(F);		\draw (4)--(B);
			\draw (5)--(B);		\draw (6)--(C);		\draw (7)--(C);		\draw (8)--(D);
			\draw (9)--(D);		\draw (10)--(E);	\draw (11)--(E);	\draw (12)--(G);
			
			\draw (a)--(A);		\draw (a)--(F);		\draw (a)--(G);
			\draw (b)--(B);		\draw (b)--(C);		\draw (b)--(F);
			\draw (c)--(D);		\draw (c)--(E);		\draw (c)--(G);
			
			\foreach \p in {A,B,C,D,E,F,G} {
				\fill[white] (\p) circle (2pt);
				\draw (\p) circle (2pt);
			}
			\foreach \p in {a,b,c}{
				\fill (\p) circle (2pt);
			}
			\node[below] at (0,-1.2) {$W_8$};
		}
	\end{minipage}
	\begin{minipage}[t]{0.19\textwidth}
		\centering
		\Disk{
			\coordinate (A) at (75:0.8);
			\coordinate (B) at (15:0.8);
			\coordinate (C) at (-75:0.8);
			\coordinate (D) at (-135:0.8);
			\coordinate (E) at (165:0.8);
			\coordinate (F) at (-15:0.5);
			\coordinate (G) at (135:0.5);
			
			\coordinate (a) at (45:0.6);
			\coordinate (b) at (-165:0.6);
			\coordinate (c) at (0:0);
			
			\draw (1)--(A);		\draw (2)--(A);		\draw (3)--(B);		\draw (4)--(B);
			\draw (5)--(F);		\draw (6)--(C);		\draw (7)--(C);		\draw (8)--(D);
			\draw (9)--(D);		\draw (10)--(E);	\draw (11)--(E);	\draw (12)--(G);
			
			\draw (a)--(A);		\draw (a)--(B);		\draw (a)--(F);
			\draw (b)--(D);		\draw (b)--(E);		\draw (b)--(G);
			\draw (c)--(C);		\draw (c)--(F);		\draw (c)--(G);
			
			\foreach \p in {A,B,C,D,E,F,G} {
				\fill[white] (\p) circle (2pt);
				\draw (\p) circle (2pt);
			}
			\foreach \p in {a,b,c}{
				\fill (\p) circle (2pt);
			}
			\node[below] at (0,-1.2) {$W_9$};
		}
	\end{minipage}
	\caption{Non-elliptic webs with 12 black boundary vertices and 1 connected component.}
	\label{fig:1con_12}
\end{figure}

\begin{proof}
	Let $W$ be a web with 1 connected component, that is $m=1$, then $|V_{int}|=12+2c-2=2c+10$ by \Cref{prp:Vint}. Furthermore, $c\leq7$ by \Cref{prp:c_Vint}.
	
	If $c=7$, then $|V_{int}|=24$, and so all of these $24$ vertices are required to construct the $7$ cycles, and at least one black vertex should be connected to the boundary by \Cref{lem:interior}. However, this is impossible because in plabic graphs, there are no white vertices which not in cycles connect the black internal vertex with a black boundary vertex. 
	
	If $c=6$, then $|V_{int}|=22$; similar to when $c=7$, all of these $22$ vertices are required to construct the $6$ cycles, but at least one black vertex should be connected to the boundary by \Cref{lem:interior}, which is also impossible. 
	
	If $c=5$, then $|V_{int}|=20$. There are $12$ white and $8$ black internal vertices since $k=4$ by \Cref{prp:n=kr}. Note that at least $9$ black vertices are needed to construct 5 cycles, this is impossible since we only have 8 black vertices. 
	
	If $c=4$, then $|V_{int}|=18$. There are $11$ white and $7$ black internal vertices since $k=4$ by \Cref{prp:n=kr}. Note that at least 8 black vertices are needed to construct 4 cycles, this is impossible since we only have 7 black vertices.
	
	If $c=3$, then $|V_{int}|=16$. There are $10$ white and $6$ black internal vertices since $k=4$ from \Cref{prp:n=kr}. Note that constructing 3 cycles requires at least $7$ white and $6$ black vertices. In this case, the 3 remaining white vertices are utilized to connect the bivalent black vertices on the cycles with 2 boundary vertices to construct a claw. Now we obtained one non-elliptic web $W_1$ which shown in \Cref{fig:1con_12}.
	
	If $c=2$, then $|V_{int}|=14$. There are $9$ white and $5$ black internal vertices since $k=4$ from \Cref{prp:n=kr}. Note that constructing 2 cycles requires at least $5$ white and $5$ black vertices. In this case, the $4$ remaining white vertices are utilized to connect the bivalent black vertices on the cycles with 2 boundary vertices to construct a claw. Now we obtained one non-elliptic web $W_2$ which shown in \Cref{fig:1con_12}.
	
	If $c=1$, then $|V_{int}|=12$. There are $8$ white internal vertices and $4$ black since $k=4$ from \Cref{prp:n=kr}. If the cycle is a hexagon, then there are $3$ white and $3$ black vertices in it, the remaining $5$ white and 1 black vertices can be used to connect the cycle to boundary vertices. Therefore, there are $2$ non-elliptic webs, $W_3$ and $W_4$ can be obtained which shown in \Cref{fig:1con_12}. If the cycle is an octagon, then there are $4$ black and $4$ white vertices in it. Therefore, the other $4$ white vertices connect the black vertices in the octagon to $2$ boundary vertices to construct a claw, there is another web which is $W_5$, in \Cref{fig:1con_12}.
	
	And if $c=0$, then $|V_{int}|=10$. There are $7$ white and $3$ black internal vertices. Consider the number of claws in web, if there are 6 claws in the web, because there are only $12$ boundary vertices in the web, there is only one case $W_6$. If there are 5 claws in the web, there are $2$ boundary vertices are not in the $5$ claws, and there are $3$ cases, which are $W_7$ to $W_9$. If there are less than 5 claws in the web, then less than 8 boundary vertices are contained in these claws, and the other boundary vertices are connected to different white internal vertices, then at least 8 white vertices are required, which is impossible. Similarly, the cases with fewer claws in the webs are also impossible.
\end{proof}

\begin{lemma}\label{lem:2con_12}
	All non-elliptic webs with $12$ black boundary vertices and $2$ connected components are the dihedral translates of the $10$ webs labeled from $W_{10}$ to $W_{19}$ in \Cref{fig:2con_12}.
\end{lemma}

\begin{figure}[h]
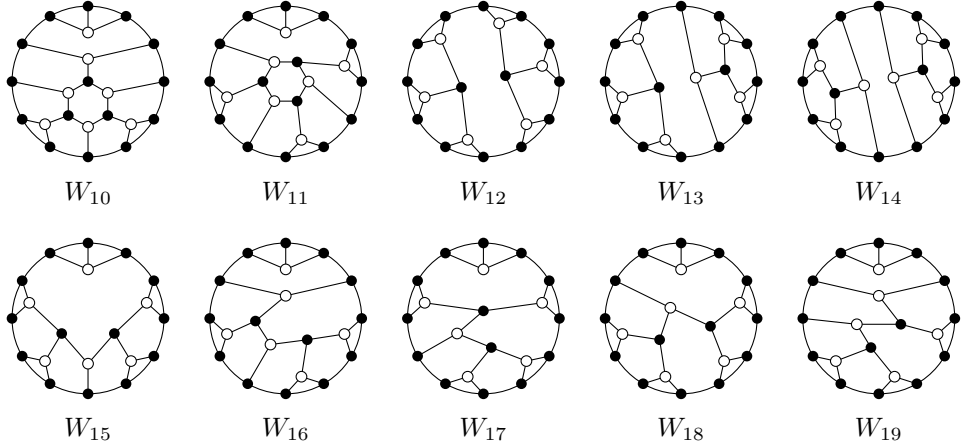

	\centering
	\begin{minipage}[t]{0.19\textwidth}
		\centering
		\Disk{
			\coordinate (A) at (90:0.65);
			\coordinate (B) at (-45:0.8);
			\coordinate (C) at (-135:0.8);
			\coordinate (D) at (-90:0.6);
			\coordinate (E) at (90:0.3);
			\coordinate (F) at (-30:0.3);
			\coordinate (G) at (-150:0.3);
			
			\coordinate (a) at (-60:{0.3*sqrt(3)});
			\coordinate (b) at (-120:{0.3*sqrt(3)});
			\coordinate (c) at (0:0);
			
			\draw (1)--(A);		\draw (2)--(A);		\draw (3)--(E);		\draw (4)--(F);
			\draw (5)--(B);		\draw (6)--(B);		\draw (7)--(D);		\draw (8)--(C);
			\draw (9)--(C);		\draw (10)--(G);	\draw (11)--(E);	\draw (12)--(A);
			
			\draw (a)--(B);		\draw (a)--(D);		\draw (a)--(F);
			\draw (b)--(C);		\draw (b)--(D);		\draw (b)--(G);
			\draw (c)--(E);		\draw (c)--(F);		\draw (c)--(G);
			
			\foreach \p in {A,B,C,D,E,F,G} {
				\fill[white] (\p) circle (2pt);
				\draw (\p) circle (2pt);
			}
			\foreach \p in {a,b,c}{
				\fill (\p) circle (2pt);
			}
			\node[below] at (0,-1.2) {$W_{10}$};
		}
	\end{minipage}
	\begin{minipage}[t]{0.19\textwidth}
		\centering
		\Disk{
			\coordinate (A) at (90:0.65);
			\coordinate (B) at (15:0.8);
			\coordinate (C) at (-75:0.8);
			\coordinate (D) at (-165:0.8);
			\coordinate (E) at (0:0.3);
			\coordinate (F) at (-120:0.3);
			\coordinate (G) at (120:0.3);
			
			\coordinate (a) at (60:0.3);
			\coordinate (b) at (-60:0.3);
			\coordinate (c) at (180:0.3);
			
			\draw (1)--(A);		\draw (2)--(A);		\draw (3)--(B);		\draw (4)--(B);
			\draw (5)--(E);		\draw (6)--(C);		\draw (7)--(C);		\draw (8)--(F);
			\draw (9)--(D);		\draw (10)--(D);	\draw (11)--(G);	\draw (12)--(A);
			
			\draw (a)--(B);		\draw (a)--(E);		\draw (a)--(G);
			\draw (b)--(C);		\draw (b)--(E);		\draw (b)--(F);
			\draw (c)--(D);		\draw (c)--(F);		\draw (c)--(G);
			
			\foreach \p in {A,B,C,D,E,F,G} {
				\fill[white] (\p) circle (2pt);
				\draw (\p) circle (2pt);
			}
			\foreach \p in {a,b,c}{
				\fill (\p) circle (2pt);
			}
			\node[below] at (0,-1.2) {$W_{11}$};
		}
	\end{minipage}
	\begin{minipage}[t]{0.19\textwidth}
		\centering
		\Disk{
			\coordinate (A) at (75:0.8);
			\coordinate (B) at (15:0.8);
			\coordinate (C) at (-45:0.8);
			\coordinate (D) at (-105:0.8);
			\coordinate (E) at (-165:0.8);
			\coordinate (F) at (135:0.8);
			
			\coordinate (a) at (15:0.3);
			\coordinate (b) at (-165:0.3);
			
			\draw (1)--(A);		\draw (2)--(A);		\draw (3)--(B);		\draw (4)--(B);
			\draw (5)--(C);		\draw (6)--(C);		\draw (7)--(D);		\draw (8)--(D);
			\draw (9)--(E);		\draw (10)--(E);	\draw (11)--(F);	\draw (12)--(F);
			
			\draw (a)--(A);		\draw (a)--(B);		\draw (a)--(C);
			\draw (b)--(D);		\draw (b)--(E);		\draw (b)--(F);
			
			\foreach \p in {A,B,C,D,E,F} {
				\fill[white] (\p) circle (2pt);
				\draw (\p) circle (2pt);
			}
			\foreach \p in {a,b}{
				\fill (\p) circle (2pt);
			}
			\node[below] at (0,-1.2) {$W_{12}$};
		}
	\end{minipage}
	\begin{minipage}[t]{0.19\textwidth}
		\centering
		\Disk{
			\coordinate (A) at (45:0.8);
			\coordinate (B) at (-15:0.8);
			\coordinate (C) at (-105:0.8);
			\coordinate (D) at (-165:0.8);
			\coordinate (E) at (135:0.8);
			\coordinate (F) at (15:0.2);
			
			\coordinate (a) at (15:0.6);
			\coordinate (b) at (-165:0.3);
			
			\draw (1)--(F);		\draw (2)--(A);		\draw (3)--(A);		\draw (4)--(B);
			\draw (5)--(B);		\draw (6)--(F);		\draw (7)--(C);		\draw (8)--(C);
			\draw (9)--(D);		\draw (10)--(D);	\draw (11)--(E);	\draw (12)--(E);
			
			\draw (a)--(A);		\draw (a)--(B);		\draw (a)--(F);
			\draw (b)--(C);		\draw (b)--(D);		\draw (b)--(E);
			
			\foreach \p in {A,B,C,D,E,F} {
				\fill[white] (\p) circle (2pt);
				\draw (\p) circle (2pt);
			}
			\foreach \p in {a,b}{
				\fill (\p) circle (2pt);
			}
			\node[below] at (0,-1.2) {$W_{13}$};
		}
	\end{minipage}
	\begin{minipage}[t]{0.19\textwidth}
		\centering
		\Disk{
			\coordinate (A) at (45:0.8);
			\coordinate (B) at (-15:0.8);
			\coordinate (C) at (-135:0.8);
			\coordinate (D) at (165:0.8);
			\coordinate (E) at (15:0.2);
			\coordinate (F) at (-165:0.2);
			
			\coordinate (a) at (15:0.6);
			\coordinate (b) at (-165:0.6);
			
			\draw (1)--(E);		\draw (2)--(A);		\draw (3)--(A);		\draw (4)--(B);
			\draw (5)--(B);		\draw (6)--(E);		\draw (7)--(F);		\draw (8)--(C);
			\draw (9)--(C);		\draw (10)--(D);	\draw (11)--(D);	\draw (12)--(F);
			
			\draw (a)--(A);		\draw (a)--(B);		\draw (a)--(E);
			\draw (b)--(C);		\draw (b)--(D);		\draw (b)--(F);
			
			\foreach \p in {A,B,C,D,E,F} {
				\fill[white] (\p) circle (2pt);
				\draw (\p) circle (2pt);
			}
			\foreach \p in {a,b}{
				\fill (\p) circle (2pt);
			}
			\node[below] at (0,-1.2) {$W_{14}$};
		}
	\end{minipage}
	\\\vspace{0.3cm}
	\begin{minipage}[t]{0.19\textwidth}
		\centering
		\Disk{
			\coordinate (A) at (90:0.65);
			\coordinate (B) at (15:0.8);
			\coordinate (C) at (-45:0.8);
			\coordinate (D) at (-135:0.8);
			\coordinate (E) at (165:0.8);
			\coordinate (F) at (-90:0.6);
			
			\coordinate (a) at (-30:0.4);
			\coordinate (b) at (-150:0.4);
			
			\draw (1)--(A);		\draw (2)--(A);		\draw (3)--(B);		\draw (4)--(B);
			\draw (5)--(C);		\draw (6)--(C);		\draw (7)--(F);		\draw (8)--(D);
			\draw (9)--(D);		\draw (10)--(E);	\draw (11)--(E);	\draw (12)--(A);
			
			\draw (a)--(B);		\draw (a)--(C);		\draw (a)--(F);
			\draw (b)--(D);		\draw (b)--(E);		\draw (b)--(F);
			
			\foreach \p in {A,B,C,D,E,F} {
				\fill[white] (\p) circle (2pt);
				\draw (\p) circle (2pt);
			}
			\foreach \p in {a,b}{
				\fill (\p) circle (2pt);
			}
			\node[below] at (0,-1.2) {$W_{15}$};
		}
	\end{minipage}
	\begin{minipage}[t]{0.19\textwidth}
		\centering
		\Disk{
			\coordinate (A) at (90:0.65);
			\coordinate (B) at (-15:0.8);
			\coordinate (C) at (-75:0.8);
			\coordinate (D) at (-165:0.8);
			\coordinate (E) at (-120:0.4);
			\coordinate (F) at (90:0.3);
			
			\coordinate (a) at (-45:0.4);
			\coordinate (b) at (-175:0.4);
			
			\draw (1)--(A);		\draw (2)--(A);		\draw (3)--(F);		\draw (4)--(B);
			\draw (5)--(B);		\draw (6)--(C);		\draw (7)--(C);		\draw (8)--(E);
			\draw (9)--(D);		\draw (10)--(D);	\draw (11)--(F);	\draw (12)--(A);
			
			\draw (a)--(B);		\draw (a)--(C);		\draw (a)--(E);
			\draw (b)--(D);		\draw (b)--(E);		\draw (b)--(F);
			
			\foreach \p in {A,B,C,D,E,F} {
				\fill[white] (\p) circle (2pt);
				\draw (\p) circle (2pt);
			}
			\foreach \p in {a,b}{
				\fill (\p) circle (2pt);
			}
			\node[below] at (0,-1.2) {$W_{16}$};
		}
	\end{minipage}
	\begin{minipage}[t]{0.19\textwidth}
		\centering
		\Disk{
			\coordinate (A) at (90:0.65);
			\coordinate (B) at (15:0.8);
			\coordinate (C) at (-45:0.8);
			\coordinate (D) at (-105:0.8);
			\coordinate (E) at (165:0.8);
			\coordinate (F) at (-150:0.4);
			
			\coordinate (a) at (-75:0.4);
			\coordinate (b) at (90:0.1);
			
			\draw (1)--(A);		\draw (2)--(A);		\draw (3)--(B);		\draw (4)--(B);
			\draw (5)--(C);		\draw (6)--(C);		\draw (7)--(D);		\draw (8)--(D);
			\draw (9)--(F);		\draw (10)--(E);	\draw (11)--(E);	\draw (12)--(A);
			
			\draw (a)--(C);		\draw (a)--(D);		\draw (a)--(F);
			\draw (b)--(B);		\draw (b)--(E);		\draw (b)--(F);
			
			\foreach \p in {A,B,C,D,E,F} {
				\fill[white] (\p) circle (2pt);
				\draw (\p) circle (2pt);
			}
			\foreach \p in {a,b}{
				\fill (\p) circle (2pt);
			}
			\node[below] at (0,-1.2) {$W_{17}$};
		}
	\end{minipage}
	\begin{minipage}[t]{0.19\textwidth}
		\centering
		\Disk{
			\coordinate (A) at (90:0.65);
			\coordinate (B) at (15:0.8);
			\coordinate (C) at (-45:0.8);
			\coordinate (D) at (-105:0.8);
			\coordinate (E) at (-165:0.8);
			\coordinate (F) at (135:0.2);
			
			\coordinate (a) at (-15:0.4);
			\coordinate (b) at (-135:0.4);
			
			\draw (1)--(A);		\draw (2)--(A);		\draw (3)--(B);		\draw (4)--(B);
			\draw (5)--(C);		\draw (6)--(C);		\draw (7)--(D);		\draw (8)--(D);
			\draw (9)--(E);		\draw (10)--(E);	\draw (11)--(F);	\draw (12)--(A);
			
			\draw (a)--(B);		\draw (a)--(C);		\draw (a)--(F);
			\draw (b)--(D);		\draw (b)--(E);		\draw (b)--(F);
			
			\foreach \p in {A,B,C,D,E,F} {
				\fill[white] (\p) circle (2pt);
				\draw (\p) circle (2pt);
			}
			\foreach \p in {a,b}{
				\fill (\p) circle (2pt);
			}
			\node[below] at (0,-1.2) {$W_{18}$};
		}
	\end{minipage}
	\begin{minipage}[t]{0.19\textwidth}
		\centering
		\Disk{
			\coordinate (A) at (90:0.65);
			\coordinate (B) at (-15:0.8);
			\coordinate (C) at (-75:0.8);
			\coordinate (D) at (-135:0.8);
			\coordinate (E) at (90:0.3);
			\coordinate (F) at (-165:0.3);
			
			\coordinate (a) at (-105:0.4);
			\coordinate (b) at (-15:0.3);
			
			\draw (1)--(A);		\draw (2)--(A);		\draw (3)--(E);		\draw (4)--(B);
			\draw (5)--(B);		\draw (6)--(C);		\draw (7)--(C);		\draw (8)--(D);
			\draw (9)--(D);		\draw (10)--(F);	\draw (11)--(E);	\draw (12)--(A);
			
			\draw (a)--(C);		\draw (a)--(D);		\draw (a)--(F);
			\draw (b)--(B);		\draw (b)--(E);		\draw (b)--(F);
			
			\foreach \p in {A,B,C,D,E,F} {
				\fill[white] (\p) circle (2pt);
				\draw (\p) circle (2pt);
			}
			\foreach \p in {a,b}{
				\fill (\p) circle (2pt);
			}
			\node[below] at (0,-1.2) {$W_{19}$};
		}
	\end{minipage}
	\caption{Non-elliptic webs with 12 black boundary vertices and 2 connected component.}
	\label{fig:2con_12}
\end{figure}

\begin{proof}
	Let $W$ be a web with 2 connected components, that is $m=2$, then $|V_{int}|=12+2c-4=2c+8$ by \Cref{prp:Vint}. Furthermore, $c\leq 4$ by \Cref{prp:c_Vint}.
	
	If $c=4$, then $|V_{int}|=16$. There are $10$ white and 6 black internal vertices since $k=4$ from \Cref{prp:n=kr}. Note that at least 8 black internal vertices are needed to construct 4 cycles, this is impossible since we only have 6 black vertices. 
	
	If $c=3$, then $|V_{int}|=14$. There are 9 white and 5 black internal vertices since $k=4$ from \Cref{prp:n=kr}. Note that at least 6 black internal vertices are needed to construct 3 cycles, this is impossible since we only have 5 black vertices. 
	
	If $c=2$, then $|V_{int}|=12$. There are $8$ white and 4 black internal vertices since $k=4$ from \Cref{prp:n=kr}. Note that at least 5 black internal vertices are needed to construct 2 cycles, this is impossible since we only have 4 black vertices.
	
	If $c=1$, then $|V_{int}|=10$. There are $7$ white and 3 black internal vertices since $k=4$ from \Cref{prp:n=kr}. Note that constructing the only cycle needs 3 white and 3 black vertices exactly, the other 4 white vertices are connected to boundary vertices. Consider the components of the web, the component containing the cycle contains at least 9 boundary vertices, and the other component contains at least 3 boundary vertices. So there are only the 2 cases $W_{10}$ and $W_{11}$ in \Cref{fig:2con_12}.
	
	If $c=0$, we have $|V_{int}|=8$. Since $k=4$ from \Cref{prp:n=kr}, there are 6 white and just 2 black internal vertices. Consider the 2 black vertices, if they are in the 2 different components, the 2 components are both contains 6 vertices, they are the 3 different cases $W_{12}$ to $W_{14}$ in \Cref{fig:2con_12}. And if the 2 black vertices are in the same component, there is a component have no black vertices, it is a tripod, the component with 2 black vertices can be seen as a non-elliptic web with 9 boundary vertices which enumerated in \cite{fraser_dimers_2019} and their dihedral translations, so there is 5 different webs from $W_{15}$ to $W_{19}$.
\end{proof}

\begin{lemma}\label{lem:3con_12}
	All non-elliptic webs with $12$ black boundary vertices and $3$ connected components are the dihedral translates of the $8$ webs from $W_{20}$ to $W_{27}$ in \Cref{fig:3con_12}.
\end{lemma}

\begin{figure}[h]
	\centering
	\begin{minipage}[t]{0.23\textwidth}
		\centering
		\Disk{
			\coordinate (A) at (90:0.65);
			\coordinate (B) at (0:0.65);
			\coordinate (C) at (-75:0.8);
			\coordinate (D) at (-135:0.8);
			\coordinate (E) at (165:0.8);
			
			\coordinate (a) at (-135:0.3);
			
			\draw (1)--(A);		\draw (2)--(A);		\draw (3)--(B);		\draw (4)--(B);
			\draw (5)--(B);		\draw (6)--(C);		\draw (7)--(C);		\draw (8)--(D);
			\draw (9)--(D);		\draw (10)--(E);	\draw (11)--(E);	\draw (12)--(A);
			
			\draw (a)--(C);		\draw (a)--(D);		\draw (a)--(E);
			
			\foreach \p in {A,B,C,D,E} {
				\fill[white] (\p) circle (2pt);
				\draw (\p) circle (2pt);
			}
			\foreach \p in {a}{
				\fill (\p) circle (2pt);
			}
			\node[below] at (0,-1.2) {$W_{20}$};
		}
	\end{minipage}
	\begin{minipage}[t]{0.23\textwidth}
		\centering
		\Disk{
			\coordinate (A) at (90:0.65);
			\coordinate (B) at (0:0.65);
			\coordinate (C) at (-105:0.8);
			\coordinate (D) at (-165:0.8);
			\coordinate (E) at (-135:0.2);
			
			\coordinate (a) at (-135:0.6);
			
			\draw (1)--(A);		\draw (2)--(A);		\draw (3)--(B);		\draw (4)--(B);
			\draw (5)--(B);		\draw (6)--(E);		\draw (7)--(C);		\draw (8)--(C);
			\draw (9)--(D);		\draw (10)--(D);	\draw (11)--(E);	\draw (12)--(A);
			
			\draw (a)--(C);		\draw (a)--(D);		\draw (a)--(E);
			
			\foreach \p in {A,B,C,D,E} {
				\fill[white] (\p) circle (2pt);
				\draw (\p) circle (2pt);
			}
			\foreach \p in {a}{
				\fill (\p) circle (2pt);
			}
			\node[below] at (0,-1.2) {$W_{21}$};
		}
	\end{minipage}
	\begin{minipage}[t]{0.23\textwidth}
		\centering
		\Disk{
			\coordinate (A) at (60:0.65);
			\coordinate (B) at (60:0.3);
			\coordinate (C) at (-75:0.8);
			\coordinate (D) at (-135:0.8);
			\coordinate (E) at (165:0.8);
			
			\coordinate (a) at (-135:0.3);
			
			\draw (1)--(A);		\draw (2)--(A);		\draw (3)--(A);		\draw (4)--(B);
			\draw (5)--(B);		\draw (6)--(C);		\draw (7)--(C);		\draw (8)--(D);
			\draw (9)--(D);		\draw (10)--(E);	\draw (11)--(E);	\draw (12)--(B);
			
			\draw (a)--(C);		\draw (a)--(D);		\draw (a)--(E);
			
			\foreach \p in {A,B,C,D,E} {
				\fill[white] (\p) circle (2pt);
				\draw (\p) circle (2pt);
			}
			\foreach \p in {a}{
				\fill (\p) circle (2pt);
			}
			\node[below] at (0,-1.2) {$W_{22}$};
		}
	\end{minipage}
	\begin{minipage}[t]{0.23\textwidth}
		\centering
		\Disk{
			\coordinate (A) at (60:0.65);
			\coordinate (B) at (60:0.3);
			\coordinate (C) at (-105:0.8);
			\coordinate (D) at (-165:0.8);
			\coordinate (E) at (-135:0.2);
			
			\coordinate (a) at (-135:0.6);
			
			\draw (1)--(A);		\draw (2)--(A);		\draw (3)--(A);		\draw (4)--(B);
			\draw (5)--(B);		\draw (6)--(E);		\draw (7)--(C);		\draw (8)--(C);
			\draw (9)--(D);		\draw (10)--(D);	\draw (11)--(E);	\draw (12)--(B);
			
			\draw (a)--(C);		\draw (a)--(D);		\draw (a)--(E);
			
			\foreach \p in {A,B,C,D,E} {
				\fill[white] (\p) circle (2pt);
				\draw (\p) circle (2pt);
			}
			\foreach \p in {a}{
				\fill (\p) circle (2pt);
			}
			\node[below] at (0,-1.2) {$W_{23}$};
		}
	\end{minipage}
	\\\vspace{0.3cm}
	\begin{minipage}[t]{0.23\textwidth}
		\centering
		\Disk{
			\coordinate (A) at (90:0.65);
			\coordinate (B) at (-30:0.65);
			\coordinate (C) at (-105:0.8);
			\coordinate (D) at (-165:0.8);
			\coordinate (E) at (90:0.2);
			
			\coordinate (a) at (-135:0.4);
			
			\draw (1)--(A);		\draw (2)--(A);		\draw (3)--(E);		\draw (4)--(B);
			\draw (5)--(B);		\draw (6)--(B);		\draw (7)--(C);		\draw (8)--(C);
			\draw (9)--(D);		\draw (10)--(D);	\draw (11)--(E);	\draw (12)--(A);
			
			\draw (a)--(C);		\draw (a)--(D);		\draw (a)--(E);
			
			\foreach \p in {A,B,C,D,E} {
				\fill[white] (\p) circle (2pt);
				\draw (\p) circle (2pt);
			}
			\foreach \p in {a}{
				\fill (\p) circle (2pt);
			}
			\node[below] at (0,-1.2) {$W_{24}$};
		}
	\end{minipage}
	\begin{minipage}[t]{0.23\textwidth}
		\centering
		\Disk{
			\coordinate (A) at (90:0.65);
			\coordinate (B) at (-60:0.65);
			\coordinate (C) at (15:0.8);
			\coordinate (D) at (-135:0.8);
			\coordinate (E) at (165:0.8);
			
			\coordinate (a) at (0:0);
			
			\draw (1)--(A);		\draw (2)--(A);		\draw (3)--(C);		\draw (4)--(C);
			\draw (5)--(B);		\draw (6)--(B);		\draw (7)--(B);		\draw (8)--(D);
			\draw (9)--(D);		\draw (10)--(E);	\draw (11)--(E);	\draw (12)--(A);
			
			\draw (a)--(C);		\draw (a)--(D);		\draw (a)--(E);
			
			\foreach \p in {A,B,C,D,E} {
				\fill[white] (\p) circle (2pt);
				\draw (\p) circle (2pt);
			}
			\foreach \p in {a}{
				\fill (\p) circle (2pt);
			}
			\node[below] at (0,-1.2) {$W_{25}$};
		}
	\end{minipage}
	\begin{minipage}[t]{0.23\textwidth}
		\centering
		\Disk{
			\coordinate (A) at (90:0.65);
			\coordinate (B) at (-60:0.65);
			\coordinate (C) at (-60:0.35);
			\coordinate (D) at (-165:0.8);
			\coordinate (E) at (90:0.35);
			
			\coordinate (a) at (0:0);
			
			\draw (1)--(A);		\draw (2)--(A);		\draw (3)--(E);		\draw (4)--(C);
			\draw (5)--(B);		\draw (6)--(B);		\draw (7)--(B);		\draw (8)--(C);
			\draw (9)--(D);		\draw (10)--(D);	\draw (11)--(E);	\draw (12)--(A);
			
			\draw (a)--(C);		\draw (a)--(D);		\draw (a)--(E);
			
			\foreach \p in {A,B,C,D,E} {
				\fill[white] (\p) circle (2pt);
				\draw (\p) circle (2pt);
			}
			\foreach \p in {a}{
				\fill (\p) circle (2pt);
			}
			\node[below] at (0,-1.2) {$W_{26}$};
		}
	\end{minipage}
	\begin{minipage}[t]{0.23\textwidth}
		\centering
		\Disk{
			\coordinate (A) at (90:0.65);
			\coordinate (B) at (-90:0.65);
			\coordinate (C) at (15:0.8);
			\coordinate (D) at (-90:0.35);
			\coordinate (E) at (165:0.8);
			
			\coordinate (a) at (0:0);
			
			\draw (1)--(A);		\draw (2)--(A);		\draw (3)--(C);		\draw (4)--(C);
			\draw (5)--(D);		\draw (6)--(B);		\draw (7)--(B);		\draw (8)--(B);
			\draw (9)--(D);		\draw (10)--(E);	\draw (11)--(E);	\draw (12)--(A);
			
			\draw (a)--(C);		\draw (a)--(D);		\draw (a)--(E);
			
			\foreach \p in {A,B,C,D,E} {
				\fill[white] (\p) circle (2pt);
				\draw (\p) circle (2pt);
			}
			\foreach \p in {a}{
				\fill (\p) circle (2pt);
			}
			\node[below] at (0,-1.2) {$W_{27}$};
		}
	\end{minipage}
	\caption{Non-elliptic webs with 12 black boundary vertices and 3 connected component.}
	\label{fig:3con_12}
\end{figure}

\begin{proof}
	Let $W$ be a web with 3 connected components, that is $m=3$. Note that the degree of all internal vertices are 3. In this case, the components can only be tripods or hexapods. Therefore, there are 2 tripods and 1 hexapod in the 3 components. Now we consider the positions of the 2 tripods and 1 hexapod, there are 8 kinds of webs $W_{20}$ -$W_{27}$ shown in \Cref{fig:3con_12}.
\end{proof}

\begin{lemma}
	All non-elliptic webs with $12$ black boundary vertices and $4$ connected components are the dihedral translates of the $5$ webs from $W_{28}$ to $W_{32}$ in \Cref{fig:4con_12}.
\end{lemma}

\begin{figure}[h]
	\centering
	\begin{minipage}[t]{0.19\textwidth}
		\centering
		\Disk{
			\coordinate (A) at (90:0.65);
			\coordinate (B) at (0:0.65);
			\coordinate (C) at (-90:0.65);
			\coordinate (D) at (180:0.65);
			
			\draw (1)--(A);		\draw (2)--(A);		\draw (3)--(B);		\draw (4)--(B);
			\draw (5)--(B);		\draw (6)--(C);		\draw (7)--(C);		\draw (8)--(C);
			\draw (9)--(D);		\draw (10)--(D);	\draw (11)--(D);	\draw (12)--(A);
			
			\foreach \p in {A,B,C,D} {
				\fill[white] (\p) circle (2pt);
				\draw (\p) circle (2pt);
			}
			\node[below] at (0,-1.2) {$W_{28}$};
		}
	\end{minipage}
	\begin{minipage}[t]{0.19\textwidth}
		\centering
		\Disk{
			\coordinate (A) at (90:0.65);
			\coordinate (B) at (-30:0.65);
			\coordinate (C) at (-120:0.65);
			\coordinate (D) at (90:0.2);
			
			\draw (1)--(A);		\draw (2)--(A);		\draw (3)--(D);		\draw (4)--(B);
			\draw (5)--(B);		\draw (6)--(B);		\draw (7)--(C);		\draw (8)--(C);
			\draw (9)--(C);		\draw (10)--(D);	\draw (11)--(D);	\draw (12)--(A);
			
			\foreach \p in {A,B,C,D} {
				\fill[white] (\p) circle (2pt);
				\draw (\p) circle (2pt);
			}
			\node[below] at (0,-1.2) {$W_{29}$};
		}
	\end{minipage}
	\begin{minipage}[t]{0.19\textwidth}
		\centering
		\Disk{
			\coordinate (A) at (90:0.65);
			\coordinate (B) at (-30:0.65);
			\coordinate (C) at (-150:0.65);
			\coordinate (D) at (0:0);
			
			\draw (1)--(A);		\draw (2)--(A);		\draw (3)--(D);		\draw (4)--(B);
			\draw (5)--(B);		\draw (6)--(B);		\draw (7)--(D);		\draw (8)--(C);
			\draw (9)--(C);		\draw (10)--(C);	\draw (11)--(D);	\draw (12)--(A);
			
			\foreach \p in {A,B,C,D} {
				\fill[white] (\p) circle (2pt);
				\draw (\p) circle (2pt);
			}
			\node[below] at (0,-1.2) {$W_{30}$};
		}
	\end{minipage}
	\begin{minipage}[t]{0.19\textwidth}
		\centering
		\Disk{
			\coordinate (A) at (90:0.65);
			\coordinate (B) at (-90:0.65);
			\coordinate (C) at (90:0.2);
			\coordinate (D) at (-90:0.2);
			
			\draw (1)--(A);		\draw (2)--(A);		\draw (3)--(C);		\draw (4)--(D);
			\draw (5)--(D);		\draw (6)--(B);		\draw (7)--(B);		\draw (8)--(B);
			\draw (9)--(D);		\draw (10)--(C);	\draw (11)--(C);	\draw (12)--(A);
			
			\foreach \p in {A,B,C,D} {
				\fill[white] (\p) circle (2pt);
				\draw (\p) circle (2pt);
			}
			\node[below] at (0,-1.2) {$W_{31}$};
		}
	\end{minipage}
	\begin{minipage}[t]{0.19\textwidth}
		\centering
		\Disk{
			\coordinate (A) at (90:0.65);
			\coordinate (B) at (-120:0.65);
			\coordinate (C) at (90:0.2);
			\coordinate (D) at (-120:0.2);
			
			\draw (1)--(A);		\draw (2)--(A);		\draw (3)--(C);		\draw (4)--(C);
			\draw (5)--(D);		\draw (6)--(D);		\draw (7)--(B);		\draw (8)--(B);
			\draw (9)--(B);		\draw (10)--(D);	\draw (11)--(C);	\draw (12)--(A);
			
			\foreach \p in {A,B,C,D} {
				\fill[white] (\p) circle (2pt);
				\draw (\p) circle (2pt);
			}
			\node[below] at (0,-1.2) {$W_{32}$};
		}
	\end{minipage}
	\caption{Non-elliptic webs with 12 black boundary vertices and 3 connected component.}
	\label{fig:4con_12}
\end{figure}

\begin{proof}
	Note that the degree of all internal vertices are 3, the 4 components in the web can only be 4 tripods. Now we just need to consider the positions of the 4 tripods. there are 5 kinds of webs $W_{28}$ -$W_{32}$ shown in \Cref{fig:4con_12}.
\end{proof}

\begin{theorem} \label{thm:Gr4_12}
	All non-elliptic webs with 12 black boundary vertices are the dihedral translates of the webs listed in Figures \ref{fig:1con_12}, \ref{fig:2con_12}, \ref{fig:3con_12} and \ref{fig:4con_12}. There are 32 such webs.
\end{theorem}

We have another method to verify all webs with that all $12$ boundary vertices are black. In \cite{tymoczko_simple_2012}, the $m$-diagram on the boundary line provides a bijection between  $3$-column standard Young tableaux and $sl_3$-webs. Utilizing the hook-length formula, it is easy to compute that there are $462$ standard Young tableaux of shape $(3,3,3,3)$. Clearly, the number of dihedral translations of webs shown in Figures \ref{fig:1con_12}, \ref{fig:2con_12}, \ref{fig:3con_12} and \ref{fig:4con_12} are exactly 462, which also show our Theorem \ref{thm:Gr4_12} is hold.

Now, we present an examples to show how non-elliptic webs with 4 black and 4 white boundary vertices obtained by non-elliptic webs with 12 black boundary vertices. Additionally, we classify the webs according to the arrangement of boundary vertices. Furthermore, in Section \ref{sec:typelist}, we provide the complete list of non-elliptic webs with 4 black and 4 white boundary vertices.

For a web, we label the vertex at the top as ``1'', and then label the other vertices sequentially in a clockwise direction. For convenience, we denote the maps $\rho$ and $\tau$ on $Z_n$ as $\rho:i\mapsto n-i+1$ (mod $n$) and $\tau:i\mapsto i-1$ (mod $n$), which act on the label of boundary vertices and correspond to reflections and rotations mentioned in Section \ref{ssc:WinPromTrip}. (Here we identify 0  (mod $n$) with $n$.)

For a claw with only two toes, we can contract it into a sink vertex. And for a claw with three toes (also called a tripod, which is a whole component), only two adjacent toes of the three can be constructed, changing the tripod into an arrow directed from the black boundary vertex to the white one. The following Example \ref{exm:3toe} illustrates the above contraction process.	
\begin{example}\label{exm:3toe}
	Consider the web $W_{21}$ in Figure \ref{fig:3con_12}, there are exactly $4$ claws, but $2$ of them consist of $3$ toes, so it is necessary to discuss the direction of arrows. The contracted $2$ toes are also denoted by the set of their labels.
	
	\begin{figure}[h]
		\begin{tikzpicture}[remember picture]
			\hspace{0.5cm}
			\node (left) {
				\Disk{
					\coordinate (A) at (90:0.65);
					\coordinate (B) at (0:0.65);
					\coordinate (C) at (-105:0.8);
					\coordinate (D) at (-165:0.8);
					\coordinate (E) at (-135:0.2);
					
					\coordinate (a) at (-135:0.6);
					
					\draw (1)--(A);		\draw (2)--(A);		\draw (3)--(B);		\draw (4)--(B);
					\draw (5)--(B);		\draw (6)--(E);		\draw (7)--(C);		\draw (8)--(C);
					\draw (9)--(D);		\draw (10)--(D);	\draw (11)--(E);	\draw (12)--(A);
					
					\draw (a)--(C);		\draw (a)--(D);		\draw (a)--(E);
					
					\foreach \p in {A,B,C,D,E} {
						\fill[white] (\p) circle (2pt);
						\draw (\p) circle (2pt);
					}
					\foreach \p in {a}{
						\fill (\p) circle (2pt);
					}
					\node[below] at (0,-1.2) {$W_{21}$};
				}
			};
		\end{tikzpicture}
		\begin{tikzpicture}[remember picture, overlay]
			\node[right=2cm of left] (top1) {
				\GrDisk{
					\coordinate (A) at (-157.5:0.3);
					\coordinate (a) at (-157.5:0.7);
					
					\draw[-stealth] (1) arc[start angle=-180, end angle=-55, radius=.414];
					\draw[-stealth] (3) arc[start angle=90, end angle=215, radius=.414];
					
					\draw (5)--(A);		\draw (6)--(a);		\draw (7)--(a);		\draw (8)--(A);
					\draw (a)--(A);			
					
					\fill[white] (A) circle (2pt);
					\draw (A) circle (2pt);
					
					\fill (a) circle (2pt);
					
					\fill[white] (2) circle (2pt);
					\draw (2) circle (2pt);
					\fill[white] (4) circle (2pt);
					\draw (4) circle (2pt);
					\fill[white] (6) circle (2pt);
					\draw (6) circle (2pt);
					\fill[white] (7) circle (2pt);
					\draw (7) circle (2pt);
					\node[below] at (0,-1.2) {$W_{21}^{(1)}$};
				}
			};
			\node[right=0cm of top1] (top2) {
				\GrDisk{
					\coordinate (A) at (-157.5:0.3);
					\coordinate (a) at (-157.5:0.7);
					
					\draw[-stealth] (1) arc[start angle=-180, end angle=-55, radius=.414];
					\draw[-stealth] (4) arc[start angle=225, end angle=100, radius=.414];
					
					\draw (5)--(A);		\draw (6)--(a);		\draw (7)--(a);		\draw (8)--(A);
					\draw (a)--(A);			
					
					\fill[white] (A) circle (2pt);
					\draw (A) circle (2pt);
					
					\fill (a) circle (2pt);
					
					\fill[white] (2) circle (2pt);
					\draw (2) circle (2pt);
					\fill[white] (3) circle (2pt);
					\draw (3) circle (2pt);
					\fill[white] (6) circle (2pt);
					\draw (6) circle (2pt);
					\fill[white] (7) circle (2pt);
					\draw (7) circle (2pt);
					\node[below] at (0,-1.2) {$W_{21}^{(2)}$};
				}
			};
			\node[right=0cm of top2] (top3) {
				\GrDisk{
					\coordinate (A) at (-157.5:0.3);
					\coordinate (a) at (-157.5:0.7);
					
					\draw[-stealth] (2) arc[start angle=-45, end angle=-170, radius=.414];
					\draw[-stealth] (3) arc[start angle=90, end angle=215, radius=.414];
					
					\draw (5)--(A);		\draw (6)--(a);		\draw (7)--(a);		\draw (8)--(A);
					\draw (a)--(A);			
					
					\fill[white] (A) circle (2pt);
					\draw (A) circle (2pt);
					
					\fill (a) circle (2pt);
					
					\fill[white] (1) circle (2pt);
					\draw (1) circle (2pt);
					\fill[white] (4) circle (2pt);
					\draw (4) circle (2pt);
					\fill[white] (6) circle (2pt);
					\draw (6) circle (2pt);
					\fill[white] (7) circle (2pt);
					\draw (7) circle (2pt);
					\node[below] at (0,-1.2) {$W_{21}^{(3)}$};
				}
			};
			
			\draw[-Latex, line width=1pt] ([xshift=0.5cm]left.east) -- ([xshift=-0.5cm]top1.west |- left);
		\end{tikzpicture}
		\caption{The 3 webs with 4 black and 4 white boundary vertices obtained by $W_{21}$ figure 4.}
		\label{fig:exam}
	\end{figure}
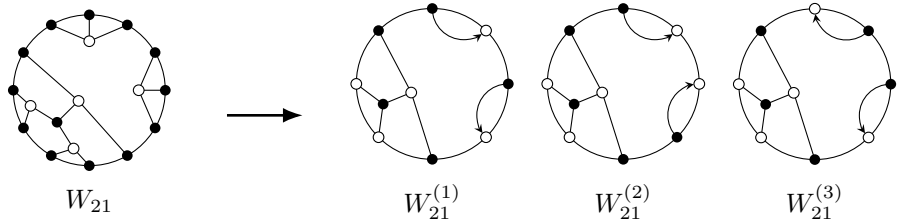
	
	Note that there are only $4$ claws, the $2$ claws with $2$ toes must be contracted, and then the direction of arrows contracted by other $2$ claws with $3$ toes need discussed. To do this,
	
	\begin{itemize}
		\item if $\{1,12\}$ and $\{3,4\}$ are contracted, then one can obtain a non-elliptic web that is a reflection of $W_{21}^{(1)}$ in \Cref{fig:exam}, by applying the map $\tau^4 \circ \rho$ to the labels of boundary vertices of $W_{21}^{(1)}$.
		
		\item if $\{1,2\}$ and $\{3,4\}$ are contracted, then one can obtain a non-elliptic web that is a rotation of $W_{21}^{(2)}$ in \Cref{fig:exam}, by applying the map $\tau$ to the labels of boundary vertices of $W_{21}^{(2)}$.
		
		\item if $\{1,12\}$ and $\{4,5\}$ are contracted, then one can obtain a non-elliptic web that is exactly $W_{21}^{(3)}$ in \Cref{fig:exam}. 
		
		\item if $\{1,2\}$ and $\{4,5\}$ are contracted, then one can obtain a non-elliptic web that is a rotation of $W_{21}^{(1)}$ in \Cref{fig:exam}, by applying the map $\tau$ to the labels of boundary vertices of $W_{21}^{(1)}$.
	\end{itemize}
	Therefore, we obtain the $3$ non-elliptic webs with 4 white and 4 black boundary vertices in \Cref{fig:exam}.
\end{example}

\hspace{\fill}

For each non-elliptic web with 12 black boundary vertices, as long as it has at least 4 claws, we can obtain the corresponding non-elliptic webs with 4 black and 4 white boundary vertices using methods similar to Example \ref{exm:3toe}, then there are a total of 182 such webs up to dihedral translations. Therefore, we have classified these webs according to the arrangement of their vertices, as shown in Figure \ref{fig:arrangements}, ensuring that vertex ``1'' is always white and vertex ``8'' is always black. For example, if the boundary vertices labeled ${2,3,6,8}$ of a web are white, the web is a reflection of type $5$, that is, applying the map $\tau^5 \circ \rho$ to the type $5$.

\begin{figure}[h]
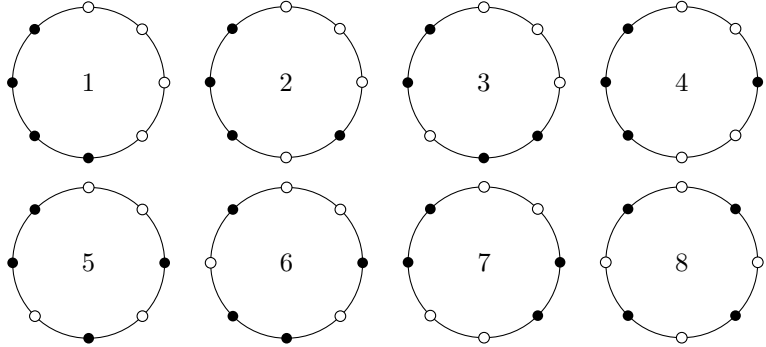

	\centering
	\begin{minipage}[t]{0.19\textwidth}
		\centering
		\Diska{\node at (0,0) {1};}
	\end{minipage}
	\begin{minipage}[t]{0.19\textwidth}
		\centering
		\Diskb{\node at (0,0) {2};}
	\end{minipage}
	\begin{minipage}[t]{0.19\textwidth}
		\centering
		\Diskc{\node at (0,0) {3};}
	\end{minipage}
	\begin{minipage}[t]{0.19\textwidth}
		\centering
		\Diskd{\node at (0,0) {4};}
	\end{minipage}
	\\\vspace{0.2cm}
	\begin{minipage}[t]{0.19\textwidth}
		\centering
		\Diske{\node at (0,0) {5};}
	\end{minipage}
	\begin{minipage}[t]{0.19\textwidth}
		\centering
		\Diskf{\node at (0,0) {6};}
	\end{minipage}
	\begin{minipage}[t]{0.19\textwidth}
		\centering
		\Diskg{\node at (0,0) {7};}
	\end{minipage}
	\begin{minipage}[t]{0.19\textwidth}
		\centering
		\Diskh{\node at (0,0) {8};}
	\end{minipage}
	\caption{The 8 arrangements of 4 white and 4 black boundary vertices.}
	\label{fig:arrangements}
\end{figure}

For each types of web, to facilitate verification, all possible cases are listed in Appendix \ref{sec:typelist}.

\subsection{Dual webs of quadratic and cubic cluster variables in $\mathbb{C}[\operatorname{Gr}(4,8)]$}\label{ssc:ComWebs}
In this section, we use the method of compatibility introduced in \cite{lam_dimers_2015} to compute the matchings and webs compatible with quadratic and cubic cluster variables in $\mathbb{C}[\operatorname{Gr}(4,8)]$. From \cite{fraser_dimers_2019} and the examples in \cite{elkin_twists_2023}, the web compatible with a Pl\"ucker polynomial is the dual web of its web diagrams.

For the 120 quadratic cluster variables, we have chosen 3 of them in Section \ref{ssc:WinInv}, whose tableaux and web diagrams are listed in Table \ref{tab:quainv}. Here, we give their Pl\"ucker polynomials, and compatible webs in Table \ref{tab:com2}.

\begin{table}[t]
	\centering
	\setcellgapes{1mm}
	\makegapedcells
	\begin{tabular}{cccccc}
		\hline
		\makecell{tableau} & \makecell{polynomial} & \makecell{matching} & \makecell{tableau} & \makecell{polynomial} & \makecell{matching}\\
		\hline
		\makecell{\scalemath{0.8}{\ytableaushort{11,24,36,57}}} &  
		\makecell{
			\scalebox{1}{
				$\thead{
					+P_{1235}P_{1467} \\ 
					-P_{1234}P_{1567}}$}}& 
		\makecell{
			\scalemath{0.7}{
				\GrDisk{
					\fill[white] (1) circle (2pt);
					\draw (1) circle (2pt);
					\draw (2)--(7);
					\draw (3) arc[start angle=90, end angle=225, radius=.414];
					\draw (5) arc[start angle=0, end angle=135, radius=.414];
		}}} &
		\makecell{\scalemath{0.8}{\ytableaushort{11,23,45,67}}} &  
		\makecell{
			\scalebox{1}{
				$\thead{
					+P_{1256}P_{1347} \\ 
					-P_{1234}P_{1567}}$}}& 
		\makecell{
			\scalemath{0.7}{
				\GrDisk{
					\fill[white] (1) circle (2pt);
					\draw (1) circle (2pt);
					\draw (2) arc[start angle=135, end angle=270, radius=.414];
					\draw (4) arc[start angle=45, end angle=180, radius=.414];
					\draw (6) arc[start angle=-45, end angle=90, radius=.414];
		}}} \\
		\hline
		\makecell{\scalemath{0.8}{\ytableaushort{14,25,37,68}}} &  
		\makecell{
			\scalebox{1}{
				$\thead{
					+P_{1234}P_{5678} \\ 
					-P_{1235}P_{4678} \\
					+P_{1236}P_{4578}}$}}& 
		\makecell{
			\scalemath{0.7}{
				\GrDisk{
					\draw (2)--(5);
					\draw (1) arc[start angle=0, end angle=-135, radius=.414];
					\draw (3) arc[start angle=90, end angle=225, radius=.414];
					\draw (6) arc[start angle=-45, end angle=90, radius=.414];
		}}} & & & \\
		\hline
	\end{tabular}
	\caption{3 different webs compatible with the quadratic cluster variables in $\Gr{4}{8}$ with 8 vertices.}
	\label{tab:com2}
\end{table}

We then presents all of webs compatible with the 174 cubic cluster variables, totaling $14$ of them are chosen which have mentioned in Section \ref{ssc:WinInv}, their Pl\"ucker polynomials and compatible webs are listed in Table \ref{tab:com3}.

\begin{table}[h]
	\centering
	\setcellgapes{1mm}
	\makegapedcells
	\begin{tabular}{cccccccc}
		\hline
		\makecell{type} & \makecell{tableau} & \makecell{polynomial} & \makecell{web} & \makecell{type} & \makecell{tableau} & \makecell{polynomial} & \makecell{web}\\
		\hline
		\makecell{1} & \makecell{\scalemath{0.7}{\ytableaushort{112,233,445,678}}} &  
		\makecell{
			\scalebox{0.8}{
				$\thead{
					+P_{1238}P_{1234}P_{4567} \\ 
					-P_{1238}P_{1456}P_{2347} \\ 
					-P_{1248}P_{1234}P_{3567} \\ 
					+P_{1248}P_{1356}P_{2347}}$}}& 
		\makecell{
			\scalemath{0.5}{
				\Diska{
					\coordinate (A) at (180:0.5);
					\coordinate (a) at (45:0.5);
					\draw (1)--(a);		\draw (2)--(a);		\draw (3)--(a);
					\draw (6)--(A);		\draw (7)--(A);		\draw (8)--(A);		
					\draw[-stealth] (5) arc[start angle=180, end angle=55, radius=.414];
					\fill[white] (A) circle (2pt);
					\draw (A) circle (2pt);
					\fill (a) circle (2pt);}}} &
		\multirow{9}{*}{\makecell{5}} & 
		\makecell{\scalemath{0.7}{\ytableaushort{112,234,456,678}}} &  
		\makecell{
			\scalebox{0.8}{
				$\thead{
					-P_{4567}P_{1246}P_{1238} \\  
					-P_{4567}P_{1234}P_{1268} \\ 
					+P_{1456}P_{1248}P_{2367} \\ 
					-P_{2456}P_{1467}P_{1238}}$}}& 
		\makecell{
			\scalemath{0.5}{
				\Diske{
					\coordinate (A) at (157.5:0.7);
					\coordinate (B) at (22.5:0.4);
					\coordinate (C) at (-97.5:0.4);
					\coordinate (a) at (67.5:0.7);
					\coordinate (b) at (-37.5:0.4);
					\coordinate (c) at (-157.5:0.4);
					\draw (1)--(a);		\draw (2)--(a);		\draw (4)--(b);		\draw (6)--(c);
					\draw (3)--(B);		\draw (5)--(C);		\draw (7)--(A);		\draw (8)--(A);
					\draw (a)--(B);		\draw (b)--(B);		\draw (b)--(C);
					\draw (c)--(A);		\draw (c)--(C);				
					\foreach \p in {A,B,C} {
						\fill[white] (\p) circle (2pt);
						\draw (\p) circle (2pt);}
					\foreach \p in {a,b,c}{
						\fill (\p) circle (2pt);}}}} \\
		\cmidrule(lr){1-4}
		\cmidrule(lr){6-8}
		\multirow{5}{*}{\makecell{2}} & 
		\makecell{\scalemath{0.7}{\ytableaushort{112,233,455,678}}} &  
		\makecell{
			\scalebox{0.8}{
				$\thead{
					+P_{5678}P_{1235}P_{1234} \\ 
					-P_{3567}P_{1258}P_{1234} \\ 
					-P_{1256}P_{3578}P_{1234} \\ 
					+P_{1356}P_{1257}P_{2348} \\
					-P_{2356}P_{1578}P_{1234}}$}}& 
		\makecell{
			\scalemath{0.5}{
				\Diskb{
					\coordinate (A) at (180:0.5);
					\coordinate (a) at (45:0.5);
					\draw (1)--(a);		\draw (2)--(a);		\draw (3)--(a);
					\draw (6)--(A);		\draw (7)--(A);		\draw (8)--(A);							
					\draw[-stealth] (4) arc[start angle=45, end angle=170, radius=.414];
					\fill[white] (A) circle (2pt);
					\draw (A) circle (2pt);
					\fill (a) circle (2pt);}}} &
		& \makecell{\scalemath{0.7}{\ytableaushort{112,244,366,578}}} &  
		\makecell{
			\scalebox{0.8}{
				$\thead{
					-P_{4567}P_{1246}P_{1238} \\ 
					-P_{4567}P_{1234}P_{1268} \\ 
					-P_{2456}P_{1467}P_{1238} \\ 
					+P_{1458}P_{2467}P_{1236}}$}}& 
		\makecell{
			\scalemath{0.5}{
				\Diske{
					\coordinate (A) at (157.5:0.7);
					\coordinate (a) at (67.5:0.7);
					\draw (1)--(a);		\draw (2)--(a);	
					\draw (7)--(A);		\draw (8)--(A);		
					\draw (a)--(A);	
					\draw[-stealth] (3) arc[start angle=90, end angle=215, radius=.414];
					\draw[-stealth] (5) arc[start angle=0, end angle=125, radius=.414];
					\fill[white] (A) circle (2pt);
					\draw (A) circle (2pt);
					\fill (a) circle (2pt);}}} \\
		\cmidrule(lr){2-4}
		\cmidrule(lr){6-8}
		& \makecell{\scalemath{0.7}{\ytableaushort{113,225,347,568}}} &  
		\makecell{
			\scalebox{0.8}{
				$\thead{
					-P_{1278}P_{1235}P_{3456} \\ 
					-P_{1378}P_{1256}P_{2345} \\ 
					+P_{2378}P_{1256}P_{1345}}$}}& 
		\makecell{
			\scalemath{0.5}{
				\Diskb{
					\coordinate (A) at (-157.5:0.7);
					\coordinate (B) at (-35:0.4);
					\coordinate (a) at (22.5:0.7);
					\coordinate (b) at (-100:0.4);
					\draw[-stealth] (8) arc[start angle=-135, end angle=-10, radius=.414];
					\draw (2)--(a);		\draw (3)--(a);		\draw (5)--(b);
					\draw (4)--(B);		\draw (6)--(A);		\draw (7)--(A);		
					\draw (a)--(B);	
					\draw (b)--(A);		\draw (b)--(B);
					\foreach \p in {A,B} {
						\fill[white] (\p) circle (2pt);
						\draw (\p) circle (2pt);}
					\foreach \p in {a,b}{
						\fill (\p) circle (2pt);}}}} &
		& \makecell{\scalemath{0.7}{\ytableaushort{113,225,447,668}}} &  
		\makecell{
			\scalebox{0.8}{
				$\thead{
					+P_{4567}P_{1234}P_{1268} \\ 
					-P_{4568}P_{1234}P_{1267} \\ 
					-P_{3456}P_{1246}P_{1278} \\ 
					+P_{1246}P_{3478}P_{1256} \\
					-P_{1246}P_{1234}P_{5678}}$}}& 
		\makecell{
			\scalemath{0.5}{
				\Diske{
					\draw[-stealth] (8) arc[start angle=-135, end angle=-10, radius=.414];
					\draw[-stealth] (3) arc[start angle=-90, end angle=-215, radius=.414];
					\draw[-stealth] (5) arc[start angle=180, end angle=55, radius=.414];
					\draw[-stealth] (7) arc[start angle=90, end angle=-35, radius=.414];}}} \\
		\hline
		\multirow{5}{*}{\makecell{3}} & 
		\makecell{\scalemath{0.7}{\ytableaushort{112,233,456,678}}} &  
		\makecell{
			\scalebox{0.8}{
				$\thead{
					+P_{5678}P_{1236}P_{1234} \\ 
					-P_{3567}P_{1268}P_{1234} \\ 
					-P_{1256}P_{3678}P_{1234} \\ 
					+P_{1356}P_{1267}P_{2348} \\
					-P_{2356}P_{1678}P_{1234}}$}}& 
		\makecell{
			\scalemath{0.5}{
				\Diskc{
					\coordinate (A) at (-67.5:0.7);
					\coordinate (B) at (157.5:0.7);
					\coordinate (a) at (45:0.5);
					\coordinate (b) at (-135:0.4);
					\draw (1)--(a);		\draw (2)--(a);		\draw (3)--(a);		\draw (6)--(b);
					\draw (4)--(A);		\draw (5)--(A);		\draw (7)--(B);		\draw (8)--(B);
					\draw (b)--(A);		\draw (b)--(B);
					\foreach \p in {A,B} {
						\fill[white] (\p) circle (2pt);
						\draw (\p) circle (2pt);}
					\foreach \p in {a,b}{
						\fill (\p) circle (2pt);}}}} &
		\makecell{6} & \makecell{\scalemath{0.7}{\ytableaushort{112,244,367,578}}} &  
		\makecell{
			\scalebox{0.8}{
				$\thead{
					-P_{2345}P_{1247}P_{1678} \\ 
					-P_{2345}P_{1278}P_{1467} \\ 
					-P_{1234}P_{2457}P_{1678} \\ 
					+P_{2367}P_{1245}P_{1478}}$}}& 
		\makecell{
			\scalemath{0.5}{
				\Diskf{
					\coordinate (A) at (-112.5:0.7);
					\coordinate (B) at (125:0.4);
					\coordinate (a) at (67.5:0.7);
					\coordinate (b) at (190:0.4);
					\draw (1)--(a);		\draw (2)--(a);		\draw (7)--(b);
					\draw (5)--(A);		\draw (6)--(A);		\draw (8)--(B);		
					\draw (a)--(B);		\draw (b)--(A);		\draw (b)--(B);	
					\draw[-stealth] (3) arc[start angle=90, end angle=215, radius=.414];
					\foreach \p in {A,B} {
						\fill[white] (\p) circle (2pt);
						\draw (\p) circle (2pt);}
					\foreach \p in {a,b}{
						\fill (\p) circle (2pt);}}}} \\
		\cmidrule(lr){2-8}
		& \makecell{\scalemath{0.7}{\ytableaushort{112,234,366,578}}} &  
		\makecell{
			\scalebox{0.8}{
				$\thead{
					+P_{1234}P_{1236}P_{5678} \\ 
					-P_{1234}P_{1567}P_{2368} \\ 
					-P_{1235}P_{1236}P_{4678} \\ 
					+P_{1235}P_{1467}P_{2368}}$}}& 
		\makecell{
			\scalemath{0.5}{
				\Diskc{
					\coordinate (A) at (157.5:0.7);
					\coordinate (a) at (67.5:0.7);
					\draw (1)--(a);		\draw (2)--(a);		
					\draw (7)--(A);		\draw (8)--(A);		
					\draw (a)--(A);	
					\draw[-stealth] (4) arc[start angle=225, end angle=100, radius=.414];
					\draw[-stealth] (5) arc[start angle=0, end angle=125, radius=.414];
					\fill[white] (A) circle (2pt);
					\draw (A) circle (2pt);
					\fill (a) circle (2pt);}}} &
		\makecell{7} & \makecell{\scalemath{0.7}{\ytableaushort{112,245,366,578}}} &  
		\makecell{
			\scalebox{0.8}{
				$\thead{
					-P_{4567}P_{1256}P_{1238} \\ 
					-P_{4567}P_{1235}P_{1268} \\ 
					-P_{2456}P_{1567}P_{1238} \\ 
					+P_{1458}P_{2567}P_{1236}}$}}& 
		\makecell{
			\scalemath{0.5}{
				\Diskg{
					\coordinate (A) at (-22.5:0.7);
					\coordinate (B) at (157.5:0.7);
					\coordinate (a) at (67.5:0.7);
					\coordinate (b) at (-112.5:0.7);
					\draw (1)--(a);		\draw (2)--(a);		\draw (5)--(b);		\draw (6)--(b);
					\draw (3)--(A);		\draw (4)--(A);		\draw (7)--(B);		\draw (8)--(B);	
					\draw (a)--(B);		\draw (b)--(A);	
					\foreach \p in {A,B} {
						\fill[white] (\p) circle (2pt);
						\draw (\p) circle (2pt);}
					\foreach \p in {a,b}{
						\fill (\p) circle (2pt);}}}} \\
		\hline
		\multirow{5}{*}{\makecell{4}} & 
		\makecell{\scalemath{0.7}{\ytableaushort{112,234,455,678}}} &  
		\makecell{
			\scalebox{0.8}{
				$\thead{
					+P_{1234}P_{1245}P_{5678} \\ 
					-P_{1234}P_{1256}P_{4578} \\ 
					-P_{1234}P_{1258}P_{4567} \\
					-P_{1234}P_{1578}P_{2456} \\ 
					+P_{1257}P_{1456}P_{2348}}$}}& 
		\makecell{
			\scalemath{0.5}{
				\Diskd{
					\coordinate (A) at (180:0.5);
					\coordinate (B) at (0:0.4);
					\coordinate (a) at (67.5:0.7);
					\coordinate (b) at (-67.5:0.7);
					\draw (1)--(a);		\draw (2)--(a);		\draw (4)--(b);		\draw (5)--(b);		
					\draw (3)--(B);		\draw (6)--(A);		\draw (7)--(A);		\draw (8)--(A);	
					\draw (a)--(B);		
					\draw (b)--(B);	
					\foreach \p in {A,B} {
						\fill[white] (\p) circle (2pt);
						\draw (\p) circle (2pt);}
					\foreach \p in {a,b}{
						\fill (\p) circle (2pt);}}}} &
		\multirow{5}{*}{\makecell{8}} & \makecell{\scalemath{0.7}{\ytableaushort{113,245,367,578}}} &  
		\makecell{
			\scalebox{0.8}{
				$\thead{
					-P_{1234}P_{1357}P_{5678} \\ 
					-P_{1234}P_{1578}P_{3567} \\ 
					+P_{1235}P_{1567}P_{3478} \\ 
					-P_{1237}P_{1345}P_{5678}}$}}& 
		\makecell{
			\scalemath{0.5}{
				\Diskh{
					\coordinate (A) at (45:0.6);
					\coordinate (B) at (-45:0.6);
					\coordinate (C) at (-135:0.6);
					\coordinate (D) at (135:0.6);
					\coordinate (a) at (0:0.6);
					\coordinate (b) at (-90:0.6);
					\coordinate (c) at (180:0.6);
					\coordinate (d) at (90:0.6);
					\draw (1)--(d);		\draw (3)--(a);		\draw (5)--(b);		\draw (7)--(c);
					\draw (2)--(A);		\draw (4)--(B);		\draw (6)--(C);		\draw (8)--(D);		
					\draw (a)--(A);		\draw (a)--(B);		\draw (b)--(B);		\draw (b)--(C);
					\draw (c)--(C);		\draw (c)--(D);		\draw (d)--(A);		\draw (d)--(D);
					\foreach \p in {A,B,C,D} {
						\fill[white] (\p) circle (2pt);
						\draw (\p) circle (2pt);}
					\foreach \p in {a,b,c,d}{
						\fill (\p) circle (2pt);}}}} \\
		\cmidrule(lr){2-4}
		\cmidrule(lr){6-8}
		& \makecell{\scalemath{0.7}{\ytableaushort{113,225,447,568}}} &  
		\makecell{
			\scalebox{0.8}{
				$\thead{
					-P_{1234}P_{1245}P_{5678} \\
					-P_{1234}P_{1256}P_{4578} \\
					+P_{1235}P_{1245}P_{4678} \\ 
					-P_{1236}P_{1245}P_{4578} \\ 
					-P_{1245}P_{1245}P_{3678} \\ 
					+P_{1245}P_{1246}P_{3578}}$}}& 
		\makecell{
			\scalemath{0.5}{
				\Diskd{
					\coordinate (A) at (-157.5:0.7);
					\coordinate (a) at (-67.5:0.7);
					\draw (4)--(a);		\draw (5)--(a);		
					\draw (6)--(A);		\draw (7)--(A);	
					\draw (a)--(A);		
					\draw[-stealth] (3) arc[start angle=-90, end angle=-215, radius=.414];
					\draw[-stealth] (8) arc[start angle=-135, end angle=-10, radius=.414];
					\fill[white] (A) circle (2pt);
					\draw (A) circle (2pt);
					\fill (a) circle (2pt);}}} &
		& \makecell{\scalemath{0.7}{\ytableaushort{112,334,556,778}}} &  
		\makecell{
			\scalebox{0.8}{
				$\thead{
					-P_{2345}P_{1357}P_{1678} \\ 
					-P_{2345}P_{1378}P_{1567} \\ 
					+P_{2357}P_{1367}P_{1458} \\ 
					-P_{1235}P_{3457}P_{1678}}$}}& 
		\makecell{
			\scalemath{0.5}{
				\Diskh{
					\draw[-stealth] (8) arc[start angle=45, end angle=-80, radius=.414];
					\draw[-stealth] (2) arc[start angle=-45, end angle=-170, radius=.414];
					\draw[-stealth] (4) arc[start angle=-135, end angle=-260, radius=.414];
					\draw[-stealth] (6) arc[start angle=135, end angle=10, radius=.414];}}} \\
		\hline
	\end{tabular}
	\caption{14 different webs compatible with the cubic cluster variables in $\Gr{4}{8}$ with 8 vertices.}
	\label{tab:com3}
\end{table}

\begin{example}
	Take the first cubic cluster variables in Table \ref{tab:com3} as an example. For the cubic cluster variable
	\begin{equation*}
		P=P_{1238}P_{1234}P_{4567} - P_{1238}P_{1456}P_{2347} \\- P_{1248}P_{1234}P_{3567} + P_{1248}P_{1356}P_{2347},
	\end{equation*}
	denote by $P_i$ the $i$th term of polynomial $P$, and by $\operatorname{sgn}_i$ the sign of $P_i$ in $P$.  By Theorem \ref{thm:fll19}, note that the web corresponding to each $P_i$ is not necessarily non-elliptic. We therefore express each as a linear combination of non-elliptic webs with coefficients given by the compatibility degrees $a(I,J,K;W)$. For convenience, we represent the web $W_i$ compatible with each $P_i$ through a multiple set $\mathcal{W}_i$, where each non-elliptic web $W$ appears with multiplicity $a(I,J,K;W)$.
	\begin{itemize}
		\item $P_1$ is correspond to multiset $\mathcal{W}_1=\{(\mathrm{i})\}$
		\item $P_2$ is correspond to multiset $\mathcal{W}_2=\{(\mathrm{i}), (\mathrm{iii})\}$
		\item $P_3$ is correspond to multiset $\mathcal{W}_3=\{(\mathrm{i}), (\mathrm{ii})\}$
		\item $P_4$ is correspond to multiset $\mathcal{W}_4=\{(\mathrm{i}), (\mathrm{ii}), (\mathrm{iii}), (\mathrm{iv})\}$.
	\end{itemize}
	
	\begin{figure}[H]
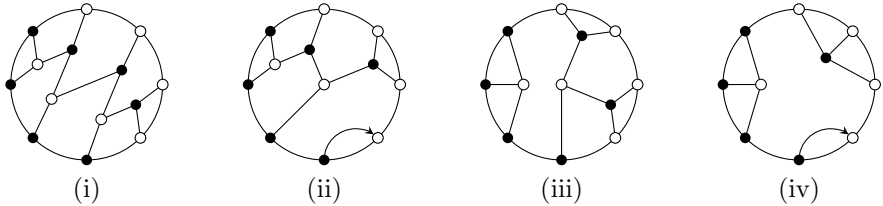

		\centering
		\begin{minipage}[t]{0.23\textwidth}
			\centering
			\Diska{
				\coordinate (A) at (157.5:0.7);
				\coordinate (B) at (-67.5:0.5);
				\coordinate (C) at (-157.5:0.5);
				
				\coordinate (a) at (-22.5:0.7);
				\coordinate (b) at (22.5:0.5);
				\coordinate (c) at (112.5:0.5);
				
				\draw (1)--(c);		\draw (2)--(b);		\draw (3)--(a);		\draw (4)--(a);
				\draw (5)--(B);		\draw (6)--(C);		\draw (7)--(A);		\draw (8)--(A);
				
				\draw (a)--(B);	
				\draw (b)--(B);		\draw (b)--(C);	
				\draw (c)--(A);		\draw (c)--(C);	
				
				\foreach \p in {A,B,C} {
					\fill[white] (\p) circle (2pt);
					\draw (\p) circle (2pt);
				}
				\foreach \p in {a,b,c}{
					\fill (\p) circle (2pt);
				}
			}
			\makebox[\textwidth][c]{(i)}
		\end{minipage}
		\begin{minipage}[t]{0.23\textwidth}
			\centering
			\Diska{
				\coordinate (A) at (157.5:0.7);
				\coordinate (B) at (0:0);
				
				\coordinate (a) at (22.5:0.7);
				\coordinate (b) at (112.5:0.5);
				
				\draw[-stealth] (5) arc[start angle=180, end angle=55, radius=.414];
				
				\draw (1)--(b);		\draw (2)--(a);		\draw (3)--(a);		
				\draw (6)--(B);		\draw (7)--(A);		\draw (8)--(A);
				
				\draw (a)--(B);	
				\draw (b)--(A);		\draw (b)--(B);
				
				\foreach \p in {A,B} {
					\fill[white] (\p) circle (2pt);
					\draw (\p) circle (2pt);
				}
				\foreach \p in {a,b}{
					\fill (\p) circle (2pt);
				}
			}
			\makebox[\textwidth][c]{(ii)}
		\end{minipage}
		\begin{minipage}[t]{0.23\textwidth}
			\centering
			\Diska{
				\coordinate (A) at (180:0.5);
				\coordinate (B) at (0:0);
				
				\coordinate (a) at (67.5:0.7);
				\coordinate (b) at (-22.5:0.7);
				
				\draw (1)--(a);		\draw (2)--(a);		\draw (3)--(b);		\draw (4)--(b);		
				\draw (5)--(B);		\draw (6)--(A);		\draw (7)--(A);		\draw (8)--(A);
				
				\draw (a)--(B);	
				\draw (b)--(B);
				
				\foreach \p in {A,B} {
					\fill[white] (\p) circle (2pt);
					\draw (\p) circle (2pt);
				}
				\foreach \p in {a,b}{
					\fill (\p) circle (2pt);
				}
			}
			\makebox[\textwidth][c]{(iii)}
		\end{minipage}
		\begin{minipage}[t]{0.23\textwidth}
			\centering
			\Diska{
				\coordinate (A) at (180:0.5);
				\coordinate (a) at (45:0.5);
				
				\draw (1)--(a);		\draw (2)--(a);		\draw (3)--(a);
				\draw (6)--(A);		\draw (7)--(A);		\draw (8)--(A);		
				
				\draw[-stealth] (5) arc[start angle=180, end angle=55, radius=.414];
				
				\fill[white] (A) circle (2pt);
				\draw (A) circle (2pt);
				\fill (a) circle (2pt);
			}
			\makebox[\textwidth][c]{(iv)}
		\end{minipage}
		\caption{The webs corresponding to $P_i$ in $P$, $P_1$ is correspond to (i), $P_2$ is correspond to webs (i) and (iii), $P_3$ is correspond to webs (i) and (ii), and $P_4$ is correspond to webs (i), (ii), (iii) and (iv).}
		\label{fig:comP}
	\end{figure}
	
	Therefore, the web compatible with $P$ is the web in the multiset $\mathcal{W}=\mathcal{W}_1 \uplus \mathcal{W}_4 - \mathcal{W}_2 - \mathcal{W}_3 = \{(\mathrm{iv})\}$. Therefore, the non-elliptic (iv) in Figure \ref{fig:comP} is the web compatible with cluster variable $P$. Concurrently, the web obtained from compatibility conditions is also the dual web of its corresponding web diagram in Section 3.3.
\end{example}

\backmatter

\bmhead{Acknowledgements}

This research was partially supported by the National Natural Science Foundation of China (Nos. 12271224, 12161062) and the Fundamental Research Funds for the Central University (No. lzujbky-2023-ey06).

\begin{appendices}

\section{The non-elliptic webs classified by the boundary conditions} \label{sec:typelist}
In this section, all of the non-elliptic webs obtain from dimers that are mentioned in Section \ref{ssc:ComSink} are listed. We classify them by their boundary conditions (type 1-8) in Figures \ref{fig:type1}-\ref{fig:type8}.

\begin{figure}[H]
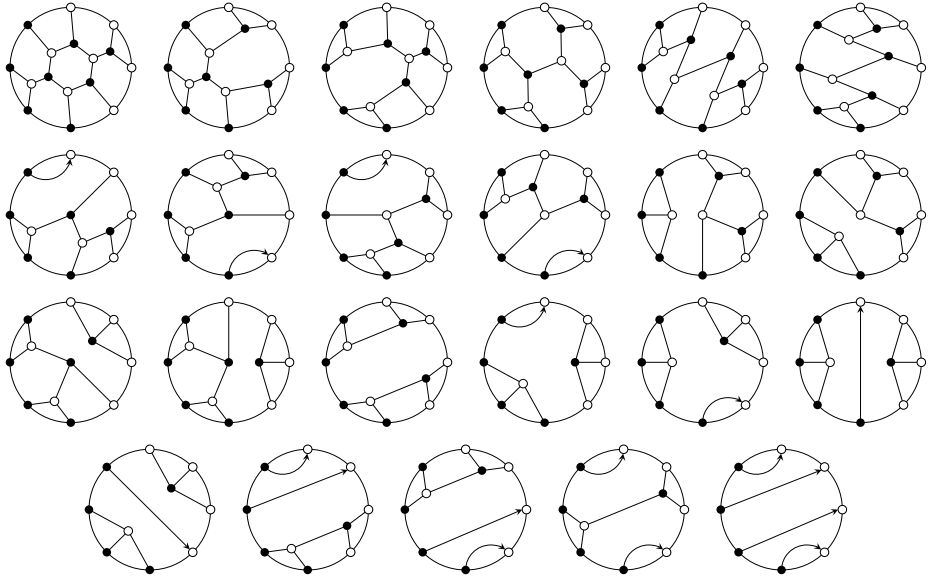

	\centering
	\begin{minipage}[t]{0.15\textwidth}
		\centering
		\scalemath{0.8}{
			\Diska{
				\coordinate (A) at (-157.5:0.7);
				\coordinate (B) at (22.5:0.4);
				\coordinate (C) at (-97.5:0.4);
				\coordinate (D) at (142.5:0.4);
				
				\coordinate (a) at (22.5:0.7);
				\coordinate (b) at (82.5:0.4);
				\coordinate (c) at (-37.5:0.4);
				\coordinate (d) at (-157.5:0.4);
				
				\draw (1)--(b);		\draw (2)--(a);		\draw (3)--(a);		\draw (4)--(c);
				\draw (5)--(C);		\draw (6)--(A);		\draw (7)--(A);		\draw (8)--(D);
				
				\draw (a)--(B);	
				\draw (b)--(B);		\draw (b)--(D);
				\draw (c)--(B);		\draw (c)--(C);	
				\draw (d)--(A);		\draw (d)--(C);		\draw (d)--(D);	
				
				\foreach \p in {A,B,C,D} {
					\fill[white] (\p) circle (2pt);
					\draw (\p) circle (2pt);
				}
				\foreach \p in {a,b,c,d}{
					\fill (\p) circle (2pt);
				}
		}}
	\end{minipage}
	\begin{minipage}[t]{0.15\textwidth}
		\centering
		\scalemath{0.8}{
			\Diska{
				\coordinate (A) at (-157.5:0.7);
				\coordinate (B) at (-97.5:0.4);
				\coordinate (C) at (142.5:0.4);
				
				\coordinate (a) at (67.5:0.7);
				\coordinate (b) at (-22.5:0.7);
				\coordinate (c) at (-157.5:0.4);
				
				\draw (1)--(a);		\draw (2)--(a);		\draw (3)--(b);		\draw (4)--(b);
				\draw (5)--(B);		\draw (6)--(A);		\draw (7)--(A);		\draw (8)--(C);
				
				\draw (a)--(C);	
				\draw (b)--(B);		
				\draw (c)--(A);		\draw (c)--(B);		\draw (c)--(C);	
				
				\foreach \p in {A,B,C} {
					\fill[white] (\p) circle (2pt);
					\draw (\p) circle (2pt);
				}
				\foreach \p in {a,b,c}{
					\fill (\p) circle (2pt);
				}
		}}
	\end{minipage}
	\begin{minipage}[t]{0.15\textwidth}
		\centering
		\scalemath{0.8}{
			\Diska{
				\coordinate (A) at (-112.5:0.7);
				\coordinate (B) at (157.5:0.7);
				\coordinate (C) at (22.5:0.4);
				
				\coordinate (a) at (22.5:0.7);
				\coordinate (b) at (87.5:0.4);
				\coordinate (c) at (-37.5:0.4);
				
				\draw (1)--(b);		\draw (2)--(a);		\draw (3)--(a);		\draw (4)--(c);
				\draw (5)--(A);		\draw (6)--(A);		\draw (7)--(B);		\draw (8)--(B);
				
				\draw (a)--(C);	
				\draw (b)--(B);		\draw (b)--(C);	
				\draw (c)--(A);		\draw (c)--(C);	
				
				\foreach \p in {A,B,C} {
					\fill[white] (\p) circle (2pt);
					\draw (\p) circle (2pt);
				}
				\foreach \p in {a,b,c}{
					\fill (\p) circle (2pt);
				}
		}}
	\end{minipage}
	\begin{minipage}[t]{0.15\textwidth}
		\centering
		\scalemath{0.8}{
			\Diska{
				\coordinate (A) at (-112.5:0.7);
				\coordinate (B) at (157.5:0.7);
				\coordinate (C) at (22.5:0.3);
				
				\coordinate (a) at (67.5:0.7);
				\coordinate (b) at (-22.5:0.7);
				\coordinate (c) at (-157.5:0.3);
				
				\draw (1)--(a);		\draw (2)--(a);		\draw (3)--(b);		\draw (4)--(b);
				\draw (5)--(A);		\draw (6)--(A);		\draw (7)--(B);		\draw (8)--(B);
				
				\draw (a)--(C);	
				\draw (b)--(C);	
				\draw (c)--(A);		\draw (c)--(B);		\draw (c)--(C);	
				
				\foreach \p in {A,B,C} {
					\fill[white] (\p) circle (2pt);
					\draw (\p) circle (2pt);
				}
				\foreach \p in {a,b,c}{
					\fill (\p) circle (2pt);
				}
		}}
	\end{minipage}
	\begin{minipage}[t]{0.15\textwidth}
		\centering
		\scalemath{0.8}{
			\Diska{
				\coordinate (A) at (157.5:0.7);
				\coordinate (B) at (-67.5:0.5);
				\coordinate (C) at (-157.5:0.5);
				
				\coordinate (a) at (-22.5:0.7);
				\coordinate (b) at (22.5:0.5);
				\coordinate (c) at (112.5:0.5);
				
				\draw (1)--(c);		\draw (2)--(b);		\draw (3)--(a);		\draw (4)--(a);
				\draw (5)--(B);		\draw (6)--(C);		\draw (7)--(A);		\draw (8)--(A);
				
				\draw (a)--(B);	
				\draw (b)--(B);		\draw (b)--(C);	
				\draw (c)--(A);		\draw (c)--(C);	
				
				\foreach \p in {A,B,C} {
					\fill[white] (\p) circle (2pt);
					\draw (\p) circle (2pt);
				}
				\foreach \p in {a,b,c}{
					\fill (\p) circle (2pt);
				}
		}}
	\end{minipage}
	\begin{minipage}[t]{0.15\textwidth}
		\centering
		\scalemath{0.8}{
			\Diska{
				\coordinate (A) at (-112.5:0.7);
				\coordinate (B) at (-157.5:0.5);
				\coordinate (C) at (112.5:0.5);
				
				\coordinate (a) at (67.5:0.7);
				\coordinate (b) at (22.5:0.5);
				\coordinate (c) at (-67.5:0.5);
				
				\draw (1)--(a);		\draw (2)--(a);		\draw (3)--(b);		\draw (4)--(c);
				\draw (5)--(A);		\draw (6)--(A);		\draw (7)--(B);		\draw (8)--(C);
				
				\draw (a)--(C);	
				\draw (b)--(B);		\draw (b)--(C);	
				\draw (c)--(A);		\draw (c)--(B);	
				
				\foreach \p in {A,B,C} {
					\fill[white] (\p) circle (2pt);
					\draw (\p) circle (2pt);
				}
				\foreach \p in {a,b,c}{
					\fill (\p) circle (2pt);
				}
		}}
	\end{minipage}
	\\\vspace{0.2cm}
	\begin{minipage}[t]{0.15\textwidth}
		\centering
		\scalemath{0.8}{
			\Diska{
				\coordinate (A) at (-157.5:0.7);
				\coordinate (B) at (-67.5:0.5);
				
				\coordinate (a) at (-22.5:0.7);
				\coordinate (b) at (0:0);
				
				\draw[-stealth] (8) arc[start angle=-135, end angle=-10, radius=.414];
				
				\draw (2)--(b);		\draw (3)--(a);		\draw (4)--(a);
				\draw (5)--(B);		\draw (6)--(A);		\draw (7)--(A);	
				
				\draw (a)--(B);	
				\draw (b)--(A);		\draw (b)--(B);
				
				\foreach \p in {A,B} {
					\fill[white] (\p) circle (2pt);
					\draw (\p) circle (2pt);
				}
				\foreach \p in {a,b}{
					\fill (\p) circle (2pt);
				}
		}}
	\end{minipage}
	\begin{minipage}[t]{0.15\textwidth}
		\centering
		\scalemath{0.8}{
			\Diska{
				\coordinate (A) at (-157.5:0.7);
				\coordinate (B) at (112.5:0.5);
				
				\coordinate (a) at (67.5:0.7);
				\coordinate (b) at (0:0);
				
				\draw[-stealth] (5) arc[start angle=180, end angle=55, radius=.414];
				
				\draw (1)--(a);		\draw (2)--(a);		\draw (3)--(b);		
				\draw (6)--(A);		\draw (7)--(A);		\draw (8)--(B);
				
				\draw (a)--(B);	
				\draw (b)--(A);		\draw (b)--(B);
				
				\foreach \p in {A,B} {
					\fill[white] (\p) circle (2pt);
					\draw (\p) circle (2pt);
				}
				\foreach \p in {a,b}{
					\fill (\p) circle (2pt);
				}
		}}
	\end{minipage}
	\begin{minipage}[t]{0.15\textwidth}
		\centering
		\scalemath{0.8}{
			\Diska{
				\coordinate (A) at (-112.5:0.7);
				\coordinate (B) at (0:0);
				
				\coordinate (a) at (22.5:0.7);
				\coordinate (b) at (-67.5:0.5);
				
				\draw[-stealth] (8) arc[start angle=-135, end angle=-10, radius=.414];
				
				\draw (2)--(a);		\draw (3)--(a);		\draw (4)--(b);		
				\draw (5)--(A);		\draw (6)--(A);		\draw (7)--(B);
				
				\draw (a)--(B);	
				\draw (b)--(A);		\draw (b)--(B);
				
				\foreach \p in {A,B} {
					\fill[white] (\p) circle (2pt);
					\draw (\p) circle (2pt);
				}
				\foreach \p in {a,b}{
					\fill (\p) circle (2pt);
				}
		}}
	\end{minipage}
	\begin{minipage}[t]{0.15\textwidth}
		\centering
		\scalemath{0.8}{
			\Diska{
				\coordinate (A) at (157.5:0.7);
				\coordinate (B) at (0:0);
				
				\coordinate (a) at (22.5:0.7);
				\coordinate (b) at (112.5:0.5);
				
				\draw[-stealth] (5) arc[start angle=180, end angle=55, radius=.414];
				
				\draw (1)--(b);		\draw (2)--(a);		\draw (3)--(a);		
				\draw (6)--(B);		\draw (7)--(A);		\draw (8)--(A);
				
				\draw (a)--(B);	
				\draw (b)--(A);		\draw (b)--(B);
				
				\foreach \p in {A,B} {
					\fill[white] (\p) circle (2pt);
					\draw (\p) circle (2pt);
				}
				\foreach \p in {a,b}{
					\fill (\p) circle (2pt);
				}
		}}
	\end{minipage}
	\begin{minipage}[t]{0.15\textwidth}
		\centering
		\scalemath{0.8}{
			\Diska{
				\coordinate (A) at (180:0.5);
				\coordinate (B) at (0:0);
				
				\coordinate (a) at (67.5:0.7);
				\coordinate (b) at (-22.5:0.7);
				
				\draw (1)--(a);		\draw (2)--(a);		\draw (3)--(b);		\draw (4)--(b);		
				\draw (5)--(B);		\draw (6)--(A);		\draw (7)--(A);		\draw (8)--(A);
				
				\draw (a)--(B);	
				\draw (b)--(B);
				
				\foreach \p in {A,B} {
					\fill[white] (\p) circle (2pt);
					\draw (\p) circle (2pt);
				}
				\foreach \p in {a,b}{
					\fill (\p) circle (2pt);
				}
		}}
	\end{minipage}
	\begin{minipage}[t]{0.15\textwidth}
		\centering
		\scalemath{0.8}{
			\Diska{
				\coordinate (A) at (-135:0.5);
				\coordinate (B) at (0:0);
				
				\coordinate (a) at (67.5:0.7);
				\coordinate (b) at (-22.5:0.7);
				
				\draw (1)--(a);		\draw (2)--(a);		\draw (3)--(b);		\draw (4)--(b);
				\draw (5)--(A);		\draw (6)--(A);		\draw (7)--(A);		\draw (8)--(B);
				
				\draw (a)--(B);	
				\draw (b)--(B);
				
				\foreach \p in {A,B} {
					\fill[white] (\p) circle (2pt);
					\draw (\p) circle (2pt);
				}
				\foreach \p in {a,b}{
					\fill (\p) circle (2pt);
				}
		}}
	\end{minipage}
	\\\vspace{0.2cm}
	\begin{minipage}[t]{0.15\textwidth}
		\centering
		\scalemath{0.8}{
			\Diska{
				\coordinate (A) at (-112.5:0.7);
				\coordinate (B) at (157.5:0.7);
				
				\coordinate (a) at (45:0.5);
				\coordinate (b) at (0:0);
				
				\draw (1)--(a);		\draw (2)--(a);		\draw (3)--(a);		\draw (4)--(b);
				\draw (5)--(A);		\draw (6)--(A);		\draw (7)--(B);		\draw (8)--(B);
				
				\draw (b)--(A);		\draw (b)--(B);
				
				\foreach \p in {A,B} {
					\fill[white] (\p) circle (2pt);
					\draw (\p) circle (2pt);
				}
				\foreach \p in {a,b}{
					\fill (\p) circle (2pt);
				}
		}}
	\end{minipage}
	\begin{minipage}[t]{0.15\textwidth}
		\centering
		\scalemath{0.8}{
			\Diska{
				\coordinate (A) at (-112.5:0.7);
				\coordinate (B) at (157.5:0.7);
				
				\coordinate (a) at (0:0.5);
				\coordinate (b) at (0:0);
				
				\draw (1)--(b);		\draw (2)--(a);		\draw (3)--(a);		\draw (4)--(a);
				\draw (5)--(A);		\draw (6)--(A);		\draw (7)--(B);		\draw (8)--(B);
				
				\draw (b)--(A);		\draw (b)--(B);
				
				\foreach \p in {A,B} {
					\fill[white] (\p) circle (2pt);
					\draw (\p) circle (2pt);
				}
				\foreach \p in {a,b}{
					\fill (\p) circle (2pt);
				}
		}}
	\end{minipage}
	\begin{minipage}[t]{0.15\textwidth}
		\centering
		\scalemath{0.8}{
			\Diska{
				\coordinate (A) at (-112.5:0.7);
				\coordinate (B) at (157.5:0.7);
				
				\coordinate (a) at (67.5:0.7);
				\coordinate (b) at (-22.5:0.7);
				
				\draw (1)--(a);		\draw (2)--(a);		\draw (3)--(b);		\draw (4)--(b);
				\draw (5)--(A);		\draw (6)--(A);		\draw (7)--(B);		\draw (8)--(B);
				
				\draw (a)--(B);	
				\draw (b)--(A);	
				
				\foreach \p in {A,B} {
					\fill[white] (\p) circle (2pt);
					\draw (\p) circle (2pt);
				}
				\foreach \p in {a,b}{
					\fill (\p) circle (2pt);
				}
		}}
	\end{minipage}
	\begin{minipage}[t]{0.15\textwidth}
		\centering
		\scalemath{0.8}{
			\Diska{
				\coordinate (A) at (-135:0.5);
				\coordinate (a) at (0:0.5);
				
				\draw (2)--(a);		\draw (3)--(a);		\draw (4)--(a);
				\draw (5)--(A);		\draw (6)--(A);		\draw (7)--(A);		
				
				\draw[-stealth] (8) arc[start angle=-135, end angle=-10, radius=.414];
				
				\fill[white] (A) circle (2pt);
				\draw (A) circle (2pt);
				\fill (a) circle (2pt);
		}}
	\end{minipage}
	\begin{minipage}[t]{0.15\textwidth}
		\centering
		\scalemath{0.8}{
			\Diska{
				\coordinate (A) at (180:0.5);
				\coordinate (a) at (45:0.5);
				
				\draw (1)--(a);		\draw (2)--(a);		\draw (3)--(a);
				\draw (6)--(A);		\draw (7)--(A);		\draw (8)--(A);		
				
				\draw[-stealth] (5) arc[start angle=180, end angle=55, radius=.414];
				
				\fill[white] (A) circle (2pt);
				\draw (A) circle (2pt);
				\fill (a) circle (2pt);
		}}
	\end{minipage}
	\begin{minipage}[t]{0.15\textwidth}
		\centering
		\scalemath{0.8}{
			\Diska{
				\coordinate (A) at (180:0.5);
				\coordinate (a) at (0:0.5);
				
				\draw (2)--(a);		\draw (3)--(a);		\draw (4)--(a);
				\draw (6)--(A);		\draw (7)--(A);		\draw (8)--(A);		
				
				\draw[-stealth] (5)--(90:0.93);
				
				\fill[white] (A) circle (2pt);
				\draw (A) circle (2pt);
				\fill (a) circle (2pt);
		}}
	\end{minipage}
	\\\vspace{0.2cm}
	\begin{minipage}[t]{0.15\textwidth}
		\centering
		\scalemath{0.8}{
			\Diska{
				\coordinate (A) at (-135:0.5);
				\coordinate (a) at (45:0.5);
				
				\draw (1)--(a);		\draw (2)--(a);		\draw (3)--(a);
				\draw (5)--(A);		\draw (6)--(A);		\draw (7)--(A);		
				
				\draw[-stealth] (8)--(-45:0.93);
				
				\fill[white] (A) circle (2pt);
				\draw (A) circle (2pt);
				\fill (a) circle (2pt);
		}}
	\end{minipage}
	\begin{minipage}[t]{0.15\textwidth}
		\centering
		\scalemath{0.8}{
			\Diska{
				\coordinate (A) at (-112.5:0.7);
				\coordinate (a) at (-22.5:0.7);
				
				\draw (3)--(a);		\draw (4)--(a);
				\draw (5)--(A);		\draw (6)--(A);		
				
				\draw[-stealth] (8) arc[start angle=-135, end angle=-10, radius=.414];
				\draw[-stealth] (7)--(45:0.93);
				
				\draw(a)--(A);
				
				\fill[white] (A) circle (2pt);
				\draw (A) circle (2pt);
				\fill (a) circle (2pt);
		}}
	\end{minipage}
	\begin{minipage}[t]{0.15\textwidth}
		\centering
		\scalemath{0.8}{
			\Diska{
				\coordinate (A) at (157.5:0.7);
				\coordinate (a) at (67.5:0.7);
				
				\draw (1)--(a);		\draw (2)--(a);		
				\draw (7)--(A);		\draw (8)--(A);		
				
				\draw[-stealth] (5) arc[start angle=180, end angle=55, radius=.414];
				\draw[-stealth] (6)--(0:0.93);
				
				\draw (a)--(A);
				
				\fill[white] (A) circle (2pt);
				\draw (A) circle (2pt);
				\fill (a) circle (2pt);
		}}
	\end{minipage}
	\begin{minipage}[t]{0.15\textwidth}
		\centering
		\scalemath{0.8}{
			\Diska{
				\coordinate (A) at (-157.5:0.7);
				\coordinate (a) at (22.5:0.7);
				
				\draw (2)--(a);		\draw (3)--(a);		
				\draw (6)--(A);		\draw (7)--(A);		
				
				\draw[-stealth] (5) arc[start angle=180, end angle=55, radius=.414];
				\draw[-stealth] (8) arc[start angle=-135, end angle=-10, radius=.414];
				
				\draw (a)--(A);
				
				\fill[white] (A) circle (2pt);
				\draw (A) circle (2pt);
				\fill (a) circle (2pt);
		}}
	\end{minipage}
	\begin{minipage}[t]{0.15\textwidth}
		\centering
		\scalemath{0.8}{
			\Diska{	
				\draw[-stealth] (5) arc[start angle=180, end angle=55, radius=.414];
				\draw[-stealth] (6)--(0:0.93);
				\draw[-stealth] (8) arc[start angle=-135, end angle=-10, radius=.414];
				\draw[-stealth] (7)--(45:0.93);
		}}
	\end{minipage}
	\caption{Non-elliptic webs with 4 black and 4 white boundary vertices of type 1.}
	\label{fig:type1}
\end{figure}

\begin{figure}[H]
	\centering
	\begin{minipage}[t]{0.15\textwidth}
		\centering
		\scalemath{0.8}{
			\Diskb{
				\coordinate (A) at (157.5:0.7);
				\coordinate (B) at (82.5:0.4);
				\coordinate (C) at (-37.5:0.4);
				\coordinate (D) at (-157.5:0.4);
				
				\coordinate (a) at (67.5:0.7);
				\coordinate (b) at (22.5:0.4);
				\coordinate (c) at (-97.5:0.4);
				\coordinate (d) at (142.5:0.4);
				
				\draw (1)--(a);		\draw (2)--(a);		\draw (3)--(b);		\draw (5)--(c);
				\draw (4)--(C);		\draw (6)--(D);		\draw (7)--(A);		\draw (8)--(A);
				
				\draw (a)--(B);	
				\draw (b)--(B);		\draw (b)--(C);
				\draw (c)--(C);		\draw (c)--(D);	
				\draw (d)--(A);		\draw (d)--(B);		\draw (d)--(D);	
				
				\foreach \p in {A,B,C,D} {
					\fill[white] (\p) circle (2pt);
					\draw (\p) circle (2pt);
				}
				\foreach \p in {a,b,c,d}{
					\fill (\p) circle (2pt);
				}
		}}
	\end{minipage}
	\begin{minipage}[t]{0.15\textwidth}
		\centering
		\scalemath{0.8}{
			\Diskb{
				\coordinate (A) at (-157.5:0.7);
				\coordinate (B) at (-22.5:0.5);
				\coordinate (C) at (145:0.4);
				
				\coordinate (a) at (22.5:0.7);
				\coordinate (b) at (-112.5:0.5);
				\coordinate (c) at (80:0.4);
				
				\draw (1)--(c);		\draw (2)--(a);		\draw (3)--(a);		\draw (5)--(b);
				\draw (4)--(B);		\draw (6)--(A);		\draw (7)--(A);		\draw (8)--(C);
				
				\draw (a)--(B);	
				\draw (b)--(A);		\draw (b)--(C);
				\draw (c)--(B);		\draw (c)--(C);	
				
				\foreach \p in {A,B,C} {
					\fill[white] (\p) circle (2pt);
					\draw (\p) circle (2pt);
				}
				\foreach \p in {a,b,c}{
					\fill (\p) circle (2pt);
				}
		}}
	\end{minipage}
	\begin{minipage}[t]{0.15\textwidth}
		\centering
		\scalemath{0.8}{
			\Diskb{
				\coordinate (A) at (157.5:0.7);
				\coordinate (B) at (-30:0.4);
				\coordinate (C) at (-150:0.4);
				
				\coordinate (a) at (22.5:0.7);
				\coordinate (b) at (-90:0.4);
				\coordinate (c) at (150:0.4);
				
				\draw (1)--(c);		\draw (2)--(a);		\draw (3)--(a);		\draw (5)--(b);
				\draw (4)--(B);		\draw (6)--(C);		\draw (7)--(A);		\draw (8)--(A);
				
				\draw (a)--(B);	
				\draw (b)--(B);		\draw (b)--(C);
				\draw (c)--(A);		\draw (c)--(C);	
				
				\foreach \p in {A,B,C} {
					\fill[white] (\p) circle (2pt);
					\draw (\p) circle (2pt);
				}
				\foreach \p in {a,b,c}{
					\fill (\p) circle (2pt);
				}
		}}
	\end{minipage}
	\begin{minipage}[t]{0.15\textwidth}
		\centering
		\scalemath{0.8}{
			\Diskb{
				\coordinate (A) at (-157.5:0.7);
				\coordinate (B) at (75:0.4);
				\coordinate (C) at (-45:0.4);
				
				\coordinate (a) at (67.5:0.7);
				\coordinate (b) at (15:0.4);
				\coordinate (c) at (-105:0.4);
				
				\draw (1)--(a);		\draw (2)--(a);		\draw (3)--(b);		\draw (5)--(c);
				\draw (4)--(C);		\draw (6)--(A);		\draw (7)--(A);		\draw (8)--(B);
				
				\draw (a)--(B);	
				\draw (b)--(B);		\draw (b)--(C);
				\draw (c)--(A);		\draw (c)--(C);	
				
				\foreach \p in {A,B,C} {
					\fill[white] (\p) circle (2pt);
					\draw (\p) circle (2pt);
				}
				\foreach \p in {a,b,c}{
					\fill (\p) circle (2pt);
				}
		}}
	\end{minipage}
	\begin{minipage}[t]{0.15\textwidth}
		\centering
		\scalemath{0.8}{
			\Diskb{
				\coordinate (A) at (-157.5:0.7);
				\coordinate (B) at (-35:0.4);
				
				\coordinate (a) at (22.5:0.7);
				\coordinate (b) at (-100:0.4);
				
				\draw[-stealth] (8) arc[start angle=-135, end angle=-10, radius=.414];
				
				\draw (2)--(a);		\draw (3)--(a);		\draw (5)--(b);
				\draw (4)--(B);		\draw (6)--(A);		\draw (7)--(A);		
				
				\draw (a)--(B);	
				\draw (b)--(A);		\draw (b)--(B);
				
				\foreach \p in {A,B} {
					\fill[white] (\p) circle (2pt);
					\draw (\p) circle (2pt);
				}
				\foreach \p in {a,b}{
					\fill (\p) circle (2pt);
				}
		}}
	\end{minipage}
	\begin{minipage}[t]{0.15\textwidth}
		\centering
		\scalemath{0.8}{
			\Diskb{
				\coordinate (A) at (-157.5:0.7);
				\coordinate (B) at (90:0.3);
				
				\coordinate (a) at (67.5:0.7);
				\coordinate (b) at (-90:0.3);
				
				\draw[-stealth] (4) arc[start angle=225, end angle=100, radius=.414];
				
				\draw (1)--(a);		\draw (2)--(a);		\draw (5)--(b);
				\draw (6)--(A);		\draw (7)--(A);		\draw (8)--(B);		
				
				\draw (a)--(B);	
				\draw (b)--(A);		\draw (b)--(B);
				
				\foreach \p in {A,B} {
					\fill[white] (\p) circle (2pt);
					\draw (\p) circle (2pt);
				}
				\foreach \p in {a,b}{
					\fill (\p) circle (2pt);
				}
		}}
	\end{minipage}
	\\\vspace{0.2cm}
	\begin{minipage}[t]{0.15\textwidth}
		\centering
		\scalemath{0.8}{
			\Diskb{
				\coordinate (A) at (157.5:0.7);
				\coordinate (B) at (-45:0.3);
				
				\coordinate (a) at (22.5:0.7);
				\coordinate (b) at (135:0.3);
				
				\draw[-stealth] (6) arc[start angle=135, end angle=10, radius=.414];
				
				\draw (1)--(b);		\draw (2)--(a);		\draw (3)--(a);
				\draw (4)--(B);		\draw (7)--(A);		\draw (8)--(A);		
				
				\draw (a)--(B);	
				\draw (b)--(A);		\draw (b)--(B);
				
				\foreach \p in {A,B} {
					\fill[white] (\p) circle (2pt);
					\draw (\p) circle (2pt);
				}
				\foreach \p in {a,b}{
					\fill (\p) circle (2pt);
				}
		}}
	\end{minipage}
	\begin{minipage}[t]{0.15\textwidth}
		\centering
		\scalemath{0.8}{
			\Diskb{
				\coordinate (A) at (-157.5:0.7);
				\coordinate (B) at (112.5:0.5);
				
				\coordinate (a) at (67.5:0.7);
				\coordinate (b) at (0:0);
				
				\draw[-stealth] (4) arc[start angle=45, end angle=170, radius=.414];
				
				\draw (1)--(a);		\draw (2)--(a);		\draw (3)--(b);
				\draw (6)--(A);		\draw (7)--(A);		\draw (8)--(B);		
				
				\draw (a)--(B);	
				\draw (b)--(A);		\draw (b)--(B);
				
				\foreach \p in {A,B} {
					\fill[white] (\p) circle (2pt);
					\draw (\p) circle (2pt);
				}
				\foreach \p in {a,b}{
					\fill (\p) circle (2pt);
				}
		}}
	\end{minipage}
	\begin{minipage}[t]{0.15\textwidth}
		\centering
		\scalemath{0.8}{
			\Diskb{
				\coordinate (A) at (157.5:0.7);
				\coordinate (B) at (0:0);
				
				\coordinate (a) at (22.5:0.7);
				\coordinate (b) at (112.5:0.5);
				
				\draw[-stealth] (4) arc[start angle=45, end angle=170, radius=.414];
				
				\draw (1)--(b);		\draw (2)--(a);		\draw (3)--(a);
				\draw (6)--(B);		\draw (7)--(A);		\draw (8)--(A);		
				
				\draw (a)--(B);	
				\draw (b)--(A);		\draw (b)--(B);
				
				\foreach \p in {A,B} {
					\fill[white] (\p) circle (2pt);
					\draw (\p) circle (2pt);
				}
				\foreach \p in {a,b}{
					\fill (\p) circle (2pt);
				}
		}}
	\end{minipage}
	\begin{minipage}[t]{0.15\textwidth}
		\centering
		\scalemath{0.8}{
			\Diskb{
				\coordinate (A) at (112.5:0.5);
				\coordinate (B) at (-112.5:0.3);
				
				\coordinate (a) at (67.5:0.7);
				\coordinate (b) at (0:0.4);
				
				\draw[-stealth] (6) arc[start angle=135, end angle=10, radius=.414];
				
				\draw (1)--(a);		\draw (2)--(a);		\draw (3)--(b);
				\draw (4)--(B);		\draw (7)--(B);		\draw (8)--(A);		
				
				\draw (a)--(A);	
				\draw (b)--(A);		\draw (b)--(B);
				
				\foreach \p in {A,B} {
					\fill[white] (\p) circle (2pt);
					\draw (\p) circle (2pt);
				}
				\foreach \p in {a,b}{
					\fill (\p) circle (2pt);
				}
		}}
	\end{minipage}
	\begin{minipage}[t]{0.15\textwidth}
		\centering
		\scalemath{0.8}{
			\Diskb{
				\coordinate (A) at (157.5:0.7);
				\coordinate (B) at (-135:0.4);
				
				\coordinate (a) at (112.5:0.5);
				\coordinate (b) at (-22.5:0.3);
				
				\draw[-stealth] (4) arc[start angle=225, end angle=100, radius=.414];
				
				\draw (1)--(a);		\draw (2)--(b);		\draw (5)--(b);
				\draw (6)--(B);		\draw (7)--(A);		\draw (8)--(A);		
				
				\draw (a)--(A);		\draw (a)--(B);	
				\draw (b)--(B);
				
				\foreach \p in {A,B} {
					\fill[white] (\p) circle (2pt);
					\draw (\p) circle (2pt);
				}
				\foreach \p in {a,b}{
					\fill (\p) circle (2pt);
				}
		}}
	\end{minipage}
	\begin{minipage}[t]{0.15\textwidth}
		\centering
		\scalemath{0.8}{
			\Diskb{
				\coordinate (A) at (180:0.5);
				\coordinate (B) at (0:0.4);
				
				\coordinate (a) at (22.5:0.7);
				\coordinate (b) at (0:0);
				
				\draw (1)--(b);		\draw (2)--(a);		\draw (3)--(a);		\draw (5)--(b);
				\draw (4)--(B);		\draw (6)--(A);		\draw (7)--(A);		\draw (8)--(A);		
				
				\draw (a)--(B);	
				\draw (b)--(B);
				
				\foreach \p in {A,B} {
					\fill[white] (\p) circle (2pt);
					\draw (\p) circle (2pt);
				}
				\foreach \p in {a,b}{
					\fill (\p) circle (2pt);
				}
		}}
	\end{minipage}
	\\\vspace{0.2cm}
	\begin{minipage}[t]{0.15\textwidth}
		\centering
		\scalemath{0.8}{
			\Diskb{
				\coordinate (A) at (-157.5:0.7);
				\coordinate (B) at (0:0);
				
				\coordinate (a) at (45:0.5);
				\coordinate (b) at (-135:0.4);
				
				\draw (1)--(a);		\draw (2)--(a);		\draw (3)--(a);		\draw (5)--(b);
				\draw (4)--(B);		\draw (6)--(A);		\draw (7)--(A);		\draw (8)--(B);

				\draw (b)--(A);		\draw (b)--(B);
				
				\foreach \p in {A,B} {
					\fill[white] (\p) circle (2pt);
					\draw (\p) circle (2pt);
				}
				\foreach \p in {a,b}{
					\fill (\p) circle (2pt);
				}
		}}
	\end{minipage}
	\begin{minipage}[t]{0.15\textwidth}
		\centering
		\scalemath{0.8}{
			\Diskb{
				\coordinate (A) at (157.5:0.7);
				
				\coordinate (a) at (67.5:0.7);
				
				\draw (1)--(a);		\draw (2)--(a);		
				\draw (7)--(A);		\draw (8)--(A);		
				
				\draw[-stealth] (4) arc[start angle=225, end angle=100, radius=.414];
				\draw[-stealth] (6) arc[start angle=135, end angle=10, radius=.414];
				
				\draw (a)--(A);	
				
				\fill[white] (A) circle (2pt);
				\draw (A) circle (2pt);
				\fill (a) circle (2pt);
		}}
	\end{minipage}
	\begin{minipage}[t]{0.15\textwidth}
		\centering
		\scalemath{0.8}{
			\Diskb{
				\coordinate (A) at (157.5:0.7);
				
				\coordinate (a) at (67.5:0.7);
				
				\draw (1)--(a);		\draw (2)--(a);		
				\draw (7)--(A);		\draw (8)--(A);		
				
				\draw[-stealth] (4) arc[start angle=45, end angle=170, radius=.414];
				\draw[-stealth] (6)--(0:0.93);
				
				\draw (a)--(A);	
				
				\fill[white] (A) circle (2pt);
				\draw (A) circle (2pt);
				\fill (a) circle (2pt);
		}}
	\end{minipage}
	\begin{minipage}[t]{0.15\textwidth}
		\centering
		\scalemath{0.8}{
			\Diskb{
				\coordinate (A) at (180:0.5);
				
				\coordinate (a) at (45:0.5);
				
				\draw (1)--(a);		\draw (2)--(a);		\draw (3)--(a);
				\draw (6)--(A);		\draw (7)--(A);		\draw (8)--(A);		
				
				\draw[-stealth] (4) arc[start angle=45, end angle=170, radius=.414];
				
				\fill[white] (A) circle (2pt);
				\draw (A) circle (2pt);
				\fill (a) circle (2pt);
		}}
	\end{minipage}
	\begin{minipage}[t]{0.15\textwidth}
		\centering
		\scalemath{0.8}{
			\Diskb{
				\coordinate (A) at (0:0);
				
				\coordinate (a) at (45:0.5);
				
				\draw (1)--(a);		\draw (2)--(a);		\draw (3)--(a);
				\draw (4)--(A);		\draw (7)--(A);		\draw (8)--(A);		
				
				\draw[-stealth] (6) arc[start angle=135, end angle=10, radius=.414];
				
				\fill[white] (A) circle (2pt);
				\draw (A) circle (2pt);
				\fill (a) circle (2pt);
		}}
	\end{minipage}
	\begin{minipage}[t]{0.15\textwidth}
		\centering
		\scalemath{0.8}{
			\Diskb{
				\coordinate (A) at (0:0);
				
				\coordinate (a) at (22.5:0.7);
				
				\draw (2)--(a);		\draw (3)--(a);		
				\draw (4)--(A);		\draw (7)--(A);		
				
				\draw[-stealth] (6) arc[start angle=135, end angle=10, radius=.414];
				\draw[-stealth] (8) arc[start angle=-135, end angle=-10, radius=.414];
				
				\draw (a)--(A);	
				
				\fill[white] (A) circle (2pt);
				\draw (A) circle (2pt);
				\fill (a) circle (2pt);
		}}
	\end{minipage}
	\\\vspace{0.2cm}
	\begin{minipage}[t]{0.15\textwidth}
		\centering
		\scalemath{0.8}{
			\Diskb{
				\coordinate (A) at (-157.5:0.7);
				
				\coordinate (a) at (22.5:0.7);
				
				\draw (2)--(a);		\draw (3)--(a);		
				\draw (6)--(A);		\draw (7)--(A);		
				
				\draw[-stealth] (4) arc[start angle=45, end angle=170, radius=.414];
				\draw[-stealth] (8) arc[start angle=-135, end angle=-10, radius=.414];
				
				\draw (a)--(A);	
				
				\fill[white] (A) circle (2pt);
				\draw (A) circle (2pt);
				\fill (a) circle (2pt);
		}}
	\end{minipage}
	\begin{minipage}[t]{0.15\textwidth}
		\centering
		\scalemath{0.8}{
			\Diskb{
				\coordinate (A) at (-157.5:0.7);
				
				\coordinate (a) at (0:0);
				
				\draw (2)--(a);		\draw (5)--(a);		
				\draw (6)--(A);		\draw (7)--(A);		
				
				\draw[-stealth] (4) arc[start angle=225, end angle=100, radius=.414];
				\draw[-stealth] (8) arc[start angle=-135, end angle=-10, radius=.414];
				
				\draw (a)--(A);	
				
				\fill[white] (A) circle (2pt);
				\draw (A) circle (2pt);
				\fill (a) circle (2pt);
		}}
	\end{minipage}
	\begin{minipage}[t]{0.15\textwidth}
		\centering
		\scalemath{0.8}{
			\Diskb{
				\coordinate (A) at (180:0.5);
				
				\coordinate (a) at (0:0);
				
				\draw (1)--(a);		\draw (2)--(a);		\draw (5)--(a);	
				\draw (6)--(A);		\draw (7)--(A);		\draw (8)--(A);		
				
				\draw[-stealth] (4) arc[start angle=225, end angle=100, radius=.414];
				
				\fill[white] (A) circle (2pt);
				\draw (A) circle (2pt);
				\fill (a) circle (2pt);
		}}
	\end{minipage}
	\begin{minipage}[t]{0.15\textwidth}
		\centering
		\scalemath{0.8}{
			\Diskb{
				\draw[-stealth] (4) arc[start angle=45, end angle=170, radius=.414];
				\draw[-stealth] (6)--(0:0.93);
				\draw[-stealth] (8) arc[start angle=-135, end angle=-10, radius=.414];
				\draw[-stealth] (7)--(45:0.93);
		}}
	\end{minipage}
	\begin{minipage}[t]{0.15\textwidth}
		\centering
		\scalemath{0.8}{
			\Diskb{
				\draw[-stealth] (4) arc[start angle=225, end angle=100, radius=.414];
				\draw[-stealth] (6) arc[start angle=135, end angle=10, radius=.414];
				\draw[-stealth] (8) arc[start angle=-135, end angle=-10, radius=.414];
				\draw[-stealth] (7)--(45:0.93);
		}}
	\end{minipage}
	\caption{Non-elliptic webs with 4 black and 4 white boundary vertices of type 2.}
	\label{fig:type2}
\end{figure}

\begin{figure}[H]
	\centering
	\begin{minipage}[t]{0.15\textwidth}
		\centering
		\scalemath{0.8}{
			\Diskc{
				\coordinate (A) at (-67.5:0.7);
				\coordinate (B) at (157.5:0.7);
				\coordinate (C) at (0:0);
				
				\coordinate (a) at (22.5:0.7);
				\coordinate (b) at (-112.5:0.5);
				\coordinate (c) at (112.5:0.5);
				
				\draw (1)--(c);		\draw (2)--(a);		\draw (3)--(a);		\draw (6)--(b);
				\draw (4)--(A);		\draw (5)--(A);		\draw (7)--(B);		\draw (8)--(B);
				
				\draw (a)--(C);	
				\draw (b)--(A);		\draw (b)--(C);
				\draw (c)--(B);		\draw (c)--(C);	
				
				\foreach \p in {A,B,C} {
					\fill[white] (\p) circle (2pt);
					\draw (\p) circle (2pt);
				}
				\foreach \p in {a,b,c}{
					\fill (\p) circle (2pt);
				}
		}}
	\end{minipage}
	\begin{minipage}[t]{0.15\textwidth}
		\centering
		\scalemath{0.8}{
			\Diskc{
				\coordinate (A) at (-67.5:0.7);
				\coordinate (B) at (157.5:0.7);
				\coordinate (C) at (0:0);
				
				\coordinate (a) at (67.5:0.7);
				\coordinate (b) at (-22.5:0.5);
				\coordinate (c) at (-157.5:0.5);
				
				\draw (1)--(a);		\draw (2)--(a);		\draw (3)--(b);		\draw (6)--(c);
				\draw (4)--(A);		\draw (5)--(A);		\draw (7)--(B);		\draw (8)--(B);
				
				\draw (a)--(C);	
				\draw (b)--(A);		\draw (b)--(C);
				\draw (c)--(B);		\draw (c)--(C);	
				
				\foreach \p in {A,B,C} {
					\fill[white] (\p) circle (2pt);
					\draw (\p) circle (2pt);
				}
				\foreach \p in {a,b,c}{
					\fill (\p) circle (2pt);
				}
		}}
	\end{minipage}
	\begin{minipage}[t]{0.15\textwidth}
		\centering
		\scalemath{0.8}{
			\Diskc{
				\coordinate (A) at (-67.5:0.7);
				\coordinate (B) at (-170:0.5);
				\coordinate (C) at (125:0.5);
				
				\coordinate (a) at (67.5:0.7);
				\coordinate (b) at (-112.5:0.5);
				\coordinate (c) at (10:0.3);
				
				\draw (1)--(a);		\draw (2)--(a);		\draw (3)--(c);		\draw (6)--(b);
				\draw (4)--(A);		\draw (5)--(A);		\draw (7)--(B);		\draw (8)--(C);
				
				\draw (a)--(C);	
				\draw (b)--(A);		\draw (b)--(B);
				\draw (c)--(B);		\draw (c)--(C);	
				
				\foreach \p in {A,B,C} {
					\fill[white] (\p) circle (2pt);
					\draw (\p) circle (2pt);
				}
				\foreach \p in {a,b,c}{
					\fill (\p) circle (2pt);
				}
		}}
	\end{minipage}
	\begin{minipage}[t]{0.15\textwidth}
		\centering
		\scalemath{0.8}{
			\Diskc{
				\coordinate (A) at (157.5:0.7);
				\coordinate (B) at (-35:0.5);
				\coordinate (C) at (-100:0.5);
				
				\coordinate (a) at (22.5:0.7);
				\coordinate (b) at (-157.5:0.5);
				\coordinate (c) at (80:0.3);
				
				\draw (1)--(c);		\draw (2)--(a);		\draw (3)--(a);		\draw (6)--(b);
				\draw (4)--(B);		\draw (5)--(C);		\draw (7)--(A);		\draw (8)--(A);
				
				\draw (a)--(B);	
				\draw (b)--(A);		\draw (b)--(C);
				\draw (c)--(B);		\draw (c)--(C);	
				
				\foreach \p in {A,B,C} {
					\fill[white] (\p) circle (2pt);
					\draw (\p) circle (2pt);
				}
				\foreach \p in {a,b,c}{
					\fill (\p) circle (2pt);
				}
		}}
	\end{minipage}
	\begin{minipage}[t]{0.15\textwidth}
		\centering
		\scalemath{0.8}{
			\Diskc{
				\coordinate (A) at (-67.5:0.7);
				\coordinate (B) at (0:0);
				
				\coordinate (a) at (22.5:0.7);
				\coordinate (b) at (-112.5:0.5);
				
				\draw (2)--(a);		\draw (3)--(a);		\draw (6)--(b);		
				\draw (4)--(A);		\draw (5)--(A);		\draw (7)--(B);
				
				\draw[-stealth] (8) arc[start angle=-135, end angle=-10, radius=.414];
				
				\draw (a)--(B);	
				\draw (b)--(A);		\draw (b)--(B);
				
				\foreach \p in {A,B} {
					\fill[white] (\p) circle (2pt);
					\draw (\p) circle (2pt);
				}
				\foreach \p in {a,b}{
					\fill (\p) circle (2pt);
				}
		}}
	\end{minipage}
	\begin{minipage}[t]{0.15\textwidth}
		\centering
		\scalemath{0.8}{
			\Diskc{
				\coordinate (A) at (157.5:0.7);
				\coordinate (B) at (0:0);
				
				\coordinate (a) at (67.5:0.7);
				\coordinate (b) at (-157.5:0.5);
				
				\draw (1)--(a);		\draw (2)--(a);		\draw (6)--(b);		
				\draw (5)--(B);		\draw (7)--(A);		\draw (8)--(A);
				
				\draw[-stealth] (4) arc[start angle=225, end angle=100, radius=.414];
				
				\draw (a)--(B);	
				\draw (b)--(A);		\draw (b)--(B);
				
				\foreach \p in {A,B} {
					\fill[white] (\p) circle (2pt);
					\draw (\p) circle (2pt);
				}
				\foreach \p in {a,b}{
					\fill (\p) circle (2pt);
				}
		}}
	\end{minipage}
	\\\vspace{0.2cm}
	\begin{minipage}[t]{0.15\textwidth}
		\centering
		\scalemath{0.8}{
			\Diskc{
				\coordinate (A) at (157.5:0.7);
				\coordinate (B) at (0:0.3);
				
				\coordinate (a) at (22.5:0.7);
				\coordinate (b) at (135:0.3);
				
				\draw (1)--(b);		\draw (2)--(a);		\draw (3)--(a);		
				\draw (4)--(B);		\draw (7)--(A);		\draw (8)--(A);
				
				\draw[-stealth] (5) arc[start angle=0, end angle=125, radius=.414];
				
				\draw (a)--(B);	
				\draw (b)--(A);		\draw (b)--(B);
				
				\foreach \p in {A,B} {
					\fill[white] (\p) circle (2pt);
					\draw (\p) circle (2pt);
				}
				\foreach \p in {a,b}{
					\fill (\p) circle (2pt);
				}
		}}
	\end{minipage}
	\begin{minipage}[t]{0.15\textwidth}
		\centering
		\scalemath{0.8}{
			\Diskc{
				\coordinate (A) at (-67.5:0.7);
				\coordinate (B) at (90:0.3);
				
				\coordinate (a) at (67.5:0.7);
				\coordinate (b) at (-45:0.3);
				
				\draw (1)--(a);		\draw (2)--(a);		\draw (3)--(b);		
				\draw (4)--(A);		\draw (5)--(A);		\draw (8)--(B);
				
				\draw[-stealth] (7) arc[start angle=90, end angle=-45, radius=.414];
				
				\draw (a)--(B);	
				\draw (b)--(A);		\draw (b)--(B);
				
				\foreach \p in {A,B} {
					\fill[white] (\p) circle (2pt);
					\draw (\p) circle (2pt);
				}
				\foreach \p in {a,b}{
					\fill (\p) circle (2pt);
				}
		}}
	\end{minipage}
	\begin{minipage}[t]{0.15\textwidth}
		\centering
		\scalemath{0.8}{
			\Diskc{
				\coordinate (A) at (-22.5:0.5);
				\coordinate (B) at (-157.5:0.3);
				
				\coordinate (a) at (22.5:0.7);
				\coordinate (b) at (90:0.4);
				
				\draw (1)--(b);		\draw (2)--(a);		\draw (3)--(a);		
				\draw (4)--(A);		\draw (5)--(B);		\draw (8)--(B);
				
				\draw[-stealth] (7) arc[start angle=90, end angle=-45, radius=.414];
				
				\draw (a)--(A);	
				\draw (b)--(A);		\draw (b)--(B);
				
				\foreach \p in {A,B} {
					\fill[white] (\p) circle (2pt);
					\draw (\p) circle (2pt);
				}
				\foreach \p in {a,b}{
					\fill (\p) circle (2pt);
				}
		}}
	\end{minipage}
	\begin{minipage}[t]{0.15\textwidth}
		\centering
		\scalemath{0.8}{
			\Diskc{
				\coordinate (A) at (112.5:0.5);
				\coordinate (B) at (-112.5:0.3);
				
				\coordinate (a) at (67.5:0.7);
				\coordinate (b) at (0:0.4);
				
				\draw (1)--(a);		\draw (2)--(a);		\draw (3)--(b);		
				\draw (4)--(B);		\draw (7)--(B);		\draw (8)--(A);
				
				\draw[-stealth] (5) arc[start angle=0, end angle=125, radius=.414];
				
				\draw (a)--(A);	
				\draw (b)--(A);		\draw (b)--(B);
				
				\foreach \p in {A,B} {
					\fill[white] (\p) circle (2pt);
					\draw (\p) circle (2pt);
				}
				\foreach \p in {a,b}{
					\fill (\p) circle (2pt);
				}
		}}
	\end{minipage}
	\begin{minipage}[t]{0.15\textwidth}
		\centering
		\scalemath{0.8}{
			\Diskc{
				\coordinate (A) at (-67.5:0.7);
				\coordinate (B) at (157.5:0.7);
				
				\coordinate (a) at (22.5:0.7);
				\coordinate (b) at (157.5:0.3);
				
				\draw (1)--(b);		\draw (2)--(a);		\draw (3)--(a);		\draw (6)--(b);
				\draw (4)--(A);		\draw (5)--(A);		\draw (7)--(B);		\draw (8)--(B);
				
				\draw (a)--(A);	
				\draw (b)--(B);
				
				\foreach \p in {A,B} {
					\fill[white] (\p) circle (2pt);
					\draw (\p) circle (2pt);
				}
				\foreach \p in {a,b}{
					\fill (\p) circle (2pt);
				}
		}}
	\end{minipage}
	\begin{minipage}[t]{0.15\textwidth}
		\centering
		\scalemath{0.8}{
			\Diskc{
				\coordinate (A) at (-67.5:0.7);
				\coordinate (B) at (157.5:0.7);
				
				\coordinate (a) at (67.5:0.7);
				\coordinate (b) at (-67.5:0.3);
				
				\draw (1)--(a);		\draw (2)--(a);		\draw (3)--(b);		\draw (6)--(b);
				\draw (4)--(A);		\draw (5)--(A);		\draw (7)--(B);		\draw (8)--(B);
				
				\draw (a)--(B);	
				\draw (b)--(A);
				
				\foreach \p in {A,B} {
					\fill[white] (\p) circle (2pt);
					\draw (\p) circle (2pt);
				}
				\foreach \p in {a,b}{
					\fill (\p) circle (2pt);
				}
		}}
	\end{minipage}
	\\\vspace{0.2cm}
	\begin{minipage}[t]{0.15\textwidth}
		\centering
		\scalemath{0.8}{
			\Diskc{
				\coordinate (A) at (-67.5:0.7);
				\coordinate (B) at (157.5:0.7);
				
				\coordinate (a) at (45:0.5);
				\coordinate (b) at (-135:0.4);
				
				\draw (1)--(a);		\draw (2)--(a);		\draw (3)--(a);		\draw (6)--(b);
				\draw (4)--(A);		\draw (5)--(A);		\draw (7)--(B);		\draw (8)--(B);
				
				\draw (b)--(A);		\draw (b)--(B);
				
				\foreach \p in {A,B} {
					\fill[white] (\p) circle (2pt);
					\draw (\p) circle (2pt);
				}
				\foreach \p in {a,b}{
					\fill (\p) circle (2pt);
				}
		}}
	\end{minipage}
	\begin{minipage}[t]{0.15\textwidth}
		\centering
		\scalemath{0.8}{
			\Diskc{
				\coordinate (A) at (157.5:0.7);
				
				\coordinate (a) at (67.5:0.7);
				
				\draw (1)--(a);		\draw (2)--(a);		
				\draw (7)--(A);		\draw (8)--(A);		
				
				\draw (a)--(A);	
				
				\draw[-stealth] (4) arc[start angle=225, end angle=100, radius=.414];
				\draw[-stealth] (5) arc[start angle=0, end angle=125, radius=.414];
				
				\fill[white] (A) circle (2pt);
				\draw (A) circle (2pt);
				\fill (a) circle (2pt);
		}}
	\end{minipage}
	\begin{minipage}[t]{0.15\textwidth}
		\centering
		\scalemath{0.8}{
			\Diskc{
				\coordinate (A) at (-67.5:0.7);
				
				\coordinate (a) at (22.5:0.7);
				
				\draw (2)--(a);		\draw (3)--(a);		
				\draw (4)--(A);		\draw (5)--(A);		
				
				\draw (a)--(A);	
				
				\draw[-stealth] (7) arc[start angle=90, end angle=-35, radius=.414];
				\draw[-stealth] (8) arc[start angle=-135, end angle=-10, radius=.414];
				
				\fill[white] (A) circle (2pt);
				\draw (A) circle (2pt);
				\fill (a) circle (2pt);
		}}
	\end{minipage}
	\begin{minipage}[t]{0.15\textwidth}
		\centering
		\scalemath{0.8}{
			\Diskc{
				\coordinate (A) at (0:0);
				
				\coordinate (a) at (45:0.5);
				
				\draw (1)--(a);		\draw (2)--(a);		\draw (3)--(a);		
				\draw (4)--(A);		\draw (7)--(A);		\draw (8)--(A);
				
				\draw[-stealth] (5) arc[start angle=0, end angle=125, radius=.414];
				
				\fill[white] (A) circle (2pt);
				\draw (A) circle (2pt);
				\fill (a) circle (2pt);
		}}
	\end{minipage}
	\begin{minipage}[t]{0.15\textwidth}
		\centering
		\scalemath{0.8}{
			\Diskc{
				\coordinate (A) at (0:0);
				
				\coordinate (a) at (45:0.5);
				
				\draw (1)--(a);		\draw (2)--(a);		\draw (3)--(a);		
				\draw (4)--(A);		\draw (5)--(A);		\draw (8)--(A);
				
				\draw[-stealth] (7) arc[start angle=90, end angle=-35, radius=.414];
				
				\fill[white] (A) circle (2pt);
				\draw (A) circle (2pt);
				\fill (a) circle (2pt);
		}}
	\end{minipage}
	\begin{minipage}[t]{0.15\textwidth}
		\centering
		\scalemath{0.8}{
			\Diskc{
				\coordinate (A) at (-67.5:0.7);
				
				\coordinate (a) at (-67.5:0.3);
				
				\draw (3)--(a);		\draw (6)--(a);		
				\draw (4)--(A);		\draw (5)--(A);	
				\draw (a)--(A);
				
				\draw[-stealth] (8) arc[start angle=-135, end angle=-10, radius=.414];
				\draw[-stealth] (7)--(45:0.93);
				
				\fill[white] (A) circle (2pt);
				\draw (A) circle (2pt);
				\fill (a) circle (2pt);
		}}
	\end{minipage}
	\\\vspace{0.2cm}
	\begin{minipage}[t]{0.15\textwidth}
		\centering
		\scalemath{0.8}{
			\Diskc{
				\coordinate (A) at (157.5:0.7);
				
				\coordinate (a) at (157.5:0.3);
				
				\draw (1)--(a);		\draw (6)--(a);		
				\draw (7)--(A);		\draw (8)--(A);	
				\draw (a)--(A);
				
				\draw[-stealth] (4) arc[start angle=225, end angle=100, radius=.414];
				\draw[-stealth] (5)--(45:0.93);
				
				\fill[white] (A) circle (2pt);
				\draw (A) circle (2pt);
				\fill (a) circle (2pt);
		}}
	\end{minipage}
	\begin{minipage}[t]{0.15\textwidth}
		\centering
		\scalemath{0.8}{
			\Diskc{
				\coordinate (A) at (0:0);
				
				\coordinate (a) at (67.5:0.7);
				
				\draw (1)--(a);		\draw (2)--(a);		
				\draw (5)--(A);		\draw (8)--(A);	
				\draw (a)--(A);
				
				\draw[-stealth] (4) arc[start angle=225, end angle=100, radius=.414];
				\draw[-stealth] (7) arc[start angle=90, end angle=-35, radius=.414];
				
				\fill[white] (A) circle (2pt);
				\draw (A) circle (2pt);
				\fill (a) circle (2pt);
		}}
	\end{minipage}
	\begin{minipage}[t]{0.15\textwidth}
		\centering
		\scalemath{0.8}{
			\Diskc{
				\coordinate (A) at (0:0);
				
				\coordinate (a) at (22.5:0.7);
				
				\draw (2)--(a);		\draw (3)--(a);		
				\draw (4)--(A);		\draw (7)--(A);	
				\draw (a)--(A);
				
				\draw[-stealth] (5) arc[start angle=0, end angle=125, radius=.414];
				\draw[-stealth] (8) arc[start angle=-135, end angle=-10, radius=.414];
				
				\fill[white] (A) circle (2pt);
				\draw (A) circle (2pt);
				\fill (a) circle (2pt);
		}}
	\end{minipage}
	\begin{minipage}[t]{0.15\textwidth}
		\centering
		\scalemath{0.8}{
			\Diskc{
				\draw[-stealth] (8) arc[start angle=-135, end angle=-10, radius=.414];
				\draw[-stealth] (7) arc[start angle=90, end angle=-45, radius=.414];
				\draw[-stealth] (4) arc[start angle=225, end angle=100, radius=.414];
				\draw[-stealth] (5)--(45:0.93);
		}}
	\end{minipage}
	\begin{minipage}[t]{0.15\textwidth}
		\centering
		\scalemath{0.8}{
			\Diskc{
				\draw[-stealth] (8) arc[start angle=-135, end angle=-10, radius=.414];
				\draw[-stealth] (5) arc[start angle=0, end angle=125, radius=.414];
				\draw[-stealth] (4) arc[start angle=225, end angle=100, radius=.414];
				\draw[-stealth] (7)--(45:0.93);
		}}
	\end{minipage}
	\caption{Non-elliptic webs with 4 black and 4 white boundary vertices of type 3.}
	\label{fig:type3}
\end{figure}

\begin{figure}[H]
	\centering
	\begin{minipage}[t]{0.15\textwidth}
		\centering
		\scalemath{0.8}{
			\Diskd{
				\coordinate (A) at (157.5:0.7);
				\coordinate (B) at (22.5:0.5);
				\coordinate (C) at (-112.5:0.5);
				
				\coordinate (a) at (67.5:0.7);
				\coordinate (b) at (-67.5:0.7);
				\coordinate (c) at (0:0);
				
				\draw (1)--(a);		\draw (2)--(a);		\draw (4)--(b);		\draw (5)--(b);
				\draw (3)--(B);		\draw (6)--(C);		\draw (7)--(A);		\draw (8)--(A);
				
				\draw (a)--(B);	
				\draw (b)--(C);	
				\draw (c)--(A);		\draw (c)--(B);		\draw (c)--(C);	
				
				\foreach \p in {A,B,C} {
					\fill[white] (\p) circle (2pt);
					\draw (\p) circle (2pt);
				}
				\foreach \p in {a,b,c}{
					\fill (\p) circle (2pt);
				}
		}}
	\end{minipage}
	\begin{minipage}[t]{0.15\textwidth}
		\centering
		\scalemath{0.8}{
			\Diskd{
				\coordinate (A) at (-157.5:0.7);
				\coordinate (B) at (-22.5:0.5);
				\coordinate (C) at (112.5:0.5);
				
				\coordinate (a) at (67.5:0.7);
				\coordinate (b) at (-67.5:0.7);
				\coordinate (c) at (0:0);
				
				\draw (1)--(a);		\draw (2)--(a);		\draw (4)--(b);		\draw (5)--(b);
				\draw (3)--(B);		\draw (6)--(A);		\draw (7)--(A);		\draw (8)--(C);
				
				\draw (a)--(C);	
				\draw (b)--(B);	
				\draw (c)--(A);		\draw (c)--(B);		\draw (c)--(C);	
				
				\foreach \p in {A,B,C} {
					\fill[white] (\p) circle (2pt);
					\draw (\p) circle (2pt);
				}
				\foreach \p in {a,b,c}{
					\fill (\p) circle (2pt);
				}
		}}
	\end{minipage}
	\begin{minipage}[t]{0.15\textwidth}
		\centering
		\scalemath{0.8}{
			\Diskd{
				\coordinate (A) at (157.5:0.7);
				\coordinate (B) at (-22.5:0.5);
				\coordinate (C) at (-145:0.3);
				
				\coordinate (a) at (-67.5:0.7);
				\coordinate (b) at (35:0.5);
				\coordinate (c) at (100:0.5);
				
				\draw (1)--(c);		\draw (2)--(b);		\draw (4)--(a);		\draw (5)--(a);
				\draw (3)--(B);		\draw (6)--(C);		\draw (7)--(A);		\draw (8)--(A);
				
				\draw (a)--(B);	
				\draw (b)--(B);		\draw (b)--(C);	
				\draw (c)--(A);		\draw (c)--(C);	
				
				\foreach \p in {A,B,C} {
					\fill[white] (\p) circle (2pt);
					\draw (\p) circle (2pt);
				}
				\foreach \p in {a,b,c}{
					\fill (\p) circle (2pt);
				}
		}}
	\end{minipage}
	\begin{minipage}[t]{0.15\textwidth}
		\centering
		\scalemath{0.8}{
			\Diskd{
				\coordinate (A) at (-157.5:0.7);
				\coordinate (B) at (22.5:0.5);
				\coordinate (C) at (145:0.3);
				
				\coordinate (a) at (67.5:0.7);
				\coordinate (b) at (-35:0.5);
				\coordinate (c) at (-100:0.5);
				
				\draw (1)--(a);		\draw (2)--(a);		\draw (4)--(b);		\draw (5)--(c);
				\draw (3)--(B);		\draw (6)--(A);		\draw (7)--(A);		\draw (8)--(C);
				
				\draw (a)--(B);	
				\draw (b)--(B);		\draw (b)--(C);	
				\draw (c)--(A);		\draw (c)--(C);	
				
				\foreach \p in {A,B,C} {
					\fill[white] (\p) circle (2pt);
					\draw (\p) circle (2pt);
				}
				\foreach \p in {a,b,c}{
					\fill (\p) circle (2pt);
				}
		}}
	\end{minipage}
	\begin{minipage}[t]{0.15\textwidth}
		\centering
		\scalemath{0.8}{
			\Diskd{
				\coordinate (A) at (-157.5:0.7);
				\coordinate (B) at (90:0.3);
				
				\coordinate (a) at (67.5:0.7);
				\coordinate (b) at (-135:0.3);
				
				\draw (1)--(a);		\draw (2)--(a);		\draw (5)--(b);		
				\draw (6)--(A);		\draw (7)--(A);		\draw (8)--(B);
				
				\draw[-stealth] (3) arc[start angle=90, end angle=215, radius=.414];
				
				\draw (a)--(B);	
				\draw (b)--(A);		\draw (b)--(B);
				
				\foreach \p in {A,B} {
					\fill[white] (\p) circle (2pt);
					\draw (\p) circle (2pt);
				}
				\foreach \p in {a,b}{
					\fill (\p) circle (2pt);
				}
		}}
	\end{minipage}
	\begin{minipage}[t]{0.15\textwidth}
		\centering
		\scalemath{0.8}{
			\Diskd{
				\coordinate (A) at (157.5:0.7);
				\coordinate (B) at (-90:0.3);
				
				\coordinate (a) at (-67.5:0.7);
				\coordinate (b) at (135:0.3);
				
				\draw (1)--(b);		\draw (4)--(a);		\draw (5)--(a);		
				\draw (6)--(B);		\draw (7)--(A);		\draw (8)--(A);
				
				\draw[-stealth] (3) arc[start angle=-90, end angle=-215, radius=.414];
				
				\draw (a)--(B);	
				\draw (b)--(A);		\draw (b)--(B);
				
				\foreach \p in {A,B} {
					\fill[white] (\p) circle (2pt);
					\draw (\p) circle (2pt);
				}
				\foreach \p in {a,b}{
					\fill (\p) circle (2pt);
				}
		}}
	\end{minipage}
	\\\vspace{0.2cm}
	\begin{minipage}[t]{0.15\textwidth}
		\centering
		\scalemath{0.8}{
			\Diskd{
				\coordinate (A) at (-157.5:0.7);
				\coordinate (B) at (135:0.4);
				
				\coordinate (a) at (-112.5:0.5);
				\coordinate (b) at (22.5:0.3);
				
				\draw (1)--(b);		\draw (4)--(b);		\draw (5)--(a);		
				\draw (6)--(A);		\draw (7)--(A);		\draw (8)--(B);
				
				\draw[-stealth] (3) arc[start angle=-90, end angle=-215, radius=.414];
				
				\draw (a)--(A);		\draw (a)--(B);	
				\draw (b)--(B);
				
				\foreach \p in {A,B} {
					\fill[white] (\p) circle (2pt);
					\draw (\p) circle (2pt);
				}
				\foreach \p in {a,b}{
					\fill (\p) circle (2pt);
				}
		}}
	\end{minipage}
	\begin{minipage}[t]{0.15\textwidth}
		\centering
		\scalemath{0.8}{
			\Diskd{
				\coordinate (A) at (157.5:0.7);
				\coordinate (B) at (-135:0.4);
				
				\coordinate (a) at (112.5:0.5);
				\coordinate (b) at (-22.5:0.3);
				
				\draw (1)--(a);		\draw (2)--(b);		\draw (5)--(b);		
				\draw (6)--(B);		\draw (7)--(A);		\draw (8)--(A);
				
				\draw[-stealth] (3) arc[start angle=90, end angle=215, radius=.414];
				
				\draw (a)--(A);		\draw (a)--(B);	
				\draw (b)--(B);
				
				\foreach \p in {A,B} {
					\fill[white] (\p) circle (2pt);
					\draw (\p) circle (2pt);
				}
				\foreach \p in {a,b}{
					\fill (\p) circle (2pt);
				}
		}}
	\end{minipage}
	\begin{minipage}[t]{0.15\textwidth}
		\centering
		\scalemath{0.8}{
			\Diskd{
				\coordinate (A) at (157.5:0.7);
				\coordinate (B) at (22.5:0.5);
				
				\coordinate (a) at (67.5:0.7);
				\coordinate (b) at (0:0);
				
				\draw (1)--(a);		\draw (2)--(a);		\draw (4)--(b);		
				\draw (3)--(B);		\draw (7)--(A);		\draw (8)--(A);
				
				\draw[-stealth] (6) arc[start angle=135, end angle=10, radius=.414];
				
				\draw (a)--(B);		
				\draw (b)--(A);		\draw (b)--(B);	
				
				\foreach \p in {A,B} {
					\fill[white] (\p) circle (2pt);
					\draw (\p) circle (2pt);
				}
				\foreach \p in {a,b}{
					\fill (\p) circle (2pt);
				}
		}}
	\end{minipage}
	\begin{minipage}[t]{0.15\textwidth}
		\centering
		\scalemath{0.8}{
			\Diskd{
				\coordinate (A) at (-157.5:0.7);
				\coordinate (B) at (-22.5:0.5);
				
				\coordinate (a) at (-67.5:0.7);
				\coordinate (b) at (0:0);
				
				\draw (2)--(b);		\draw (4)--(a);		\draw (5)--(a);		
				\draw (3)--(B);		\draw (6)--(A);		\draw (7)--(A);
				
				\draw[-stealth] (8) arc[start angle=-135, end angle=-10, radius=.414];
				
				\draw (a)--(B);		
				\draw (b)--(A);		\draw (b)--(B);	
				
				\foreach \p in {A,B} {
					\fill[white] (\p) circle (2pt);
					\draw (\p) circle (2pt);
				}
				\foreach \p in {a,b}{
					\fill (\p) circle (2pt);
				}
		}}
	\end{minipage}
	\begin{minipage}[t]{0.15\textwidth}
		\centering
		\scalemath{0.8}{
			\Diskd{
				\coordinate (A) at (-157.5:0.7);
				\coordinate (B) at (67.5:0.3);
				
				\coordinate (a) at (67.5:0.7);
				\coordinate (b) at (-67.5:0.7);
				
				\draw (1)--(a);		\draw (2)--(a);		\draw (4)--(b);		\draw (5)--(b);		
				\draw (3)--(B);		\draw (6)--(A);		\draw (7)--(A);		\draw (8)--(B);	
				
				\draw (a)--(B);		
				\draw (b)--(A);	
				
				\foreach \p in {A,B} {
					\fill[white] (\p) circle (2pt);
					\draw (\p) circle (2pt);
				}
				\foreach \p in {a,b}{
					\fill (\p) circle (2pt);
				}
		}}
	\end{minipage}
	\begin{minipage}[t]{0.15\textwidth}
		\centering
		\scalemath{0.8}{
			\Diskd{
				\coordinate (A) at (157.5:0.7);
				\coordinate (B) at (-67.5:0.3);
				
				\coordinate (a) at (67.5:0.7);
				\coordinate (b) at (-67.5:0.7);
				
				\draw (1)--(a);		\draw (2)--(a);		\draw (4)--(b);		\draw (5)--(b);		
				\draw (3)--(B);		\draw (6)--(B);		\draw (7)--(A);		\draw (8)--(A);	
				
				\draw (a)--(A);		
				\draw (b)--(B);	
				
				\foreach \p in {A,B} {
					\fill[white] (\p) circle (2pt);
					\draw (\p) circle (2pt);
				}
				\foreach \p in {a,b}{
					\fill (\p) circle (2pt);
				}
		}}
	\end{minipage}
	\\\vspace{0.2cm}
	\begin{minipage}[t]{0.15\textwidth}
		\centering
		\scalemath{0.8}{
			\Diskd{
				\coordinate (A) at (180:0.5);
				\coordinate (B) at (0:0.4);
				
				\coordinate (a) at (67.5:0.7);
				\coordinate (b) at (-67.5:0.7);
				
				\draw (1)--(a);		\draw (2)--(a);		\draw (4)--(b);		\draw (5)--(b);		
				\draw (3)--(B);		\draw (6)--(A);		\draw (7)--(A);		\draw (8)--(A);	
				
				\draw (a)--(B);		
				\draw (b)--(B);	
				
				\foreach \p in {A,B} {
					\fill[white] (\p) circle (2pt);
					\draw (\p) circle (2pt);
				}
				\foreach \p in {a,b}{
					\fill (\p) circle (2pt);
				}
		}}
	\end{minipage}
	\begin{minipage}[t]{0.15\textwidth}
		\centering
		\scalemath{0.8}{
			\Diskd{
				\coordinate (A) at (-157.5:0.7);
				
				\coordinate (a) at (-67.5:0.7);
				
				\draw (4)--(a);		\draw (5)--(a);		
				\draw (6)--(A);		\draw (7)--(A);	
				\draw (a)--(A);		
				
				\draw[-stealth] (3) arc[start angle=-90, end angle=-215, radius=.414];
				\draw[-stealth] (8) arc[start angle=-135, end angle=-10, radius=.414];
				
				\fill[white] (A) circle (2pt);
				\draw (A) circle (2pt);
				\fill (a) circle (2pt);
		}}
	\end{minipage}
	\begin{minipage}[t]{0.15\textwidth}
		\centering
		\scalemath{0.8}{
			\Diskd{
				\coordinate (A) at (157.5:0.7);
				
				\coordinate (a) at (67.5:0.7);
				
				\draw (1)--(a);		\draw (2)--(a);		
				\draw (7)--(A);		\draw (8)--(A);	
				\draw (a)--(A);		
				
				\draw[-stealth] (3) arc[start angle=90, end angle=215, radius=.414];
				\draw[-stealth] (6) arc[start angle=135, end angle=10, radius=.414];
				
				\fill[white] (A) circle (2pt);
				\draw (A) circle (2pt);
				\fill (a) circle (2pt);
		}}
	\end{minipage}
	\begin{minipage}[t]{0.15\textwidth}
		\centering
		\scalemath{0.8}{
			\Diskd{
				\coordinate (A) at (67.5:0.3);
				
				\coordinate (a) at (67.5:0.7);
				
				\draw (1)--(a);		\draw (2)--(a);		
				\draw (3)--(A);		\draw (8)--(A);	
				\draw (a)--(A);		
				
				\draw[-stealth] (6) arc[start angle=135, end angle=10, radius=.414];
				\draw[-stealth] (7)--(-45:0.93);
				
				\fill[white] (A) circle (2pt);
				\draw (A) circle (2pt);
				\fill (a) circle (2pt);
		}}
	\end{minipage}
	\begin{minipage}[t]{0.15\textwidth}
		\centering
		\scalemath{0.8}{
			\Diskd{
				\coordinate (A) at (-67.5:0.3);
				
				\coordinate (a) at (-67.5:0.7);
				
				\draw (4)--(a);		\draw (5)--(a);		
				\draw (3)--(A);		\draw (6)--(A);	
				\draw (a)--(A);
				
				\draw[-stealth] (8) arc[start angle=-135, end angle=-10, radius=.414];
				\draw[-stealth] (7)--(45:0.93);
				
				\fill[white] (A) circle (2pt);
				\draw (A) circle (2pt);
				\fill (a) circle (2pt);
		}}
	\end{minipage}
	\begin{minipage}[t]{0.15\textwidth}
		\centering
		\scalemath{0.8}{
			\Diskd{
				\coordinate (A) at (180:0.5);
				
				\coordinate (a) at (0:0);
				
				\draw (1)--(a);		\draw (2)--(a);		\draw (5)--(a);
				\draw (6)--(A);		\draw (7)--(A);		\draw (8)--(A);			
				
				\draw[-stealth] (3) arc[start angle=90, end angle=215, radius=.414];
				
				\fill[white] (A) circle (2pt);
				\draw (A) circle (2pt);
				\fill (a) circle (2pt);
		}}
	\end{minipage}
	\\\vspace{0.2cm}
	\begin{minipage}[t]{0.15\textwidth}
		\centering
		\scalemath{0.8}{
			\Diskd{
				\coordinate (A) at (180:0.5);
				
				\coordinate (a) at (0:0);
				
				\draw (1)--(a);		\draw (4)--(a);		\draw (5)--(a);
				\draw (6)--(A);		\draw (7)--(A);		\draw (8)--(A);			
				
				\draw[-stealth] (3) arc[start angle=-90, end angle=-215, radius=.414];
				
				\fill[white] (A) circle (2pt);
				\draw (A) circle (2pt);
				\fill (a) circle (2pt);
		}}
	\end{minipage}
	\begin{minipage}[t]{0.15\textwidth}
		\centering
		\scalemath{0.8}{
			\Diskd{
				\coordinate (A) at (-157.5:0.7);
				
				\coordinate (a) at (0:0);
				
				\draw (2)--(a);		\draw (5)--(a);		
				\draw (6)--(A);		\draw (7)--(A);		
				\draw (a)--(A);		
				
				\draw[-stealth] (3) arc[start angle=90, end angle=215, radius=.414];
				\draw[-stealth] (8) arc[start angle=-135, end angle=-10, radius=.414];
				
				\fill[white] (A) circle (2pt);
				\draw (A) circle (2pt);
				\fill (a) circle (2pt);
		}}
	\end{minipage}
	\begin{minipage}[t]{0.15\textwidth}
		\centering
		\scalemath{0.8}{
			\Diskd{
				\coordinate (A) at (157.5:0.7);
				
				\coordinate (a) at (0:0);
				
				\draw (1)--(a);		\draw (4)--(a);		
				\draw (7)--(A);		\draw (8)--(A);		
				\draw (a)--(A);		
				
				\draw[-stealth] (3) arc[start angle=-90, end angle=-215, radius=.414];
				\draw[-stealth] (6) arc[start angle=135, end angle=10, radius=.414];
				
				\fill[white] (A) circle (2pt);
				\draw (A) circle (2pt);
				\fill (a) circle (2pt);
		}}
	\end{minipage}
	\begin{minipage}[t]{0.15\textwidth}
		\centering
		\scalemath{0.8}{
			\Diskd{
				\draw[-stealth] (8) arc[start angle=-135, end angle=-10, radius=.414];
				\draw[-stealth] (3) arc[start angle=-90, end angle=-215, radius=.414];
				\draw[-stealth] (6) arc[start angle=135, end angle=10, radius=.414];
				\draw[-stealth] (7)--(-45:0.93);
		}}
	\end{minipage}
	\begin{minipage}[t]{0.15\textwidth}
		\centering
		\scalemath{0.8}{
			\Diskd{
				\draw[-stealth] (8) arc[start angle=-135, end angle=-10, radius=.414];
				\draw[-stealth] (3) arc[start angle=90, end angle=215, radius=.414];
				\draw[-stealth] (6) arc[start angle=135, end angle=10, radius=.414];
				\draw[-stealth] (7)--(45:0.93);
		}}
	\end{minipage}
	\caption{Non-elliptic webs with 4 black and 4 white boundary vertices of type 4.}
	\label{fig:type4}
\end{figure}

\begin{figure}[H]
	\centering
	\begin{minipage}[t]{0.15\textwidth}
		\centering
		\scalemath{0.8}{
			\Diske{
				\coordinate (A) at (-90:0.6);
				\coordinate (B) at (112.5:0.5);
				\coordinate (C) at (0:{0.2*sqrt(3)});
				\coordinate (D) at (180:{0.2*sqrt(3)});
				
				\coordinate (a) at (67.5:0.7);
				\coordinate (b) at ({-90+acos(2/(sqrt(7)))}:{0.2*sqrt(7)});
				\coordinate (c) at ({-90-acos(2/(sqrt(7)))}:{0.2*sqrt(7)});
				\coordinate (d) at (90:0.2);
				
				\draw (1)--(a);		\draw (2)--(a);		\draw (4)--(b);		\draw (6)--(c);
				\draw (3)--(C);		\draw (5)--(A);		\draw (7)--(D);		\draw (8)--(B);
				
				\draw (a)--(B);		\draw (b)--(A);		\draw (b)--(C);		
				\draw (c)--(A);		\draw (c)--(D);		
				\draw (d)--(B);		\draw (d)--(C);		\draw (d)--(D);	
				
				\foreach \p in {A,B,C,D} {
					\fill[white] (\p) circle (2pt);
					\draw (\p) circle (2pt);
				}
				\foreach \p in {a,b,c,d}{
					\fill (\p) circle (2pt);
				}
		}}
	\end{minipage}
	\begin{minipage}[t]{0.15\textwidth}
		\centering
		\scalemath{0.8}{
			\Diske{
				\coordinate (A) at (157.5:0.7);
				\coordinate (B) at ({-45+acos(2/(sqrt(7)))}:{0.2*sqrt(7)});
				\coordinate (C) at ({-45-acos(2/(sqrt(7)))}:{0.2*sqrt(7)});
				\coordinate (D) at (135:0.2);
				
				\coordinate (a) at (-45:0.6);
				\coordinate (b) at (112.5:0.5);
				\coordinate (c) at (45:{0.2*sqrt(3)});
				\coordinate (d) at (-135:{0.2*sqrt(3)});
				
				\draw (1)--(b);		\draw (2)--(c);		\draw (4)--(a);		\draw (6)--(d);
				\draw (3)--(B);		\draw (5)--(C);		\draw (7)--(A);		\draw (8)--(A);
				
				\draw (a)--(B);		\draw (a)--(C);		\draw (b)--(A);		\draw (b)--(D);
				\draw (c)--(B);		\draw (c)--(D);		\draw (d)--(C);		\draw (d)--(D);	
				
				\foreach \p in {A,B,C,D} {
					\fill[white] (\p) circle (2pt);
					\draw (\p) circle (2pt);
				}
				\foreach \p in {a,b,c,d}{
					\fill (\p) circle (2pt);
				}
		}}
	\end{minipage}
	\begin{minipage}[t]{0.15\textwidth}
		\centering
		\scalemath{0.8}{
			\Diske{
				\coordinate (A) at (157.5:0.7);
				\coordinate (B) at (22.5:0.4);
				\coordinate (C) at (-97.5:0.4);
				
				\coordinate (a) at (67.5:0.7);
				\coordinate (b) at (-37.5:0.4);
				\coordinate (c) at (-157.5:0.4);
				
				\draw (1)--(a);		\draw (2)--(a);		\draw (4)--(b);		\draw (6)--(c);
				\draw (3)--(B);		\draw (5)--(C);		\draw (7)--(A);		\draw (8)--(A);
				
				\draw (a)--(B);		\draw (b)--(B);		\draw (b)--(C);
				\draw (c)--(A);		\draw (c)--(C);		
				
				\foreach \p in {A,B,C} {
					\fill[white] (\p) circle (2pt);
					\draw (\p) circle (2pt);
				}
				\foreach \p in {a,b,c}{
					\fill (\p) circle (2pt);
				}
		}}
	\end{minipage}
	\begin{minipage}[t]{0.15\textwidth}
		\centering
		\scalemath{0.8}{
			\Diske{
				\coordinate (A) at (22.5:0.4);
				\coordinate (B) at (-97.5:0.4);
				\coordinate (C) at (142.5:0.4);
				
				\coordinate (a) at (82.5:0.4);
				\coordinate (b) at (-37.5:0.4);
				\coordinate (c) at (-157.5:0.4);
				
				\draw (2)--(a);		\draw (4)--(b);		\draw (6)--(c);
				\draw (3)--(A);		\draw (5)--(B);		\draw (7)--(C);		
				
				\draw (a)--(A);		\draw (a)--(C);		\draw (b)--(A);		\draw (b)--(B);
				\draw (c)--(B);		\draw (c)--(C);		
				
				\draw[-stealth] (8) arc[start angle=-135, end angle=-10, radius=.414];
				
				\foreach \p in {A,B,C} {
					\fill[white] (\p) circle (2pt);
					\draw (\p) circle (2pt);
				}
				\foreach \p in {a,b,c}{
					\fill (\p) circle (2pt);
				}
		}}
	\end{minipage}
	\begin{minipage}[t]{0.15\textwidth}
		\centering
		\scalemath{0.8}{
			\Diske{
				\coordinate (A) at (22.5:0.5);
				\coordinate (B) at (-157.5:0.2);
				
				\coordinate (a) at (67.5:0.7);
				\coordinate (b) at (-45:0.4);
				
				\draw (1)--(a);		\draw (2)--(a);		\draw (4)--(b);
				\draw (3)--(A);		\draw (5)--(B);		\draw (8)--(B);		
				
				\draw (a)--(A);		\draw (b)--(A);		\draw (b)--(B);	
				
				\draw[-stealth] (7) arc[start angle=90, end angle=-35, radius=.414];
				
				\foreach \p in {A,B} {
					\fill[white] (\p) circle (2pt);
					\draw (\p) circle (2pt);
				}
				\foreach \p in {a,b}{
					\fill (\p) circle (2pt);
				}
		}}
	\end{minipage}
	\begin{minipage}[t]{0.15\textwidth}
		\centering
		\scalemath{0.8}{
			\Diske{
				\coordinate (A) at (157.5:0.7);
				\coordinate (B) at (-90:0.4);
				
				\coordinate (a) at (-157.5:0.5);
				\coordinate (b) at (22.5:0.2);
				
				\draw (1)--(b);		\draw (4)--(b);		\draw (6)--(a);
				\draw (5)--(B);		\draw (7)--(A);		\draw (8)--(A);		
				
				\draw (a)--(A);		\draw (a)--(B);		\draw (b)--(B);	
				
				\draw[-stealth] (3) arc[start angle=-90, end angle=-215, radius=.414];
				
				\foreach \p in {A,B} {
					\fill[white] (\p) circle (2pt);
					\draw (\p) circle (2pt);
				}
				\foreach \p in {a,b}{
					\fill (\p) circle (2pt);
				}
		}}
	\end{minipage}
	\\\vspace{0.2cm}
	\begin{minipage}[t]{0.15\textwidth}
		\centering
		\scalemath{0.8}{
			\Diske{
				\coordinate (A) at (157.5:0.7);
				\coordinate (B) at (22.5:0.5);
				
				\coordinate (a) at (67.5:0.7);
				\coordinate (b) at (0:0);
				
				\draw (1)--(a);		\draw (2)--(a);		\draw (4)--(b);
				\draw (3)--(B);		\draw (7)--(A);		\draw (8)--(A);		
				
				\draw (a)--(B);		\draw (b)--(A);		\draw (b)--(B);	
				
				\draw[-stealth] (5) arc[start angle=0, end angle=125, radius=.414];
				
				\foreach \p in {A,B} {
					\fill[white] (\p) circle (2pt);
					\draw (\p) circle (2pt);
				}
				\foreach \p in {a,b}{
					\fill (\p) circle (2pt);
				}
		}}
	\end{minipage}
	\begin{minipage}[t]{0.15\textwidth}
		\centering
		\scalemath{0.8}{
			\Diske{
				\coordinate (A) at (157.5:0.7);
				\coordinate (B) at (0:0);
				
				\coordinate (a) at (67.5:0.7);
				\coordinate (b) at (-157.5:0.5);
				
				\draw (1)--(a);		\draw (2)--(a);		\draw (6)--(b);
				\draw (5)--(B);		\draw (7)--(A);		\draw (8)--(A);		
				
				\draw (a)--(B);		\draw (b)--(A);		\draw (b)--(B);	
				
				\draw[-stealth] (3) arc[start angle=90, end angle=215, radius=.414];
				
				\foreach \p in {A,B} {
					\fill[white] (\p) circle (2pt);
					\draw (\p) circle (2pt);
				}
				\foreach \p in {a,b}{
					\fill (\p) circle (2pt);
				}
		}}
	\end{minipage}
	\begin{minipage}[t]{0.15\textwidth}
		\centering
		\scalemath{0.8}{
			\Diske{
				\coordinate (A) at (0:0);
				
				\coordinate (a) at (67.5:0.7);
				
				\draw (1)--(a);		\draw (2)--(a);	
				\draw (5)--(A);		\draw (8)--(A);		
				\draw (a)--(A);	
				
				\draw[-stealth] (3) arc[start angle=90, end angle=215, radius=.414];
				\draw[-stealth] (7) arc[start angle=90, end angle=-35, radius=.414];
				
				\fill[white] (A) circle (2pt);
				\draw (A) circle (2pt);
				\fill (a) circle (2pt);
		}}
	\end{minipage}
	\begin{minipage}[t]{0.15\textwidth}
		\centering
		\scalemath{0.8}{
			\Diske{
				\coordinate (A) at (157.5:0.7);
				
				\coordinate (a) at (0:0);
				
				\draw (1)--(a);		\draw (4)--(a);	
				\draw (7)--(A);		\draw (8)--(A);		
				\draw (a)--(A);	
				
				\draw[-stealth] (3) arc[start angle=-90, end angle=-215, radius=.414];
				\draw[-stealth] (5) arc[start angle=0, end angle=125, radius=.414];
				
				\fill[white] (A) circle (2pt);
				\draw (A) circle (2pt);
				\fill (a) circle (2pt);
		}}
	\end{minipage}
	\begin{minipage}[t]{0.15\textwidth}
		\centering
		\scalemath{0.8}{
			\Diske{
				\coordinate (A) at (-157.5:0.3);
				
				\coordinate (a) at (22.5:0.3);
				
				\draw (1)--(a);		\draw (4)--(a);	
				\draw (5)--(A);		\draw (8)--(A);		
				\draw (a)--(A);	
				
				\draw[-stealth] (3) arc[start angle=-90, end angle=-215, radius=.414];
				\draw[-stealth] (7) arc[start angle=90, end angle=-35, radius=.414];
				
				\fill[white] (A) circle (2pt);
				\draw (A) circle (2pt);
				\fill (a) circle (2pt);
		}}
	\end{minipage}
	\begin{minipage}[t]{0.15\textwidth}
		\centering
		\scalemath{0.8}{
			\Diske{
				\coordinate (A) at (157.5:0.7);
				
				\coordinate (a) at (157.5:0.3);
				
				\draw (1)--(a);		\draw (6)--(a);	
				\draw (7)--(A);		\draw (8)--(A);		
				\draw (a)--(A);	
				
				\draw[-stealth] (3) arc[start angle=-90, end angle=-215, radius=.414];
				\draw[-stealth] (5) arc[start angle=180, end angle=55, radius=.414];
				
				\fill[white] (A) circle (2pt);
				\draw (A) circle (2pt);
				\fill (a) circle (2pt);
		}}
	\end{minipage}
	\\\vspace{0.2cm}
	\begin{minipage}[t]{0.15\textwidth}
		\centering
		\scalemath{0.8}{
			\Diske{
				\coordinate (A) at (157.5:0.7);
				
				\coordinate (a) at (67.5:0.7);
				
				\draw (1)--(a);		\draw (2)--(a);	
				\draw (7)--(A);		\draw (8)--(A);		
				\draw (a)--(A);	
				
				\draw[-stealth] (3) arc[start angle=90, end angle=215, radius=.414];
				\draw[-stealth] (5) arc[start angle=0, end angle=125, radius=.414];
				
				\fill[white] (A) circle (2pt);
				\draw (A) circle (2pt);
				\fill (a) circle (2pt);
		}}
	\end{minipage}
	\begin{minipage}[t]{0.15\textwidth}
		\centering
		\scalemath{0.8}{
			\Diske{
				\coordinate (A) at (67.5:0.3);
				
				\coordinate (a) at (67.5:0.7);
				
				\draw (1)--(a);		\draw (2)--(a);	
				\draw (3)--(A);		\draw (8)--(A);		
				\draw (a)--(A);	
				
				\draw[-stealth] (5) arc[start angle=180, end angle=55, radius=.414];
				\draw[-stealth] (7) arc[start angle=90, end angle=-35, radius=.414];
				
				\fill[white] (A) circle (2pt);
				\draw (A) circle (2pt);
				\fill (a) circle (2pt);
		}}
	\end{minipage}
	\begin{minipage}[t]{0.15\textwidth}
		\centering
		\scalemath{0.8}{
			\Diske{
				\coordinate (A) at (157.5:0.7);
				
				\coordinate (a) at (157.5:0.3);
				
				\draw (1)--(a);		\draw (6)--(a);	
				\draw (7)--(A);		\draw (8)--(A);		
				\draw (a)--(A);	
				
				\draw[-stealth] (3) arc[start angle=90, end angle=215, radius=.414];
				\draw[-stealth] (5)--(45:0.93);
				
				\fill[white] (A) circle (2pt);
				\draw (A) circle (2pt);
				\fill (a) circle (2pt);
		}}
	\end{minipage}
	\begin{minipage}[t]{0.15\textwidth}
		\centering
		\scalemath{0.8}{
			\Diske{
				\coordinate (A) at (157.5:0.7);
				
				\coordinate (a) at (67.5:0.7);
				
				\draw (1)--(a);		\draw (2)--(a);	
				\draw (7)--(A);		\draw (8)--(A);		
				\draw (a)--(A);	
				
				\draw[-stealth] (5) arc[start angle=180, end angle=55, radius=.414];
				\draw[-stealth] (3)--(-135:0.93);
				
				\fill[white] (A) circle (2pt);
				\draw (A) circle (2pt);
				\fill (a) circle (2pt);
		}}
	\end{minipage}
	\begin{minipage}[t]{0.15\textwidth}
		\centering
		\scalemath{0.8}{
			\Diske{
				\coordinate (A) at (67.5:0.3);
				
				\coordinate (a) at (67.5:0.7);
				
				\draw (1)--(a);		\draw (2)--(a);	
				\draw (3)--(A);		\draw (8)--(A);		
				\draw (a)--(A);	
				
				\draw[-stealth] (5) arc[start angle=0, end angle=125, radius=.414];
				\draw[-stealth] (7)--(-45:0.93);
				
				\fill[white] (A) circle (2pt);
				\draw (A) circle (2pt);
				\fill (a) circle (2pt);
		}}
	\end{minipage}
	\begin{minipage}[t]{0.15\textwidth}
		\centering
		\scalemath{0.8}{
			\Diske{
				\draw[-stealth] (8) arc[start angle=-135, end angle=-10, radius=.414];
				\draw[-stealth] (3) arc[start angle=-90, end angle=-215, radius=.414];
				\draw[-stealth] (5) arc[start angle=180, end angle=55, radius=.414];
				\draw[-stealth] (7) arc[start angle=90, end angle=-35, radius=.414];
		}}
	\end{minipage}
	\\\vspace{0.2cm}
	\begin{minipage}[t]{0.15\textwidth}
		\centering
		\scalemath{0.8}{
			\Diske{
				\draw[-stealth] (8) arc[start angle=-135, end angle=-10, radius=.414];
				\draw[-stealth] (3) arc[start angle=-90, end angle=-215, radius=.414];
				\draw[-stealth] (5) arc[start angle=0, end angle=125, radius=.414];
				\draw[-stealth] (7)--(-45:0.93);
		}}
	\end{minipage}
	\begin{minipage}[t]{0.15\textwidth}
		\centering
		\scalemath{0.8}{
			\Diske{
				\draw[-stealth] (8) arc[start angle=-135, end angle=-10, radius=.414];
				\draw[-stealth] (7) arc[start angle=90, end angle=-35, radius=.414];
				\draw[-stealth] (3) arc[start angle=90, end angle=215, radius=.414];
				\draw[-stealth] (5)--(45:0.93);
		}}
	\end{minipage}
	\begin{minipage}[t]{0.15\textwidth}
		\centering
		\scalemath{0.8}{
			\Diske{
				\draw[-stealth] (8) arc[start angle=-135, end angle=-10, radius=.414];
				\draw[-stealth] (5) arc[start angle=0, end angle=125, radius=.414];
				\draw[-stealth] (3) arc[start angle=90, end angle=215, radius=.414];
				\draw[-stealth] (7)--(45:0.93);
		}}
	\end{minipage}
	\begin{minipage}[t]{0.15\textwidth}
		\centering
		\scalemath{0.8}{
			\Diske{
				\draw[-stealth] (8) arc[start angle=-135, end angle=-10, radius=.414];
				\draw[-stealth] (5) arc[start angle=180, end angle=55, radius=.414];
				\draw[-stealth] (3)--(-135:0.93);
				\draw[-stealth] (7)--(45:0.93);
		}}
	\end{minipage}
	\caption{Non-elliptic webs with 4 black and 4 white boundary vertices of type 5.}
	\label{fig:type5}
\end{figure}

\begin{figure}[H]
	\centering
	\begin{minipage}[t]{0.15\textwidth}
		\centering
		\scalemath{0.8}{
			\Diskf{
				\coordinate (A) at ({67.5-acos(1/sqrt(13))}:{0.2*sqrt(13)});
				\coordinate (B) at ({67.5+acos(1/sqrt(13))}:{0.2*sqrt(13)});
				\coordinate (C) at ({-112.5+acos(2/sqrt(7))}:{0.2*sqrt(7)});
				\coordinate (D) at ({-112.5-acos(2/sqrt(7))}:{0.2*sqrt(7)});
				\coordinate (E) at (67.5:0.2);
				
				\coordinate (a) at ({-112.5+acos(1/sqrt(13))}:{0.2*sqrt(13)});
				\coordinate (b) at ({-112.5-acos(1/sqrt(13))}:{0.2*sqrt(13)});
				\coordinate (c) at ({67.5-acos(2/sqrt(7))}:{0.2*sqrt(7)});
				\coordinate (d) at ({67.5+acos(2/sqrt(7))}:{0.2*sqrt(7)});
				\coordinate (e) at (-112.5:0.2);
				
				\draw (1)--(d);		\draw (2)--(c);		\draw (4)--(a);		\draw (7)--(b);
				\draw (3)--(A);		\draw (5)--(C);		\draw (6)--(D);		\draw (8)--(B);
				
				\draw (a)--(A);		\draw (a)--(C);		\draw (b)--(B);		\draw (b)--(D);
				\draw (c)--(A);		\draw (c)--(E);		\draw (d)--(B);		\draw (d)--(E);	
				\draw (e)--(C);		\draw (e)--(D);		\draw (e)--(E);	
				
				\foreach \p in {A,B,C,D,E} {
					\fill[white] (\p) circle (2pt);
					\draw (\p) circle (2pt);
				}
				\foreach \p in {a,b,c,d,e}{
					\fill (\p) circle (2pt);
				}
		}}
	\end{minipage}
	\begin{minipage}[t]{0.15\textwidth}
		\centering
		\scalemath{0.8}{
			\Diskf{
				\coordinate (A) at (-112.5:0.7);
				\coordinate (B) at (22.5:0.5);
				\coordinate (C) at (112.5:0.5);
				
				\coordinate (a) at (67.5:0.7);
				\coordinate (b) at (-67.5:0.5);
				\coordinate (c) at (-157.5:0.5);
				
				\draw (1)--(a);		\draw (2)--(a);		\draw (4)--(b);		\draw (7)--(c);
				\draw (3)--(B);		\draw (5)--(A);		\draw (6)--(A);		\draw (8)--(C);
				
				\draw (a)--(B);			
				\draw (b)--(B);		\draw (b)--(C);
				\draw (c)--(A);		\draw (c)--(C);		
				
				\foreach \p in {A,B,C} {
					\fill[white] (\p) circle (2pt);
					\draw (\p) circle (2pt);
				}
				\foreach \p in {a,b,c}{
					\fill (\p) circle (2pt);
				}
		}}
	\end{minipage}
	\begin{minipage}[t]{0.15\textwidth}
		\centering
		\scalemath{0.8}{
			\Diskf{
				\coordinate (A) at (-112.5:0.7);
				\coordinate (B) at (22.5:0.5);
				\coordinate (C) at (112.5:0.5);
				
				\coordinate (a) at (67.5:0.7);
				\coordinate (b) at (-67.5:0.5);
				\coordinate (c) at (-157.5:0.5);
				
				\draw (1)--(a);		\draw (2)--(a);		\draw (4)--(b);		\draw (7)--(c);
				\draw (3)--(B);		\draw (5)--(A);		\draw (6)--(A);		\draw (8)--(C);
				
				\draw (a)--(C);			
				\draw (b)--(A);		\draw (b)--(B);
				\draw (c)--(B);		\draw (c)--(C);		
				
				\foreach \p in {A,B,C} {
					\fill[white] (\p) circle (2pt);
					\draw (\p) circle (2pt);
				}
				\foreach \p in {a,b,c}{
					\fill (\p) circle (2pt);
				}
		}}
	\end{minipage}
	\begin{minipage}[t]{0.15\textwidth}
		\centering
		\scalemath{0.8}{
			\Diskf{
				\coordinate (A) at (-112.5:0.7);
				\coordinate (B) at (67.5:0.3);
				
				\coordinate (a) at (67.5:0.7);
				\coordinate (b) at (-112.5:0.3);
				
				\draw (1)--(a);		\draw (2)--(a);		\draw (4)--(b);		\draw (7)--(b);
				\draw (3)--(B);		\draw (5)--(A);		\draw (6)--(A);		\draw (8)--(B);	
				
				\draw (a)--(B);		\draw (b)--(A);	
				
				\foreach \p in {A,B} {
					\fill[white] (\p) circle (2pt);
					\draw (\p) circle (2pt);
				}
				\foreach \p in {a,b}{
					\fill (\p) circle (2pt);
				}
		}}
	\end{minipage}
	\begin{minipage}[t]{0.15\textwidth}
		\centering
		\scalemath{0.8}{
			\Diskf{
				\coordinate (A) at (-112.5:0.7);
				\coordinate (B) at (125:0.4);
				
				\coordinate (a) at (67.5:0.7);
				\coordinate (b) at (190:0.4);
				
				\draw (1)--(a);		\draw (2)--(a);		\draw (7)--(b);
				\draw (5)--(A);		\draw (6)--(A);		\draw (8)--(B);		
				
				\draw (a)--(B);		\draw (b)--(A);		\draw (b)--(B);	
				
				\draw[-stealth] (3) arc[start angle=90, end angle=215, radius=.414];
				
				\foreach \p in {A,B} {
					\fill[white] (\p) circle (2pt);
					\draw (\p) circle (2pt);
				}
				\foreach \p in {a,b}{
					\fill (\p) circle (2pt);
				}
		}}
	\end{minipage}
	\begin{minipage}[t]{0.15\textwidth}
		\centering
		\scalemath{0.8}{
			\Diskf{
				\coordinate (A) at (-112.5:0.7);
				\coordinate (B) at (10:0.4);
				
				\coordinate (a) at (67.5:0.7);
				\coordinate (b) at (-55:0.4);
				
				\draw (1)--(a);		\draw (2)--(a);		\draw (4)--(b);
				\draw (3)--(B);		\draw (5)--(A);		\draw (6)--(A);		
				
				\draw (a)--(B);		\draw (b)--(A);		\draw (b)--(B);	
				
				\draw[-stealth] (8) arc[start angle=45, end angle=-80, radius=.414];
				
				\foreach \p in {A,B} {
					\fill[white] (\p) circle (2pt);
					\draw (\p) circle (2pt);
				}
				\foreach \p in {a,b}{
					\fill (\p) circle (2pt);
				}
		}}
	\end{minipage}
	\\\vspace{0.2cm}
	\begin{minipage}[t]{0.15\textwidth}
		\centering
		\scalemath{0.8}{
			\Diskf{
				\coordinate (A) at (22.5:0.5);
				\coordinate (B) at (-157.5:0.2);
				
				\coordinate (a) at (67.5:0.7);
				\coordinate (b) at (-45:0.4);
				
				\draw (1)--(a);		\draw (2)--(a);		\draw (4)--(b);
				\draw (3)--(A);		\draw (5)--(B);		\draw (8)--(B);		
				
				\draw (a)--(A);		\draw (b)--(A);		\draw (b)--(B);	
				
				\draw[-stealth] (6) arc[start angle=-45, end angle=80, radius=.414];
				
				\foreach \p in {A,B} {
					\fill[white] (\p) circle (2pt);
					\draw (\p) circle (2pt);
				}
				\foreach \p in {a,b}{
					\fill (\p) circle (2pt);
				}
		}}
	\end{minipage}
	\begin{minipage}[t]{0.15\textwidth}
		\centering
		\scalemath{0.8}{
			\Diskf{
				\coordinate (A) at (112.5:0.5);
				\coordinate (B) at (-67.5:0.2);
				
				\coordinate (a) at (67.5:0.7);
				\coordinate (b) at (180:0.4);
				
				\draw (1)--(a);		\draw (2)--(a);		\draw (7)--(b);
				\draw (3)--(B);		\draw (6)--(B);		\draw (8)--(A);		
				
				\draw (a)--(A);		\draw (b)--(A);		\draw (b)--(B);	
				
				\draw[-stealth] (5) arc[start angle=180, end angle=55, radius=.414];
				
				\foreach \p in {A,B} {
					\fill[white] (\p) circle (2pt);
					\draw (\p) circle (2pt);
				}
				\foreach \p in {a,b}{
					\fill (\p) circle (2pt);
				}
		}}
	\end{minipage}
	\begin{minipage}[t]{0.15\textwidth}
		\centering
		\scalemath{0.8}{
			\Diskf{
				\coordinate (A) at (-112.5:0.7);
				\coordinate (B) at (135:0.4);
				
				\coordinate (a) at (-157.5:0.5);
				\coordinate (b) at (22.5:0.2);
				
				\draw (1)--(b);		\draw (4)--(b);		\draw (7)--(a);
				\draw (5)--(A);		\draw (6)--(A);		\draw (8)--(B);		
				
				\draw (a)--(A);		\draw (a)--(B);		\draw (b)--(B);	
				
				\draw[-stealth] (3) arc[start angle=-90, end angle=-215, radius=.414];
				
				\foreach \p in {A,B} {
					\fill[white] (\p) circle (2pt);
					\draw (\p) circle (2pt);
				}
				\foreach \p in {a,b}{
					\fill (\p) circle (2pt);
				}
		}}
	\end{minipage}
	\begin{minipage}[t]{0.15\textwidth}
		\centering
		\scalemath{0.8}{
			\Diskf{
				\coordinate (A) at (-112.5:0.7);
				\coordinate (B) at (0:0.4);
				
				\coordinate (a) at (-67.5:0.5);
				\coordinate (b) at (112.5:0.2);
				
				\draw (2)--(b);		\draw (4)--(a);		\draw (7)--(b);
				\draw (3)--(B);		\draw (5)--(A);		\draw (6)--(A);		
				
				\draw (a)--(A);		\draw (a)--(B);		\draw (b)--(B);	
				
				\draw[-stealth] (8) arc[start angle=-135, end angle=-10, radius=.414];
				
				\foreach \p in {A,B} {
					\fill[white] (\p) circle (2pt);
					\draw (\p) circle (2pt);
				}
				\foreach \p in {a,b}{
					\fill (\p) circle (2pt);
				}
		}}
	\end{minipage}
	\begin{minipage}[t]{0.15\textwidth}
		\centering
		\scalemath{0.8}{
			\Diskf{
				\coordinate (A) at (-112.5:0.7);
				
				\coordinate (a) at (-112.5:0.3);
				
				\draw (4)--(a);		\draw (7)--(a);	
				\draw (5)--(A);		\draw (6)--(A);		
				\draw (a)--(A);	
				
				\draw[-stealth] (3) arc[start angle=-90, end angle=-215, radius=.414];
				\draw[-stealth] (8) arc[start angle=-135, end angle=-10, radius=.414];
				
				\fill[white] (A) circle (2pt);
				\draw (A) circle (2pt);
				\fill (a) circle (2pt);
		}}
	\end{minipage}
	\begin{minipage}[t]{0.15\textwidth}
		\centering
		\scalemath{0.8}{
			\Diskf{
				\coordinate (A) at (67.5:0.3);
				
				\coordinate (a) at (67.5:0.7);
				
				\draw (1)--(a);		\draw (2)--(a);	
				\draw (3)--(A);		\draw (8)--(A);		
				\draw (a)--(A);	
				
				\draw[-stealth] (5) arc[start angle=180, end angle=55, radius=.414];
				\draw[-stealth] (6) arc[start angle=-45, end angle=80, radius=.414];
				
				\fill[white] (A) circle (2pt);
				\draw (A) circle (2pt);
				\fill (a) circle (2pt);
		}}
	\end{minipage}
	\\\vspace{0.2cm}
	\begin{minipage}[t]{0.15\textwidth}
		\centering
		\scalemath{0.8}{
			\Diskf{
				\coordinate (A) at (0:0);
				
				\coordinate (a) at (67.5:0.7);
				
				\draw (1)--(a);		\draw (2)--(a);	
				\draw (5)--(A);		\draw (8)--(A);		
				\draw (a)--(A);	
				
				\draw[-stealth] (3) arc[start angle=90, end angle=215, radius=.414];
				\draw[-stealth] (6) arc[start angle=-45, end angle=80, radius=.414];
				
				\fill[white] (A) circle (2pt);
				\draw (A) circle (2pt);
				\fill (a) circle (2pt);
		}}
	\end{minipage}
	\begin{minipage}[t]{0.15\textwidth}
		\centering
		\scalemath{0.8}{
			\Diskf{
				\coordinate (A) at (0:0);
				
				\coordinate (a) at (67.5:0.7);
				
				\draw (1)--(a);		\draw (2)--(a);	
				\draw (3)--(A);		\draw (6)--(A);		
				\draw (a)--(A);	
				
				\draw[-stealth] (5) arc[start angle=180, end angle=55, radius=.414];
				\draw[-stealth] (8) arc[start angle=45, end angle=-80, radius=.414];
				
				\fill[white] (A) circle (2pt);
				\draw (A) circle (2pt);
				\fill (a) circle (2pt);
		}}
	\end{minipage}
	\begin{minipage}[t]{0.15\textwidth}
		\centering
		\scalemath{0.8}{
			\Diskf{
				\coordinate (A) at (-112.5:0.7);
				
				\coordinate (a) at (0:0);
				
				\draw (1)--(a);		\draw (4)--(a);	
				\draw (5)--(A);		\draw (6)--(A);		
				\draw (a)--(A);	
				
				\draw[-stealth] (3) arc[start angle=-90, end angle=-215, radius=.414];
				\draw[-stealth] (8) arc[start angle=45, end angle=-80, radius=.414];
				
				\fill[white] (A) circle (2pt);
				\draw (A) circle (2pt);
				\fill (a) circle (2pt);
		}}
	\end{minipage}
	\begin{minipage}[t]{0.15\textwidth}
		\centering
		\scalemath{0.8}{
			\Diskf{
				\coordinate (A) at (-112.5:0.7);
				
				\coordinate (a) at (0:0);
				
				\draw (2)--(a);		\draw (7)--(a);	
				\draw (5)--(A);		\draw (6)--(A);		
				\draw (a)--(A);	
				
				\draw[-stealth] (3) arc[start angle=90, end angle=215, radius=.414];
				\draw[-stealth] (8) arc[start angle=-135, end angle=-10, radius=.414];
				
				\fill[white] (A) circle (2pt);
				\draw (A) circle (2pt);
				\fill (a) circle (2pt);
		}}
	\end{minipage}
	\begin{minipage}[t]{0.15\textwidth}
		\centering
		\scalemath{0.8}{
			\Diskf{
				\coordinate (A) at (-67.5:0.3);
				
				\coordinate (a) at (112.5:0.3);
				
				\draw (2)--(a);		\draw (7)--(a);	
				\draw (3)--(A);		\draw (6)--(A);		
				\draw (a)--(A);	
				
				\draw[-stealth] (5) arc[start angle=180, end angle=55, radius=.414];
				\draw[-stealth] (8) arc[start angle=-135, end angle=-10, radius=.414];
				
				\fill[white] (A) circle (2pt);
				\draw (A) circle (2pt);
				\fill (a) circle (2pt);
		}}
	\end{minipage}
	\begin{minipage}[t]{0.15\textwidth}
		\centering
		\scalemath{0.8}{
			\Diskf{
				\coordinate (A) at (-112.5:0.7);
				
				\coordinate (a) at (67.5:0.7);
				
				\draw (1)--(a);		\draw (2)--(a);	
				\draw (5)--(A);		\draw (6)--(A);		
				\draw (a)--(A);	
				
				\draw[-stealth] (3) arc[start angle=90, end angle=215, radius=.414];
				\draw[-stealth] (8) arc[start angle=45, end angle=-80, radius=.414];
				
				\fill[white] (A) circle (2pt);
				\draw (A) circle (2pt);
				\fill (a) circle (2pt);
		}}
	\end{minipage}
	\\\vspace{0.2cm}
	\begin{minipage}[t]{0.15\textwidth}
		\centering
		\scalemath{0.8}{
			\Diskf{
				\draw[-stealth] (8) arc[start angle=-135, end angle=-10, radius=.414];
				\draw[-stealth] (3) arc[start angle=-90, end angle=-215, radius=.414];
				\draw[-stealth] (5) arc[start angle=180, end angle=55, radius=.414];
				\draw[-stealth] (6) arc[start angle=-45, end angle=80, radius=.414];
		}}
	\end{minipage}
	\begin{minipage}[t]{0.15\textwidth}
		\centering
		\scalemath{0.8}{
			\Diskf{
				\draw[-stealth] (8) arc[start angle=-135, end angle=-10, radius=.414];
				\draw[-stealth] (3) arc[start angle=90, end angle=215, radius=.414];
				\draw[-stealth] (6) arc[start angle=-45, end angle=80, radius=.414];
				\draw[-stealth] (5)--(45:0.93);
		}}
	\end{minipage}
	\begin{minipage}[t]{0.15\textwidth}
		\centering
		\scalemath{0.8}{
			\Diskf{
				\draw[-stealth] (3) arc[start angle=-90, end angle=-215, radius=.414];
				\draw[-stealth] (5) arc[start angle=180, end angle=55, radius=.414];
				\draw[-stealth] (8) arc[start angle=45, end angle=-80, radius=.414];
				\draw[-stealth] (6)--(90:0.93);
		}}
	\end{minipage}
	\begin{minipage}[t]{0.15\textwidth}
		\centering
		\scalemath{0.8}{
			\Diskf{
				\draw[-stealth] (3) arc[start angle=90, end angle=215, radius=.414];
				\draw[-stealth] (8) arc[start angle=45, end angle=-80, radius=.414];
				\draw[-stealth] (6)--(90:0.93);
				\draw[-stealth] (5)--(45:0.93);
		}}
	\end{minipage}
	\caption{Non-elliptic webs with 4 black and 4 white boundary vertices of type 6.}
	\label{fig:type6}
\end{figure}

\begin{figure}[H]
	\centering
	\begin{minipage}[t]{0.15\textwidth}
		\centering
		\scalemath{0.8}{
			\Diskg{
				\coordinate (A) at (-22.5:0.7);
				\coordinate (B) at (157.5:0.7);
				\coordinate (C) at (67.5:0.1);
				
				\coordinate (a) at (67.5:0.7);
				\coordinate (b) at (-67.5:0.5);
				\coordinate (c) at (-157.5:0.5);
				
				\draw (1)--(a);		\draw (2)--(a);		\draw (5)--(b);		\draw (6)--(c);
				\draw (3)--(A);		\draw (4)--(A);		\draw (7)--(B);		\draw (8)--(B);
				
				\draw (a)--(C);			
				\draw (b)--(A);		\draw (b)--(C);
				\draw (c)--(B);		\draw (c)--(C);		
				
				\foreach \p in {A,B,C} {
					\fill[white] (\p) circle (2pt);
					\draw (\p) circle (2pt);
				}
				\foreach \p in {a,b,c}{
					\fill (\p) circle (2pt);
				}
		}}
	\end{minipage}
	\begin{minipage}[t]{0.15\textwidth}
		\centering
		\scalemath{0.8}{
			\Diskg{
				\coordinate (A) at (-22.5:0.7);
				\coordinate (B) at (157.5:0.7);
				\coordinate (C) at (-112.5:0.1);
				
				\coordinate (a) at (-112.5:0.7);
				\coordinate (b) at (22.5:0.5);
				\coordinate (c) at (112.5:0.5);
				
				\draw (1)--(c);		\draw (2)--(b);		\draw (5)--(a);		\draw (6)--(a);
				\draw (3)--(A);		\draw (4)--(A);		\draw (7)--(B);		\draw (8)--(B);
				
				\draw (a)--(C);			
				\draw (b)--(A);		\draw (b)--(C);
				\draw (c)--(B);		\draw (c)--(C);		
				
				\foreach \p in {A,B,C} {
					\fill[white] (\p) circle (2pt);
					\draw (\p) circle (2pt);
				}
				\foreach \p in {a,b,c}{
					\fill (\p) circle (2pt);
				}
		}}
	\end{minipage}
	\begin{minipage}[t]{0.15\textwidth}
		\centering
		\scalemath{0.8}{
			\Diskg{
				\coordinate (A) at (-22.5:0.7);
				\coordinate (B) at (-157.5:0.5);
				\coordinate (C) at (112.5:0.5);
				
				\coordinate (a) at (67.5:0.7);
				\coordinate (b) at (-112.5:0.7);
				\coordinate (c) at (-22.5:0.1);
				
				\draw (1)--(a);		\draw (2)--(a);		\draw (5)--(b);		\draw (6)--(b);
				\draw (3)--(A);		\draw (4)--(A);		\draw (7)--(B);		\draw (8)--(C);
				
				\draw (a)--(C);			
				\draw (b)--(B);		
				\draw (c)--(A);		\draw (c)--(B);		\draw (c)--(C);		
				
				\foreach \p in {A,B,C} {
					\fill[white] (\p) circle (2pt);
					\draw (\p) circle (2pt);
				}
				\foreach \p in {a,b,c}{
					\fill (\p) circle (2pt);
				}
		}}
	\end{minipage}
	\begin{minipage}[t]{0.15\textwidth}
		\centering
		\scalemath{0.8}{
			\Diskg{
				\coordinate (A) at (157.5:0.7);
				\coordinate (B) at (22.5:0.5);
				\coordinate (C) at (-67.5:0.5);
				
				\coordinate (a) at (67.5:0.7);
				\coordinate (b) at (-112.5:0.7);
				\coordinate (c) at (157.5:0.1);
				
				\draw (1)--(a);		\draw (2)--(a);		\draw (5)--(b);		\draw (6)--(b);
				\draw (3)--(B);		\draw (4)--(C);		\draw (7)--(A);		\draw (8)--(A);
				
				\draw (a)--(B);			
				\draw (b)--(C);		
				\draw (c)--(A);		\draw (c)--(B);		\draw (c)--(C);		
				
				\foreach \p in {A,B,C} {
					\fill[white] (\p) circle (2pt);
					\draw (\p) circle (2pt);
				}
				\foreach \p in {a,b,c}{
					\fill (\p) circle (2pt);
				}
		}}
	\end{minipage}
	\begin{minipage}[t]{0.15\textwidth}
		\centering
		\scalemath{0.8}{
			\Diskg{
				\coordinate (A) at (-22.5:0.7);
				\coordinate (B) at (-145:0.3);
				
				\coordinate (a) at (-112.5:0.7);
				\coordinate (b) at (10:0.3);
				
				\draw (2)--(b);		\draw (5)--(a);		\draw (6)--(a);
				\draw (3)--(A);		\draw (4)--(A);		\draw (7)--(B);	
				
				\draw (a)--(B);		\draw (b)--(A);		\draw (b)--(B);	
				
				\draw[-stealth] (8) arc[start angle=-135, end angle=-10, radius=.414];
				
				\foreach \p in {A,B} {
					\fill[white] (\p) circle (2pt);
					\draw (\p) circle (2pt);
				}
				\foreach \p in {a,b}{
					\fill (\p) circle (2pt);
				}
		}}
	\end{minipage}
	\begin{minipage}[t]{0.15\textwidth}
		\centering
		\scalemath{0.8}{
			\Diskg{
				\coordinate (A) at (157.5:0.7);
				\coordinate (B) at (-80:0.3);
				
				\coordinate (a) at (-112.5:0.7);
				\coordinate (b) at (125:0.3);
				
				\draw (1)--(b);		\draw (5)--(a);		\draw (6)--(a);
				\draw (4)--(B);		\draw (7)--(A);		\draw (8)--(A);	
				
				\draw (a)--(B);		\draw (b)--(A);		\draw (b)--(B);	
				
				\draw[-stealth] (3) arc[start angle=-90, end angle=-215, radius=.414];
				
				\foreach \p in {A,B} {
					\fill[white] (\p) circle (2pt);
					\draw (\p) circle (2pt);
				}
				\foreach \p in {a,b}{
					\fill (\p) circle (2pt);
				}
		}}
	\end{minipage}
	\\\vspace{0.2cm}
	\begin{minipage}[t]{0.15\textwidth}
		\centering
		\scalemath{0.8}{
			\Diskg{
				\coordinate (A) at (157.5:0.7);
				\coordinate (B) at (35:0.3);
				
				\coordinate (a) at (67.5:0.7);
				\coordinate (b) at (-170:0.3);
				
				\draw (1)--(a);		\draw (2)--(a);		\draw (6)--(b);
				\draw (3)--(B);		\draw (7)--(A);		\draw (8)--(A);	
				
				\draw (a)--(B);		\draw (b)--(A);		\draw (b)--(B);	
				
				\draw[-stealth] (4) arc[start angle=45, end angle=170, radius=.414];
				
				\foreach \p in {A,B} {
					\fill[white] (\p) circle (2pt);
					\draw (\p) circle (2pt);
				}
				\foreach \p in {a,b}{
					\fill (\p) circle (2pt);
				}
		}}
	\end{minipage}
	\begin{minipage}[t]{0.15\textwidth}
		\centering
		\scalemath{0.8}{
			\Diskg{
				\coordinate (A) at (-22.5:0.7);
				\coordinate (B) at (100:0.3);
				
				\coordinate (a) at (67.5:0.7);
				\coordinate (b) at (-55:0.3);
				
				\draw (1)--(a);		\draw (2)--(a);		\draw (5)--(b);
				\draw (3)--(A);		\draw (4)--(A);		\draw (8)--(B);	
				
				\draw (a)--(B);		\draw (b)--(A);		\draw (b)--(B);	
				
				\draw[-stealth] (7) arc[start angle=90, end angle=-35, radius=.414];
				
				\foreach \p in {A,B} {
					\fill[white] (\p) circle (2pt);
					\draw (\p) circle (2pt);
				}
				\foreach \p in {a,b}{
					\fill (\p) circle (2pt);
				}
		}}
	\end{minipage}
	\begin{minipage}[t]{0.15\textwidth}
		\centering
		\scalemath{0.8}{
			\Diskg{
				\coordinate (A) at (-22.5:0.7);
				\coordinate (B) at (157.5:0.7);
				
				\coordinate (a) at (67.5:0.7);
				\coordinate (b) at (-112.5:0.7);
				
				\draw (1)--(a);		\draw (2)--(a);		\draw (5)--(b);		\draw (6)--(b);
				\draw (3)--(A);		\draw (4)--(A);		\draw (7)--(B);		\draw (8)--(B);	
				
				\draw (a)--(B);		\draw (b)--(A);	
				
				\foreach \p in {A,B} {
					\fill[white] (\p) circle (2pt);
					\draw (\p) circle (2pt);
				}
				\foreach \p in {a,b}{
					\fill (\p) circle (2pt);
				}
		}}
	\end{minipage}
	\begin{minipage}[t]{0.15\textwidth}
		\centering
		\scalemath{0.8}{
			\Diskg{
				\coordinate (A) at (-22.5:0.7);
				\coordinate (B) at (157.5:0.7);
				
				\coordinate (a) at (67.5:0.7);
				\coordinate (b) at (-112.5:0.7);
				
				\draw (1)--(a);		\draw (2)--(a);		\draw (5)--(b);		\draw (6)--(b);
				\draw (3)--(A);		\draw (4)--(A);		\draw (7)--(B);		\draw (8)--(B);	
				
				\draw (a)--(A);		\draw (b)--(B);	
				
				\foreach \p in {A,B} {
					\fill[white] (\p) circle (2pt);
					\draw (\p) circle (2pt);
				}
				\foreach \p in {a,b}{
					\fill (\p) circle (2pt);
				}
		}}
	\end{minipage}
	\begin{minipage}[t]{0.15\textwidth}
		\centering
		\scalemath{0.8}{
			\Diskg{
				\coordinate (A) at (-22.5:0.7);
				\coordinate (B) at (157.5:0.7);
				
				\coordinate (a) at (-22.5:0.3);
				\coordinate (b) at (157.5:0.3);
				
				\draw (1)--(b);		\draw (2)--(a);		\draw (5)--(a);		\draw (6)--(b);
				\draw (3)--(A);		\draw (4)--(A);		\draw (7)--(B);		\draw (8)--(B);	
				
				\draw (a)--(A);		\draw (b)--(B);	
				
				\foreach \p in {A,B} {
					\fill[white] (\p) circle (2pt);
					\draw (\p) circle (2pt);
				}
				\foreach \p in {a,b}{
					\fill (\p) circle (2pt);
				}
		}}
	\end{minipage}
	\begin{minipage}[t]{0.15\textwidth}
		\centering
		\scalemath{0.8}{
			\Diskg{
				\coordinate (A) at (67.5:0.3);
				\coordinate (B) at (-112.5:0.3);
				
				\coordinate (a) at (67.5:0.7);
				\coordinate (b) at (-112.5:0.7);
				
				\draw (1)--(a);		\draw (2)--(a);		\draw (5)--(b);		\draw (6)--(b);
				\draw (3)--(A);		\draw (4)--(B);		\draw (7)--(B);		\draw (8)--(A);	
				
				\draw (a)--(A);		\draw (b)--(B);	
				
				\foreach \p in {A,B} {
					\fill[white] (\p) circle (2pt);
					\draw (\p) circle (2pt);
				}
				\foreach \p in {a,b}{
					\fill (\p) circle (2pt);
				}
		}}
	\end{minipage}
	\\\vspace{0.2cm}
	\begin{minipage}[t]{0.15\textwidth}
		\centering
		\scalemath{0.8}{
			\Diskg{
				\coordinate (A) at (-22.5:0.7);
				
				\coordinate (a) at (-22.5:0.3);
				
				\draw (2)--(a);		\draw (5)--(a);	
				\draw (3)--(A);		\draw (4)--(A);		
				\draw (a)--(A);	
				
				\draw[-stealth] (7) arc[start angle=90, end angle=-35, radius=.414];
				\draw[-stealth] (8) arc[start angle=-135, end angle=-10, radius=.414];
				
				\fill[white] (A) circle (2pt);
				\draw (A) circle (2pt);
				\fill (a) circle (2pt);
		}}
	\end{minipage}
	\begin{minipage}[t]{0.15\textwidth}
		\centering
		\scalemath{0.8}{
			\Diskg{
				\coordinate (A) at (157.5:0.7);
				
				\coordinate (a) at (157.5:0.3);
				
				\draw (1)--(a);		\draw (6)--(a);	
				\draw (7)--(A);		\draw (8)--(A);		
				\draw (a)--(A);	
				
				\draw[-stealth] (3) arc[start angle=-90, end angle=-215, radius=.414];
				\draw[-stealth] (4) arc[start angle=45, end angle=170, radius=.414];
				
				\fill[white] (A) circle (2pt);
				\draw (A) circle (2pt);
				\fill (a) circle (2pt);
		}}
	\end{minipage}
	\begin{minipage}[t]{0.15\textwidth}
		\centering
		\scalemath{0.8}{
			\Diskg{
				\coordinate (A) at (67.5:0.3);
				
				\coordinate (a) at (67.5:0.7);
				
				\draw (1)--(a);		\draw (2)--(a);	
				\draw (3)--(A);		\draw (8)--(A);		
				\draw (a)--(A);	
				
				\draw[-stealth] (4) arc[start angle=45, end angle=170, radius=.414];
				\draw[-stealth] (7) arc[start angle=90, end angle=-35, radius=.414];
				
				\fill[white] (A) circle (2pt);
				\draw (A) circle (2pt);
				\fill (a) circle (2pt);
		}}
	\end{minipage}
	\begin{minipage}[t]{0.15\textwidth}
		\centering
		\scalemath{0.8}{
			\Diskg{
				\coordinate (A) at (-112.5:0.3);
				
				\coordinate (a) at (-112.5:0.7);
				
				\draw (5)--(a);		\draw (6)--(a);	
				\draw (4)--(A);		\draw (7)--(A);		
				\draw (a)--(A);	
				
				\draw[-stealth] (3) arc[start angle=-90, end angle=-215, radius=.414];
				\draw[-stealth] (8) arc[start angle=-135, end angle=-10, radius=.414];
				
				\fill[white] (A) circle (2pt);
				\draw (A) circle (2pt);
				\fill (a) circle (2pt);
		}}
	\end{minipage}
	\begin{minipage}[t]{0.15\textwidth}
		\centering
		\scalemath{0.8}{
			\Diskg{
				\coordinate (A) at (-22.5:0.7);
				
				\coordinate (a) at (-112.5:0.7);
				
				\draw (5)--(a);		\draw (6)--(a);	
				\draw (3)--(A);		\draw (4)--(A);		
				\draw (a)--(A);	
				
				\draw[-stealth] (8) arc[start angle=-135, end angle=-10, radius=.414];
				\draw[-stealth] (7)--(45:0.93);
				
				\fill[white] (A) circle (2pt);
				\draw (A) circle (2pt);
				\fill (a) circle (2pt);
		}}
	\end{minipage}
	\begin{minipage}[t]{0.15\textwidth}
		\centering
		\scalemath{0.8}{
			\Diskg{
				\coordinate (A) at (157.5:0.7);
				
				\coordinate (a) at (-112.5:0.7);
				
				\draw (5)--(a);		\draw (6)--(a);	
				\draw (7)--(A);		\draw (8)--(A);		
				\draw (a)--(A);	
				
				\draw[-stealth] (3) arc[start angle=-90, end angle=-215, radius=.414];
				\draw[-stealth] (4)--(90:0.93);
				
				\fill[white] (A) circle (2pt);
				\draw (A) circle (2pt);
				\fill (a) circle (2pt);
		}}
	\end{minipage}
	\\\vspace{0.2cm}
	\begin{minipage}[t]{0.15\textwidth}
		\centering
		\scalemath{0.8}{
			\Diskg{
				\coordinate (A) at (157.5:0.7);
				
				\coordinate (a) at (67.5:0.7);
				
				\draw (1)--(a);		\draw (2)--(a);	
				\draw (7)--(A);		\draw (8)--(A);		
				\draw (a)--(A);	
				
				\draw[-stealth] (4) arc[start angle=45, end angle=170, radius=.414];
				\draw[-stealth] (3)--(-135:0.93);
				
				\fill[white] (A) circle (2pt);
				\draw (A) circle (2pt);
				\fill (a) circle (2pt);
		}}
	\end{minipage}
	\begin{minipage}[t]{0.15\textwidth}
		\centering
		\scalemath{0.8}{
			\Diskg{
				\coordinate (A) at (-22.5:0.7);
				
				\coordinate (a) at (67.5:0.7);
				
				\draw (1)--(a);		\draw (2)--(a);	
				\draw (3)--(A);		\draw (4)--(A);		
				\draw (a)--(A);	
				
				\draw[-stealth] (7) arc[start angle=90, end angle=-35, radius=.414];
				\draw[-stealth] (8)--(-90:0.93);
				
				\fill[white] (A) circle (2pt);
				\draw (A) circle (2pt);
				\fill (a) circle (2pt);
		}}
	\end{minipage}
	\begin{minipage}[t]{0.15\textwidth}
		\centering
		\scalemath{0.8}{
			\Diskg{
				\draw[-stealth] (8) arc[start angle=-135, end angle=-10, radius=.414];
				\draw[-stealth] (3) arc[start angle=-90, end angle=-215, radius=.414];
				\draw[-stealth] (4) arc[start angle=45, end angle=170, radius=.414];
				\draw[-stealth] (7) arc[start angle=90, end angle=-45, radius=.414];
		}}
	\end{minipage}
	\begin{minipage}[t]{0.15\textwidth}
		\centering
		\scalemath{0.8}{
			\Diskg{
				\draw[-stealth] (3) arc[start angle=-90, end angle=-215, radius=.414];
				\draw[-stealth] (7) arc[start angle=90, end angle=-45, radius=.414];
				\draw[-stealth] (4)--(90:0.93);
				\draw[-stealth] (8)--(-90:0.93);
		}}
	\end{minipage}
	\begin{minipage}[t]{0.15\textwidth}
		\centering
		\scalemath{0.8}{
			\Diskg{
				\draw[-stealth] (8) arc[start angle=-135, end angle=-10, radius=.414];
				\draw[-stealth] (4) arc[start angle=45, end angle=170, radius=.414];
				\draw[-stealth] (3)--(-135:0.93);
				\draw[-stealth] (7)--(45:0.93);
		}}
	\end{minipage}
	\caption{Non-elliptic webs with 4 black and 4 white boundary vertices of type 7.}
	\label{fig:type7}
\end{figure}

\begin{figure}[H]
	\centering	
	\begin{minipage}[t]{0.15\textwidth}
		\centering
		\scalemath{0.8}{
			\Diskh{
				\coordinate (A) at (45:0.6);
				\coordinate (B) at (-45:0.6);
				\coordinate (C) at (-135:0.6);
				\coordinate (D) at (135:0.6);
				
				\coordinate (a) at (0:0.6);
				\coordinate (b) at (-90:0.6);
				\coordinate (c) at (180:0.6);
				\coordinate (d) at (90:0.6);
				
				\draw (1)--(d);		\draw (3)--(a);		\draw (5)--(b);		\draw (7)--(c);
				\draw (2)--(A);		\draw (4)--(B);		\draw (6)--(C);		\draw (8)--(D);		
				
				\draw (a)--(A);		\draw (a)--(B);		\draw (b)--(B);		\draw (b)--(C);
				\draw (c)--(C);		\draw (c)--(D);		\draw (d)--(A);		\draw (d)--(D);
				
				\foreach \p in {A,B,C,D} {
					\fill[white] (\p) circle (2pt);
					\draw (\p) circle (2pt);
				}
				\foreach \p in {a,b,c,d}{
					\fill (\p) circle (2pt);
				}
		}}
	\end{minipage}
	\begin{minipage}[t]{0.15\textwidth}
		\centering
		\scalemath{0.8}{
			\Diskh{
				\coordinate (A) at (-22.5:0.4);
				\coordinate (B) at (-142.5:0.4);
				\coordinate (C) at (97.5:0.4);
				
				\coordinate (a) at (37.5:0.4);
				\coordinate (b) at (-82.5:0.4);
				\coordinate (c) at (157.5:0.4);
				
				\draw (3)--(a);		\draw (5)--(b);		\draw (7)--(c);
				\draw (4)--(A);		\draw (6)--(B);		\draw (8)--(C);		
				
				\draw (a)--(A);		\draw (a)--(C);		\draw (b)--(A);		\draw (b)--(B);
				\draw (c)--(B);		\draw (c)--(C);		
				
				\draw[-stealth] (2) arc[start angle=-45, end angle=-170, radius=.414];
				
				\foreach \p in {A,B,C} {
					\fill[white] (\p) circle (2pt);
					\draw (\p) circle (2pt);
				}
				\foreach \p in {a,b,c}{
					\fill (\p) circle (2pt);
				}
		}}
	\end{minipage}
	\begin{minipage}[t]{0.15\textwidth}
		\centering
		\scalemath{0.8}{
			\Diskh{
				\coordinate (A) at (82.5:0.4);
				\coordinate (B) at (-37.5:0.4);
				\coordinate (C) at (-157.5:0.4);
				
				\coordinate (a) at (22.5:0.4);
				\coordinate (b) at (-97.5:0.4);
				\coordinate (c) at (142.5:0.4);
				
				\draw (3)--(a);		\draw (5)--(b);		\draw (7)--(c);
				\draw (2)--(A);		\draw (4)--(B);		\draw (6)--(C);		
				
				\draw (a)--(A);		\draw (a)--(B);		\draw (b)--(B);		\draw (b)--(C);
				\draw (c)--(A);		\draw (c)--(C);		
				
				\draw[-stealth] (8) arc[start angle=-135, end angle=-10, radius=.414];
				
				\foreach \p in {A,B,C} {
					\fill[white] (\p) circle (2pt);
					\draw (\p) circle (2pt);
				}
				\foreach \p in {a,b,c}{
					\fill (\p) circle (2pt);
				}
		}}
	\end{minipage}
	\begin{minipage}[t]{0.15\textwidth}
		\centering
		\scalemath{0.8}{
			\Diskh{
				\coordinate (A) at (7.5:0.4);
				\coordinate (B) at (-112.5:0.4);
				\coordinate (C) at (127.5:0.4);
				
				\coordinate (a) at (67.5:0.4);
				\coordinate (b) at (-52.5:0.4);
				\coordinate (c) at (-172.5:0.4);
				
				\draw (1)--(a);		\draw (5)--(b);		\draw (7)--(c);
				\draw (2)--(A);		\draw (6)--(B);		\draw (8)--(C);		
				
				\draw (a)--(A);		\draw (a)--(C);		\draw (b)--(A);		\draw (b)--(B);
				\draw (c)--(B);		\draw (c)--(C);		
				
				\draw[-stealth] (4) arc[start angle=-135, end angle=-260, radius=.414];
				
				\foreach \p in {A,B,C} {
					\fill[white] (\p) circle (2pt);
					\draw (\p) circle (2pt);
				}
				\foreach \p in {a,b,c}{
					\fill (\p) circle (2pt);
				}
		}}
	\end{minipage}
	\begin{minipage}[t]{0.15\textwidth}
		\centering
		\scalemath{0.8}{
			\Diskh{
				\coordinate (A) at (-7.5:0.4);
				\coordinate (B) at (-127.5:0.4);
				\coordinate (C) at (112.5:0.4);
				
				\coordinate (a) at (52.5:0.4);
				\coordinate (b) at (-67.5:0.4);
				\coordinate (c) at (172.5:0.4);
				
				\draw (1)--(a);		\draw (5)--(b);		\draw (7)--(c);
				\draw (4)--(A);		\draw (6)--(B);		\draw (8)--(C);		
				
				\draw (a)--(A);		\draw (a)--(C);		\draw (b)--(A);		\draw (b)--(B);
				\draw (c)--(B);		\draw (c)--(C);		
				
				\draw[-stealth] (2) arc[start angle=135, end angle=260, radius=.414];
				
				\foreach \p in {A,B,C} {
					\fill[white] (\p) circle (2pt);
					\draw (\p) circle (2pt);
				}
				\foreach \p in {a,b,c}{
					\fill (\p) circle (2pt);
				}
		}}
	\end{minipage}
	\begin{minipage}[t]{0.15\textwidth}
		\centering
		\scalemath{0.8}{
			\Diskh{
				\coordinate (A) at (37.5:0.4);
				\coordinate (B) at (-82.5:0.4);
				\coordinate (C) at (157.5:0.4);
				
				\coordinate (a) at (-22.5:0.4);
				\coordinate (b) at (-142.5:0.4);
				\coordinate (c) at (97.5:0.4);
				
				\draw (1)--(c);		\draw (3)--(a);		\draw (7)--(b);
				\draw (2)--(A);		\draw (4)--(B);		\draw (8)--(C);		
				
				\draw (a)--(A);		\draw (a)--(B);		\draw (b)--(B);		\draw (b)--(C);
				\draw (c)--(A);		\draw (c)--(C);		
				
				\draw[-stealth] (6) arc[start angle=135, end angle=10, radius=.414];
				
				\foreach \p in {A,B,C} {
					\fill[white] (\p) circle (2pt);
					\draw (\p) circle (2pt);
				}
				\foreach \p in {a,b,c}{
					\fill (\p) circle (2pt);
				}
		}}
	\end{minipage}
	\\\vspace{0.2cm}
	\begin{minipage}[t]{0.15\textwidth}
		\centering
		\scalemath{0.8}{
			\Diskh{
				\coordinate (A) at (22.5:0.4);
				\coordinate (B) at (-97.5:0.4);
				\coordinate (C) at (142.5:0.4);
				
				\coordinate (a) at (82.5:0.4);
				\coordinate (b) at (-37.5:0.4);
				\coordinate (c) at (-157.5:0.4);
				
				\draw (1)--(a);		\draw (3)--(b);		\draw (7)--(c);
				\draw (2)--(A);		\draw (6)--(B);		\draw (8)--(C);		
				
				\draw (a)--(A);		\draw (a)--(C);		\draw (b)--(A);		\draw (b)--(B);
				\draw (c)--(B);		\draw (c)--(C);		
				
				\draw[-stealth] (4) arc[start angle=45, end angle=170, radius=.414];
				
				\foreach \p in {A,B,C} {
					\fill[white] (\p) circle (2pt);
					\draw (\p) circle (2pt);
				}
				\foreach \p in {a,b,c}{
					\fill (\p) circle (2pt);
				}
		}}
	\end{minipage}
	\begin{minipage}[t]{0.15\textwidth}
		\centering
		\scalemath{0.8}{
			\Diskh{
				\coordinate (A) at (67.5:0.4);
				\coordinate (B) at (-52.5:0.4);
				\coordinate (C) at (-172.5:0.4);
				
				\coordinate (a) at (7.5:0.4);
				\coordinate (b) at (-112.5:0.4);
				\coordinate (c) at (127.5:0.4);
				
				\draw (1)--(c);		\draw (3)--(a);		\draw (5)--(b);
				\draw (2)--(A);		\draw (4)--(B);		\draw (6)--(C);		
				
				\draw (a)--(A);		\draw (a)--(B);		\draw (b)--(B);		\draw (b)--(C);
				\draw (c)--(A);		\draw (c)--(C);		
				
				\draw[-stealth] (8) arc[start angle=45, end angle=-80, radius=.414];
				
				\foreach \p in {A,B,C} {
					\fill[white] (\p) circle (2pt);
					\draw (\p) circle (2pt);
				}
				\foreach \p in {a,b,c}{
					\fill (\p) circle (2pt);
				}
		}}
	\end{minipage}
	\begin{minipage}[t]{0.15\textwidth}
		\centering
		\scalemath{0.8}{
			\Diskh{
				\coordinate (A) at (52.5:0.4);
				\coordinate (B) at (-67.5:0.4);
				\coordinate (C) at (172.5:0.4);
				
				\coordinate (a) at (-7.5:0.4);
				\coordinate (b) at (-127.5:0.4);
				\coordinate (c) at (112.5:0.4);
				
				\draw (1)--(c);		\draw (3)--(a);		\draw (5)--(b);
				\draw (2)--(A);		\draw (4)--(B);		\draw (8)--(C);		
				
				\draw (a)--(A);		\draw (a)--(B);		\draw (b)--(B);		\draw (b)--(C);
				\draw (c)--(A);		\draw (c)--(C);		
				
				\draw[-stealth] (6) arc[start angle=-45, end angle=80, radius=.414];
				
				\foreach \p in {A,B,C} {
					\fill[white] (\p) circle (2pt);
					\draw (\p) circle (2pt);
				}
				\foreach \p in {a,b,c}{
					\fill (\p) circle (2pt);
				}
		}}
	\end{minipage}
	\begin{minipage}[t]{0.15\textwidth}
		\centering
		\scalemath{0.8}{
			\Diskh{
				\draw[-stealth] (8) arc[start angle=-135, end angle=-10, radius=.414];
				\draw[-stealth] (2) arc[start angle=135, end angle=260, radius=.414];
				\draw[-stealth] (4) arc[start angle=45, end angle=170, radius=.414];
				\draw[-stealth] (6) arc[start angle=-45, end angle=80, radius=.414];
		}}
	\end{minipage}
	\begin{minipage}[t]{0.15\textwidth}
		\centering
		\scalemath{0.8}{
			\Diskh{
				\draw[-stealth] (8) arc[start angle=45, end angle=-80, radius=.414];
				\draw[-stealth] (2) arc[start angle=-45, end angle=-170, radius=.414];
				\draw[-stealth] (4) arc[start angle=-135, end angle=-260, radius=.414];
				\draw[-stealth] (6) arc[start angle=135, end angle=10, radius=.414];
		}}
	\end{minipage}
	\begin{minipage}[t]{0.15\textwidth}
		\centering
		\scalemath{0.8}{
			\Diskh{
				\draw[-stealth] (8) arc[start angle=-135, end angle=-10, radius=.414];
				\draw[-stealth] (4) arc[start angle=-135, end angle=-260, radius=.414];
				\draw[-stealth] (6) arc[start angle=135, end angle=10, radius=.414];
				\draw[-stealth] (2)--(180:0.93);
		}}
	\end{minipage}
	\\\vspace{0.2cm}
	\begin{minipage}[t]{0.15\textwidth}
		\centering
		\scalemath{0.8}{
			\Diskh{
				\draw[-stealth] (2) arc[start angle=-45, end angle=-170, radius=.414];
				\draw[-stealth] (4) arc[start angle=45, end angle=170, radius=.414];
				\draw[-stealth] (6) arc[start angle=-45, end angle=80, radius=.414];
				\draw[-stealth] (8)--(0:0.93);
		}}
	\end{minipage}
	\begin{minipage}[t]{0.15\textwidth}
		\centering
		\scalemath{0.8}{
			\Diskh{
				\draw[-stealth] (8) arc[start angle=45, end angle=-80, radius=.414];
				\draw[-stealth] (2) arc[start angle=135, end angle=260, radius=.414];
				\draw[-stealth] (6) arc[start angle=135, end angle=10, radius=.414];
				\draw[-stealth] (4)--(90:0.93);
		}}
	\end{minipage}
	\begin{minipage}[t]{0.15\textwidth}
		\centering
		\scalemath{0.8}{
			\Diskh{
				\draw[-stealth] (8) arc[start angle=-135, end angle=-10, radius=.414];
				\draw[-stealth] (6) arc[start angle=-45, end angle=80, radius=.414];
				\draw[-stealth] (4) arc[start angle=-135, end angle=-260, radius=.414];
				\draw[-stealth] (2)--(-90:0.93);
		}}
	\end{minipage}
	\begin{minipage}[t]{0.15\textwidth}
		\centering
		\scalemath{0.8}{
			\Diskh{
				\draw[-stealth] (8) arc[start angle=45, end angle=-80, radius=.414];
				\draw[-stealth] (2) arc[start angle=-45, end angle=-170, radius=.414];
				\draw[-stealth] (4) arc[start angle=45, end angle=170, radius=.414];
				\draw[-stealth] (6)--(0:0.93);
		}}
	\end{minipage}
	\begin{minipage}[t]{0.15\textwidth}
		\centering
		\scalemath{0.8}{
			\Diskh{
				\draw[-stealth] (8) arc[start angle=-135, end angle=-10, radius=.414];
				\draw[-stealth] (2) arc[start angle=135, end angle=260, radius=.414];
				\draw[-stealth] (6) arc[start angle=135, end angle=10, radius=.414];
				\draw[-stealth] (4)--(180:0.93);
		}}
	\end{minipage}
	\begin{minipage}[t]{0.15\textwidth}
		\centering
		\scalemath{0.8}{
			\Diskh{
				\draw[-stealth] (2) arc[start angle=-45, end angle=-170, radius=.414];
				\draw[-stealth] (4) arc[start angle=-135, end angle=-260, radius=.414];
				\draw[-stealth] (6) arc[start angle=-45, end angle=80, radius=.414];
				\draw[-stealth] (8)--(-90:0.93);
		}}
	\end{minipage}
	\\\vspace{0.2cm}
	\begin{minipage}[t]{0.15\textwidth}
		\centering
		\scalemath{0.8}{
			\Diskh{
				\draw[-stealth] (2) arc[start angle=135, end angle=260, radius=.414];
				\draw[-stealth] (4) arc[start angle=45, end angle=170, radius=.414];
				\draw[-stealth] (8) arc[start angle=45, end angle=-80, radius=.414];
				\draw[-stealth] (6)--(90:0.93);
		}}
	\end{minipage}
	\begin{minipage}[t]{0.15\textwidth}
		\centering
		\scalemath{0.8}{
			\Diskh{
				\draw[-stealth] (8) arc[start angle=-135, end angle=-10, radius=.414];
				\draw[-stealth] (4) arc[start angle=45, end angle=170, radius=.414];
				\draw[-stealth] (2)--(180:0.93);
				\draw[-stealth] (6)--(0:0.93);
		}}
	\end{minipage}
	\begin{minipage}[t]{0.15\textwidth}
		\centering
		\scalemath{0.8}{
			\Diskh{
				\draw[-stealth] (2) arc[start angle=-45, end angle=-170, radius=.414];
				\draw[-stealth] (6) arc[start angle=135, end angle=10, radius=.414];
				\draw[-stealth] (4)--(180:0.93);
				\draw[-stealth] (8)--(0:0.93);
		}}
	\end{minipage}
	\begin{minipage}[t]{0.15\textwidth}
		\centering
		\scalemath{0.8}{
			\Diskh{
				\draw[-stealth] (2) arc[start angle=135, end angle=260, radius=.414];
				\draw[-stealth] (6) arc[start angle=-45, end angle=80, radius=.414];
				\draw[-stealth] (4)--(90:0.93);
				\draw[-stealth] (8)--(-90:0.93);
		}}
	\end{minipage}
	\begin{minipage}[t]{0.15\textwidth}
		\centering
		\scalemath{0.8}{
			\Diskh{
				\draw[-stealth] (4) arc[start angle=-135, end angle=-260, radius=.414];
				\draw[-stealth] (8) arc[start angle=45, end angle=-80, radius=.414];
				\draw[-stealth] (6)--(90:0.93);
				\draw[-stealth] (2)--(-90:0.93);
		}}
	\end{minipage}
	\caption{Non-elliptic webs with 4 black and 4 white boundary vertices of type 8.}
	\label{fig:type8}
\end{figure}

\end{appendices}


\bibliography{Grassmannian}

\end{document}